\documentclass[a4paper,11pt]{scrartcl}

\usepackage[utf8]{inputenc}
\usepackage[T1]{fontenc}
\usepackage{lmodern}
\usepackage{amsmath}
\usepackage{amssymb}
\usepackage{amsthm}
\usepackage{a4wide}
\usepackage{tikz}\usetikzlibrary{arrows}
\graphicspath{{img/}}
\usepackage{gensymb}
\usepackage[super]{nth}
\usepackage[english]{babel}
\usepackage{csquotes}

\usepackage{bookmark}
\usepackage[style=alphabetic, maxcitenames=4, mincitenames=4,
maxbibnames=99, minbibnames=99 ]{biblatex}
\makeatletter
\AtEveryBibitem{\def\@currentlabel{\thefield{labelnumber}}\label{labelnumber-\thefield{entrykey}}}
\makeatother

\usepackage{hyperref}
\usepackage[inline,final]{showlabels}

\usepackage{cleveref}

\newtheorem{theorem}{Theorem}
\newtheorem*{theorem*}{Theorem}
\newtheorem{proposition}[theorem]{Proposition}
\newtheorem*{proposition*}{Proposition}
\newtheorem{corollary}[theorem]{Corollary}
\newtheorem*{corollary*}{Corollary}
\newtheorem{lemma}[theorem]{Lemma}
\newtheorem*{lemma*}{Lemma}

\theoremstyle{definition}

\newtheorem{definition}[theorem]{Definition}
\newtheorem*{definition*}{Definition}

\newtheorem*{question*}{Question}
\newtheorem{remark}[theorem]{Remark}
\newtheorem*{remark*}{Remark}

\renewcommand{\cal}[1]{\mathcal{#1}}

\newcommand{\C}{\mathbb{C}}
\newcommand{\Z}{\mathbb{Z}}

\newcommand{\on}[1]{\operatorname{#1}}

\newcommand{\term}[1]{\textsf{#1}}

\newcommand{\iV}{\mathrm{V}}
\newcommand{\iI}{\mathrm{I}}
\newcommand{\iA}{\mathrm{A}}
\newcommand{\iB}{\mathrm{B}}
\newcommand{\iC}{\mathrm{C}}

\newcommand{\hT}{\mathrm{T}}

\newcommand{\iD}{\mathrm{D}}

\usepackage[margin=10pt,font=small,labelfont=bf,labelsep=endash]{caption}

\newcommand{\image}[3]{
\bigskip
\noindent
\begin{minipage}{\linewidth}
\centering%
{#3}%
\\
\captionsetup{type=figure}\captionof{figure}{#1}%
\ifx&#2&%
\else
\label{#2}
\fi
\end{minipage}
\par\medskip
}

\newcommand{\latin}[1]{\emph{#1}}

\title{Observations on the hex clusters of the Spectre tilings}
\date{}
\author{Arnaud Chéritat\thanks{CNRS researcher, Institut de Mathématiques de Toulouse, UMR 5219, Université de Toulouse}}

\begin{document}

\maketitle

\abstract{Decorating the Spectre tile with hexagons reveals triangular hexagonal clusters whose structure we study. In the process we reprove that the Spectre tilings exist and are uniquely hierarchical. The proof is not computer-assisted.}

\tableofcontents

\section{Introduction}\label{sec:intro}

This article focuses on the second aperiodic monotile discovered by Smith \latin{et alii}, the chiral one called Tile(1,1) a.k.a.\ the Spectre\footnote{We deviate slightly from the terminology of \cite{chiral}, see \Cref{note:2}.} and announced in \cite{chiral}.

\medskip

The tile has a shape that is (bounded by) a non-convex 13-gon, which, as the authors explain, is better viewed as a 14-gon whose edges all have the same length but which has one flat angle.

\nopagebreak

\image{The Spectre. The 7 non-right or flat angles measure 120\textdegree.}{}{
\includegraphics[scale=1.5]{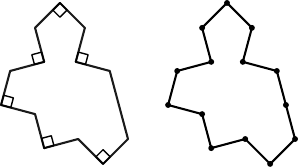}
}

It was proved in \cite{chiral} that it can tile the plane without reflections, but only non-periodically. The key ingredient is that tilings with this shape are \emph{uniquely hierarchical}. We will not reprove this classical implication but among other things we give an independent proof of the unique hierarchical nature of tilings with the Spectre.

\medskip

Here is a sample of a whole plane tiling:

\nopagebreak

\image{}{}{
\includegraphics[width=12cm]{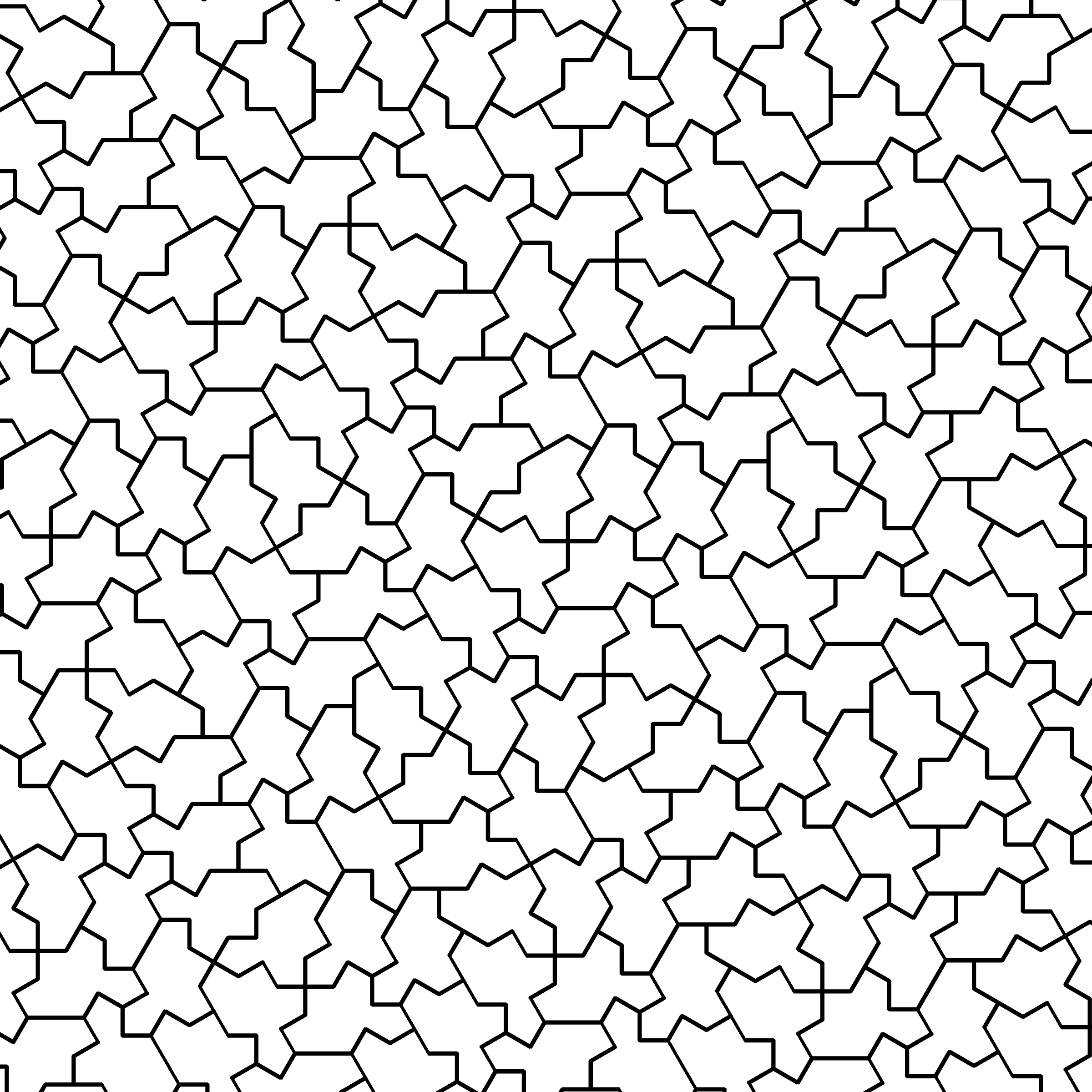}
}

It was obtained by applying the construction algorithm explained in \cite{chiral}, Figure~A1, page~21, to construct a finite patch that is ensured to be extendable into a whole plane tiling.

In such tessellations, we do not allow ourselves to reflect the tile.\footnote{In \cite{chiral} the authors explain how to modify the tile outline so as to prevent reflected tiles to assemble with unreflected tiles. The result is actually what they call a Spectre. Here, since mixing the tile with its reflection is never considered, we also call Spectre the undeformed Tile(1,1).\label{note:2}}
In \cite{chiral}, it was proved that all tiles appear in 12 possible orientations differing of a turn$/12$, with a majority of them differing only by a turn$/6$. They are called the \term{even} tiles and the other ones are called the \term{odd} tiles, which is coherent with their bearing w.r.t.\ a reference even tile being written as $k\times 30\degree$, with $k$ even or odd.
The asymptotic proportion of odd tiles is known (see \cite{chiral}) to be $1/(5+\sqrt{15})\approx 1/8.87$.

\medskip

The construction can be pushed to get a tiling as big as one wants, for instance:

\nopagebreak

\image{}{}{
\includegraphics[width=13cm]{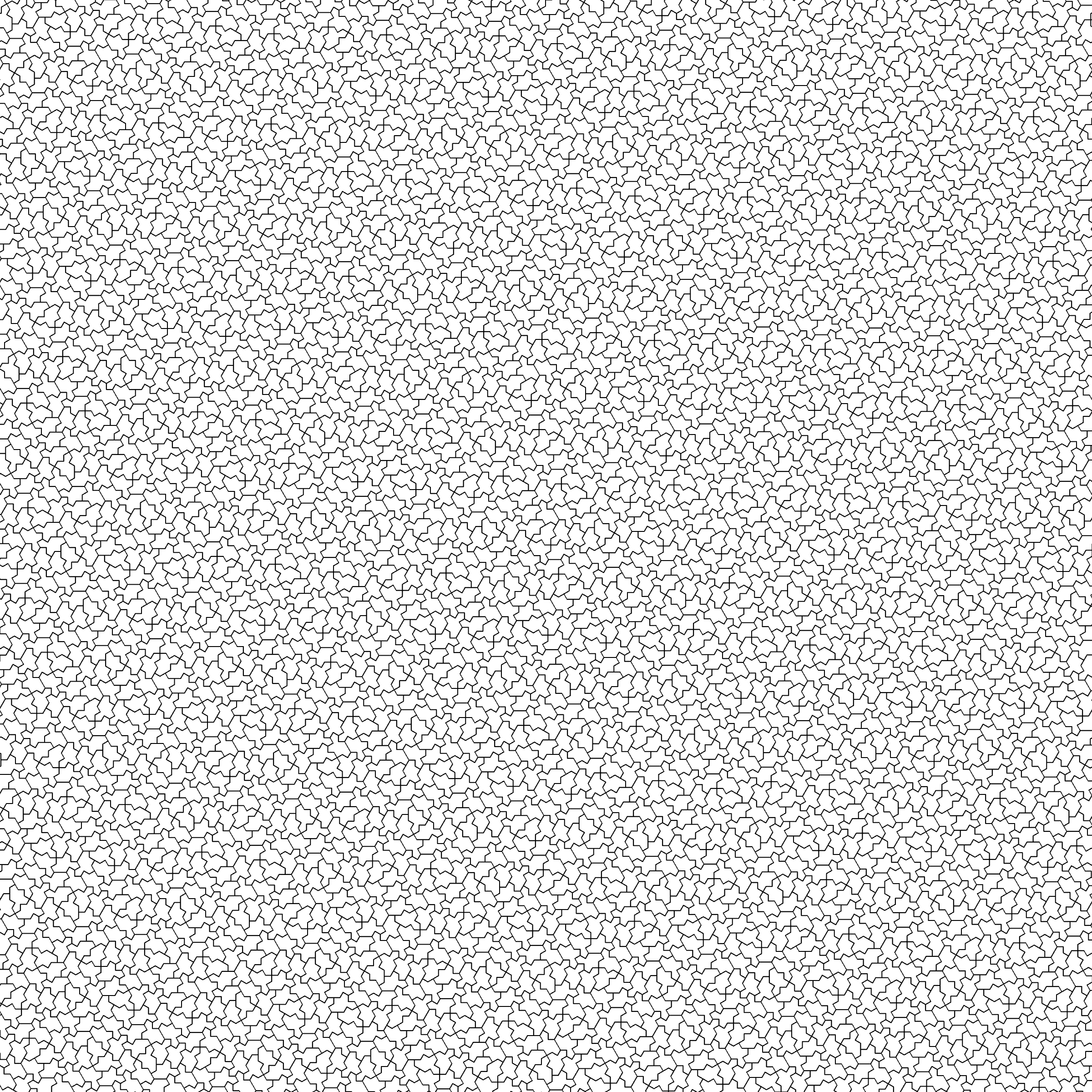}
}

\subsection{Decoration}

Without further markings, it is hard to distinguish any structure. In \cite{chiral} the authors investigate this structure to prove that the shape can tile the whole Euclidean plane, but also that no whole plane tiling can have a translation period.

\medskip

Yoshiaki Araki\footnote{Inspired by a remark Craig Kaplan made on another decoration of the tile by Dale Walton.} realized one could decorate Tile(1,1) as follows:

\nopagebreak

\image{}{}{
\includegraphics[scale =1]{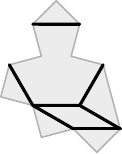}
}

so as to reveal a beautiful subjacent tiling into hexagons, squares and rhombs:

\nopagebreak

\image{}{}{
\includegraphics[width=13cm]{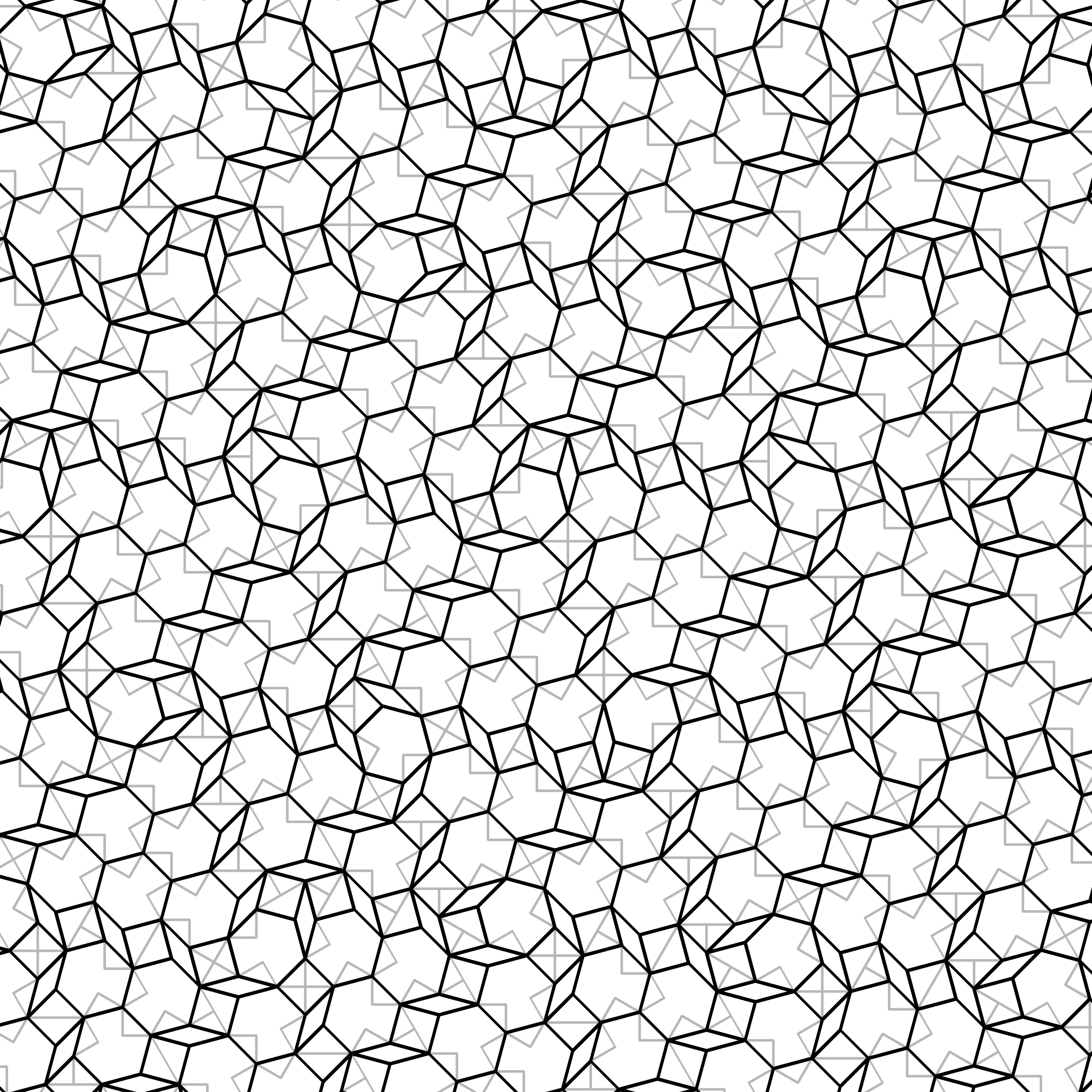}
}

%

Here is the same picture but showing the hexagons only, using yellow for odd tiles and blue for even ones:

\nopagebreak

\image{}{}{
\includegraphics[width=12cm]{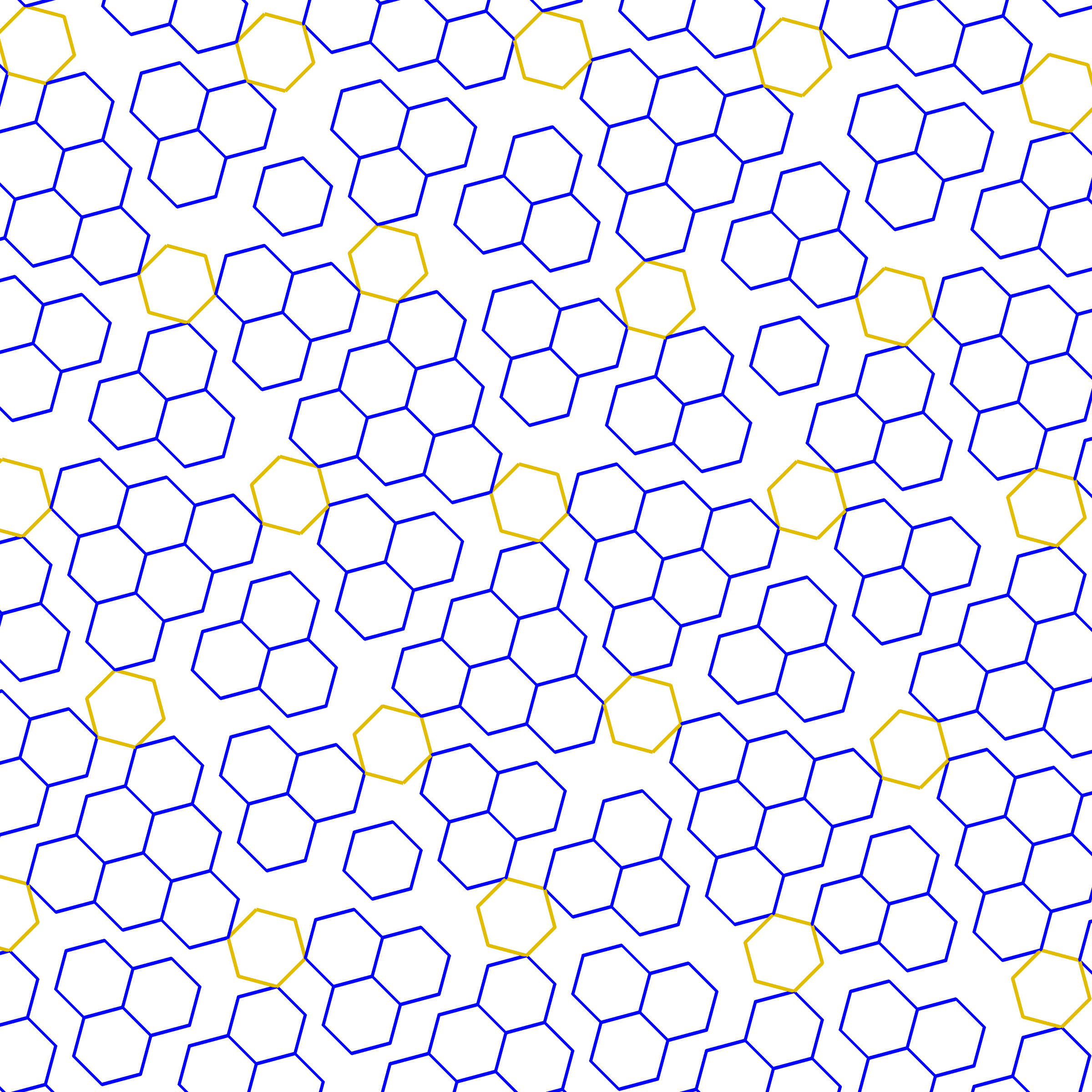}
}

In the sequel we will often abbreviate hexagon as \term{hex} and hexagons as hexes.
An interesting blue/yellow colouring of the whole hex/rhomb/squares graph is possible using the left decoration for even tiles and the right one for odd tiles:

\nopagebreak

\image{}{}{
\includegraphics[scale=1]{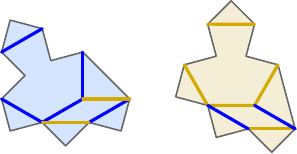}
}

Varying the shade of blue of Spectres to distinguish better their boundaries, the decorated Spectre tiling looks as follows:

\nopagebreak

\image{}{fig:superpo}{
\includegraphics[width=13cm]{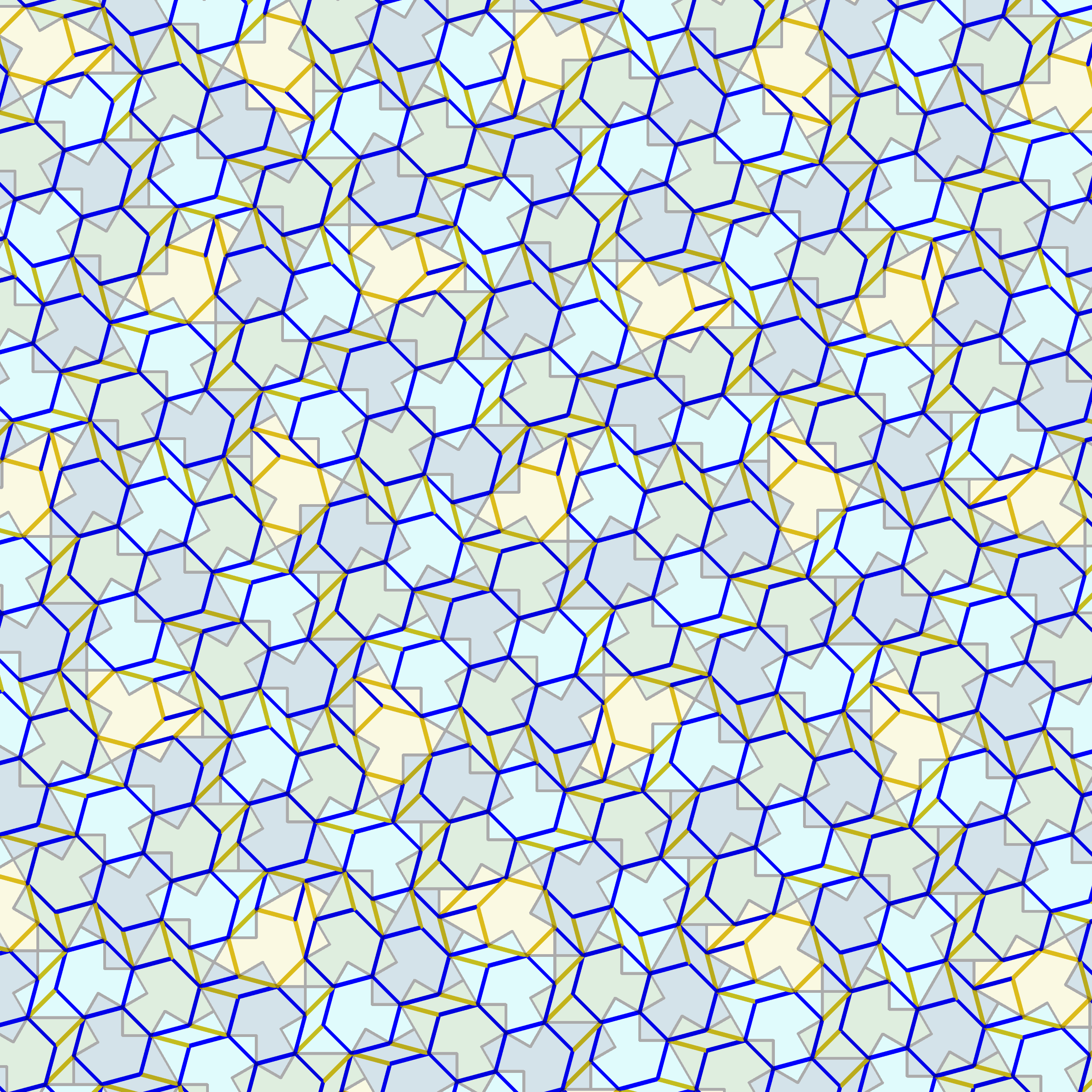}
}

%
%
%

\subsection{Compaction, observations}

Let us show the decorations only, on a different and bigger patch:

\nopagebreak

\image{}{}{
\includegraphics[width=13cm]{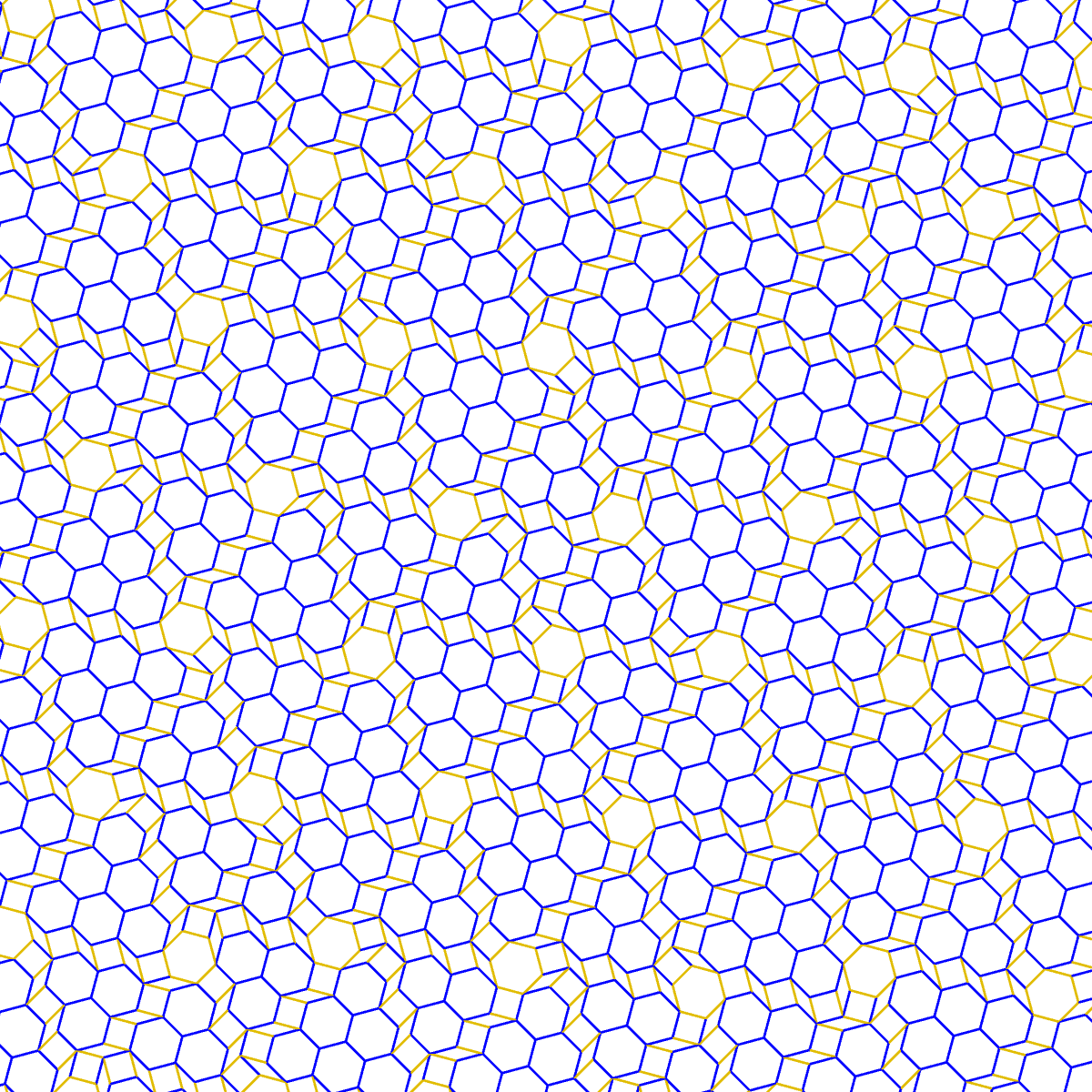}
}

We observe that the blue graph seems to have only bounded connected components made of hexagons that are isolated or arranged in triangles, possibly with protruding ``antennae'', each consisting in a segment, or a broken line made of two segments, attached to the hexagon arrangement at one end.

\begin{definition*}
We will call \term{cc} the connected components of the blue graph.
\end{definition*}

Each cc seems to be in contact with three yellow hexagons and each yellow hexagon seems to be in contact with 6 blue components, one at each of its corners.
It is as if the blue cc's correspond to the triangles in a regular triangular tessellation of the plane which would have been deformed (such a correspondence is qualified as \term{combinatorial}), and the yellow hexagons to the vertices of such a tessellation, see the upper right frame of \Cref{fig:dual}.

The blue and yellow segments all have the same length.
If we multiply the length of the yellow segments by a common factor to make them bigger we get:

\nopagebreak

\image{}{fig:yb-2}{
\includegraphics[width=13cm]{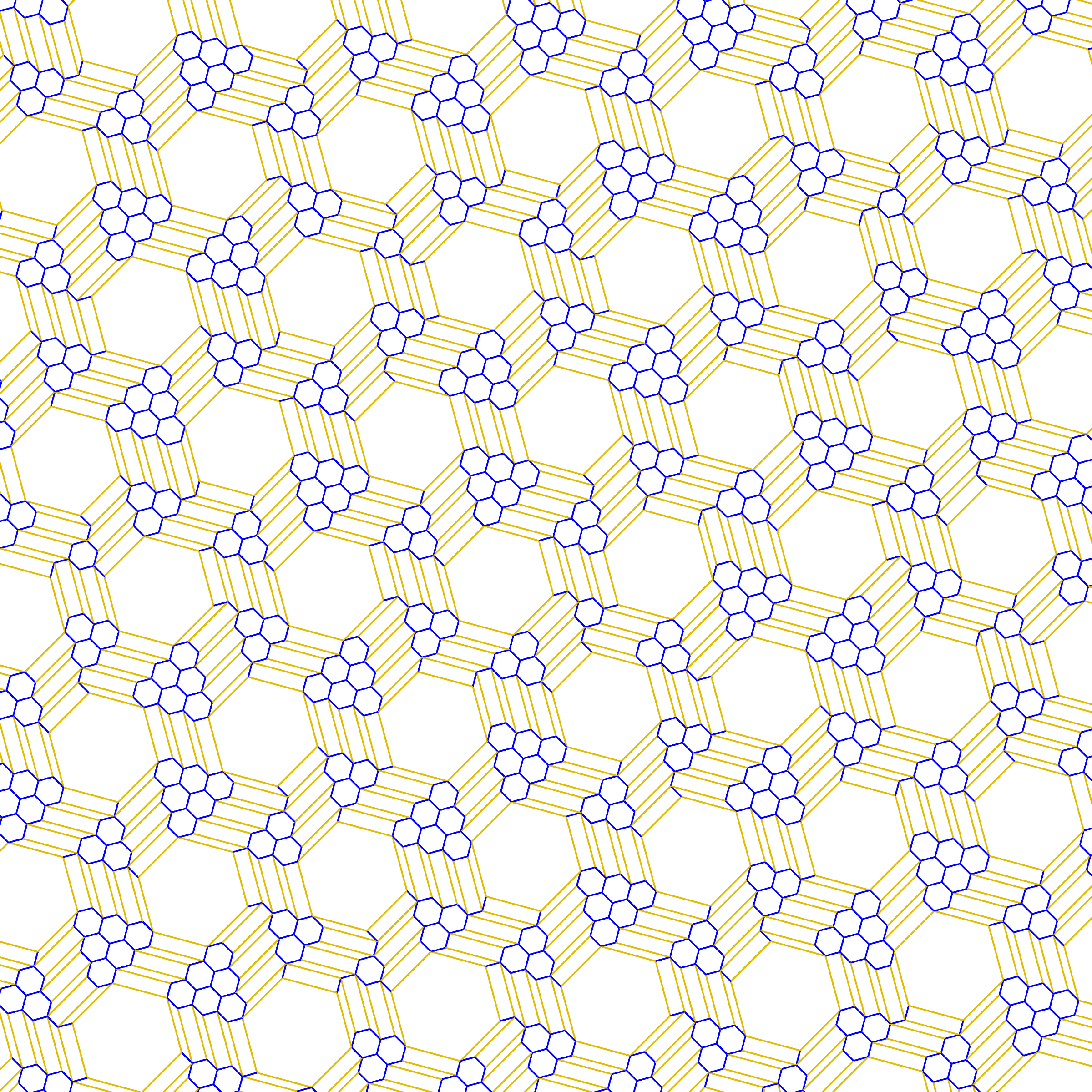}
}

In a way dual to the previous, the yellow hexagons would correspond to a regular hexagonal tessellation of the plane and the blue cc's to its vertices (see the upper right frame of \Cref{fig:dual}).

\image{Four ways to represent duality between blue triangles and yellow hexes.}{fig:dual}{
\includegraphics[height=5cm,trim={0 1.9cm 0 0.1cm},clip]{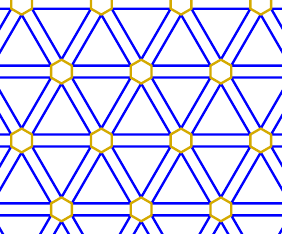}
\includegraphics[height=5cm,trim={0 1.9cm 0 0.1cm},clip]{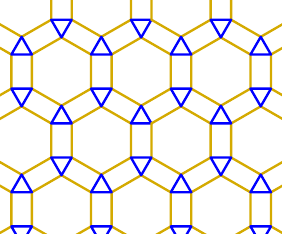}

\includegraphics[height=5cm,trim={0 1.9cm 0 0.1cm},clip]{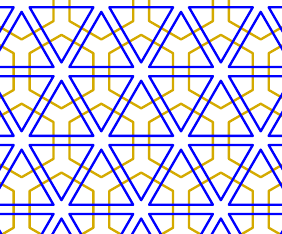}
\includegraphics[height=5cm,trim={0 1.9cm 0 0.1cm},clip]{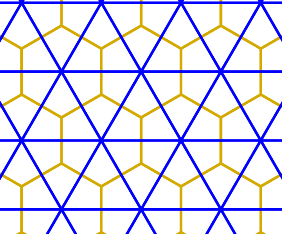}
}

In this article we will call \term{honeycomb} any regular hexagonal tessellation of the whole plane.

\medskip

When squeezing the blue segments, the yellow hexes seem to compact into a honeycomb.
If instead we squeeze the yellow segments, the blue hexes seem to compact into a honeycomb too.

\nopagebreak

\image{}{fig:yb-3}{
\includegraphics[width=13cm]{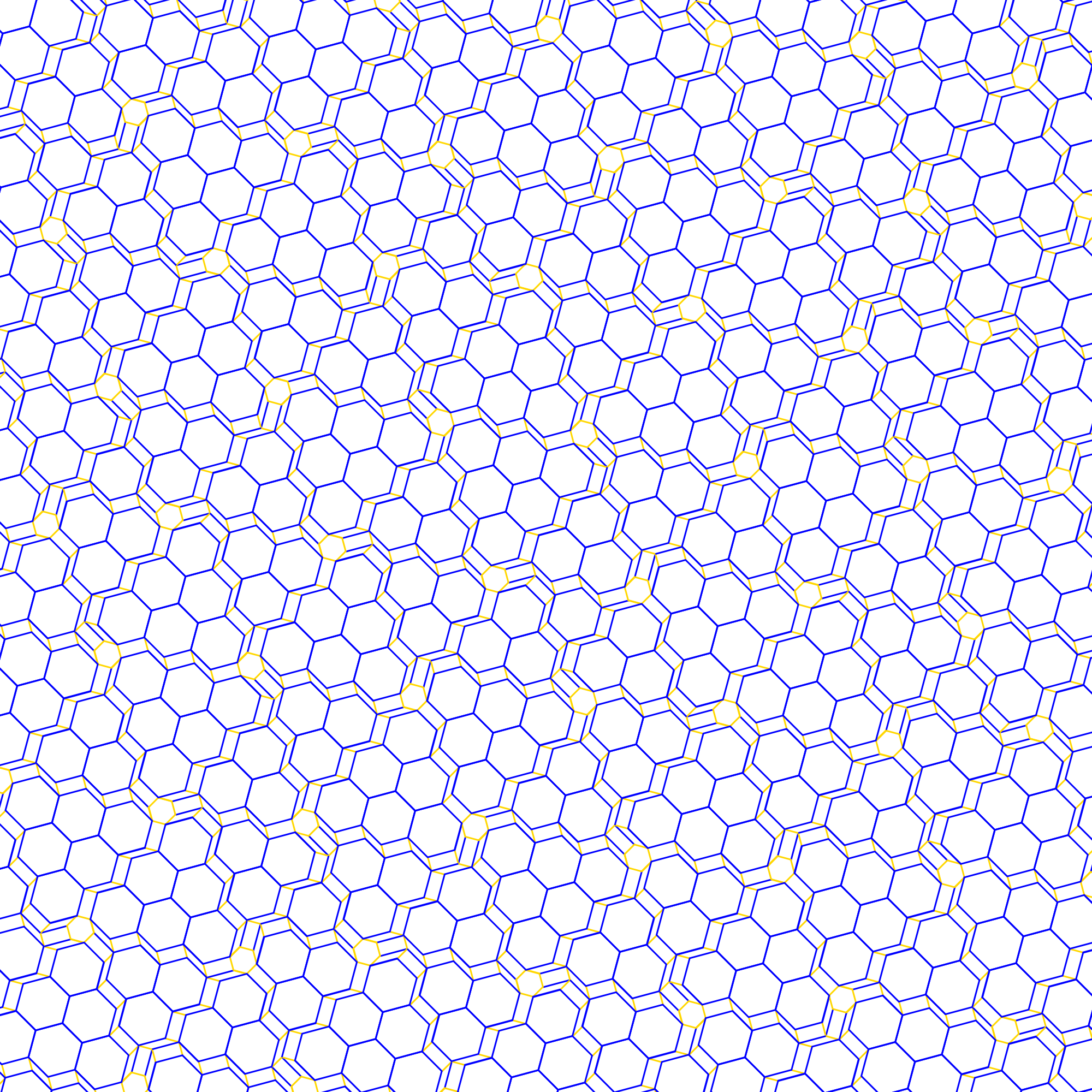}
}

Let us call \term{cluster} the maximal groups of blue hexes that touch each other. Clusters are subsets of cc's, possibly strict subsets.
We observe that each cluster comes into sets of $T_1=1$, $T_2=3$ or $T_3=6$ hexes arranged in a triangle.
We will denote them $\hT1$, $\hT2$, $\hT3$ where $\hT$ stands for triangle.
Each cc seems to contain exactly one cluster.

These observations will be proved in \Cref{prop:T}.
If they are true, clusters and cc's are in one-to-one correspondence, and also in one-to-one correspondence with the vertices (corners) of a yellow honeycomb. Let us call $HY$ the latter tessellation.
Each $\hT2$ or $\hT3$ cluster come in two possible orientations (modulo $1/3$ of a turn), and in fact this seems to follow the orientation of the triangle it corresponds to in the dual tessellation to $HY$, a fact that we will prove too.
Call these orientations \term{point up} and \term{point down}.

This hints at the fact that the $\hT1$ clusters also have a natural up/down orientation too: just attribute to it the orientation of the triangle it correspond to in the dual tesselation to $HY$.

There is a nice interpretation of these up/down orientations: actually, the Spectre tiles underlying the blue hexagons come in only 6 possible orientations, that differ by multiples of $1/6$ of a turn, these orientations come into two groups differing by multiples of $1/3$, and we will see that in every cluster all the orientations belong to a single group.
Moreover, we will see that the up pointing triangular clusters correspond to one group of orientations, and the down pointing ones to the other group.
For a cluster containing only one blue hexagon, we can use the underlying Spectre tile orientation to declare that cluster pointing up or down. This coincides with the attribution made in the previous paragraph (and this will be proved).

\medskip

Colouring the regular blue hex tiling with respect to whether each cluster points up or down gives pictures like:

\nopagebreak

\image{Condensed blue graph where each blue hexagon is materialized by a thin gray line and blue hex clusters are surrounded by a thick black line. We coloured in violet the up pointing clusters and green the down pointing ones.}{fig:bo}{
\includegraphics[width=13cm]{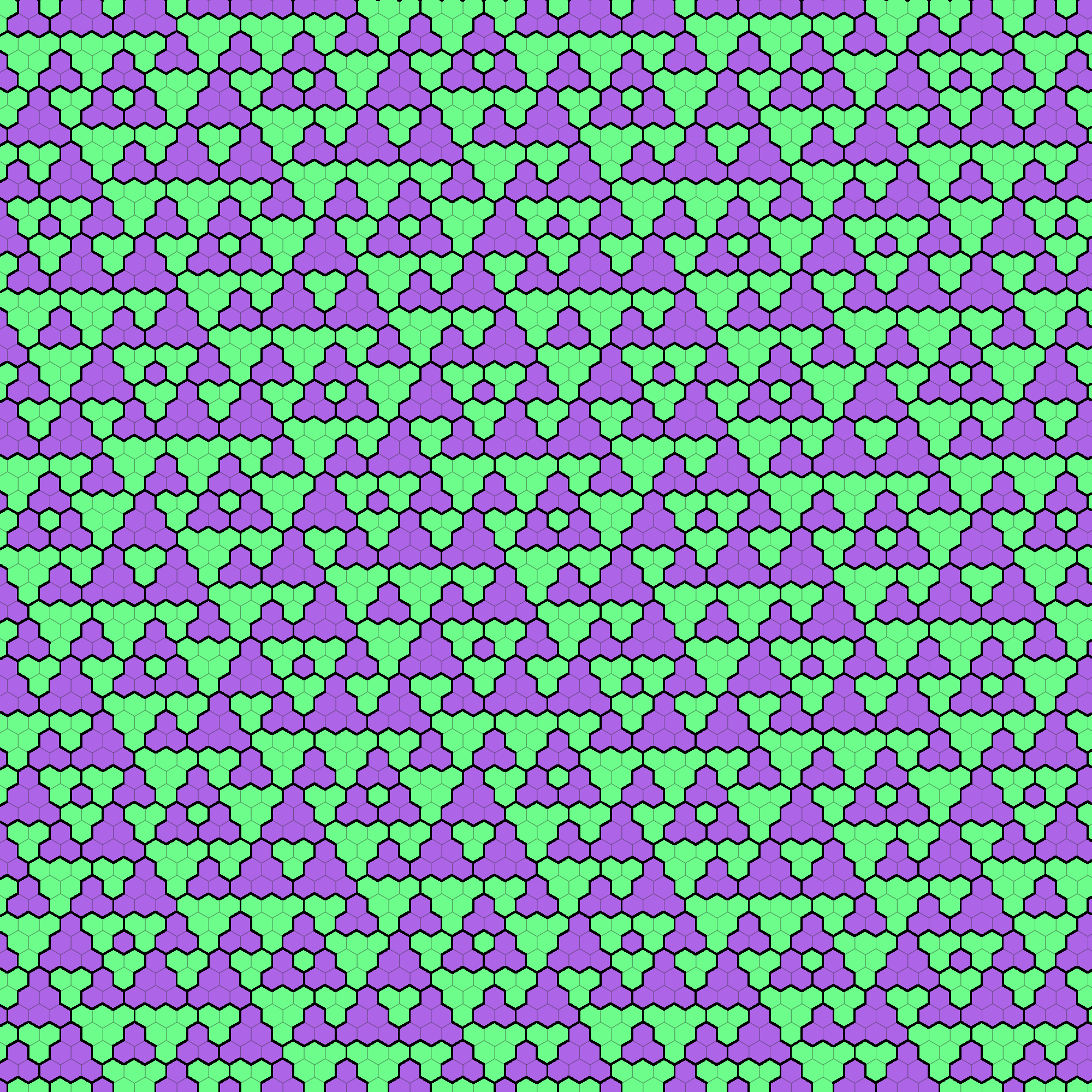}
}

\medskip

One of the main focuses of this article is to investigate the self-similar structure of this partition.

The keyword \emph{self-similar} is not that much used in the field as \emph{hierarchical structure}. Many proofs of existence of whole plane tilings, for the Spectre and Hat monotiles \cite{chiral,march} and before them of many other non-periodic tilesets, including Penrose tiles, are done via substitution systems, where one starts from a finite arrangement of shapes, then replace each shape by an arrangement of one or several pieces of the same shapes (this requires a rescaling of the original arrangement, and in some situation a deformation thereof), and the procedure is repeated to cover large pieces of the plane.
The proof of non-periodicity is usually done by reversing the procedure and proving that any whole plane tiling can be de-substituted. But this is not enough: for instance the periodic tiling by the square can be de-substituted too.
A key point for non-periodicity is \emph{uniqueness} of the way the de-substitution can be done, and tilings with this property are called \emph{uniquely hierarchical}. In general, unique hierarchy automatically provides a substitution system.

For the Spectre and the Hat, substitution systems and/or unique hierarchy have been proved in several ways, both by the original authors and by others. See for instance the works of Shigeki Akiyama, Yoshiaki Araki, Erhard Künzel and James Smith, some of which are available as \cite{AA,jS}. Unpublished work of Künzel include a substitution system in terms of tree structure. Unpublished work of Pieter Mostert uses edge substitution.

\image{This picture uses the colour scheme of \Cref{fig:bo} but burst as in \Cref{fig:yb-2}, and here the blue lines  have been turned in one direction and the yellow lines in the other direction, for better alignment. Incidentally it has some aesthetic value.}{fig:bu}{
\includegraphics[width=13cm]{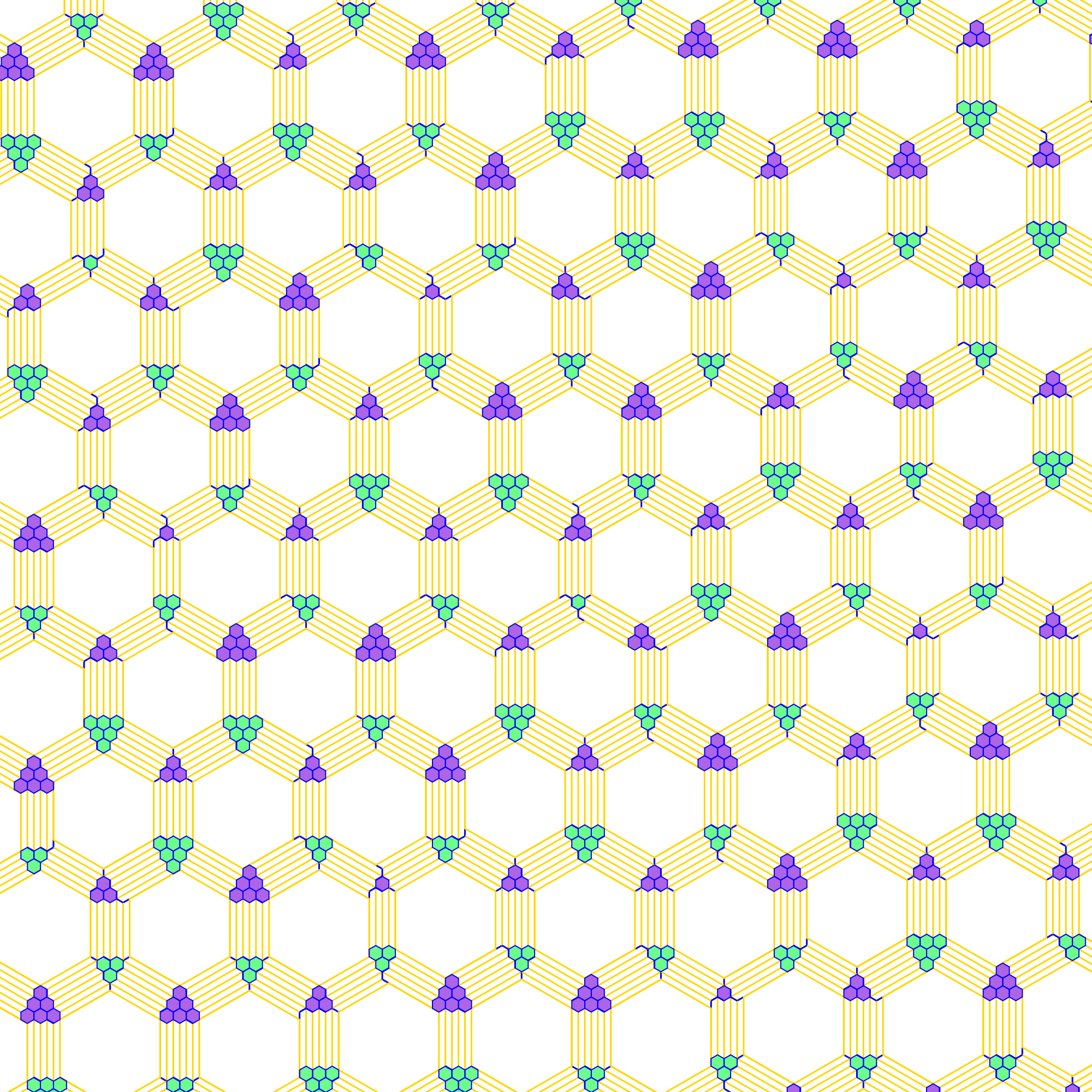}
}

\bigskip

In \cite{chiral}, the authors partitioned the tiles in any whole plane tiling by the Spectre into patches of 8 or 9 tiles, called\footnote{They actually used an equivalent tileset made of Hats and Turtles.} H8 and H9, each containing exactly one odd tile.
Up to deformation they form a honeycomb, arranged exactly like $HY$.
The blue $\hT n$ partition and the H8/H9 partition are sort of dual to each other. We investigate this duality further in \Cref{ss:oh}

%
%

\bigskip

When passing from a tiling with the Spectre to the yellow/blue graph tracing hexes, squares and rhombs, we a priori lost information.
The orientation of the Spectre associated to a hex in a $\hT n$ or to a yellow hex is only retained up to a rotation of a sixth of a turn: we only know if the hex is it is even (blue) or odd (yellow).
In other words, if we represent the orientations of each Spectre tile in a whole plane tiling with an integer modulo 12, then the hexes, squares and rhombs yellow/blue graph \emph{a priori} only tells us the class modulo 2 of each of these integers.
We saw that the up or down nature of the clusters allow to specify further the possible orientation of the blue hexes down to a set of three (we know their class modulo 4).\footnote{Observe that each single hex cluster is surrounded by the opposite colour, so one can infer the up/down nature of single hex clusters solely from the observation of the partition. Of course this observation will be proved in this article.}
But maybe the arrangement of the nearby hexes forces its value?
We will see that this is the case, provided we use the yellow hexes and retain how they are attached to nearby clusters/cc's.

Indeed, above, we compacted the blue hexes into a partition of the honeycomb in sets of type $\hT1$, $\hT2$, $\hT3$. Doing so we removed the yellow hexes from the picture and a priori lost two more pieces of information: 
which blue cc's have antennae and how the yellow hexes are attached to the blue cc's.
During the compaction, the 6 corners of a yellow hex are merged into a single vertex of the blue honeycomb.
What we can do to keep a piece of information is to locate where these vertices are: call them \term{dots}.
It turns out that their position is enough to recover the two pieces of information above: antennae and yellow hex position in the yellow/blue decorated tiling graph: see
\Cref{prop:ch2sp}.

But what if we do not put the dots?

\begin{question*}
  Can the tiling be deduced from the partition only?
\end{question*}

We will see in \Cref{sec:double} that the answer is: not always.
The process cannot be purely local.
And we may end up with whole plane partitions where there associated to more than one tiling.

Here is an extract of the condensed blue hex decomposition of a whole plane Spectre tiling, with supplementary information: clusters, dots making the condensed yellow hexes, and a line in the blue hexes indicating the orientation of the underlying spectre as on the figure below:

\nopagebreak

\image{}{}{
\includegraphics[scale=1]{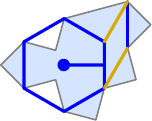}
}

\image{}{fig:partition_arrows_dots}{
\includegraphics[width=13cm]{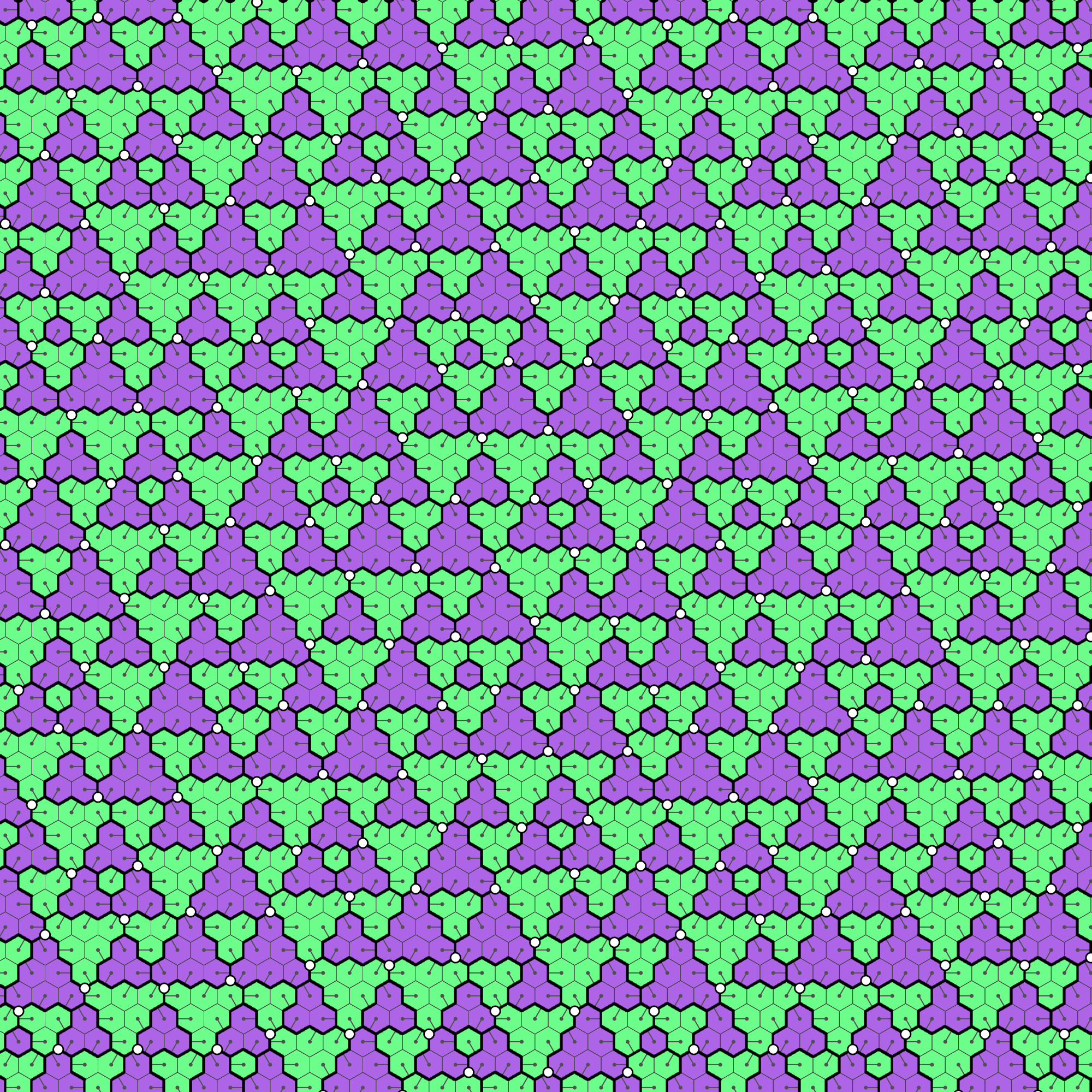}
}

The authors of \cite{chiral} proved that whole plane tilings with the Spectre are (uniquely) hierarchical.
More precisely the H8 and H9 patches are arranged in a way so that they can be themselves grouped in bigger patches that follow the same kind of arrangement, but \emph{with a mirror reflection}, so can be grouped themselves further, infinitely many times for a whole plane tiling.\footnote{There is a subtlety here: the H9 patches have to be marked as distinct according to the way they touch their neighbouring patches, yielding 8 different kinds of H9, see Section~4 of \cite{chiral}, especially Figures~4.1 and~4.2.}
The even Spectre tiles themselves can be added at the bottom of this hierarchy, but the odd tiles have to be disregarded.\footnote{The authors of \cite{chiral} decided to merge them with a specific adjacent even tile yielding a piece they called the \term{mystic} and that plays the role of an H8 patch while the other even tiles play the role of an H9.}
However, as centres of the H8 and H9 patches, the odd tiles correspond in some way to the first level of patches. 

So in some way or another, the odd tiles are arranged in a way similar to the even tiles, but at a bigger scale, and with a mirror reflection.
We saw that the even tiles arrange in triangular clusters.
It is thus to be expected that the dots in \Cref{fig:partition_arrows_dots} actually come into their own triangular clusters of 1, 3 or 6 dots.
And we can observe that if we colour them according to a bipartite colouring of the vertices of the blue hex graph, then they seem to be coming in such clusters, which moreover have a nice geometrical property: \emph{they trace the vertices of regular triangles}.
This will be a side consequence of our proof of unique hierarchy, where we will observe that the dot clusters with 3 or 6 dots necessarily has its dots at the end of the antennae of specific blue cc shapes: their cluster is of type $\hT2$ and they have three antennae, each composed of a single segment.
In the picture below, we removed some information for more visibility.

\image{We invite the reader to spot the regular triangles traced by dots of a given colour. A solution is given in the next figure.}{fig:trace}{
\includegraphics[width=13cm]{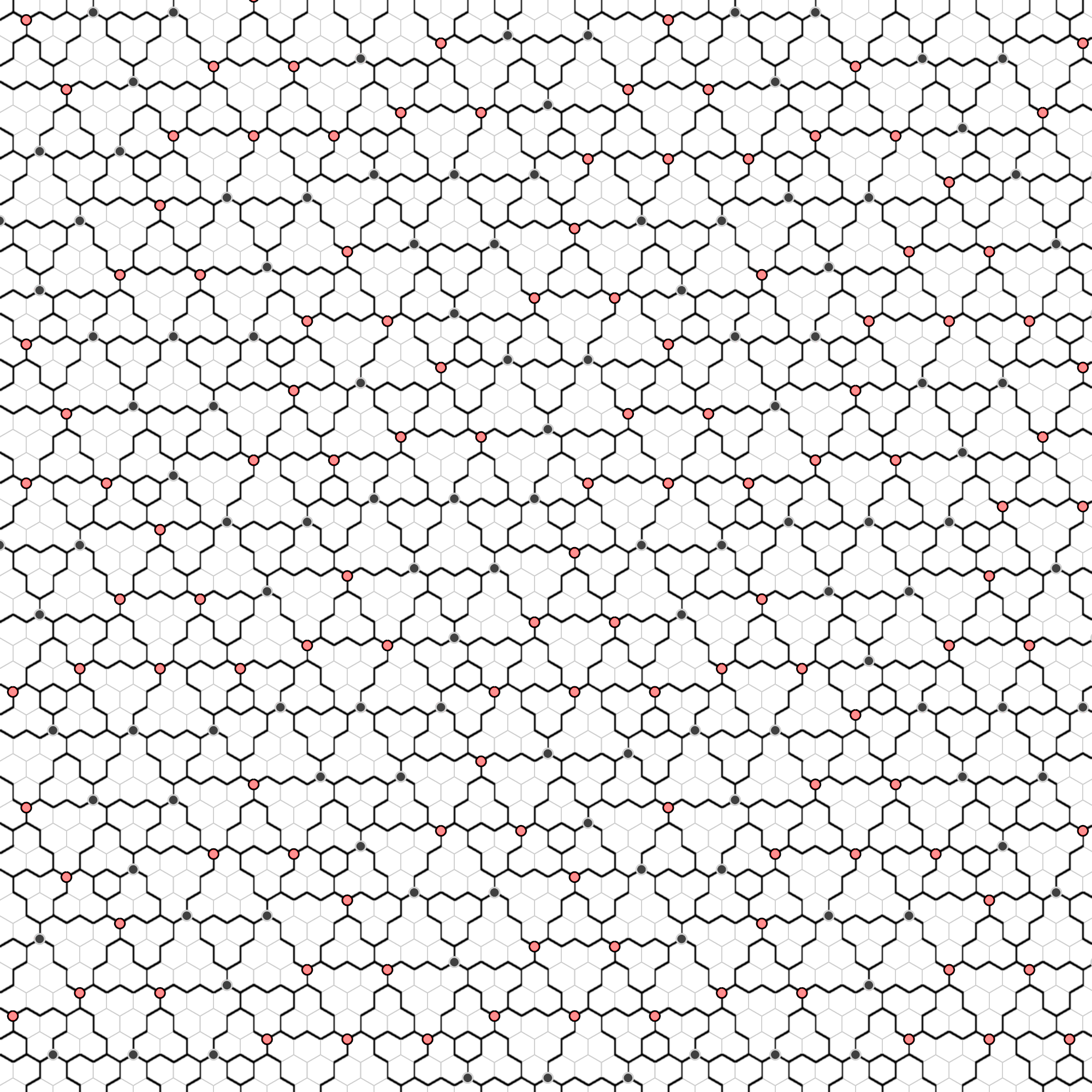}
}
%
%

If we believe that the yellow hexes cluster exactly in the same patterns as the blue hex do, but reflected,  this means the triangular clusters in which the yellow hexagons group, equivalently those in which the dots group, are \emph{combinatorially} (as opposed to \emph{geometrically}) dual to a, say green, honeycomb, i.e.\ they can be associated to the vertices of such a graph: see \Cref{fig:ov} (more on this in \Cref{ss:3-levels}).
If we choose the green honeycomb $HG$ to be in the same scale and orientation as the blue one $HB$, one can quess a formula, which we prove in \Cref{prop:B-Y-G-coord}, relating, for a given yellow hex $h$:
\begin{itemize}
\item the vertex $b$ in $HB$ of the dot associated to $h$
\item the yellow hex centre $y$ of $h$ in $HY$
\item the vertex $g$ in $HG$ associated to the yellow cluster $h$ is in.
\end{itemize}
If we use complex numbers $a,b,c\in\C$ to locate these points, the formula should look like, provided we chose the orientation of $HB$, $HY$ and $HG$ appropriately:
\[ b=3.y+j.g+\on{cst}\]
where $j\in\C$ is the principal \nth{3} root of unity, $\on{cst}$ is some constant depending on the position of the origin in $HB$, $HY$ and $HG$.

In particular the green coordinate could be deduced as a simple linear combination of the blue and yellow coordinates.

\image{}{fig:ov}{
\scalebox{0.2}{
\begin{tikzpicture}
\clip(-30.441,-24.5) rectangle (30.441,24.5); 
\node at (-0.35,1.72)
{\includegraphics[scale=1]{img-5g.png}};
\node at (0,0) {\includegraphics[scale=1]{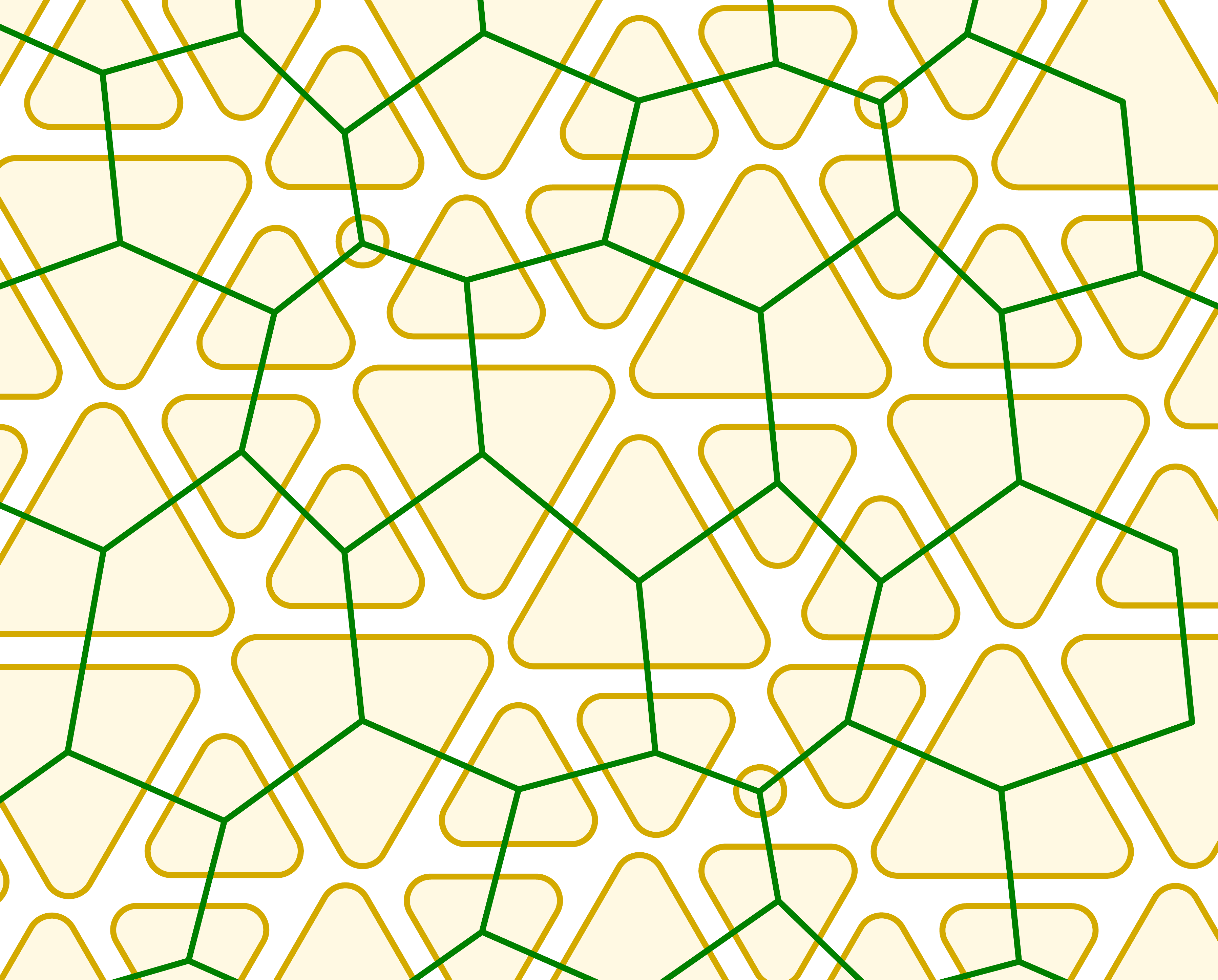}};
\end{tikzpicture}
}
}

\subsection{Summary of main results}

We summarize here how we reprove, independently of \cite{chiral} and without the use of computer searches, the uniquely hierarchical character of the Spectre tiling, and give corresponding substitution systems.

\newcounter{mycounter}
\newcommand{\nh}{\medskip\noindent\stepcounter{mycounter}\arabic{mycounter}.\ }

\nh We prove in \Cref{thm:pieces-ne,thm:pieces-su} that a Spectre tiling is combinatorially equivalent to an a assembly of the following pieces, by superposition of their square/rhomb/(yellow hexes) interfaces, which must be identical ($+$ goes on $-$), and with conditions on the markings of the yellow hexes: each must be given a solid marking and if one is given two solid markings, they must coincide.

\image{Copy of \Cref{fig:labeled-cc-2}.}{labeled-cc-2-copy}{\makebox[\textwidth][c]{\includegraphics[scale=0.575]{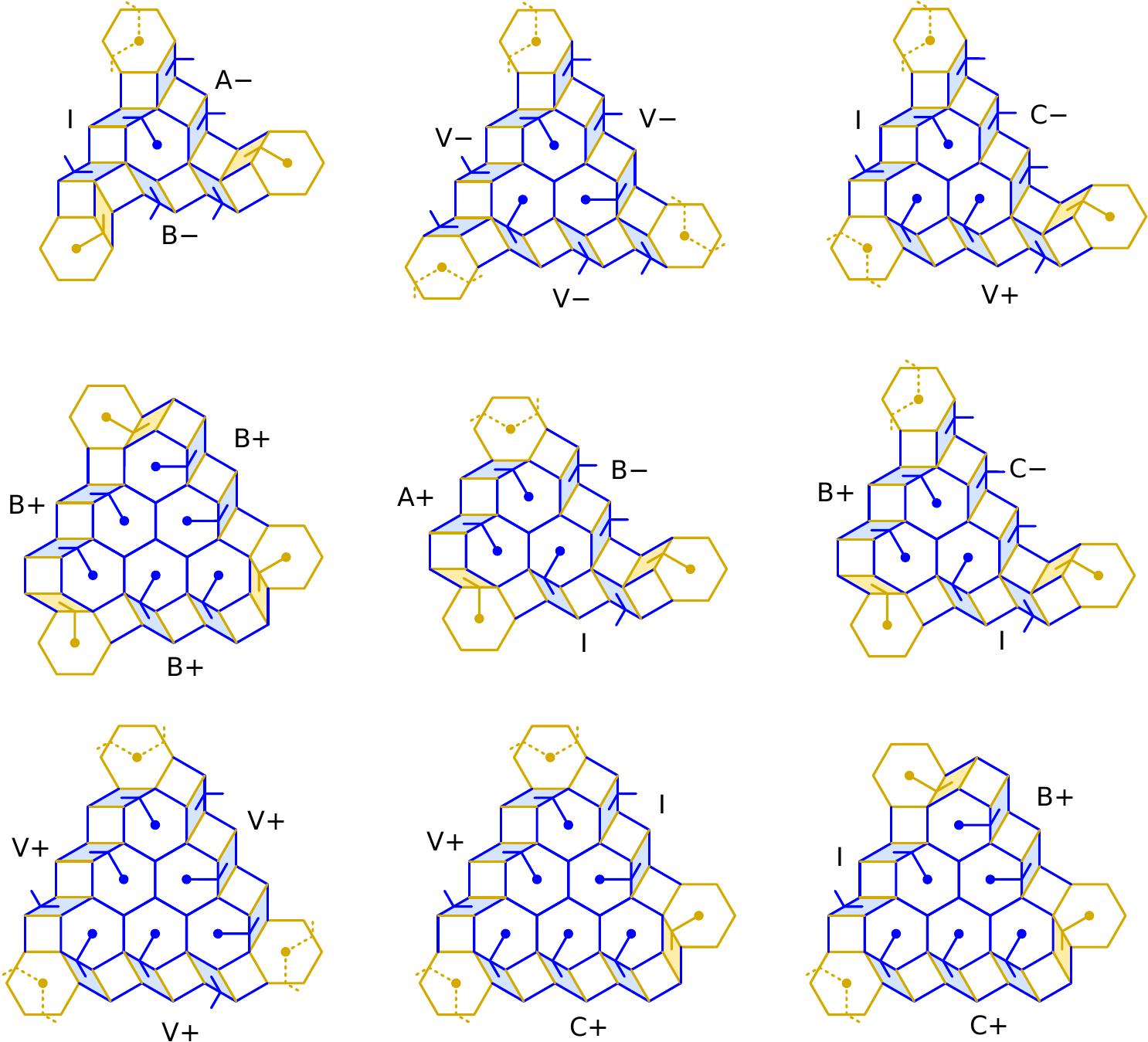}}
}

\nh We deduce from this in \Cref{prop:ces,prop:cts} that a  
the data P$+$D of a partition $P$ of the blue honeycomb $HB$ and a collection $D$ of dots on some of the vertices of $HB$, comes from a Spectre tiling if and only if: 
\begin{enumerate}
  \item It is an assemblage of pieces of \Cref{fig:contracted-pieces-copy}, where two pieces in contact either only touch on one dot of each piece, or they are in contact exactly along a whole identical interface between two respective dots.
  \item Any dot in $D$ has one and only one of the three segments of $HB$ ending on it which is a pink segment of an interface pair as on \Cref{fig:contracted-pieces-copy}.
\end{enumerate}

\image{Copy of \Cref{fig:contracted-pieces}.}{fig:contracted-pieces-copy}{\includegraphics[scale=0.5]{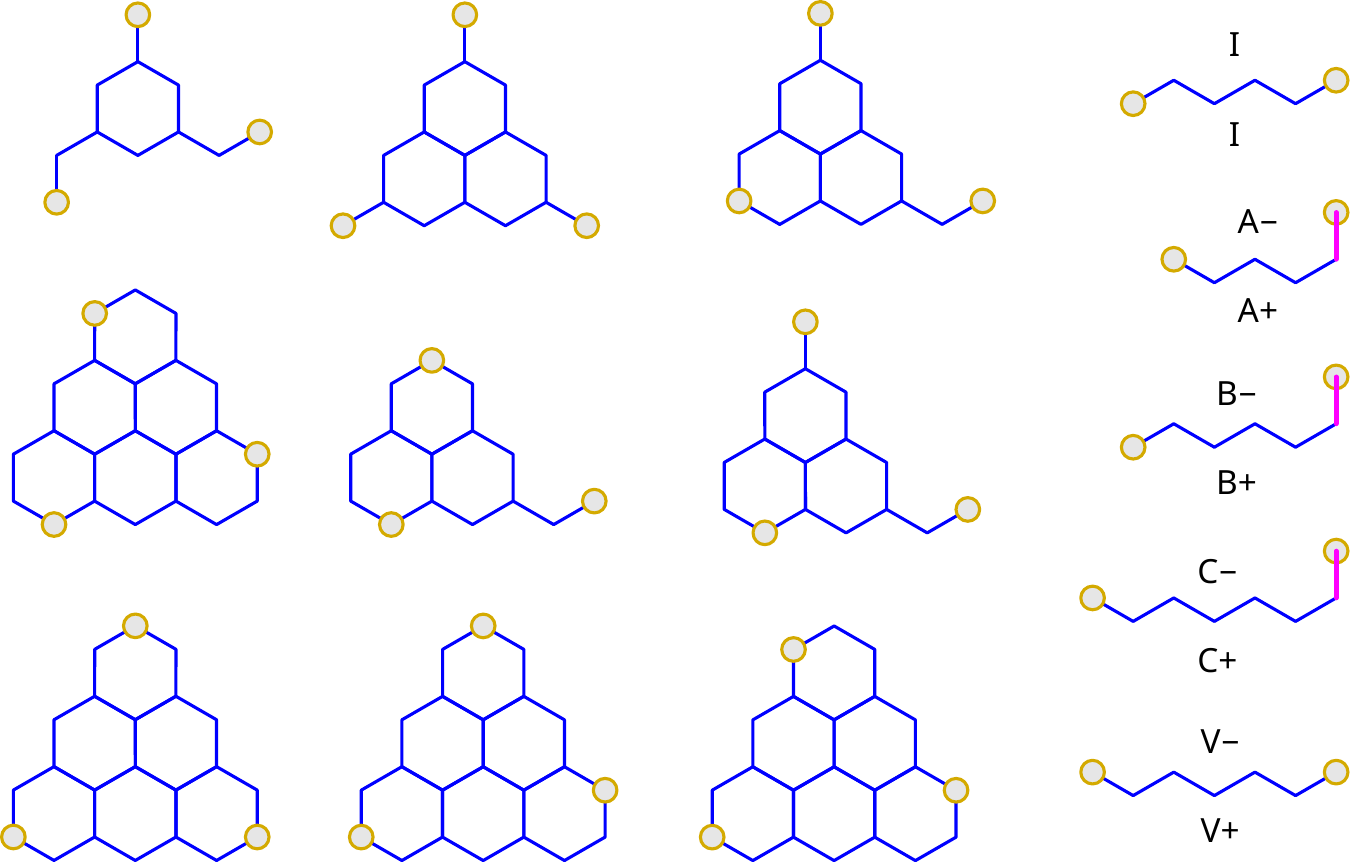}}

\nh We also deduce a combinatorially equivalent tileset, for which we propose many variants in \Cref{fig:triset-1,fig:tri-6b,fig:triset-2,fig:triset-7h} and our favorite:

\image{Copy of \Cref{fig:triset}}{fig:triset-copy}{
  \includegraphics[scale=0.55]{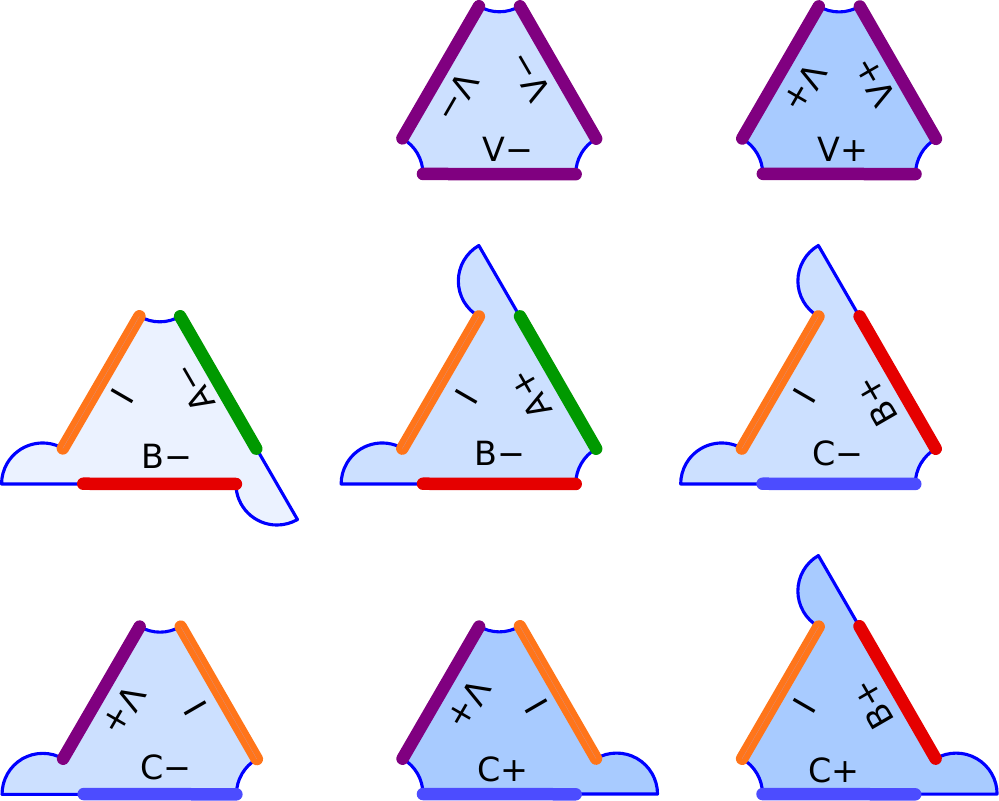}
}

\noindent where the markings must match: same letter, $+$ with $-$, no two pieces can superimpose, and the pieces must cover the whole plane.

\nh We prove that whole plane tilings with the above tileset necessarily packs as follows: 

\image{Copy of \Cref{fig:pack}}{fig:pack-copy}{
\makebox[\textwidth][c]{\includegraphics[scale=0.5]{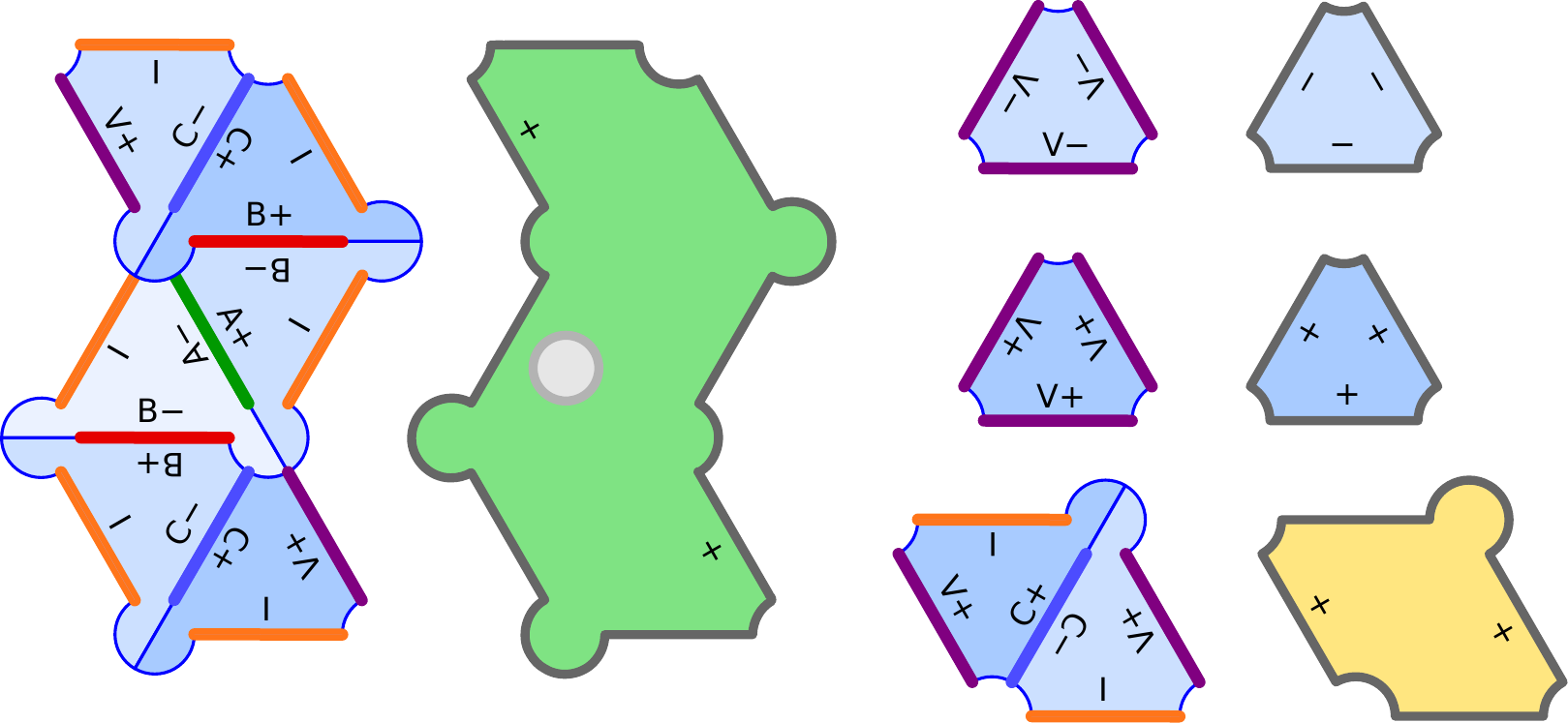}}
}

\noindent and we get an equivalent tileset, which we call \term{triangle packs}, where the $+$ must match a $-$.

\nh We introduce a decoration on the packs:

\image{Extract of \Cref{fig:packed-deco}}{fig:pack-deco}{
  \includegraphics[scale=0.45]{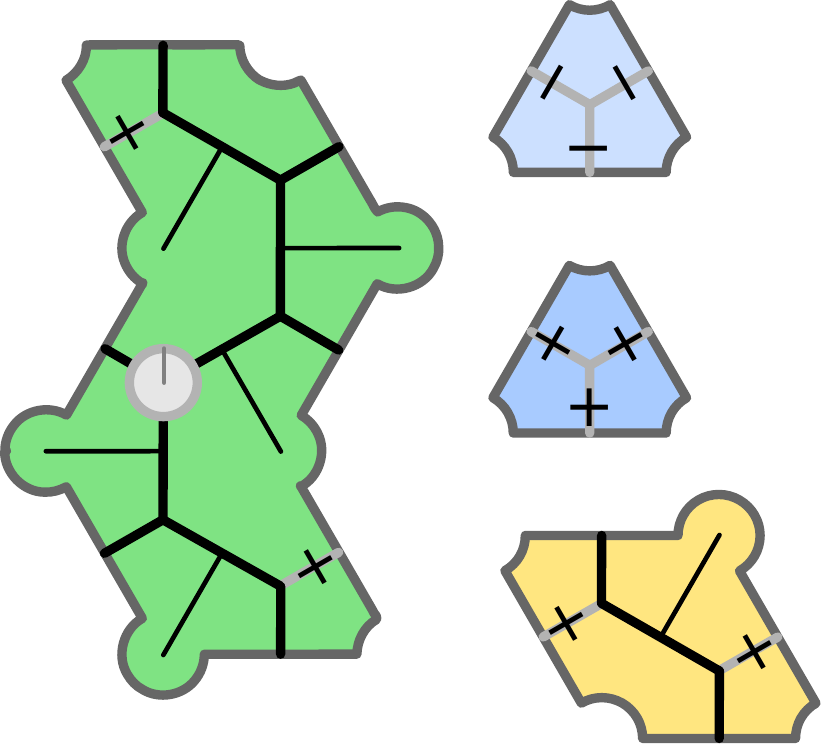}
}

These decoration include the traces on the pack triangles of the dual hexagonal tiling to the regular triangle tiling, as on the lower right of \Cref{fig:dual}.
These trace are in thick black line and thick gray line (the latter hold a sign).
For a whole plane tiling with the packed triangles, these trace a honeycomb of big hexagons.
We prove in \Cref{sub:env} that the thick black lines trace the boundary of clusters of big hexagons, and that together with the big gray dot, they are necessarily assemblies of the \emph{reflected} pieces of \Cref{fig:contracted-pieces-copy}.
This proves unique hierarchy for the Spectre tilings.

\nh We prove in \Cref{prop:inv} that the procedure can be reverted, giving rise to a substitution system, hence implying that there exists whole plane tilings by the Spectre. We explicit in \Cref{sub:subsys} several equivalent substitution systems. \emph{Let us stress that it is a two-step substitution scheme,} i.e.\ we decomposed the substitution into two substitutions.

In \Cref{ss:hpdosub}, we give a (one step) substitution system in terms of the honeycomb partition plus dots (P$+$D) figures, thus elucidating the hierarchical structure of \Cref{fig:bo,fig:trace}.

\nh There are a few bonuses in \Cref{ss:variants,sec:down,sec:double}, including a proof in \Cref{sec:double} that the dots are necessary in the P$+$D patterns (as on \Cref{fig:bo,fig:trace}) for recovering the underlying Spectre tiling or to go one level up or down in the hierarchy, but also the interesting remark that \emph{going up or down two levels in the hierarchy of the patterns can be performed without the dots}.

\subsection{Acknowledgements}

I would like to thank, besides the authors of \cite{chiral,march}, David Smith, Joseph Samuel Myers, Craig S. Kaplan, and Chaim Goodman-Strauss, discoverers of this wonderful tile, the following people with whom I had many discussions:
Nan Ma who boosted my interest in the subject by explaining us how to lift the tiling to 4D space (not used in the present article), and with which I started a more in-depth study of the tiling: see \cite{web}.
Pieter Mostert, who revealed to me the importance of the hexagons and some insight about the cut-and-project methods (not used in the present article).
Several persons in a newsgroup discussion forum, including Yoshiaki Araki,
George Baloglou, Maurizio Paolini, Joshua Socolar and Dale Walton.
Marianne Imperor for her comments improving readability of the present introduction.

The free program \textsf{Inkscape} was extremely useful in producing the many pictures of the present document.
The figures make extensive use of colour and I apologize to the colour-blind people, for they will have to use specific computer programs or physical filters to distinguish some of them.

This research was possible thanks to the great freedom left by the CNRS to its researchers, and the facilities of the Toulouse Mathematics Institute at Université de Toulouse.

For the Spectre and the Hat, existence and aperiodicity of tilings has been proved in several ways, both by the authors of \cite{march,chiral} and by others.
These two original articles used computer assisted searches in some parts.
For examples of non computer-assisted proofs of \emph{existence} for Hat and Spectre, and of aperiodicity for Hat tilings, see for instance the works of Shigeki Akiyama, Yoshiaki Araki, Erhard Künzel and James Smith, some of which are available in \cite{AA,jS}.
We believe that having a substitution system, which proves existence of whole plane tilings with a monotile, is a strong step towards the proof of aperiodicity of all tilings with this monotile.
Yet, to our knowledge, there was no announced non computer assisted proof of non-periodicity of Spectre tilings before the present article appeared on arXiv.

Finally, the informed reader can look at \cite{BGMS2} for a characterisation of the Spectre tilings with the Cut-And-Project method using cohomological methods and associated dynamical system. We do not use any of these methods and results here.


\section{Analysis}\label{sec:analysis}

From now on all angles are expressed in terms of fractions of a full turn. For instance $1/12$ means $1/12$-th of a turn, i.e.\ $360\degree/12=30\degree$.

We call \term{honeycomb} a tessellation of the whole plane by regular hexagons.

The word \term{orientation} will often be used to designate the bearing of an object, i.e.\ the amount it has been turned by, with respect to a reference position. This is not to be confused with other classical uses of this term in mathematics: for instance given a segment, we can give it a preferred direction, i.e.\ decide which of its vertices is first and which is last. Here we call this a \term{directed} segment instead of \emph{oriented} segment. With this terminology, a segment's orientation is a number modulo $1/2$ (of a turn) and an oriented segment's orientation is a number modulo~1.


\subsection{The Spectre and the \texorpdfstring{$\iD2$}{D2} and \texorpdfstring{$\iD3$}{D3} tiles}

People observed that the Spectre shape splits into two shapes, both mirror symmetric and rotationally symmetric of order respectively 3 and 2.
They are sometimes called $\iD3$ and $\iD2$ by reference to their symmetry group and we will use this denomination here.

\image{Splitting of the Spectre into pieces $\iD3$ (top) and $\iD2$ (bottom). Splitting of the Spectre decoration.}{fig:split}{
\includegraphics[scale=1]{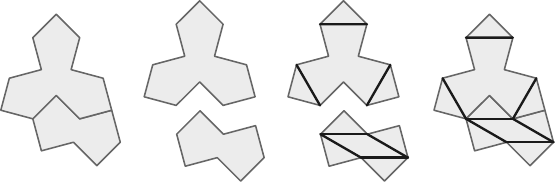}
}

Any tiling by the Spectre (with or without reflections) gives a tiling by $\iD3$ and $\iD2$, as on \Cref{fig:SpD32}.
But the converse is false: 
notwithstanding the fact that $\iD3$ or $\iD2$ alone can tile the plane, there are also tilings mixing the two but for which there are on average more tiles of one type than the other, as on \Cref{fig:unbal}, whereas to come from a Spectre tiling, the $\iD3$ and $\iD2$ must come in pairs (attached in a specific way).

\image{$\iD3$ and $\iD2$ tiling coming from a whole plane Spectre tiling.}{fig:SpD32}{
\includegraphics[scale=0.75]{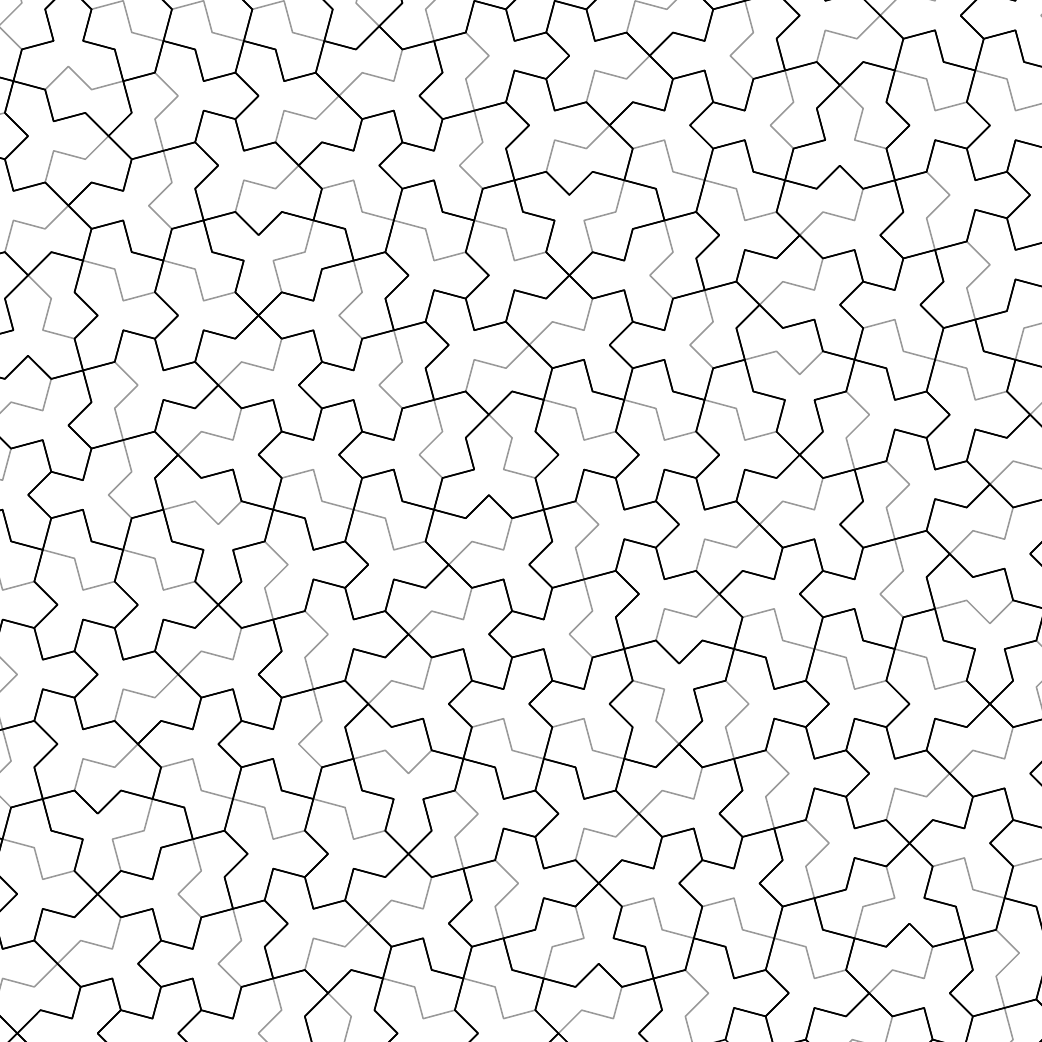}
}

\image{A periodic tiling with $\iD3$ and $\iD2$, for which there are on fundamental domains (such as the one indicated by a tichk black contour) for its translation subgroup, $3$ tiles of shape $\iD2$ for $2$ tiles of shape $\iD3$.}{fig:unbal}{
\includegraphics[scale=0.85]{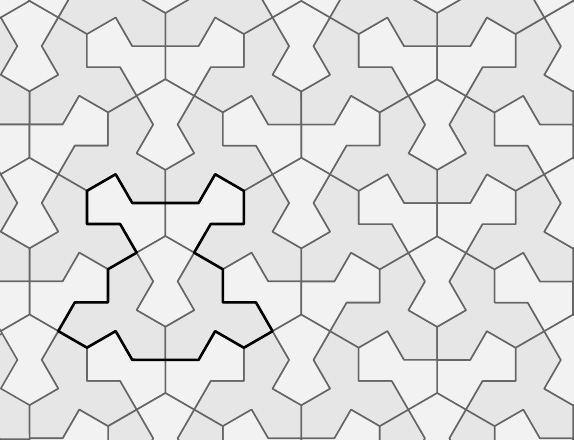}
}

Shape $\iD3$ can be constructed from a regular hexagon by placing along each edge the long side of an isosceles right triangle (angles $1/8$, $1/8$, $1/4$). They are alternatively placed inside and outside and we call them in/out dents.
The same holds for $\iD2$, starting from a rhomb (with small angle $1/12$) but its 4 dents are placed outside.

\image{}{fig:D32-1}{
\includegraphics[scale=1.2]{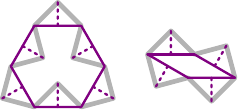}
}

The internal angles of the shapes $\iD3$ and $\iD2$ are all equal to $1/4$, $3/4$ and $1/3$ for the first and $1/4$, $1/3$ and $2/3$ for the second.
In any $\iD3$ and $\iD2$ tiling, the vertices must match with vertices, edges with edges. Moreover a vertex with angle $k/4$ can only match with a $k'/4$ and a $k/3$ with a $k'/3$.

In particular the inward dents of a $\iD3$ can only receive an outward dent of a $\iD3$ or $\iD2$, as illustrated below:

\image{The three possible ways an inward dent of a $\iD3$ can be completed.}{fig:notch-fill}{
\includegraphics[scale=1.2]{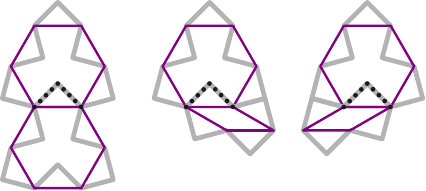}
}

The rhomb from which $\iD2$ is constructed is also its decoration. The decoration of $\iD3$ traces only three sides of the hexagon from which it is constructed.
However by the above, in any whole plane tiling by $\iD3$ and $\iD2$, the decoration of the adjacent tiles filling the three inward dents must complete the hexagon (see \Cref{fig:split,fig:notch-fill}).

Notice that on both shapes $\iD3$ and $\iD2$, the vertices whose angles are multiple of $1/4$ alternate with those whose angles are multiple of $1/3$.
When laying out these shapes, their decoration trace a graph that we call the \term{decoration graph}. 
Note the set of vertices of the decoration graph coincides with the set of vertices of the two shapes $\iD2$ and $\iD3$ whose angle is a multiple of $1/3$.

\begin{proposition}\label{prop:d32_hrs}
Consider any whole plane tiling by the shapes $\iD3$ and $\iD2$. The decorations graph cuts the plane into regular hexagons associated to the $\iD3$ pieces, rhombs associated to the $\iD2$ pieces, and squares, filled with 4 dents.
\end{proposition}
\begin{proof}
Any cut part $P$ which is not the hex of a $\iD3$ or the rhomb of a $\iD2$ is covered by outward dents, which are isosceles right triangles. The long side of these triangles are part of the graph, the other two are not. Such piece $P$ can only be a square covered by 4 dents.
\end{proof}

%

We call \term{decoration tiling} the tiling into hexes, rhombs and squares traced by the decoration graph.
Since the graph is planar and all polygons of this graph have an even number of sides, it follows that it carries a \emph{bipartite colouring} of its vertices, i.e.\ a colouring by two colours such that any two vertices linked by an edge have different colours. The decoration graph is connected,\footnote{A way to justify connectedness of the decoration graph is: its complement has simply connected component (open squares, hexes and rhombs).} and a bipartite colouring on a connected graph is unique up to colour renaming.
We call \term{improved graph} the decoration graph together with a bipartite colouring.

\begin{proposition}\label{prop:bipar}
In a whole plane tiling by $\iD3$ and $\iD2$, the colour of a vertex $v$ of the improved graph is determined by the class modulo $2$ of $k$, where $k/12$ designates the orientation\footnote{Relative to any fixed reference direction.} of any segment ending on $v$ on the boundary of any tile $\iD3$ or $\iD2$ having $v$ as vertex.
\end{proposition}
\begin{proof}
Indeed the edges of the decoration graph correspond to right angles on the boundary of the shapes $\iD3$ and $\iD2$. Jumping from one vertex to the other one along such an edge hence changes the orientation index $k$ by adding/subtracting 3.
\end{proof}

Like for the Spectre, in any whole plane tiling with the shapes $\iD3$ and $\iD2$ the pieces can only come in orientations differing by a multiple of $1/12$.
The decorations of $\iD3$ and $\iD2$ are composed of segments of the same length, that come in orientations differing by a multiple of $1/3$ in $\iD3$ and of $1/12$ in $\iD2$, and when we turn and assemble them, they all differ by multiples of $1/12$.
In particular they come into two classes (or one if only $\iD3$ is used) modulo $1/6$, which we can colour blue and yellow.

The pieces of type $\iD3$ and $\iD2$ come respectively in 4 and 6 orientations, and the colouring of their marking only depends on this orientation, as on the figure below:

\nopagebreak

\image{}{}{
\includegraphics[scale=1.1]{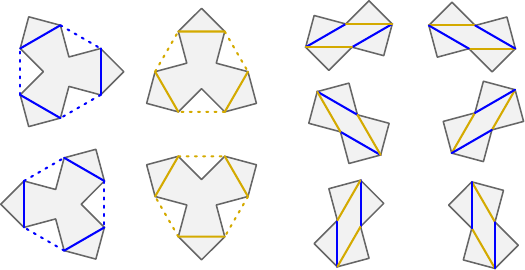}
}

\nopagebreak

By \Cref{prop:bipar}, the bipartite colouring of the vertices of the improved graph also only depends on the orientation of the pieces, see the figure below. Dents are on the right of the blue segments when followed from their black dot to their white dot, and on the left of the yellow segments followed the same way: from black to white dot.

\image{}{fig:byrg}{
\includegraphics[scale=1.1]{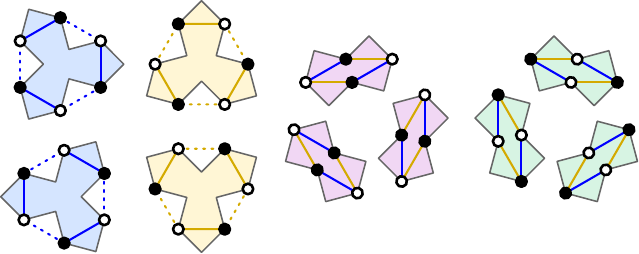}
}

The improved decoration for the $\iD3$ and $\iD2$ tiling of \Cref{fig:unbal} looks as follows:

\nopagebreak

\image{Improved decoration graph for a whole plane tiling by $\iD3$ and $\iD2$ not coming from a Spectre tiling. The orientation is turned by $1/24$ compared to \Cref{fig:byrg}.}{}{
\includegraphics[scale=0.85]{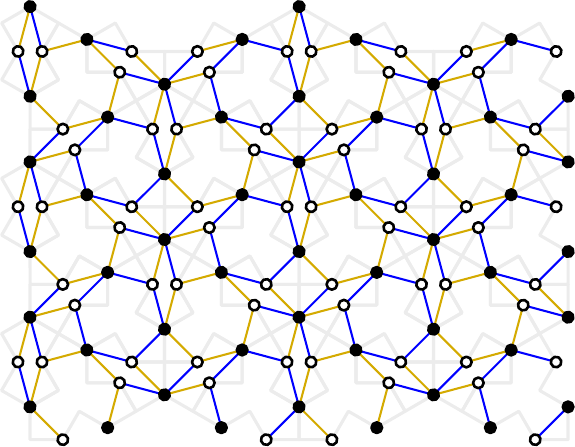}
}

We conclude this section by a summary of simple properties of the decoration tiling into hexes, rhombs and squares.
These tiles can only touch at vertices or along a common edge.
Call two tiles \term{adjacent} if they share an edge.

\begin{proposition}\label{prop:dec-adj}
For the decoration of a $\iD3$ and $\iD2$ tiling:
no two rhombs can be adjacent; no two squares can be adjacent.
\end{proposition}
\begin{proof}
This is an immediate consequence of the orientation of dents.
\end{proof}

Consider a whole plane tiling by $\iD3$ and $\iD2$.

\begin{proposition}\label{prop:blue-hex-touch}
Two hexes of the same coulour in contact in the decoration graph must touch along an edge. 
\end{proposition}
\begin{proof}
Indeed if they were to touch only at a vertex, since hexes of the same colour are parallel, this means that in the decoration graph, there would remain two angles of $1/6$, and some dents to fill, each of which can each only be filled by two acute angles of rhombs. But rhombs cannot share an edge (\Cref{prop:dec-adj}).
\end{proof}

\emph{Note:} For the above proposition, and many below, it is actually enough that the zone covered by the tiling be big enough.
However, the focus of the present article being whole plane tilings, we will not seek to optimize the assumptions of each intermediate result and will put ourselves in the situation of whole plane tilings in most of the coming statements.

Two adjacent hexes not only have parallel edges but their vertices colours are identical by translation. This reflects that two $\iD3$ with one filling an inward dent of the other have the same orientation as in the figure below.

\nopagebreak

\image{}{}{
\includegraphics[scale=1]{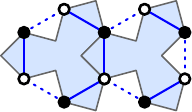}
}

\subsection{Spectre tilings}

\subsubsection{Clusters}\label{subsub:cl}

Consider now a whole plane tiling by the Spectre, and the associated improved decoration, as on \Cref{fig:enriched-deco} below.

\image{}{fig:enriched-deco}{
\includegraphics[scale=0.75]{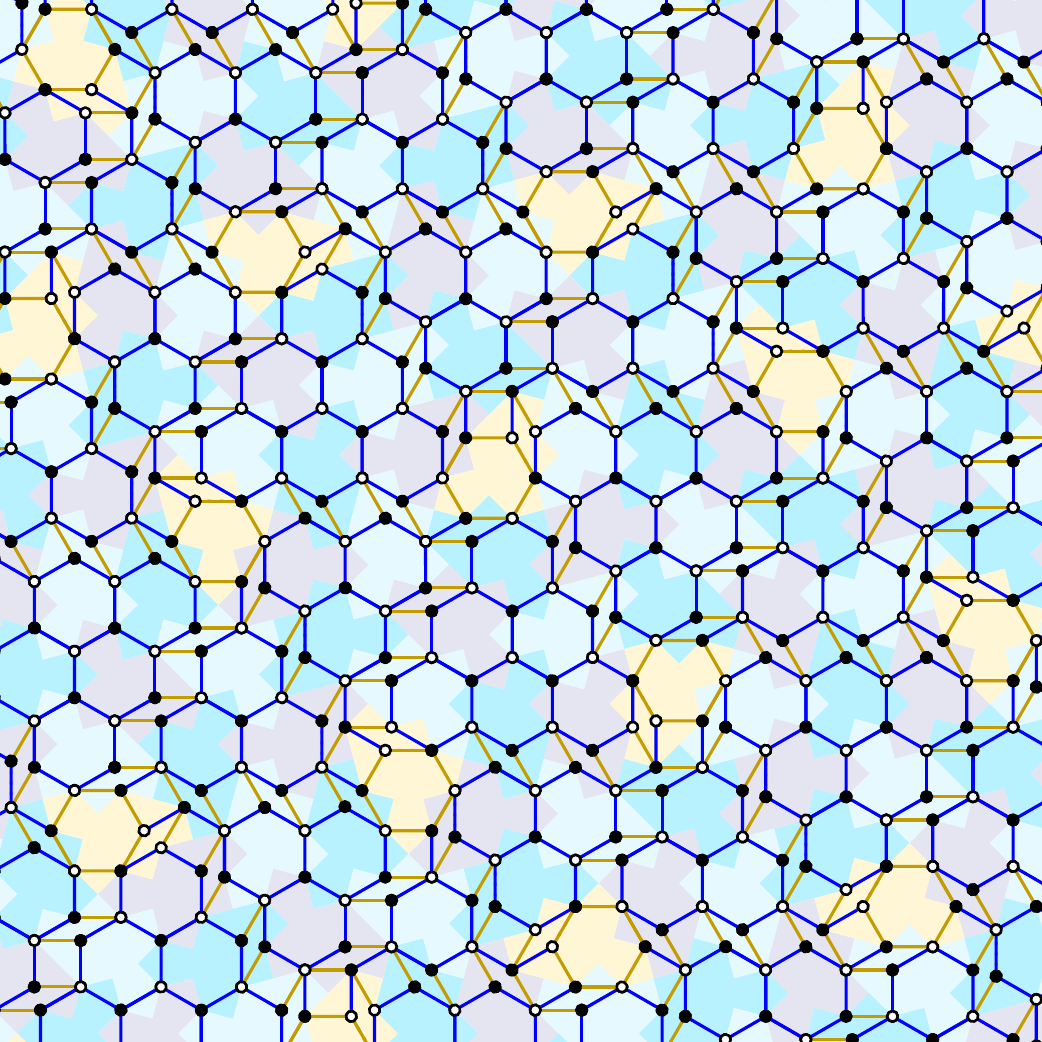}
}

Note that for the corresponding $\iD3$, $\iD2$ tiling, there must be a pairing between adjacent $\iD3$ and $\iD2$ tiles, i.e.\ a pairing of the hexagons and the adjacent rhombs. Because the Spectre cannot be reflected, pink $\iD2$ (third column) in \Cref{fig:byrg} can only be associated to blue $\iD3$ and the green $\iD2$ (fourth column) to yellow $\iD3$ so we may as well colour them as follows:

\image{Orientations match those of \Cref{fig:enriched-deco}}{fig:col-conv-spec}{
\includegraphics[scale=1.1]{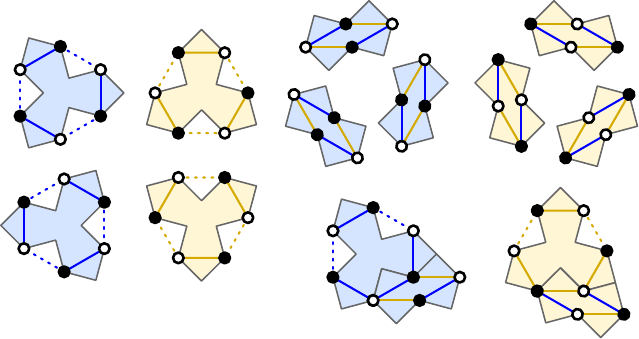}
}

We now study how blue hexes can cluster. Recall that by \Cref{prop:blue-hex-touch}, two blue hexes in contact must touch along an edge. 

\begin{definition}
We call \term{clusters} equivalence classes of blue hexes under adjacency.
\end{definition}

Clusters are not cc's. Let us start with a simple completion rule:

\begin{lemma}\label{lem:completion}
For any edge shared between two blue hexes, the black dot of this edge is surrounded by 3 blue hexes.
In other words: if we have

\nopagebreak

\image{}{fig:A0}{
\includegraphics[scale=1]{hextouch-A0.pdf}
}

then we have the darker piece below:

\nopagebreak

\image{}{}{
\includegraphics[scale=1]{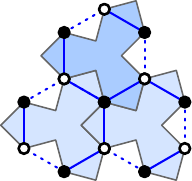}
}
\end{lemma}
\begin{proof}
Indeed in the figure below, the pink inward dent can only be filled in by a blue $\iD3$ or a yellow $\iD2$, and only in one way in each case.

\image{}{}{
\includegraphics[scale=1]{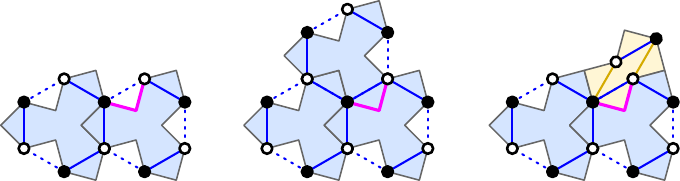}
}

But in the second case, the companion yellow $\iD3$ of the yellow $\iD2$ has no room to be placed (this is easier to see on the decoration graph: where the yellow rhomb is in contact with the two hexes, there only remains room for squares, so we cannot fit its companion yellow hex).
\end{proof}

\begin{proposition}\label{prop:T}
  A blue hex cluster is necessarily a $\hT1$, $\hT2$ or $\hT3$.
\end{proposition}
\begin{proof}
If a cluster is not a $\hT1$ then there are at least two adjacent blue hexes.
%
By \Cref{lem:completion}, the cluster is necessarily triangular (or is infinite containing arbitrarily large triangular arrangements).
This triangle has a side made of at most 3 hexes, otherwise there is at least a hex inside that is completely surrounded by hexes, preventing it to have a paired rhomb.

\image{The central $\iD3$ cannot be paired with a $\iD2$.}{}{
\includegraphics[scale=1]{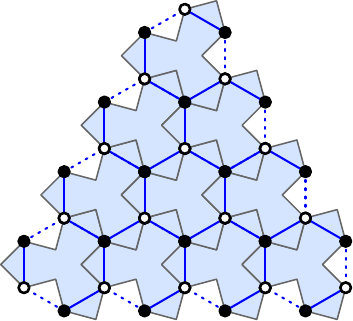}
}
\end{proof}

Here are the three possible clusters, up to rotation:

\nopagebreak

\image{Clusters $\hT1$, $\hT2$ and $\hT3$, up to rotation}{fig:clus}{
\includegraphics[scale=1]{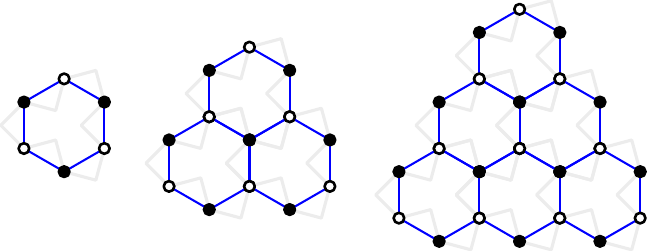}
}

Each $\iD3$ corresponding to these hexes has a paired $\iD2$, and this pair completely surrounds a black dot.

Blue hexes with bipartite colouring can come in two orientations differing by $1/3$, that correspond to the two possible orientations of the underlying blue $\iD3$.
This means that the clusters $\hT1$, $\hT2$ and $\hT3$ also come in two possible orientations. For $\hT2$ and $\hT3$, this orientation can be deduced from the orientation of the hex arrangement, but for $\hT1$ this is invisible if we only look at the blue hexagon.
By abuse of language we say that the clusters in \Cref{fig:clus} \term{point up}, though they also point in two other directions.
And we will say that the other orientation \term{points down}.

\image{Down pointing clusters.}{fig:clus-2}{
\includegraphics[scale=1]{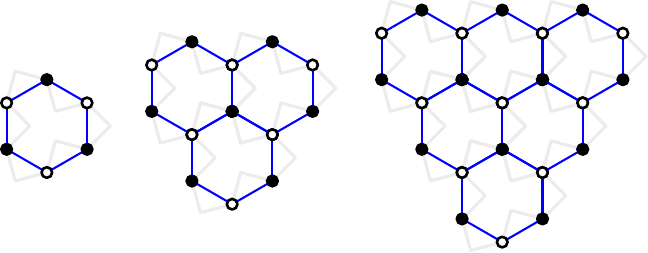}
}

\subsection{From cluster to cc}

\subsubsection{Visual aids for Spectre tilings}\label{ss:visual-aids}

In this section we indicate boundary shapes that can be visually identified and used to exclude a situation, or force the presence of a given tile in a given orientation.

The reader should bear in mind the following: the Spectre can be seen as a 14-gon with equal length edges with vertices of two types: 7 with angles multiples of 1/3 and 7 with angles multiple of 1/4 (one of them flat).
Then in a whole plane tiling by the Spectre, for two tiles in contact, each vertex of one tile touches a vertex of the other one, and their type is the same.

\begin{proposition}\label{prop:impo-11}
The following polygonal line cannot be included, even after rotation, in the union of boundaries of tiles in a whole plane tiling by the Spectre, where the red lines have the same length as the Spectre edge (seen as a 14-gon).

\nopagebreak

\image{}{fig:impo-11}{
\includegraphics[scale=0.85]{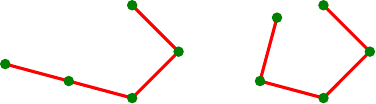}
}
\end{proposition}
\begin{proof}
A direct inspection shows that the Spectre cannot fit the 1/3 angle at the central vertex. (On the left, only the reflected Spectre fits.)
\end{proof}

\begin{proposition}\label{prop:autofill-1}
In each of the figures below, if the polygonal line on the left of the arrow is included in the union of boundaries of tiles in a whole plane tiling by the Spectre, then a tile must be placed as the red tile on the right of the arrow.

\nopagebreak

\image{}{}{
\begin{tikzpicture}
\node at(0,0){\includegraphics[scale=1]{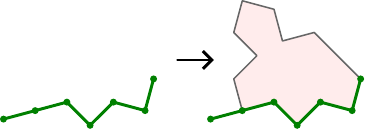}};

\node at(7,0){\includegraphics[scale=1]{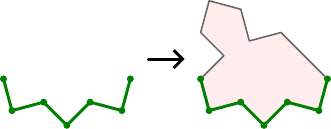}};
\end{tikzpicture}
}

\image{}{}{
\begin{tikzpicture}
\node at(0,0){\includegraphics[scale=1]{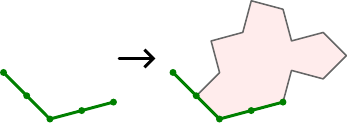}};

\node at(7,0){\includegraphics[scale=1]{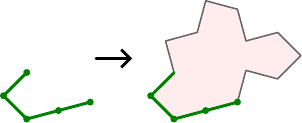}};
\end{tikzpicture}
}

\image{}{}{
\begin{tikzpicture}
\node at(0,0){\includegraphics[scale=1]{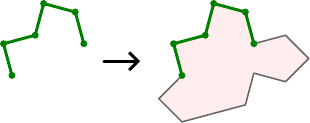}};
\node at(7,0){\includegraphics[scale=1]{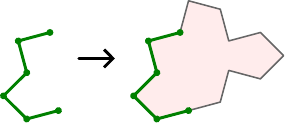}};
\end{tikzpicture}
}
\end{proposition}
\begin{proof}
In the first two situations, every other way to fill the dent at the central vertex gives rise to an outline forbidden by \Cref{prop:impo-11} (the red polygonal lines in the figure below).

\image{}{}{
\includegraphics[scale=0.8]{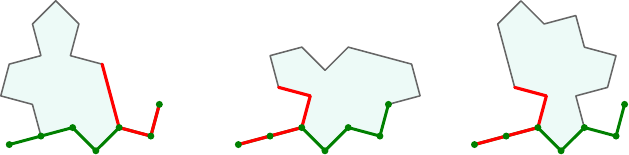}
}

\image{}{}{
\includegraphics[scale=0.8]{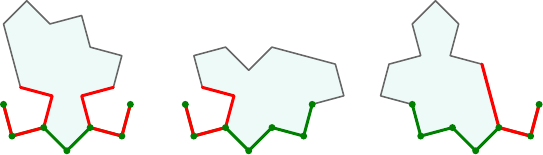}
}

In the next two situations, there is only one way to fit a tile at the 1/3 angle of the central vertex.

In the last two, the only other way to fill the circled vertex is as follows:

\image{}{}{
\begin{tikzpicture}
\node at (0,0){\includegraphics[scale=0.8]{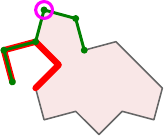}};
\node at (4,0){\includegraphics[scale=0.8]{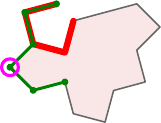}};
\end{tikzpicture}
}
\end{proof}

\subsubsection{Properties of odd tiles}

We recall that, as before, we use the colours blue and yellow on the Spectre tiles and the $\iD3$ and $\iD2$ tiles, to distinguish their two possible orientations modulo multiples of $1/6$. See \Cref{fig:col-conv-spec} again.

Note that, up to now we have not identified odd/even tiles, i.e.\ which of the blue or yellow tiles are majority.

\begin{proposition}\label{prop:no-full-blue}
No whole plane tiling by the Spectre can have only blue or only yellow tiles.
\end{proposition}
\begin{proof}
By rotation, is enough to prove the claim for blue tiles.
By contradiction, assume this is possible and consider the decoration graph and also the induced $\iD3$ and $\iD2$ tiling.
In this case we have that the yellow edges belong only to $\iD2$ pieces.
The dent on these yellow edges can only belong to squares. So each $\iD2$ is part of an infinitely long chain as on the left column of \Cref{fig:chain}.

The squares cannot be filled by pieces of type $\iD2$ because the dent must be supported by a blue segment and the piece would have to be oriented as on the second column and hence cannot fit.

So they are filled by pieces of type $\iD3$ as on the third column. But then we get clusters of blue hexes that are too big and contradict \Cref{prop:T}.

\image{A chain of 5 pieces $\iD2$ placed as on the left column (blue $\iD2$ with touching dents supported by yellow segments) cannot appear in a whole plane tiling of the plane by Spectres. }{fig:chain}{
\includegraphics[scale=0.95]{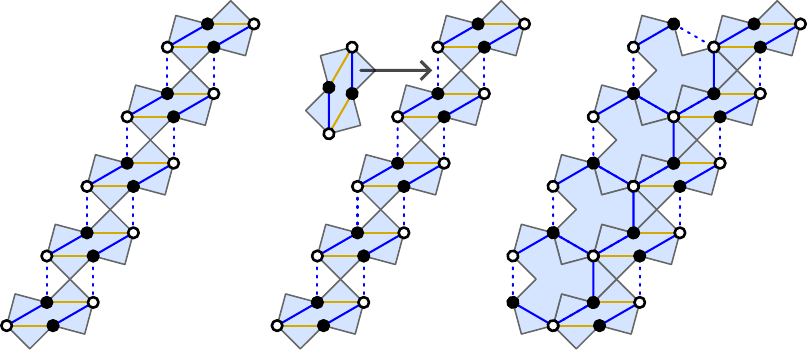}
}
\end{proof}

The central argument in \Cref{fig:chain} will be generalized in \Cref{lem:no-spec-rs-chain}.

\medskip

We reprove (and improve) the forced environment of odd tiles explained in Section~4 of \cite{chiral}.
The proof of the following key proposition goes through the examination of many cases.
It is not obvious how to optimize that and I did not manage to take a significant advantage of the tile decorations for that.
It would be nice to have a more conceptual proof of \Cref{prop:odd-env-1}.

\begin{remark*}
By seeing a draft of the present notes, the author of \cite{jS} informed us that for Hats tilings, the latter article includes a conceptual proof of the existence of a parity class which cannot form clusters with more than one tile, from which a neighbouring environment is somewhat easy to deduce.
It is quite possible that one could adapt these arguments for Spectre tilings.
\end{remark*}

\begin{proposition}\label{prop:odd-env-1}
Consider in a whole plane tiling by the Spectre, two tiles of different parity and in contact along at least one edge.
Then they are, up to a rotation, part of the following arrangement, with one of them being the red tile. The non-red tiles all have a parity opposite to the red one.

\image{}{fig:odd-env}{
\includegraphics[scale=0.666]{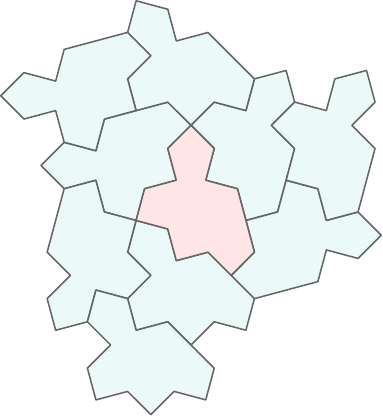}
}
\end{proposition}

Since the proof is a bit long, we moved its content in \Cref{sub:pf-odd-env-1}. It makes an extensive use of the visual aids of \Cref{ss:visual-aids}.

\medskip

This allows us to reprove one of a key observations of \cite{chiral} (for this we could have stopped at \Cref{fig:last-1} in the previous proof).

\begin{corollary}\label{cor:odds}
In every whole plane tiling by the Spectre, there is a parity class for which the tiles are isolated.
\end{corollary}
\begin{proof}
We proceed as in \cite{chiral}, Section~4, page~13.
Every two tiles in the tiling can be connected by a chain of adjacent tiles.
Two adjacent tiles of different parity must be in the configuration of \Cref{prop:odd-env-1} up to rotation, and we call (in this proof) \term{iso} the one of the two tiles corresponding to the red central tile of the configuration.
Only one of the two tiles is iso since in the configuration, all cyan tiles adjacent to the red one are also adjacent to another cyan tile, and all cyan tiles of the configuration have the same parity, opposite to the red.
Now consider any chain of adjacent tiles. In it, the iso tiles are isolated: the next and previous tiles cannot be iso, and they have a different parity from the iso. Moreover for two consecutive tiles in the chain of different parity, one must be iso. By contraposition, two consecutive non-iso tiles must have the same parity. It follows that, along the chain, iso tiles must all have the same parity.
\end{proof}

We call this parity class the \emph{odd}. By performing if needed a rotation of 1/12, we may always assume that this is the yellow class and this is what we do on the rest of this article. 

\subsubsection{Properties of cc}\label{ss:prop-cc}

Consider the blue/yellow decoration graph generated by a tiling with the Spectre.
By convention we use use blue tiles, with a blue hexagon, for even tiles and yellow tiles, with a yellow hexagon, for odd tiles, where we assume as we just explained, that the odd tiles are the isolated ones.

We recall the definition we gave in the introduction:

\begin{definition}\label{def:cc}
We call \term{cc} the connected components of the blue graph in the decoration graph (i.e.\ without the yellow segments).
\end{definition}

There is the following observation: 

\begin{proposition}
The only vertices of the decoration graph such that all edges sharing it have the same colour, are vertices at the meeting point of three blue hexes.
\end{proposition}

\begin{proof} 
  The edges of a rhomb or a square alternate in colour.
  So a vertex as in the statement necessarily only belongs to hexagons of the same colour. Yellow hexagons come from odd tiles but by \Cref{prop:odd-env-1} all odd tiles are isolated, so the hexagons must actually come from even tiles and thus be blue.  
%
%
%
%
\end{proof}

\begin{proposition}\label{prop:impo-2}
The following configuration cannot appear, even rotated, in a whole plane tiling by the Spectre.

\nopagebreak

\image{}{}{
\includegraphics[scale=0.75]{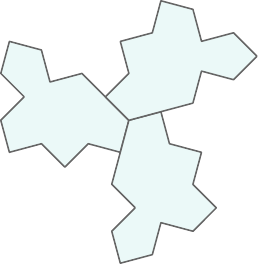}
}
\end{proposition}
\begin{proof}
It forces the following sequence by repeated application of \Cref{prop:autofill-1} which ends up in a configuration with non-extendable parts (or we can invoke \Cref{prop:odd-env-1} at the second frame).

\image{}{}{
\includegraphics[scale=0.6]{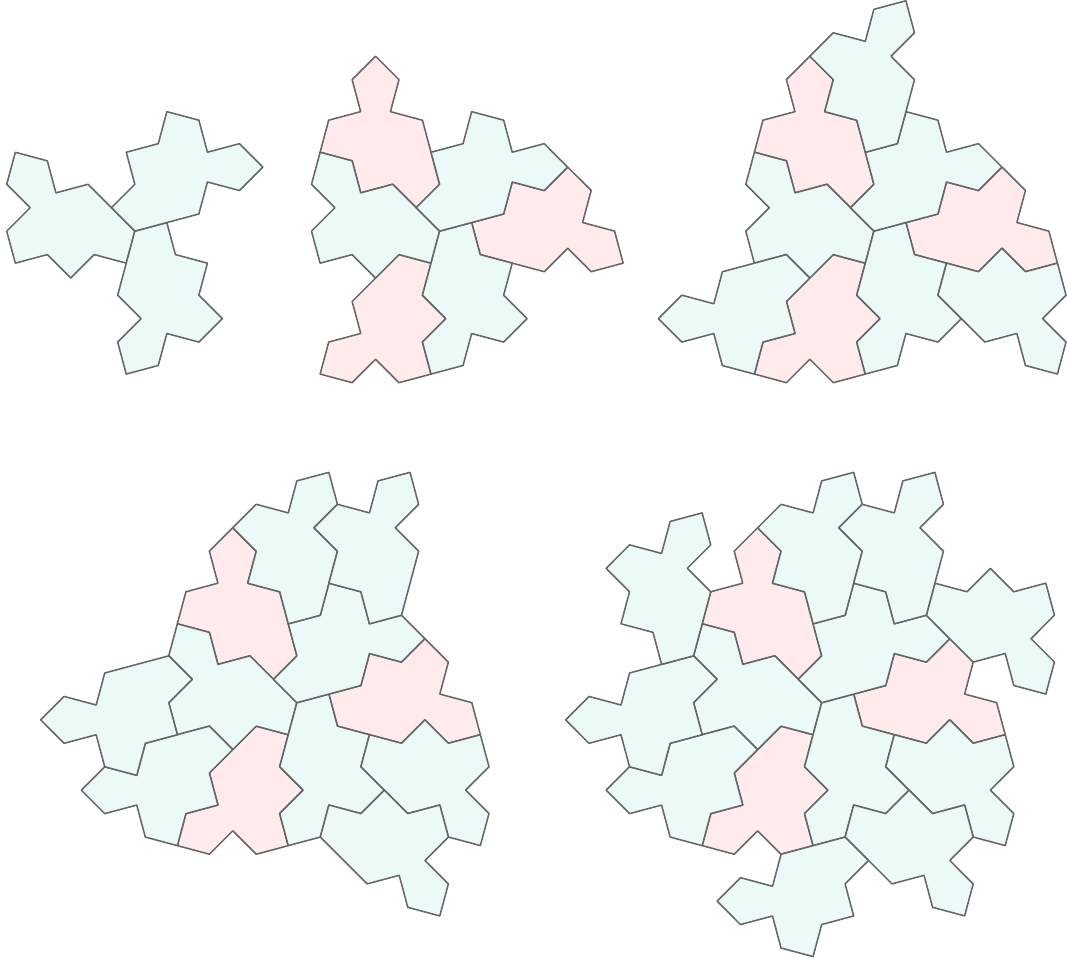}
}
\end{proof}

\begin{proposition}\label{prop:ne}
Each cc contains at least one blue hex.
\end{proposition}

\proof
A blue segment not bounding a hex bounds a square.
Disappointingly, we have to proceed by inspection of many cases.

First, the blue segment of an even tile that is outlined in red below:

\nopagebreak

\image{}{}{
\includegraphics[scale=1]{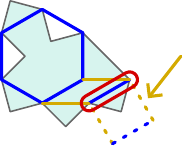}
}

The yellow segment indicated by the arrow will be in a tile, which must be one of the following cases.

\image{}{}{
\includegraphics[scale=0.6]{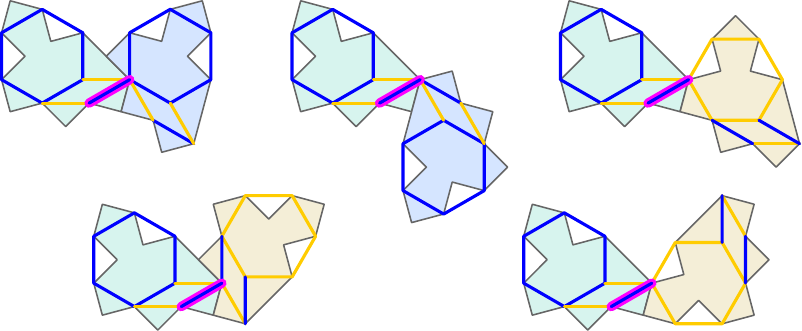}
}

However, not all cases can be realized: two cases have red outlines which fit no tile, and another one forces two odd tiles to be adjacent.

\image{}{}{
\includegraphics[scale=0.6]{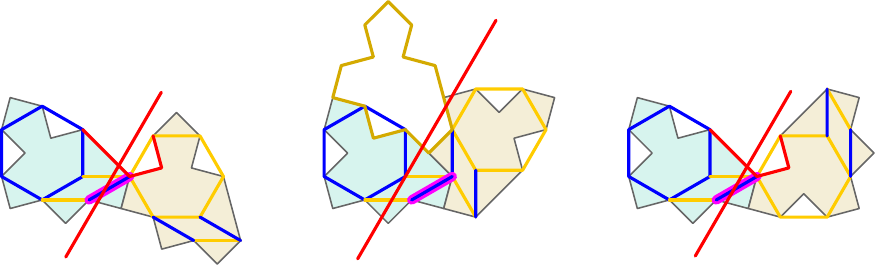}
}

\noindent Alternatively we can invoke \Cref{prop:odd-env-1} to rule out these same four cases.

In the two remaining situations, the leftmost one has a hex touching the initial blue segment, and in the the second one we look at what can fit in the upper right free space :

\image{}{}{
\includegraphics[scale=0.66]{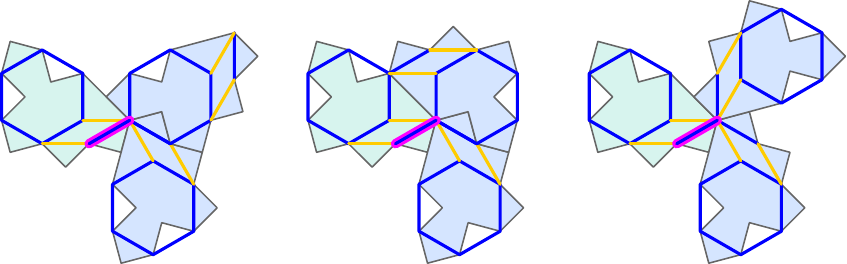}
}

The first two have a hex, and the last one cannot appear by 
\Cref{prop:impo-2}.
%

We now study the case of the blue segments of an odd tile. By \Cref{prop:odd-env-1} a yellow tile has a fixed immediate environment, which includes the following two blue tiles:

\nopagebreak

%
%

\image{}{}{
\includegraphics[scale=0.75]{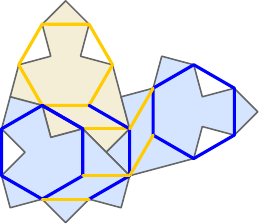}
}

The left blue segment of the odd tile is on a hex so we only need to inspect the right one.
The free corner in the common segment of the blue tiles can only complete as follows:

\image{}{}{
\includegraphics[scale=0.66]{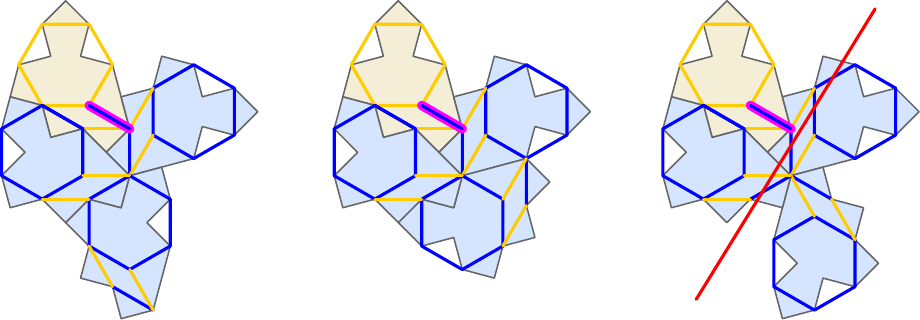}
}

The right one has already been excluded above and the other two have a hex.
\qed\medskip

\begin{proposition}\label{prop:impo-16}
In the $\iD3$ and $\iD2$ tiling induced by a whole plane tiling by the Spectre, the following arrangement or its rotations cannot appear:

\nopagebreak

\image{}{}{
\includegraphics[scale=0.75]{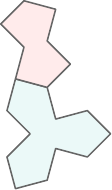}
}
\end{proposition}
\begin{proof}
The $\iD2$ and $\iD3$ must pair so as to form spectres.
The red $\iD2$ can only be paired with the red $\iD3$ as in the central frame in the figure below. The green inward dent can only be filled by the hatched red $\iD2$. But there is no room left for the associated $\iD3$ of the latter.

\image{}{}{
\includegraphics[scale=0.75]{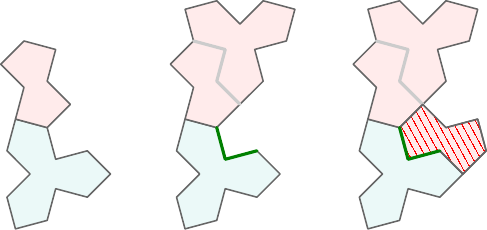}
}
\end{proof}

The following consequence extends \Cref{lem:completion}.

\begin{corollary}\label{cor:protr-1}
If a blue segment protrudes from the black dot of a blue hex, then this segment belongs to two adjacent blue hexes.

\image{}{}{
\includegraphics[scale=0.75]{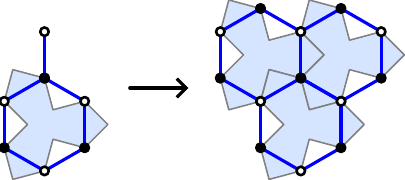}
}
\end{corollary}
\begin{proof}
Only a blue hex or a yellow $\iD2$ can be placed to contain this blue segment with the colours of its extremities.
The latter has been excluded by \Cref{prop:impo-16}, so it is a blue hex.
Then we get the two possibilities on the right of the following figure:

\image{}{}{
\includegraphics[scale=0.75]{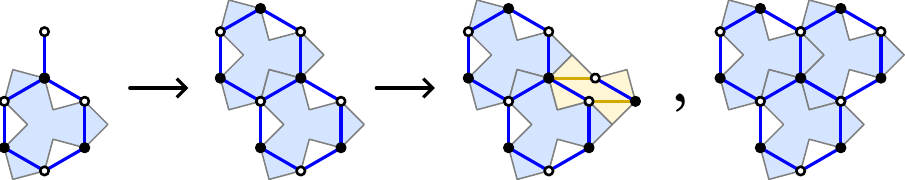}
}

\noindent The one with a yellow $\iD2$ does not leave room for its associated yellow $\iD3$, so we are in the last configuration. (We can also conclude this from \Cref{lem:completion}.)
\end{proof}

Below we show the only two ways a specific arrangement of $\iD3$ and $\iD2$ can be filled at a specific dent (in green). In one of the two ways, the blue segment with a violet highlight belongs to a blue hex. 

\image{}{fig:f1}{
\includegraphics[scale=1]{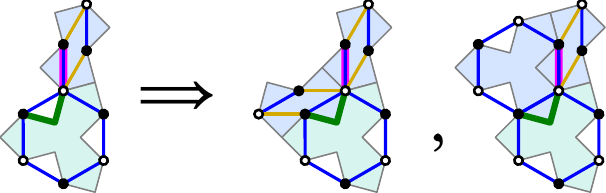}
}

\begin{lemma}
If a blue segment protrudes from a blue hex without being itself in a hex then we are in one of the following situations, up to rotations (not all may appear in an infinite tiling), restricting the possible antennae to a short list:

\nopagebreak

\image{A list of possible antennae attached to a blue hex of a given blue tile (highlighted in cyan). Not all may actually be allowed in a whole plane tiling by the Spectre.}{fig:antenatips}{
  \includegraphics[scale=0.666]{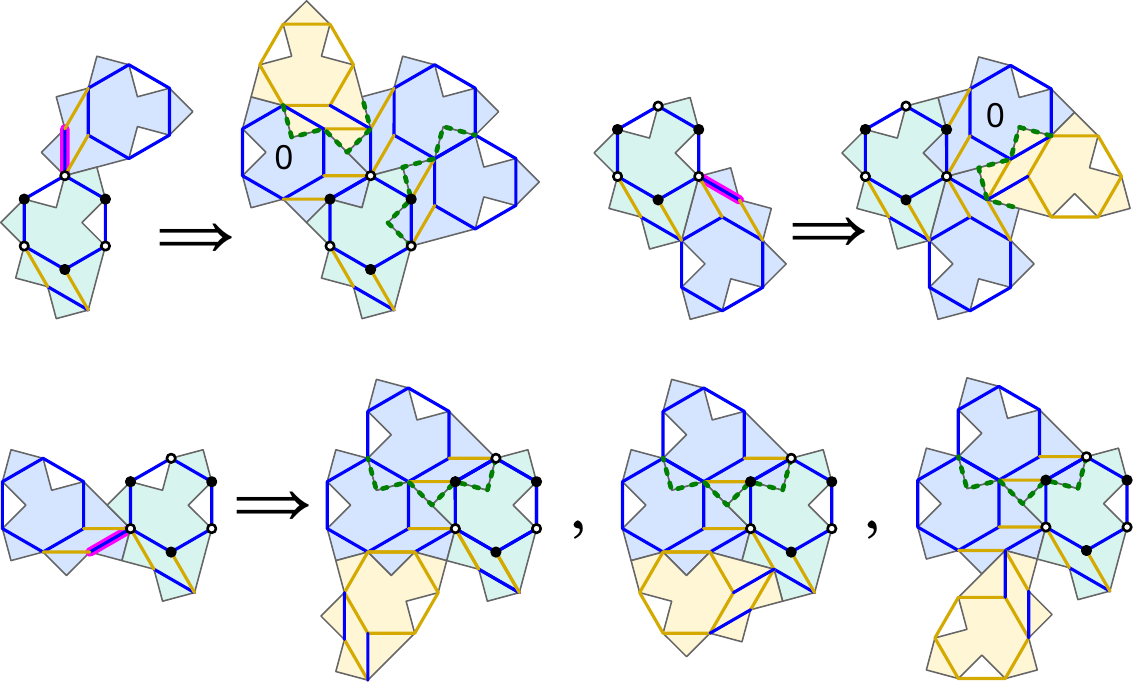}
}
\end{lemma}
\begin{proof}
By \Cref{cor:protr-1}, the protruding blue segment $S$ has to stem from a white dot of the hex.
Now $S$ cannot belong to a yellow Spectre tile: of the two blue segments in the latter, one has its white dot inside the tile, so this is the other segment and its white dot that must fit, but on the figure below it appears that the D2 of the yellow Spectre would then collide with the blue D3 piece.

\image{The protruding segment is highlighted in purple. Middle: if it were to belong to a yellow Spectre tile, its white dot would have to be the one circled in red, and the yellow tile would have to be oriented as on the figure. But then the yellow tile does not fit. So we are in the situation on the right.}{fig:wd1}{
  \includegraphics[scale=0.85]{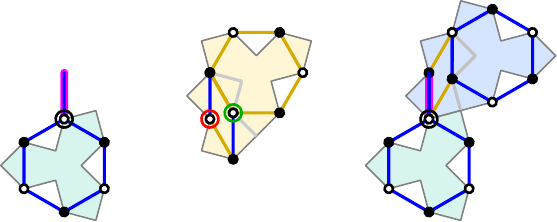}
}

So $S$ is the blue segment of a blue Spectre tile. Not being in a blue hex we necessarily have the situation on the right of \Cref{fig:wd1}.
The cyan D3 of the figure above then completes in three possible ways into a blue Spectre tile, coloured cyan in the figure below:

\image{}{fig:wd2}{
\includegraphics[scale=0.85]{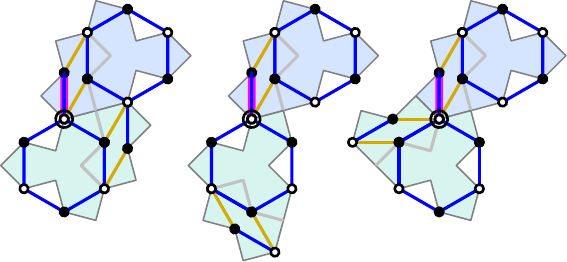}
}

\Cref{fig:antenatips} is then obtained by rotating these three possibilities so as to set the cyan tile in a common orientation, and then by simple inspection and application of the visual aids of \Cref{prop:autofill-1} at the level ot the green dashed lines, and
the tiles marked $0$ are deduced from the paragraph above \Cref{fig:f1} and the fact that the highlighted blue segment shall not be on the boundary of a blue hex.
\end{proof}

\medskip

Recall that we defined \emph{clusters} as maximal sets of adjacent blue hexes and proved that in whole plane tilings they can only come in three possible triangular configurations, that we called $\hT1$, $\hT2$ and $\hT3$.
We also noted that each $\hT n$ can come in two orientations, related to the orientation of the underlying piece $\iD3$.

\begin{definition}\label{def:tips}
For a $\hT n$, let us define \term{tips} as below:

\nopagebreak

\image{Position of the tips of a $\hT n$ pointing up}{}{
\includegraphics[scale=0.75]{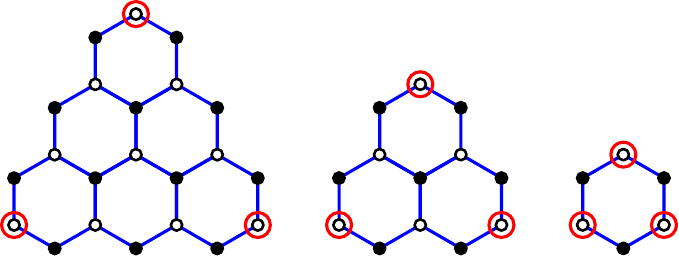}
}
In the picture above, the $\hT n$ point up and the red circles indicate what we call the tips. They are in particular white dots of the white/black dot colouring of vertices introduced earlier.

\image{Position of the tips of a $\hT n$ pointing down}{}{
\scalebox{1}[-1]{\includegraphics[scale=0.75]{tips-b.pdf}}
}
\end{definition}

From the analysis above it follows that:

\begin{proposition}
A cc contains at most one hex cluster and a possible antenna at each of its three tips.
\end{proposition}

\begin{proof} Indeed we saw that a cc contains a cluster, that a cluster is a $\hT n$ for $n=1$, $2$ or $3$, and \Cref{cor:protr-1} implies that no tip can grow on the sides of a $\hT n$. 
\Cref{fig:antenatips} shows than an antenna touches one blue hex and no more.
\end{proof}

\subsubsection{Interface between cc}

In the sequel, as in the previous section and unless explicitly stated, all statements concern whole plane tilings by the Spectre.

\medskip

We already mentioned that blue hexes with the bipartite vertex colouring come in two possible orientations.
Actually we can extend this to all blue segments, in a compatible way, using the improved decoration, the one with black/white dots at the vertices of the blue/yellow decorated graph.
Blue segments come in three orientations, six if we direct them. 
Let us direct them from the black dot to the white.
The six orientations differ by a rotation by multiples of $1/6$. 
Let us group them into their two classes modulo rotation by multiples of $1/3$.

\begin{proposition}\label{prop:oc}
  In the improved decorated graph of a tiling by $\iD3$ and $\iD2$, 
  two blue edges sharing a vertex and directed as above belong to the same orientation class.
  Two directed blue edges connected by a yellow edge belong to opposite classes.
\end{proposition}
\begin{proof}
  The first claim follows from the fact that the colouring is bipartite and that the undirected blue edges have only 3 orientations.
  The second claim follows from the fact that the colouring is bipartite and that in the decoration graph we have, besides the hexagons, only squares and rhombs, where the parallel sides have the same colour, so opposite directions w.r.t.\ the vertex colouring.
\end{proof}

Recall that the notion of \emph{cc} was defined in \Cref{def:cc} as
the connected components of the blue graph in the decoration graph (i.e.\ without the yellow segments).
Consider a cc. By \Cref{prop:oc}, all its directed segments belong to the same class modulo rotation by a multiple of $1/3$.
We can thus associate a class to the whole cc.
Consider now two cc's. We say that they are \term{adjacent} if there is a yellow segment linking them. From the same proposition:

\begin{corollary}
Two adjacent cc's have opposite classes. The corresponding $\hT n$ thus have opposite classes.
\end{corollary}

%


We now proceed with a more precise description of the environment of a cc.

\begin{proposition}\label{prop:cc-descr}
 Each cc touches three yellow hexagons at three contact vertices related as follows to the three tips of the $\hT n$ the cc contains.
 For each tip the contact vertex is either: 
 \begin{itemize}
 \item the other end of the antenna if there is one attached to the tip,
 \item the tip itself,
 \item the next vertex on the $\hT n$ boundary followed in the counter-clockwise direction.
 \end{itemize}
 Between two such contact vertices, there are 3 interfaces to 3 adjacent cc's. Each interface consists in squares and rhombs in alternation, with parallel yellow edges. They end with two yellow hexagons.
\end{proposition}

\image{Example of a cc and its three interfaces to adjacent cc's.}{}{
\includegraphics[width=8cm]{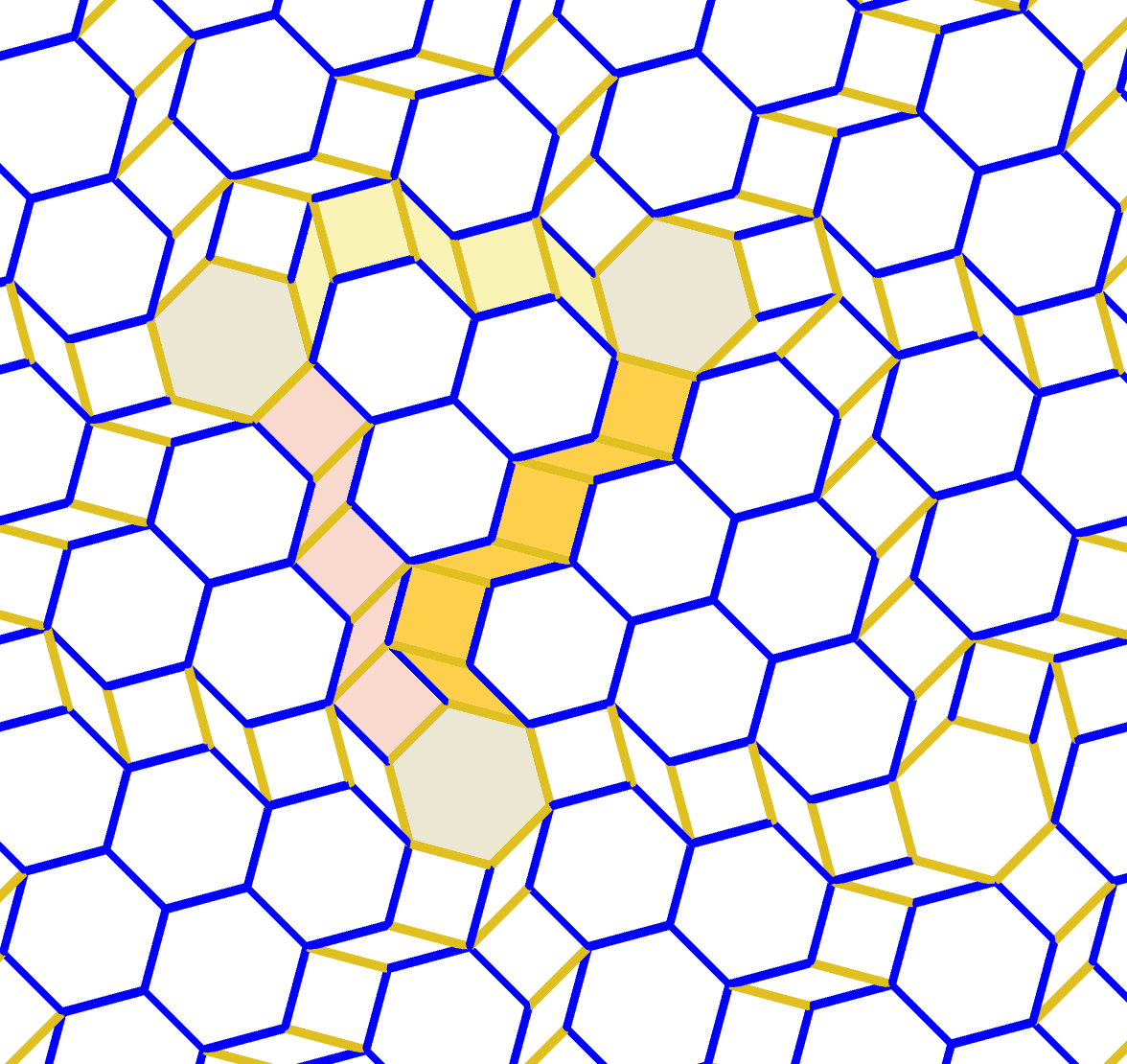}
}

\begin{proof}
Possible antennae ends were enumerated in Figure~\ref{fig:antenatips}. In particular there is always a yellow hex at the end.

Let us study further the environment of a $\hT n$ tip. In the next figure we show all the possible top tip neighbourhoods. It is a subset of all the possible neighbourhoods, where we exclude the presence of another hex touching the tip.

\image{A list of possible top tip neighbourhoods. The vertex circled in red represent the top tip of an up pointing $\hT n$. Not all those configurations can actually appear in a whole plane tiling: for instance the second one is forbidden by \Cref{prop:odd-env-1}}{fig:tips}{
\includegraphics[scale=0.74]{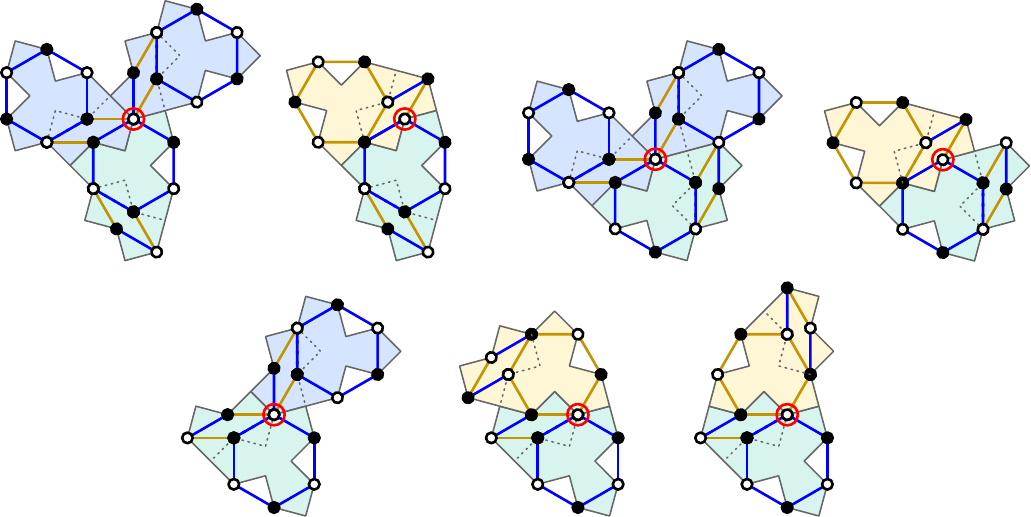}
}

We see that those which do not have an antenna touch a yellow hex either at the tip or at the next vertex following the blue hex boundary in the counter clockwise direction.

In the figure below, we sum up the ways a top tip of an up pointing cluster might be connected to a yellow hex, as a consequence of \Cref{fig:antenatips,fig:tips}:

\image{Top tip circled in red, contact point in green.}{fig:tips-2}{
\includegraphics[scale=0.75]{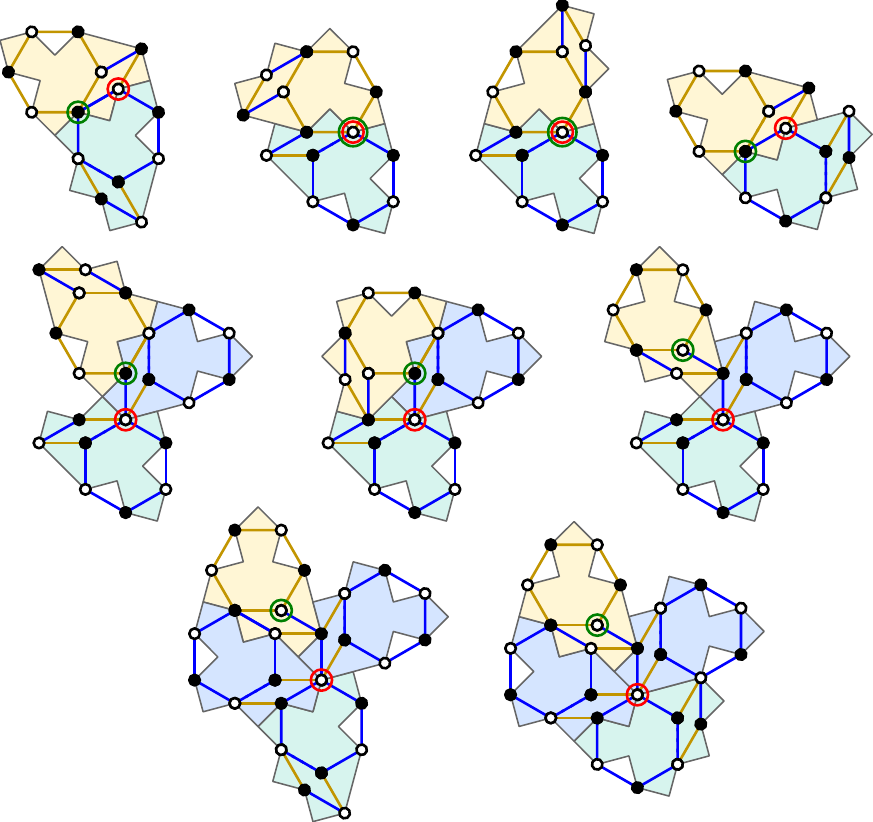}
}

This gives us at least three yellow hexes in contact with the cc.
Actually there cannot be a fourth one: for one thing, the list above shows that on the union of a tip, the possible antenna, and the black dot following the tip in anticlockwise order, there is only one yellow hex; for another, a yellow hex on a non-tip white dot of a hex cluster intersects one of the blue hexes so this is excluded; the last case is a black dot not following a tip in anticlockwise order: in that case we get (a rotation of) the figure below, which does not work either.
Alternatively, this can be deduced from the left part of \Cref{fig:odd-env-2}.

\image{}{fig:wd3}{
  \includegraphics[scale=0.85]{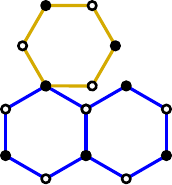}
}

Consider one of the three yellow hexes in contact with the cc.
The two sides of this yellow hex that are adjacent to the contact point can each be seen as the starting element of a finite sequence of parallel yellow segments attached to the boundary of the cc and ending on another yellow hex in contact.

\image{How we start following the boundary on two examples}{}{
\includegraphics[scale=1.25]{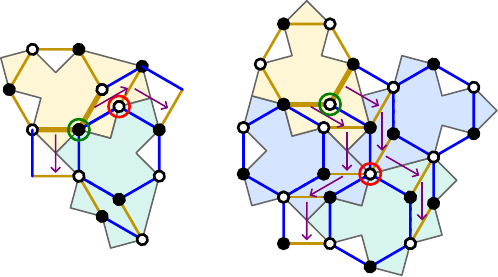}
}

This claim follows from induction by following the cc boundary: on \Cref{fig:tips-2}, inspection of the cases show that the first blue segment in this boundary 
together with the yellow segment are part of the boundary of a quadrilateral which is either a square or a rhomb.
The opposite sides of this quadrilateral are parallel segments of the same colours.
Opposite to the yellow segment, there can be a square or a rhomb or a hex.
No two squares or two rhombs can follow each other (\Cref{prop:dec-adj}).
\end{proof}

\subsection{A list of cc}\label{sub:cc-list}

We will enumerate all possible types of cc an their interfaces
(this information includes the contact points with the yellow hexes).
We will also see that a $\hT 3$ never has antennae while a $\hT 2$ has between 1, 2 or 3 of them, and a $\hT 1$ necessarily has 3.

\medskip

In the sequel we use the following \term{marking} in the hexagons to indicate which pieces of type $\iD3$ and $\iD2$ pair to form Spectres:

\nopagebreak

\image{Marking of the hexagon of a spectre. The marking is a dot at the centre of the hexagon plus a broken line that links the centre of the hexagon to the midpoint of the side of the hexagon in common with the rhomb of the associated $\iD2$, then to the centre of the rhomb. }{fig:marking-a}{
\includegraphics[scale=1]{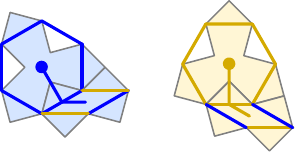}
}

\image{Left: We removed the Spectre, and hence its composing $\iD3$ and $\iD2$, from the previous picture, but maintained the decoration (hex$+$rhomb) and the marking. Right: we only kept the hexagon and the marking, which now has a small segment sticking out.}{fig:marking-b}{
\includegraphics[scale=0.9]{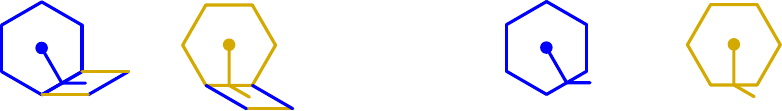}
}

We find interesting to look at the interpretation of \Cref{fig:odd-env} in terms of marked hexagons.

\nopagebreak

\image{Environment of an odd tile, from the point of view of the decoration graph. The two dashed hexes are respectively deduced from \Cref{lem:completion} and the description of antennae in \Cref{prop:ne,fig:antenatips} (or more simply a direct inspection of what can fit there).
}{fig:odd-env-2}{
\includegraphics[scale=0.25]{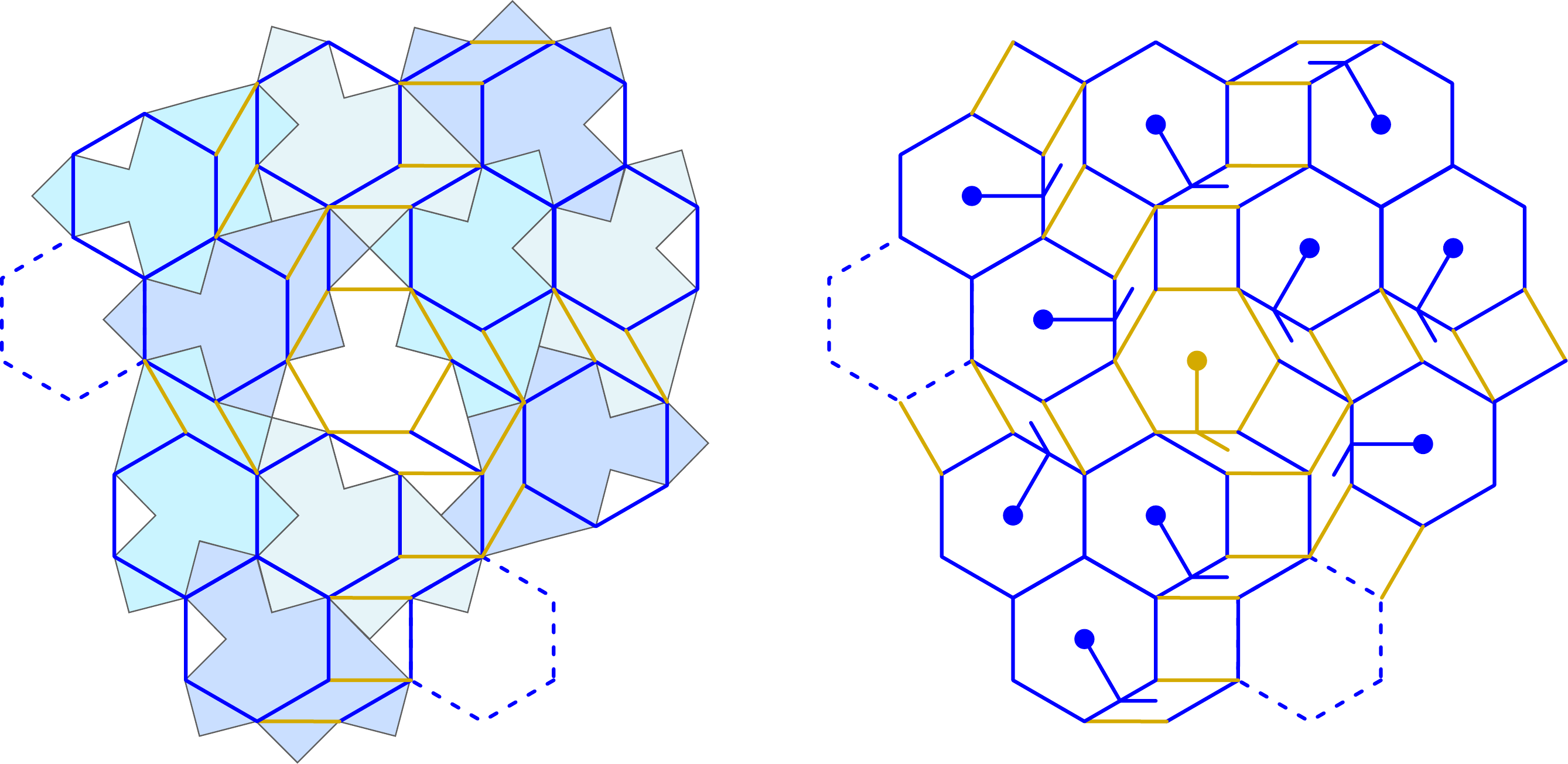}
}

%
%
%
%

The next six statements, from \Cref{prop:T1-cc} to \Cref{prop:T3-cc}, have their proofs moved to \Cref{ss:pf-cc-types}.

\begin{proposition}\label{prop:T1-cc}
Every cc of a $\hT 1$ is necessarily as on the picture below:\footnote{Actually we will see in \Cref{cor:T1-tto} that the lower left arrow can only have one orientation.}

\nopagebreak

\image{}{fig:T1-a}{
\includegraphics[scale=0.75]{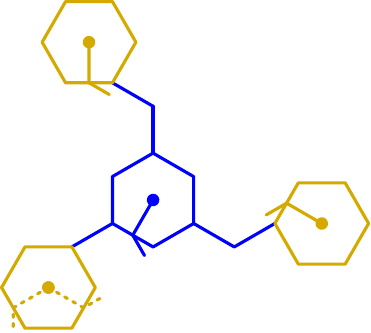}
}
\end{proposition}

\begin{proposition}\label{prop:T2-impo}
A $\hT 2$ with markings as below cannot occur.

\nopagebreak

\image{}{}{
\includegraphics[scale=0.75]{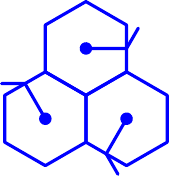}
}
\end{proposition}

%
%
%
%

\begin{proposition}\label{prop:T2-cc-1}
A $\hT 2$ with markings as below left has a cc as below right:

\nopagebreak

\image{}{}{
\begin{tikzpicture}
\node at(-2,0) {\includegraphics[scale=0.75]{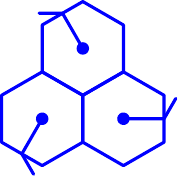}};
\node at(2,0) {\includegraphics[scale=0.66]{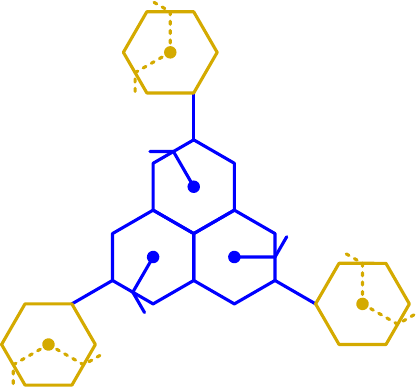}};
\end{tikzpicture}
}
\end{proposition}

\begin{proposition}\label{prop:T2-cc-2}
The cc of a T2 can only be as in the previous proposition or as follows:

\nopagebreak

\image{}{}{
\includegraphics[scale=0.66]{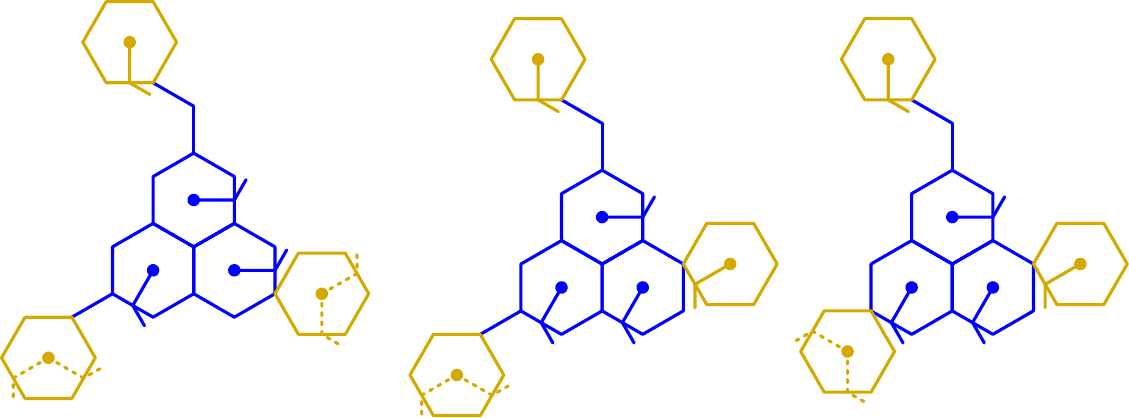}
}
\end{proposition}

\begin{lemma}\label{lem:T3-env}
A $\hT3$ is environed as follows, in terms of the $\iD3$ and $\iD2$ induced tiling (equivalently, in terms of the yellow/blue decoration graph):

\image{Two equivalent depictions of the environment of a $\hT3$. Some of the $\iD2$ and $\iD3$ have a determined pairing into a Spectre (such pairs are filled with the same colour and their common boundary is in white). For others (in two pale blue shades), we do not know in advance. All the $\iD3$ and $\iD2$ and Spectre are ``blue'' in our previous denomination.}{fig:T3-env}{
\includegraphics[scale=0.75]{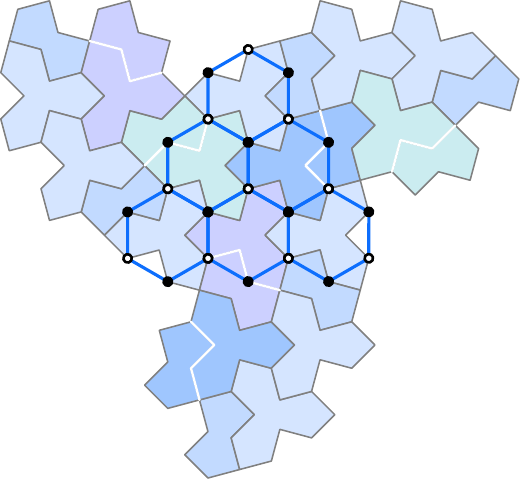}
\quad
\includegraphics[scale=0.75]{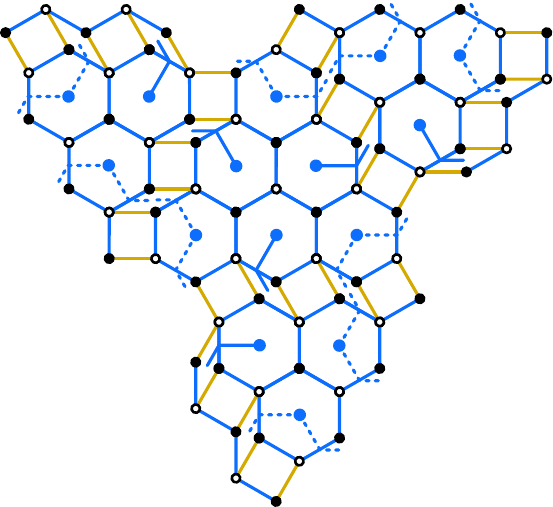}
}
\end{lemma}

\begin{proposition}\label{prop:T3-cc}
A T3 has no antenna: it is its own cc. It can only be as follows on its lower left corner, and similarly for the other two by rotation.

\nopagebreak

\image{}{fig:T3}{
\includegraphics[scale=0.66]{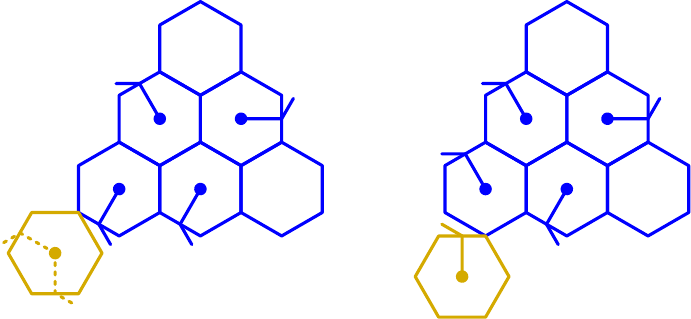}
}
\end{proposition}

As already remarked, in each case $n=1$, $2$ or $3$, the yellow hex associated to a tip of a $\hT n$ cluster is either at the end of an antenna if there is one, or at the tip, or at a vertex on the boundary of the cluster, situated one blue segment away from the tip, counter-clockwise.

From this, we get the following possibilities for the cc's and their \emph{interfaces} as defined in \Cref{prop:cc-descr}.
Not all those possibilities may appear in an actual tiling of the whole plane by the Spectre. This is one of the main figures of the article and we will refer to it often.

\image{In this list, we marked interfaces between adjacent cc's with a letter corresponding to their shape. The plus or minus is there to indicate that a $+$ can only fit a $-$ and vice-versa. The `$\iI$' has no sign because it fits with itself. We also indicated which rhombs are blue and which are yellow. The blue rhombs that match external blue hexes have a small marking.
Note that the the yellow hexes with a solid yellow marking are those linked to a yellow rhomb of one of the interfaces. The other ones have a dotted marking indicating the two indents of the underlying $\iD3$ that are still free for a matching yellow rhomb. 
}{fig:labeled-cc-2}{
\makebox[\textwidth][c]{\includegraphics[scale=0.575]{inter-ant-2c.pdf}}
}

Two neighbouring cc's in the decoration graph of a whole plane Spectre tiling must have a common interface, in the sense that two yellow hex and a chain of rhombs and squares linking them must be the same.
In other words, the interfaces in the pieces of \Cref{fig:labeled-cc-2} must superimpose.

Let us sum up what we obtained:

\begin{theorem}\label{thm:pieces-ne}
Consider the hex, squares and rhombs tiling associated via the decoration graph to a whole plane tiling by the Spectre.
It decomposes as a union of pieces as in \Cref{fig:labeled-cc-2}, rotated by multiples of $1/6$, superimposing at their interfaces and yellow hexes.
\end{theorem}

Conversely:

\begin{theorem}\label{thm:pieces-su}
 Any arrangement of the pieces above such that
 \begin{enumerate}
 \item the pieces cover the plane,
 \item the blue hexagons of distinct pieces are disjoint,
 \item the interfaces overlap in a matching way (disregarding the yellow hex markings): $\iA+$ with $\iA-$, $\iB+$ with $\iB-$, $C+$ with $C-$, $\iV+$ with $\iV-$ and $\iI$ with $\iI$, (then the yellow hexagons match),
 \item\label{item:yellow-rule} each yellow hex has been given a marking by a solid (not dotted) yellow marking by at least one piece; no yellow hex has been given two different solid yellow marking,
 \end{enumerate}
 gives rise to a tiling by Spectres.
\end{theorem}
\begin{proof}
Indeed the blue hex/rhomb association always match when the letters and sign follow the rules:
below we show the interfaces facing each other.
Bear in mind that the number of hexes behind a given interface could depend on the piece: for instance $\iB+$ can bound cc's with a $T3$ and cc's with a $T2$, $\iB-$ can bound cc's with a $T1$ and cc's with a $T2$, etc.

\image{List of interfaces and their correspondence}{fig:interface-list}{
\includegraphics[scale=0.41]{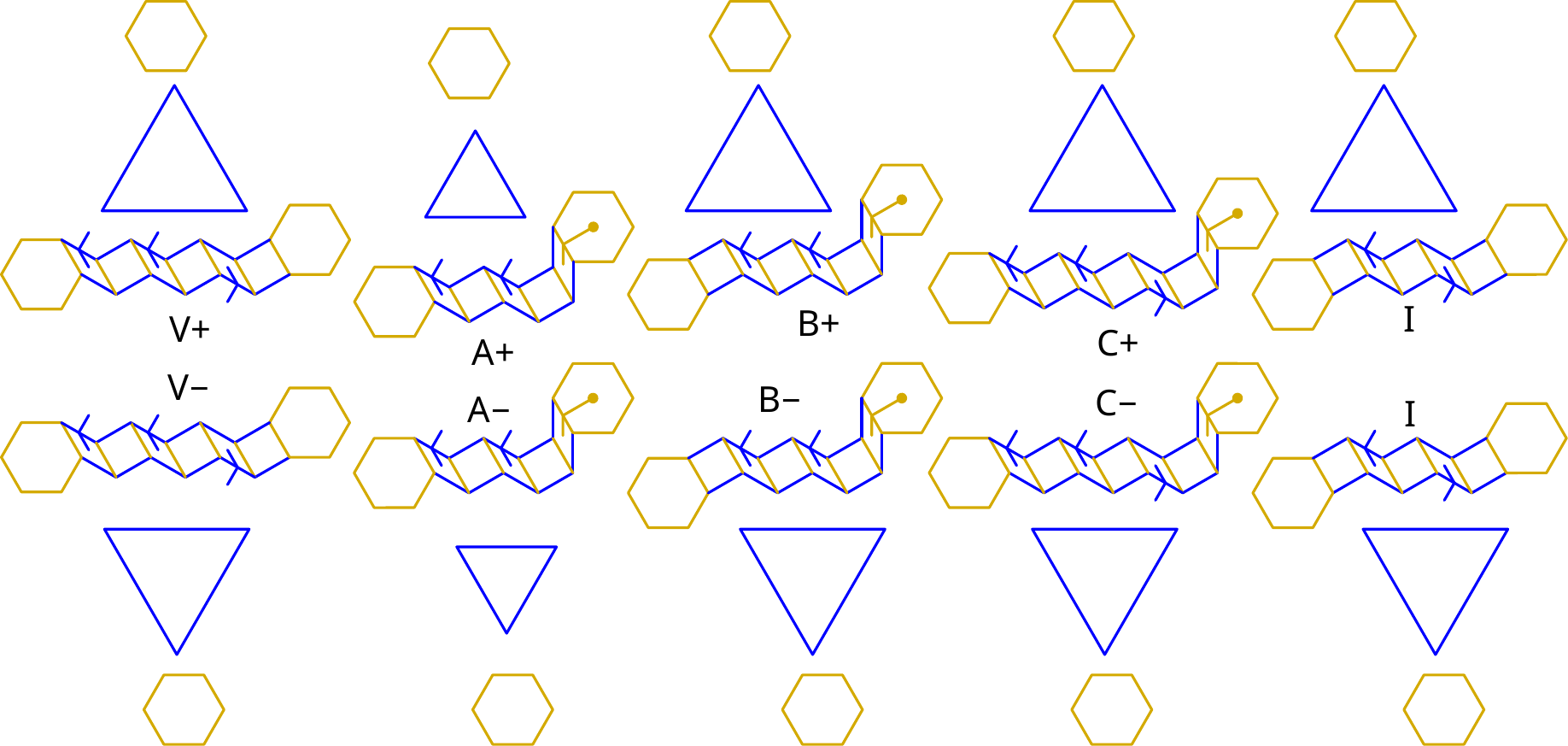}
}

Let us be more explicit. In the figure above, the triangles represent cc's.
Consider an interface of a cc and its rhombs, with markings pointing towards the cc (towards the triangle in the sketches), or away.
We know that in the cc, there is a hex matching with each rhomb whose marking points towards the triangle. The rhombs whose marking points away will point towards the adjacent triangle, for which we thus know that there is a matching hex.

Now replace all hexes by $\iD3$'s, whose orientations are given by the marking (it points to an indent of the $\iD3$), and replace all rhombs by $\iD2$'s.
They will be disjoint and if the conditions of the theorem are satisfied, we pair all $\iD3$ and $\iD2$ and get disjoint Spectres.

There remains to check that the Spectres cover the whole plane. In our situation the whole plane is tiled by hexes, rhombs and squares, so there remains to check that the squares are covered.
All squares belong to an interface and inspection of all the pieces in \Cref{fig:labeled-cc-2} (take the hexagon orientation into account, knowing that the marking points to an indent of the $\iD3$) show that at least three of their four sides are covered by a dent: the only side which may fail to be is the outer part of the interface (alternatively, the way the pieces were made already ensures that). But since that part is the inner part of the same interface for the adjacent cc, it is covered too.
\end{proof}

Rule \ref{item:yellow-rule} is important: in the figure below the central yellow hex has no matching yellow rhomb. In other examples, there could be more than one yellow rhomb for only one hex.

\image{}{}{
\includegraphics[scale=0.7]{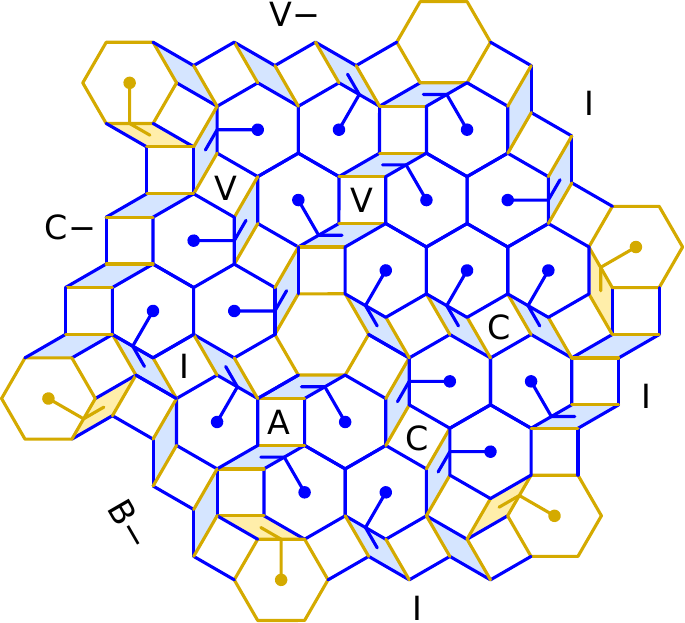}
}

\subsection{Contracted graph and partition of the honeycomb}\label{sub:condensed}

\subsubsection{For a general \texorpdfstring{$\iD3$}{D3} and \texorpdfstring{$\iD2$}{D2} tiling}

This section is not strictly needed for the analysis done here and can be safely skipped.

\medskip

Given a D2/D3 tiling and its yellow/blue decoration graph, we may contract every yellow hex to a point, and every square and rhomb to a blue segment, by collapsing each parallel yellow lines as below:

\nopagebreak

\image{}{}{
\includegraphics[scale=1]{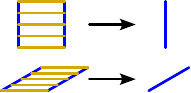}
}

All the collapsing classes on an example:

\nopagebreak

\image{}{}{
\includegraphics[scale=0.66]{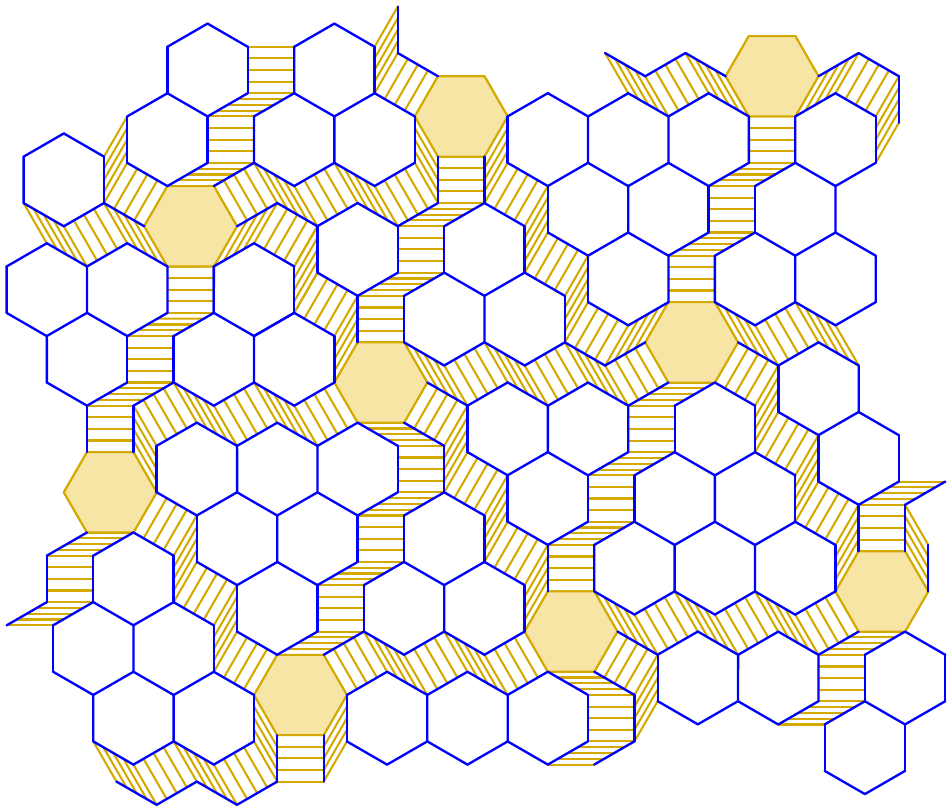}
}

Collapsing only the following parallel strand:

\nopagebreak

\image{}{}{
\includegraphics[scale=0.66]{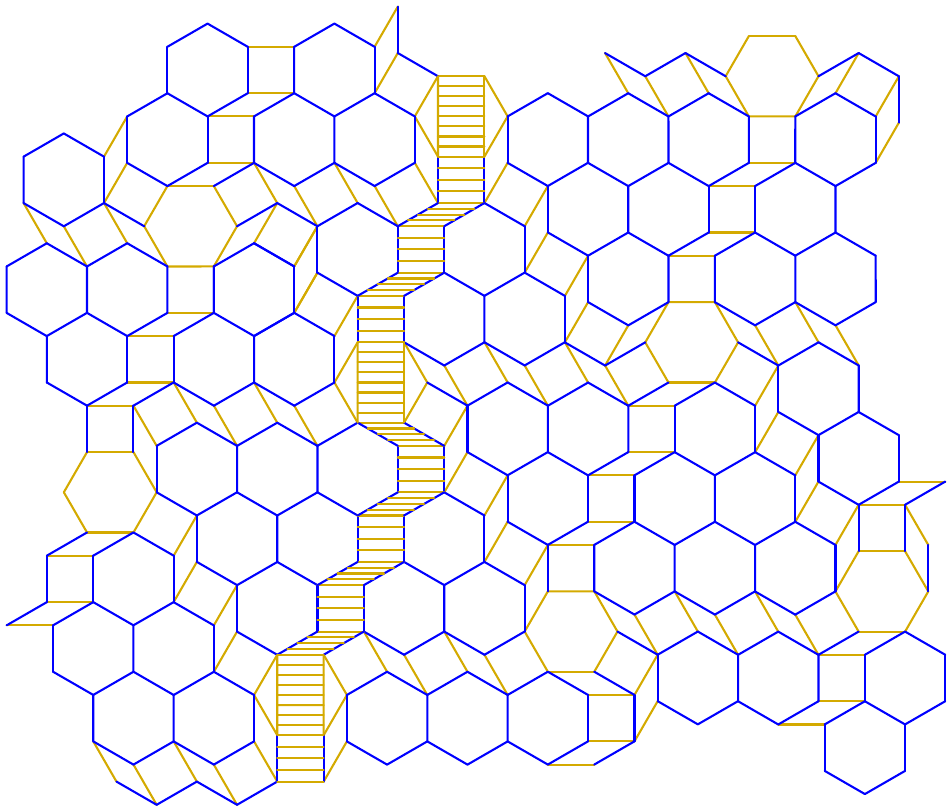}
}

gives:

\nopagebreak

\image{}{}{
\includegraphics[scale=0.66]{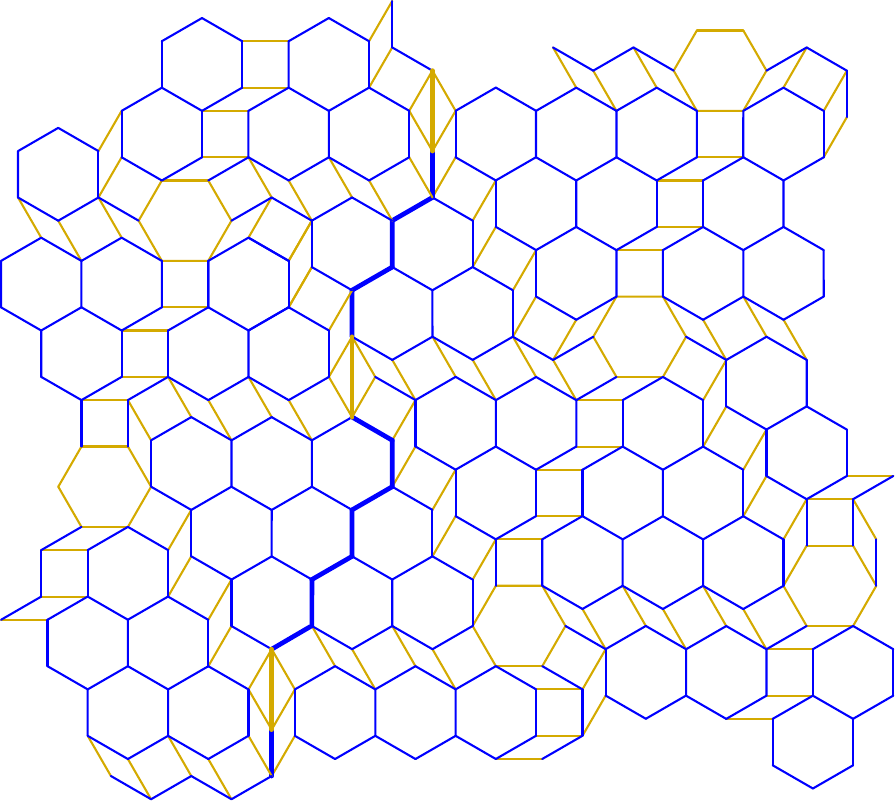}
}

The contracted graph is a subset of the honeycomb.
It may or may not be a strict subset.
For instance, if one starts from a tiling made solely of $\iD2$ tiles\ldots

\image{}{}{
\includegraphics[scale=0.66]{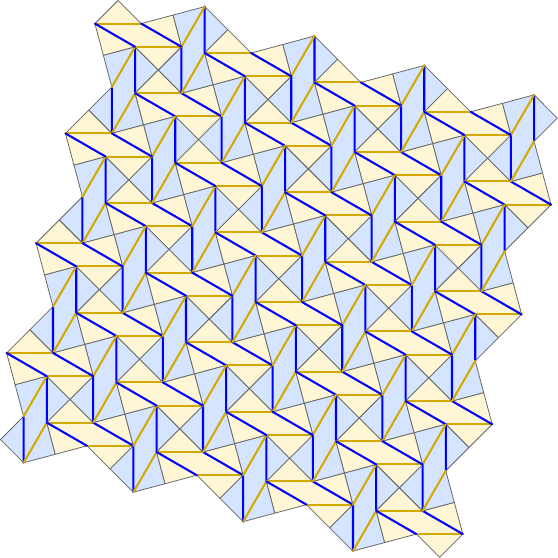}
}

\ldots\ then after collapsing all the yellow segments, there remains only one broken blue line.

\nopagebreak

\image{}{}{
\includegraphics[scale=0.65]{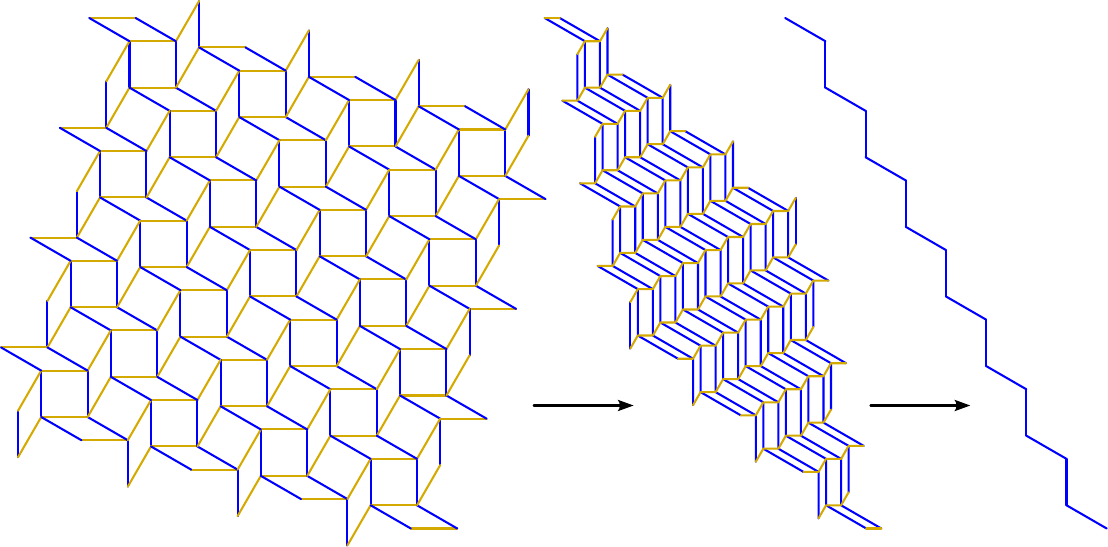}
}

On the decoration graph a whole plane tiling by D2/D3's, each fracture extends both ways to infinity.
Because we can proceed fracture by fracture, the quotient actually naturally maps to the plane so that every blue segment and every blue hex of the non-quotiented graph is injected in the collapsed graph by a translation.

\begin{remark*}
A word of caution: the white/black vertex colouring of the improved decoration graph does not pass to the quotient.
However the quotient is a subset of a hexagonal lattice, whose vertices has its own bipartite colouring.
\end{remark*}

The quotient map can be seen as a map from the plane to the plane, sending every blue segment or hex to a blue segment or a hexagon of some hexagonal tiling of the plane (honeycomb).

\begin{definition}
This projection is called the \term{yellow contraction}. \end{definition}

\begin{proposition}\label{prop:contr-inj}
The yellow contraction maps different blue hexagons of the decoration graph to different hexagons of the honeycomb.
\end{proposition}
\begin{proof}
We saw that the yellow contraction can be obtained by collapsing fractures one after the other. In such collapses, the blue hexagons (their interior) remain disjoint.
\end{proof}

As we have seen, surjectivity depends on the $\iD3$/$\iD2$ tiling under study.
We will see that if it comes from a Spectre tiling, then it is always surjective.

\medskip

Similarly we define the \term{blue contraction}, mapping yellow segments and hexes to (a subset of) a yellow hexagonal tiling of the plane.
Every cc is contracted to a point by the blue contraction, hence the cc is a connected component of the preimage of a vertex by the blue contraction.
Actually, the cc must be the whole preimage of the vertex: this can be proved by uncollapsing the fractures one by one: preimages remain connected throughout the process, and actually \emph{simply connected}.
(This is also a consequence of Moore's theorem relating plane quotients and end of isotopies.)

\subsubsection{For Spectre tilings}

In this section we assume that the $\iD3$ and $\iD2$ tiling comes from a Spectre tiling.
Consider the decoration graph and the tessellation into hexes, squares and rhomb it induces.
Consider as above the yellow contraction: it induces a mapping of the plane hosting the decoration graph to the plane hosting a pure regular hexagonal tessellation (honeycomb).

\begin{proposition}\label{prop:spec-cond}
The yellow contraction induces a bijective correspondence between the blue hexes of the decoration graph and the hexagons of the honeycomb.
\end{proposition}
\begin{proof}
Injectivity has already been seen in \Cref{prop:contr-inj}.
For surjectivity it is enough to prove that for any hexagon $H$ that is an image of a hex, any adjacent hexagon $H'$ is also an image of a hex.
Let $B$ be a blue hex whose image is $H$.
The edge of $H$ along which $H'$ is adjacent, corresponds to an edge $e$ of $B$. If there is a hex adjacent to $B$ along $e$ then we are done.
Otherwise, from $e$ starts a chain of rhombs and squares, whose blue edges are parallel to $e$. Such a chain cannot be infinitely long by the lemma below, so must end on a hexagon. The image of this hexagon will then be $H'$.
\end{proof}

Above, we used the following result:

\begin{lemma}\label{lem:no-spec-rs-chain}
Consider a chain of squares and rhomb adjacent along their yellow edges, which are thus all parallel. The number of elements of this chain is bounded for Spectre tilings. The same holds for adjacency via blue edges.
\end{lemma}
\begin{proof}
We will give a proof that is independent of \Cref{cor:odds} and \Cref{prop:odd-env-1}, so we can assume the adjacency edges are yellow.

Because dents protrude from rhombs and must come inside squares,
squares and rhombs must alternate along the chain.
We will call \term{consecutive} rhombs of the chain that touch via a single vertex, i.e.\ those that are adjacent to the same square of the chain.
Consider the colour of the rhombs, i.e.\ the colour, yellow or blue, of the associated $\iD2$ according to \Cref{fig:col-conv-spec}.

Because we are coming from a Spectre tiling, yellow rhombs must have one dent in a yellow hex and this dent is supported on a yellow rhomb edge.
So if there is a yellow rhomb, it terminates the chain.
We saw in the proof of \Cref{prop:no-full-blue} that the chain cannot contain $5$ consecutive blue rhombs (\Cref{fig:chain}).
It follows that the chain is limited to 6 rhombs and 5 squares.\footnote{Actually they are shorter: we saw in \Cref{fig:labeled-cc-2} that for one colour of the parallel direction there is at most 4 rhombs and 3 squares, for the other, it follows from the same figure and \Cref{fig:odd-env-2} that there is at most 2 rhombs and 1 square.}
\end{proof}

A consequence of \Cref{prop:spec-cond} is that the partition of the blue hexes into clusters induces a partition of the honeycomb into triangular groups of $1$, $3$ and $6$ hexagons.
\Cref{fig:bo} showed an example of contracted graph, where we used thick lines keep track of the $\hT n$.
\Cref{fig:trace} showed an example where we added the information of the position of the contracted yellow hexes using dots (in this picture the dots have been coloured using the bipartite colouring of the vertices of the honeycomb).

\begin{proposition}\label{prop:ch2sp}
  The data of the partition together with the dots (P$+$D) allows to recover the initial Spectre tiling.
  This can be done with local inspection only: the recovery of the orientation of the tile corresponding to a given blue hex or to a dot can be done by looking at a small neighbourhood of the hex/dot in the P$+$D pattern, of universal size, and the recovery of a connected patch of any size by inspecting each of its individual tiles.
\end{proposition}
\begin{proof}
Inspect the figure below, which is deduced from \Cref{fig:odd-env-2}.

\image{Automatic environment of a yellow hex, contracted version.}{fig:odd-env-2f}{
\includegraphics[scale=0.23]{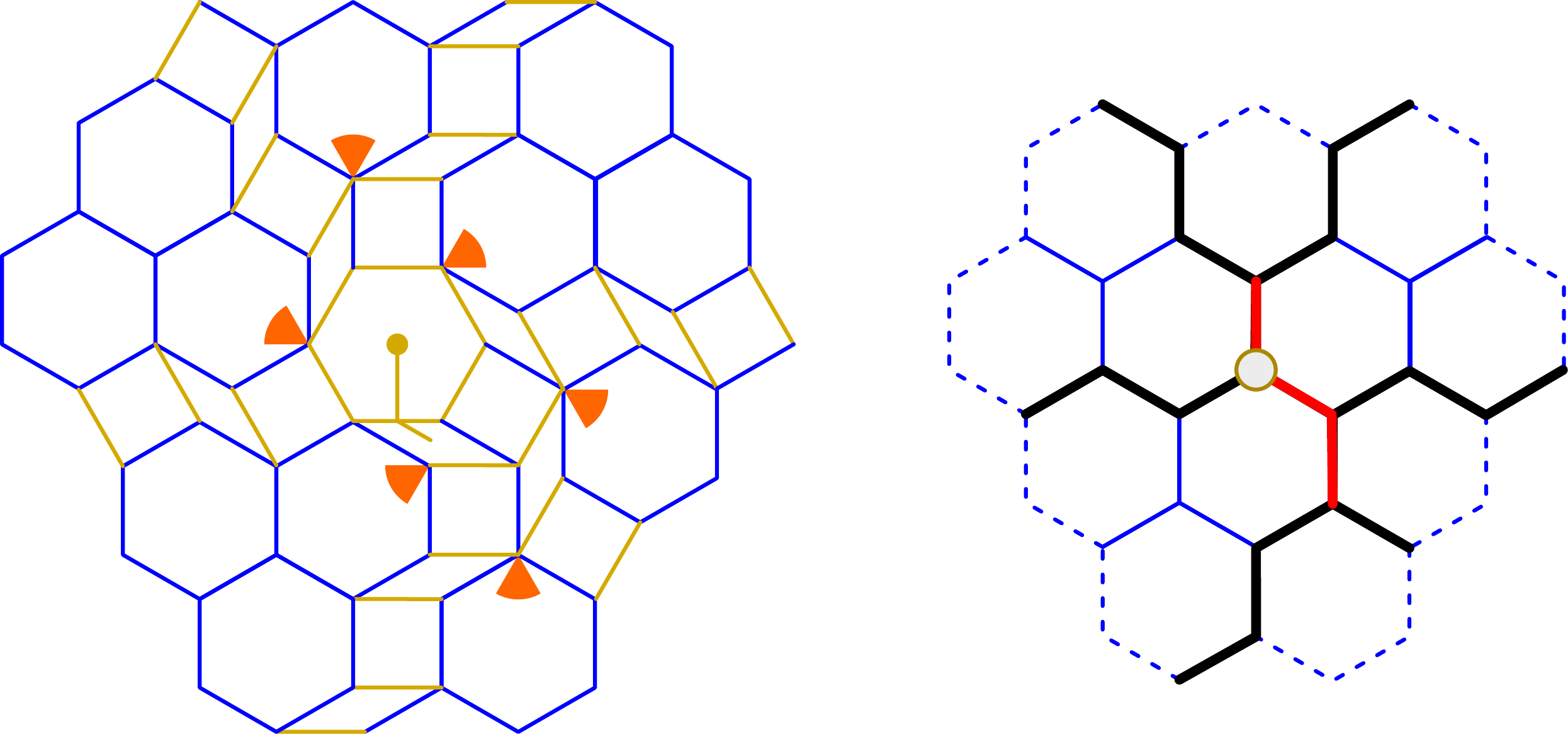}
}

In the above figure, the left part recalls the automatic local environment of a yellow hex (i.e.\ of an odd tile) in the yellow/blue decoration graph.
It shows parts of six clusters, each of which may be complete or extend beyond the picture.
We added orange wedges to indicate tips for each of the clusters.
Contraction of the yellow segments gives the right part.
The gray dot marks where the yellow hex has been contracted. The thick red and black lines must be contained in the graph $\cal G$ consisting of the boundaries of all hex clusters.
The solid blue lines in the middle picture are not in the boundary of a cluster and the dashed one may or may not be.
The red line indicates where the antennae have been contracted.
The right figure gives us the automatic environment of a dot in $\cal G$.
From the asymmetry of the right figure around the dot, we know how the environment of the yellow hex is oriented. From the left figure, the presence and shape of antennae associated to each of the indicated tips of the six clusters is deduced.
Since we started from a Spectre tiling, we know that every cluster tip is associated to a dot, so the P$+$D data gives us the presence and nature of antennae for all tips.
\end{proof}

With a little bit of training, one learns to quickly identify the antennae in the P$+$D drawings, using the right part of \Cref{fig:odd-env-2f}.

\medskip

We can also determine which partitions together with dots (P$+$D) come from a Spectre tiling: 

First, on the figure below, on the first three columns we present the allowed cc shapes together with the three yellow hex they touch, contracted to a dot. (The fourth cc shape will never appear in a situation coming for a  whole plane Spectre tiling but this does not matter at this point.) We still consider that the boundary of the cc is the union (not necessarily disjoint) of broken lines between dots, that we still call an \term{interface}.

\image{}{fig:contracted-pieces}{
\includegraphics[scale=0.5]{inter-ant-4g.pdf}
}

On the right column of the figure above, we stress the following important fact:
\begin{lemma}\label{lem:iss}
the nature of an interface only depends on its shape as a broken line.
\end{lemma}
It follows from pure observation of \Cref{fig:labeled-cc-2,fig:interface-list}.
Note that on three of the possible interface pairs, we highlighted a particular segment in pink: the ones of type $\iA$, $\iB$ and $\iC$.
It follows directly from \Cref{thm:pieces-ne,thm:pieces-su} that:

\begin{proposition}\label{prop:ces}
The data of a partition $P$ of the honeycomb $HB$ and a collection $D$ of dots on some of the vertices of $HB$, comes from a Spectre tiling if and only if: 
\begin{enumerate}
\item It is an assemblage of pieces of \Cref{fig:contracted-pieces}, where two pieces in contact either only touch on one dot of each piece, or they are in contact exactly along a whole identical interface between two respective dots.
\item Any dot in $D$ has one and only one of the three segments of $HB$ ending on it which is a pink segment of an interface pair as on \Cref{fig:contracted-pieces}.
\end{enumerate}
\end{proposition}

Moreover, since in \Cref{fig:labeled-cc-2}, the interface completely determine the rhomb/hex pairing for the blue hexes, and the yellow rhomb/hex pairing is determined by the unique $\iA$, $\iB$ or $\iC$ interface going to it:

\begin{proposition}\label{prop:cts}
Given data as in the previous proposition, the Spectre tiling is completely determined.
\end{proposition}

We end this section by a property that will be used to prove later a relation between various coordinates (\Cref{prop:B-Y-G-coord}).
We make use of complex numbers. It can be safely skipped since it is not used for most of the statements. The proof is immediate.

\begin{lemma}\label{lem:coord}
Let $\rho=e^{2\pi i /6}\in\C$ denote the principal primitive \nth{6} root of unity. Assume that the centre of the hexagons of the blue honeycomb forms the lattice $\Lambda = \Z+\rho\Z$, and hence the sides vectors of the hexagons are of the form $\rho^k c$ with\footnote{Alternatively we can take $c = 1/(\rho+\rho^2)=i\sqrt{3}$.} $c=1/(1+\rho)$.
Then for each interface in the right column of \Cref{fig:contracted-pieces}, the vectors from the left dot to the right dot are respectively
\begin{itemize}
\item for $\iI$-type, $(2+3\rho)\times c$;
\item for $\iA$-type, $(1+3\rho)\times c$;
\item for $\iB$-type, $(1+4\rho)\times c$;
\item for $\iC$-type, $(2+4\rho)\times c$;
\item for $\iV$-type, $(3+3\rho)\times c$.
\end{itemize}
On any valid hex partition with dots, the vector between two dots on an interface is the product by a power of $\rho$ and of the vector given above for the interface type $\iI$, $\iA$, \ldots, $\iV$.
\end{lemma}
For type $\iA$, we used that $\rho^2=\rho-1$ to transform the expression $2+2\rho+\rho^2$ into $1+3\rho$. We did similarly for types $\iB$ and $\iC$.

\subsection{Equivalent tilesets}\label{sec:triset}

By \Cref{thm:pieces-ne,thm:pieces-su}, tiling the whole plane with the Spectre is combinatorially equivalent to tiling the whole plane with the following marked regular triangles following the rules:
\begin{enumerate}
\item the interfaces match (disregarding the yellow hex arrows),
\item each vertex is in contact with one and only one yellow segment, where
\end{enumerate}
the two yellow segments superimposed in $A+/A-$, $B+/B-$ and $C+/C-$ interfaces are counted as one.

\image{Some triangles have been rotated for faster visual lookup of markings.}{fig:triset-1}{
\includegraphics[scale=0.55]{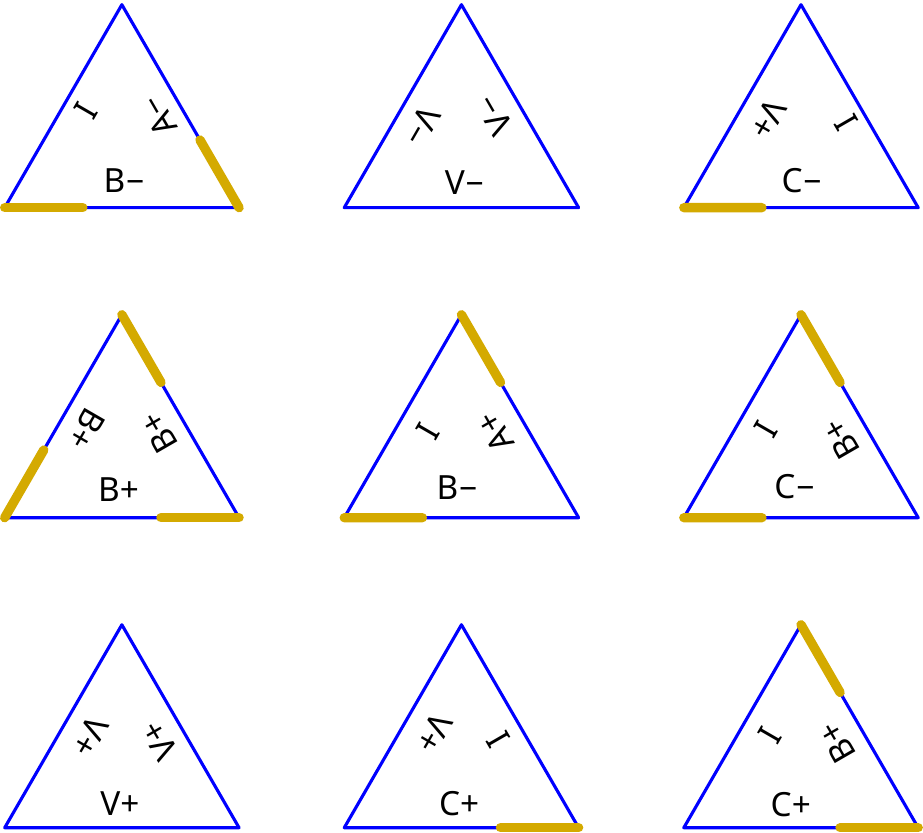}
}

Alternatively we can use the following tilesets, each of which enforces the yellow edge rule.

\nopagebreak

\image{A correct tiling with these tiles have to tile the plane in the classical sense and respect the markings.}{fig:tri-6b}{
\includegraphics[scale=0.4]{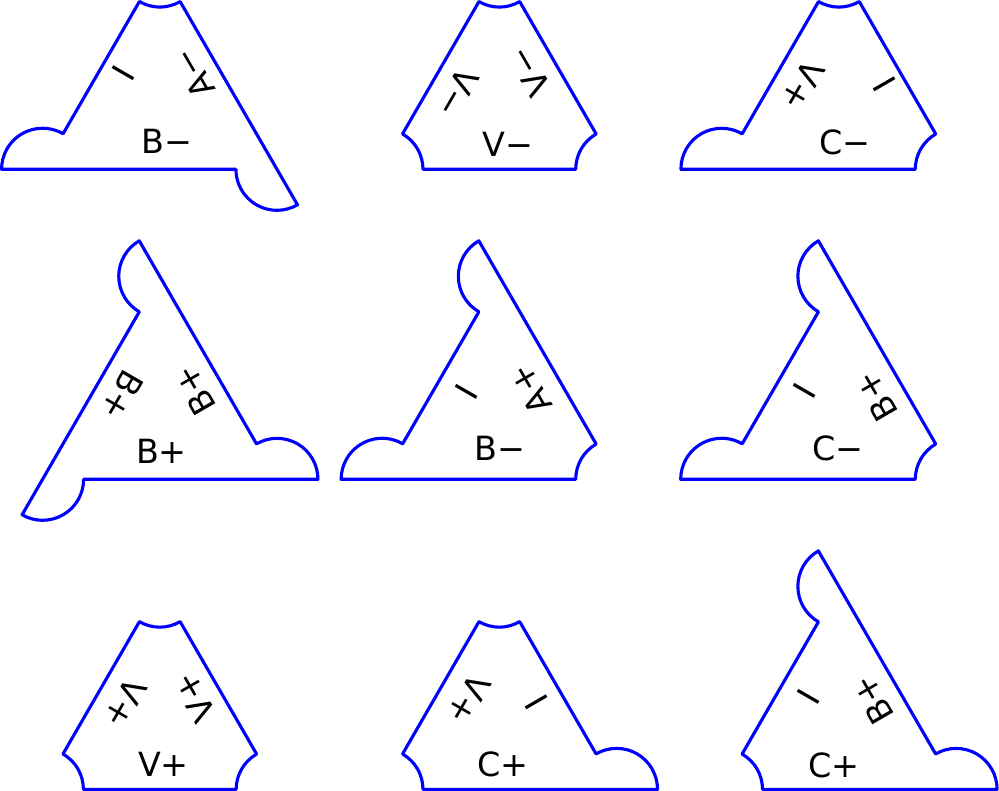}
}

or:

\nopagebreak

\image{Two more equivalent tilesets. The yellow marking at some circle centres is optional and indicates the orientation of the underlying yellow Spectre.}{fig:triset-2}{
  \makebox[\textwidth][c]{\begin{tikzpicture}
    \node at (0,0) {\includegraphics[scale=0.4]{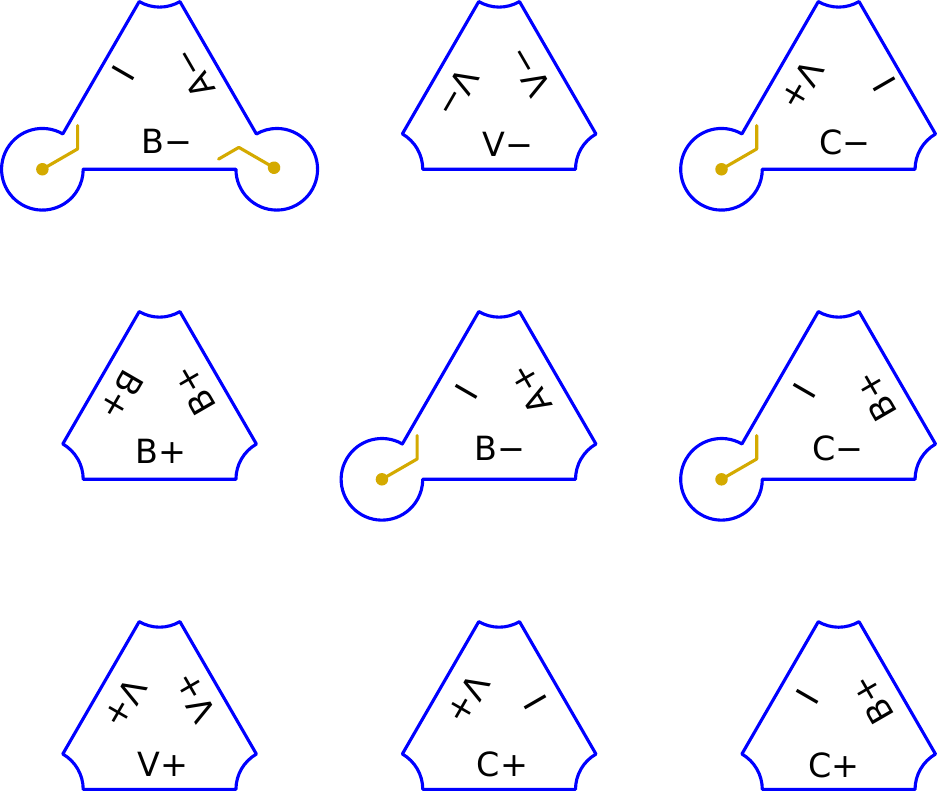}};
    \node at (8,0) {\includegraphics[scale=0.4]{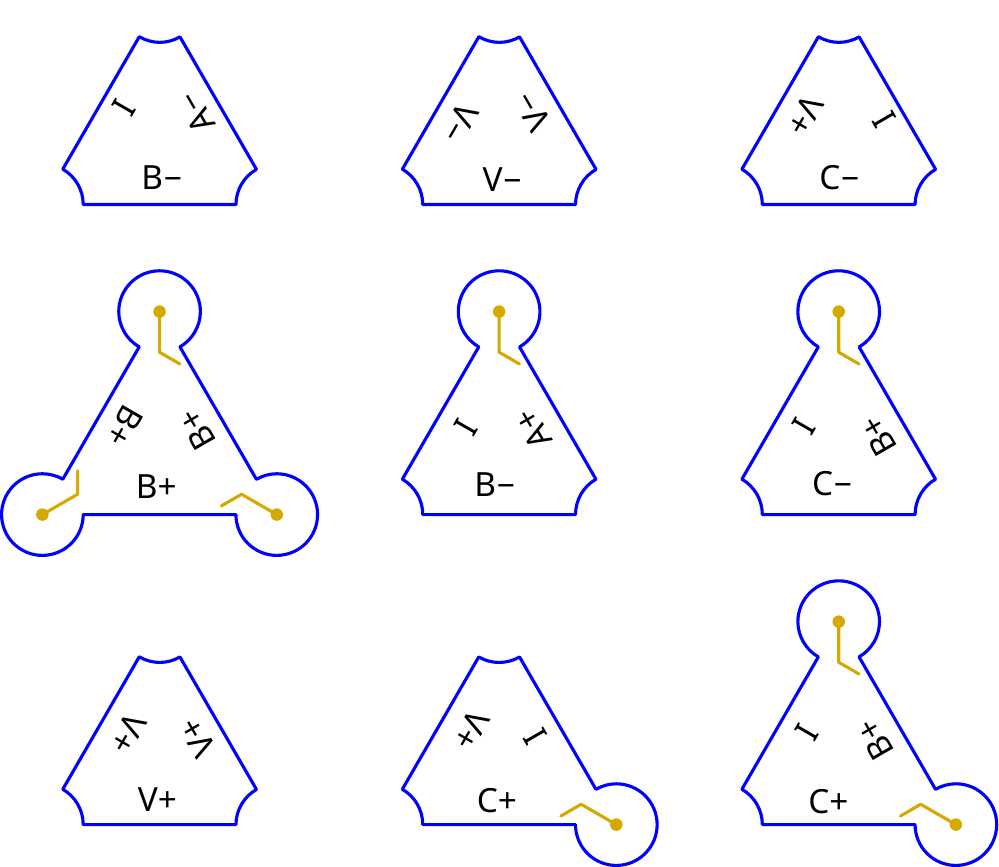}};    
  \end{tikzpicture}}
}

Tiling with all those tilesets are combinatorially equivalent.
Note that we can add some supplementary marking as below to ease laying them out:

\image{Tilings with the pieces of \Cref{fig:triset-1} automatically respect the beige markings, and at each vertex, stem \emph{alternatively} unmarked and marked half-segments. Recall that to correspond to a Spectre tiling, we must also follow the rule that every vertex should have exactly one yellow marking, as shown on the left.}{fig:triset-1c}{
\includegraphics[scale=0.55]{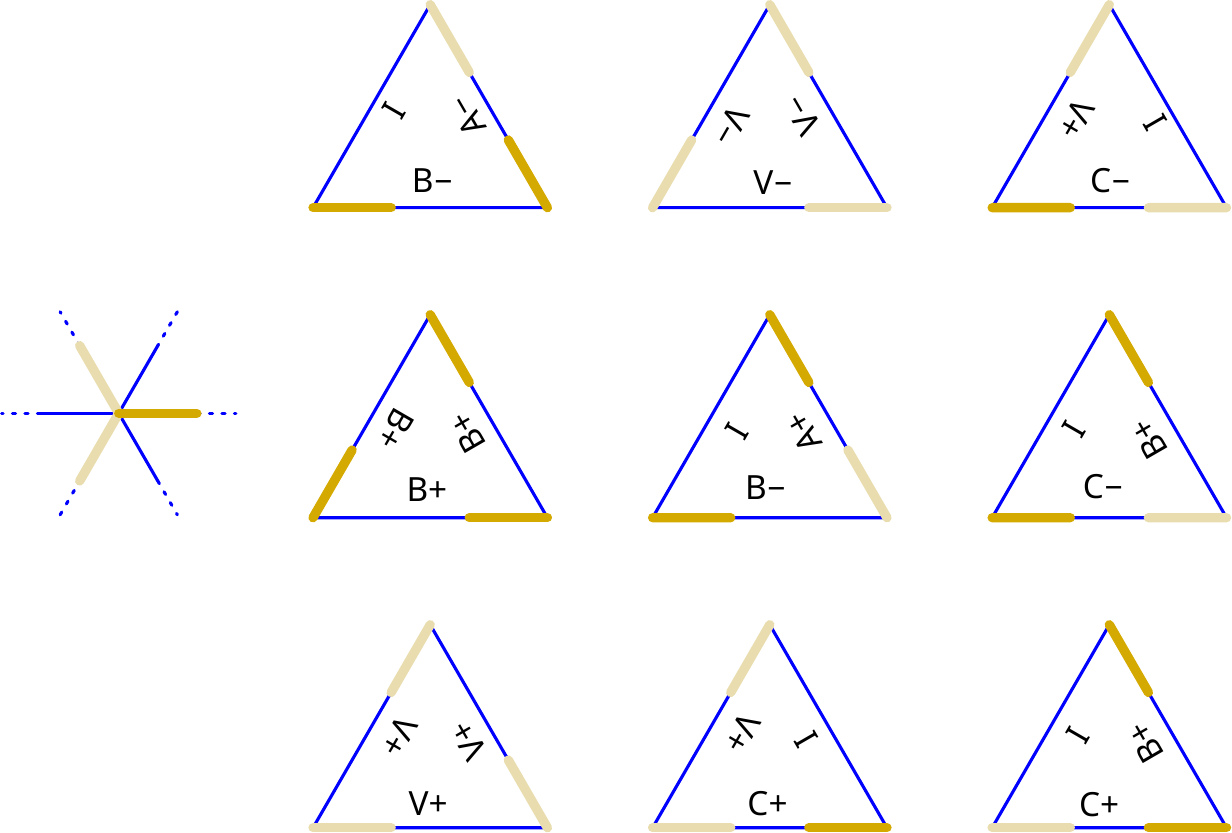}
}

It can be directly checked that the interface labels must respect the new introduced colour. It can also be deduced using \Cref{thm:pieces-su}, and the fact that it corresponds to the orientation of the yellow hexes.

With these easier depictions to deal with, let us first prove:

\begin{proposition}
The triangle marked $\mathrm{B}+$, $\mathrm{B}+$, $\mathrm{B}+$ cannot appear in a whole plane tiling.
\end{proposition}
\begin{proof}
Let us use the tileset of \Cref{fig:triset-1}
On the bottom line can only go two types of marked triangle but the one on the left creates a vertex with two yellow lines.

\image{}{}{
\includegraphics[scale=0.5]{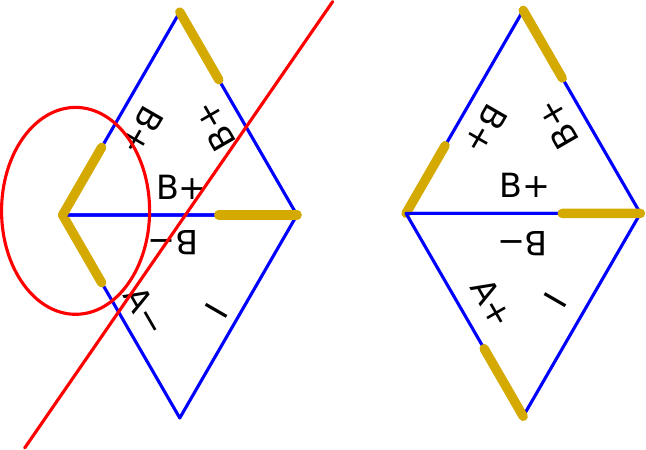}
}

By symmetry this forces the 3 adjacent tiles to the initial one.
Since there is only one $\iA-$ we have:

\image{}{}{
\includegraphics[scale=0.5]{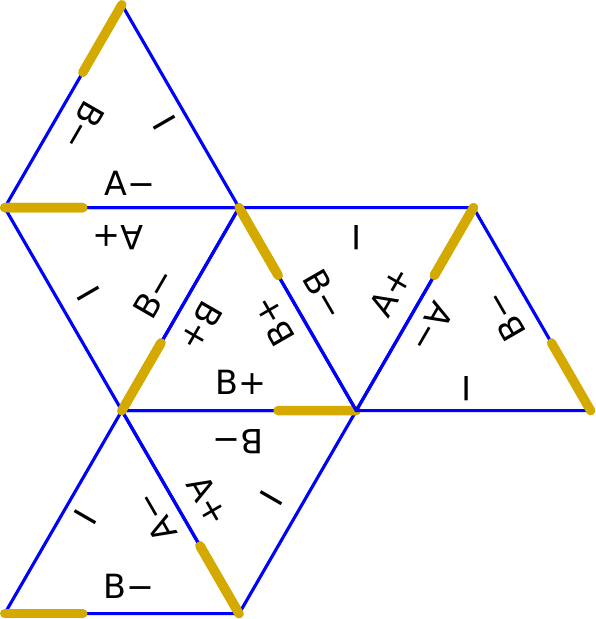}
}

Then around a corner where there is already 4 triangles there are two $\iI$ interfaces.
Moreover the $\iI$ that must fit can have no yellow segment touching its ends, because the placed tiles already have such segments.
There is only one such triangle, with markings $\iI$, $\iV+$, $\iC-$.
So we have

\image{The last two tiles do not fit: the interfaces circled in red do not match.}{}{
\includegraphics[scale=0.5]{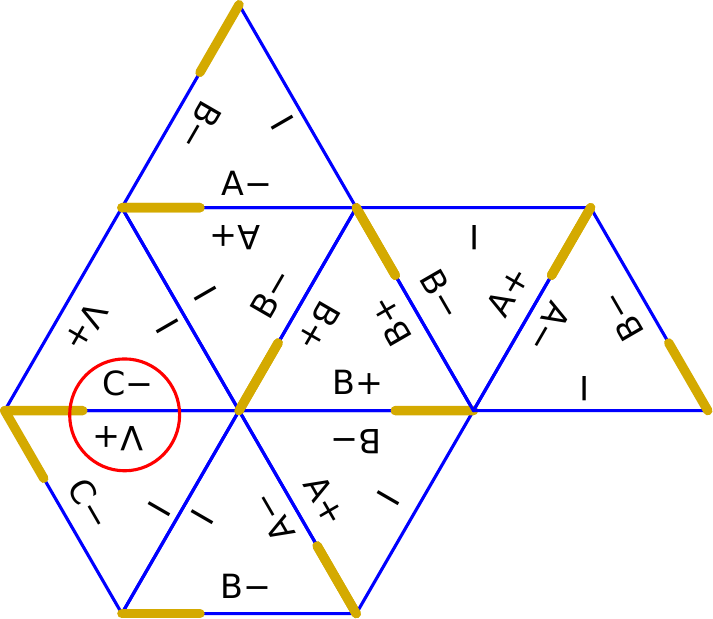}
}

but the two tiles just added do not fit with the rest.
\end{proof}

\subsection{Triangle vertex neighbourhood}

Among the several proposed equivalent tilesets presented before, we choose to use the following pieces:

\nopagebreak

\image{The lightness of the blue shade has been chosen to reflect the size of the $\hT n$ of the corresponding cc: light blue: $\hT1$, medium blue: $\hT2$, dark blue: $\hT3$.
The white and black dots must match with dots of the same colour and represent the orientation of the yellow hex (they correspond to the white and black dots on the improved decoration of the yellow hexes, and the white dots also correspond to the beige and yellow markings of \Cref{fig:triset-1c}).
Each triangle ``vertex'' must be as on the centre left. We recall that the white dots are a visual aid and are not necessary: a whole plane tiling with the pieces above, disregarding the white dots, automatically has to match them.}{fig:triset-7h}{
\includegraphics[scale=0.6]{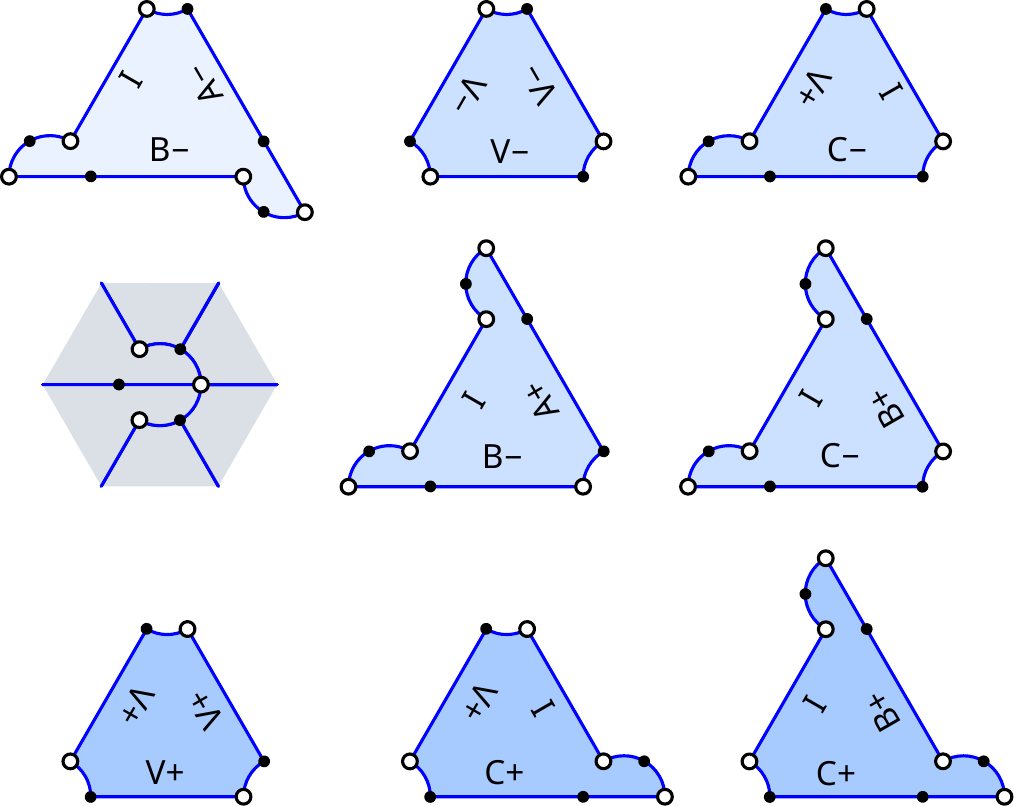}
}

These triangles will arrange as a regular triangular lattice with circular sectors added/removed. The vertices of the underlying triangular lattice are not indicated in the figure above but are ``implied''. We need a name to distinguish them from the actual vertices of the pieces in the figure.

\begin{definition}
We call \term{triangle vertices} the vertices of the triangles that the pieces are built from (before addition/removal of the circular sectors).
\end{definition}

The complete list of compatible arrangements around a triangle vertex is not very long to figure out by hand (not all may appear in a complete plane tiling):

\nopagebreak

\image{In this figure all the central yellow hexes have the same orientation of the associated marking (telling where is the paired yellow rhomb), depicted on top.}{fig:vertex-list}{
\includegraphics[scale=0.36]{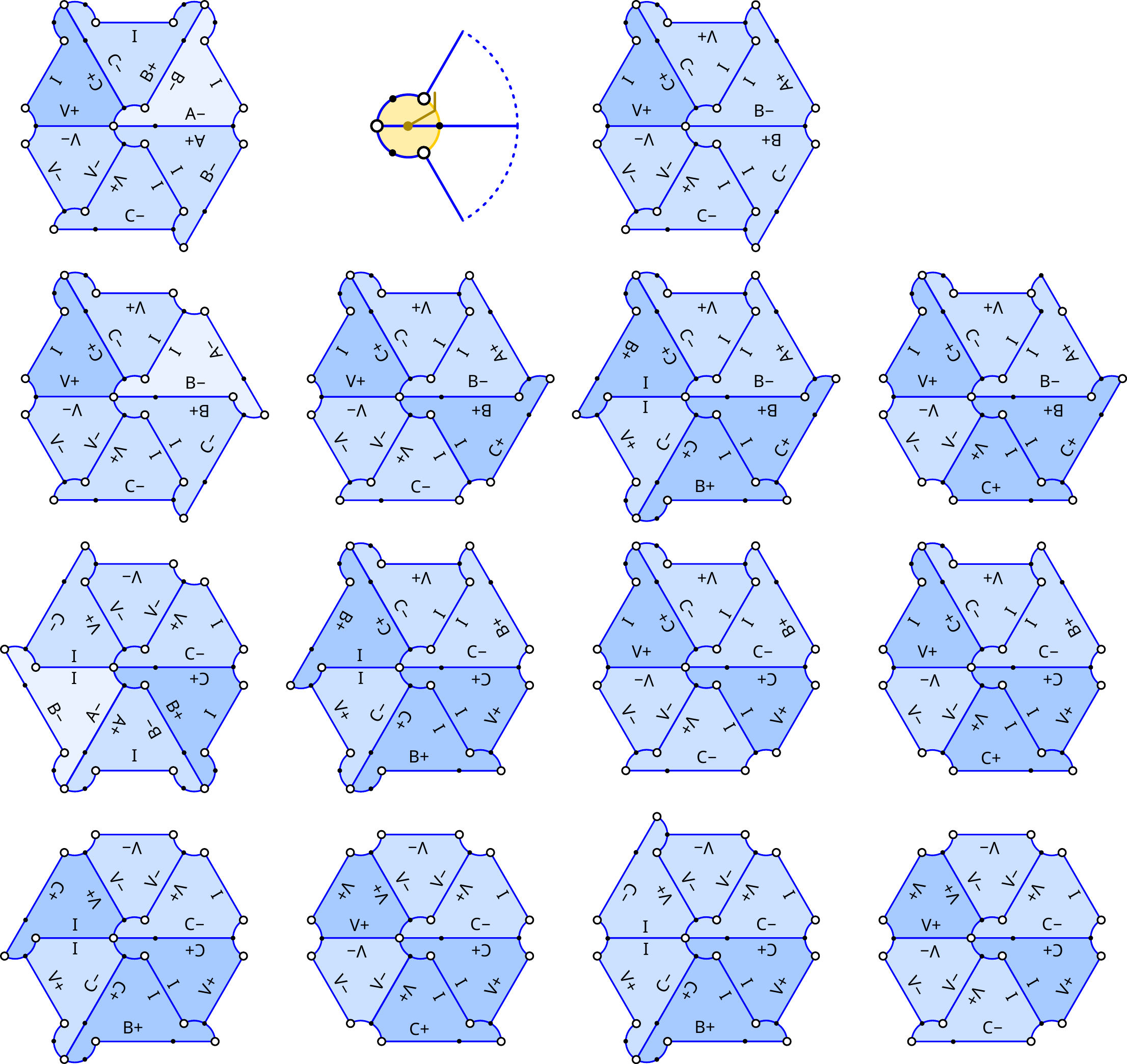}
}

\subsection{Yellow clusters \texorpdfstring{$\hT'n$}{T'n}}\label{ss:hTp}

\subsubsection{Definitions}

To prove that tilings with these triangles are uniquely hierarchical, we will identify the next generation of triangles as spotted in the introduction.
We will call them $\hT'n$. They are supposed to be like the triangular clusters $\hT n$ but for the yellow hexes.
Observation guides us into considering candidates following the rules:
\begin{enumerate}
\item A $\iV$-type interface should link two yellow hexes in the same cluster. 
\item An $\iI$-type interface should link two adjacent cluster.
\end{enumerate}

So we start enforcing this as terminology:

\begin{definition}\label{def:V-clu}
In a whole plane tiling by the pieces of \Cref{fig:triset-7h}, we declare as \term{$\iV$-connected} any two triangle vertices that are linked by a $\iV$-type segment.
Components for this connection relation will be called \term{yellow clusters}. We use this terminology as well for the yellow hexes and yellow Spectres of the associated Spectre tiling.
\end{definition}

Recall that the vertices correspond to yellow hexes, which is why we chose that name for this higher level cluster notion.
We will prove in the present section that yellow clusters come in triangular sets $\hT'n$ for $n=1$, $2$ and $3$.

\begin{definition}\label{def:I-adj}
We declare two components \term{adjacent} if they are linked by an $\iI$-type triangle edge.
\end{definition}

For the moment this definition is a bit artificial. A priori a cluster could be adjacent to itself. We will see that this cannot be the case.
Keep also in mind that two yellow clusters that are linked by another type of triangle edge may or may not be adjacent according to the definition above.

\subsubsection{Shape of yellow clusters}

Note that there is only one triangle having a $\iV-$ side, and it is the one of type $\iV-$, $\iV-$, $\iV-$. 
So any two $\iV$-connected vertices are on the edge of some $\iV-$, $\iV-$, $\iV-$ triangle.

Note also that any tile with a $\iV+$ side has a $\iV-$, $\iV-$, $\iV-$ tile attached to this side.
In particular the tile of type $\iV+$, $\iV+$, $\iV+$ has a
$\iV-$, $\iV-$, $\iV-$ tile on each side:

\image{}{}{
\includegraphics[scale=0.5]{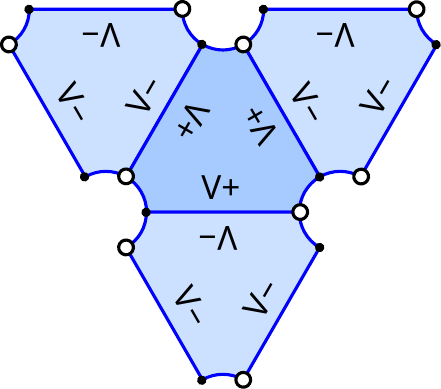}
}

\begin{proposition}\label{prop:Tp3}
Two triangles of type $\iV-$, $\iV-$, $\iV-$ that share a triangle vertex are necessarily part of a configuration as above.
Such a configuration shares no triangle vertex with any other type $\iV$ interface.
The triangle vertices of the configuration thus form a $\hT'3$.
\end{proposition}
\begin{proof}
Because the white and black dots must alternate around a vertex, two $\iV-$, $\iV-$, $\iV-$ triangles sharing a triangle vertex must be oriented similarly, thus must be adjacent to a common triangle. The latter must have two $\iV+$ interfaces, so can only be of type $\iV+$, $\iV+$, $\iV+$.
The third edge of the latter must match a $\iV-$, so a $\iV-$, $\iV-$, $\iV-$ piece.

If there would be another $\iV$-type edge connecting to one of the triangle vertices of the arrangement, it would imply by the same argument the presence of a $\iV+$, $\iV+$, $\iV+$ piece in contact with the arrangement as on the figure below left.
Then there would be a $\iV-$, $\iV-$, $\iV-$ piece as on below right.
But the triangle vertex hole marked in red cannot be filled.

\image{}{}{
\includegraphics[scale=0.5]{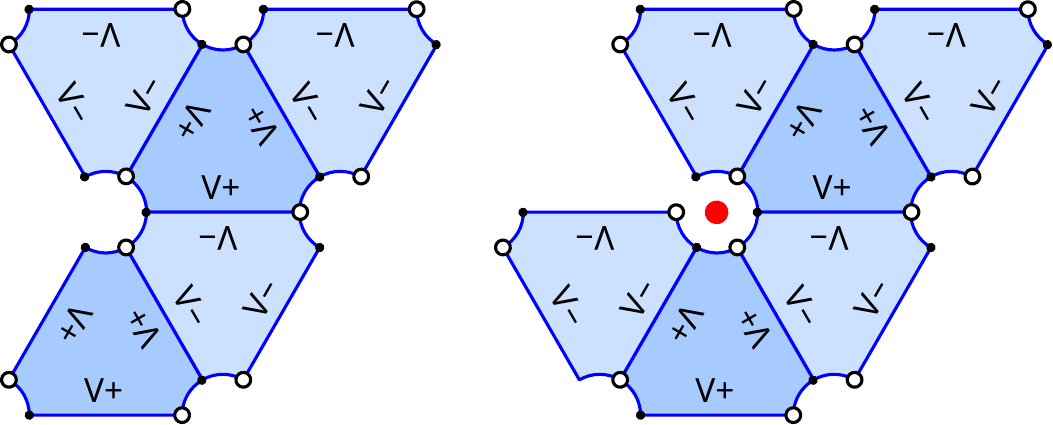}
}
\end{proof}

\begin{proposition}\label{prop:Tp2}
If a triangle of type $\iV-$, $\iV-$, $\iV-$ does not share a triangle vertex with another tile of the same type, then it shares no triangle vertex with any other type $\iV$ edge.
Its triangle vertices thus form a $\hT'2$.
\end{proposition}
\begin{proof}
If such an edge existed, it would be a $\iV+$/$\iV-$ pair, so another triangle with a $\iV-$ would be there.
Since the only triangle with a $\iV-$ is the $\iV-$, $\iV-$, $\iV-$ one, this contradicts the assumption.
\end{proof}

It follows that yellow clusters can only be triangular of type $\hT'n$ for $n=1$, $2$ and $3$, as it was the case for $\hT n$ blue hex clusters.

\medskip

Like for $\hT n$, the $\hT'n$ come in two orientations, which we also call \term{point up} and \term{point down}.
Like for $\hT n$, the orientation is linked to the orientation modulo $1/3$ of the marking of the disk at the centre of its hexagons (i.e.\ \latin{in fine} of the yellow hex of the decorated graph that this disk represents): this comes from the dot marking on \Cref{fig:triset-7h} for the $(\iV+,\iV+,\iV+)$ and the $(\iV-,\iV-,\iV-)$ pieces.
(Alternatively we can argue that on \Cref{fig:labeled-cc-2}, the $\iV$ interfaces all have ending yellow hexes in the same class modulo $1/3$ of the orientation, and this is logically equivalent to the fact that $\iV$ interfaces have an even number of edges on the cc, so link vertices of the same colour in a bipartite colouring of the blue honeycomb.)
This allows to extend the notion of orientation to $\hT'1$ clusters too.

The $\iI$-type interfaces change the orientation class modulo $1/3$: to prove it one can adapt any of the three arguments given for $\iV$ above. It follows in particular that: a cluster is not adjacent to itself.

We will characterize better the adjacency of yellow clusters and prove that it mimics a regular triangle tiling, like the blue clusters did.

\subsection{First properties of the  triangular tileset}

Below we change the information on the tiles and we reorder for even faster visual lookup: 

\nopagebreak

\image{}{fig:triset}{
\includegraphics[scale=0.55]{tri-7f.pdf}
}

In this case the local vertex arrangements look like this:

\nopagebreak

\image{}{fig:vlist}{
\includegraphics[scale=0.36]{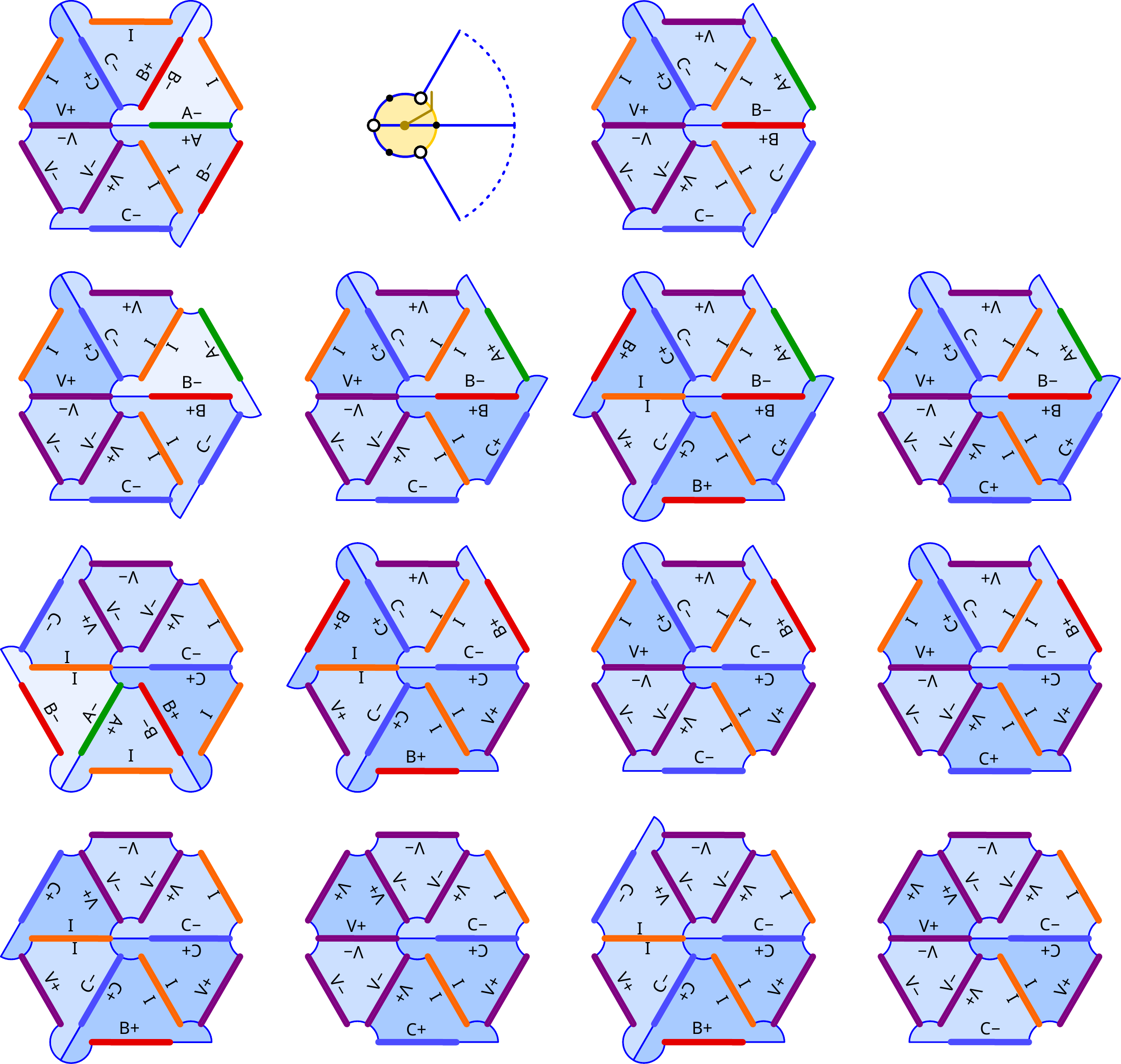}
}

\begin{remark}
In the sequel we are going to make a lot of deductions from those triangle tilings. Several of these may be analogue to those made earlier with Spectre tiles up to \Cref{sub:cc-list} (some the proofs are in \Cref{sec:later}) and there may be a way to unify them by starting at a lower level, see \Cref{ss:pts}.
\end{remark}

\begin{proposition}\label{prop:around-tri}
The only possible arrangement around the unique T1 triangle, i.e.\ the one marked ($\iB-$,$\iA-$, $\iI$), is as follows:

\nopagebreak

\image{}{fig:N1}{
\includegraphics[scale=0.5]{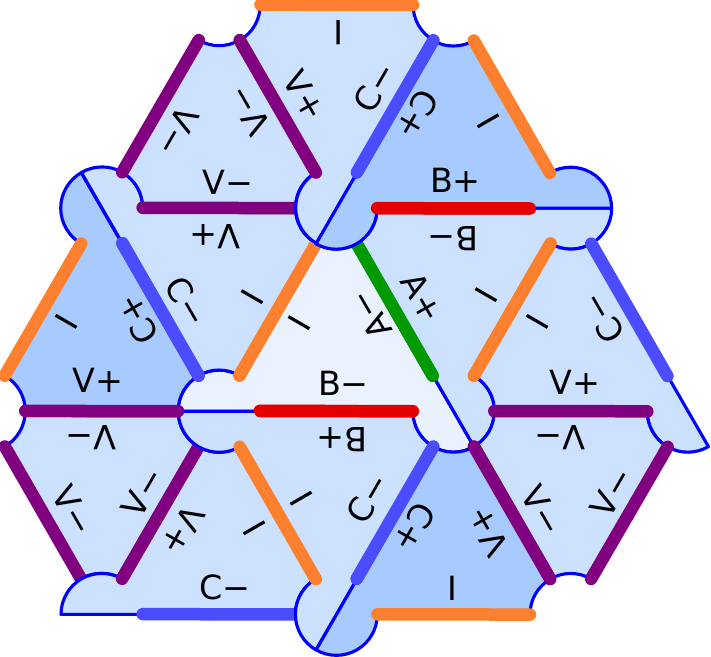}
}
\end{proposition}
\begin{proof}
This follows from \Cref{fig:vlist}. Alternatively we can proceed as follows.

Starting from the $\iA-$ and going around the neighbouring tiles in clockwise order, we see that for the first 5 neighbours, there is each time only one tile that can fit the previously placed tiles: for the first one this is because there is only one triangle with an $\iA+$. For the next, because there is only one tile with an I interface that has both ends empty of half disks (any half disk would overlap with the already present ones). Etc.

\image{}{}{
\includegraphics[scale=0.5]{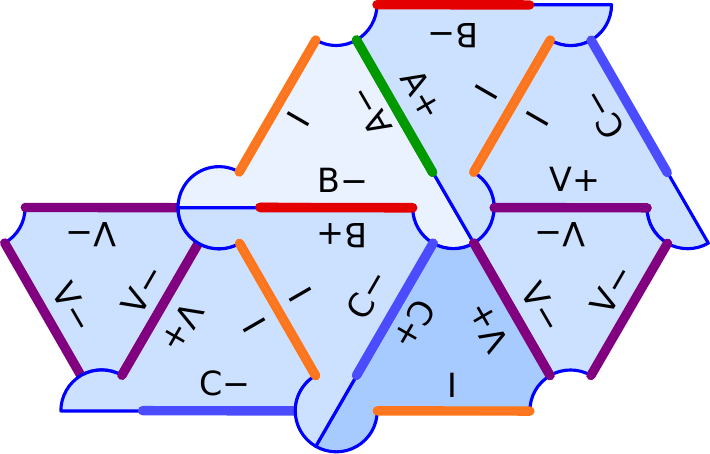}
}

The next tile should have a $\iV+$ interface with a free end on its right (left and right ends of an interface are determined by rotating the triangle so that the interface is at its bottom).
There are only two possibilities. One of them is $\iV+$, $\iV+$, $\iV+$, but it implies the presence of a $\iV-$, $\iV-$, $\iV-$ piece on its right, creating a contact between an $\iV-$ and the $\iI$ of the initial triangle.
So we get the other one, $\iV+$, $\iC+$, $\iI$.
The next two tiles, still in the clockwise order, are uniquely determined by the interfaces too.

\image{}{}{
\includegraphics[scale=0.5]{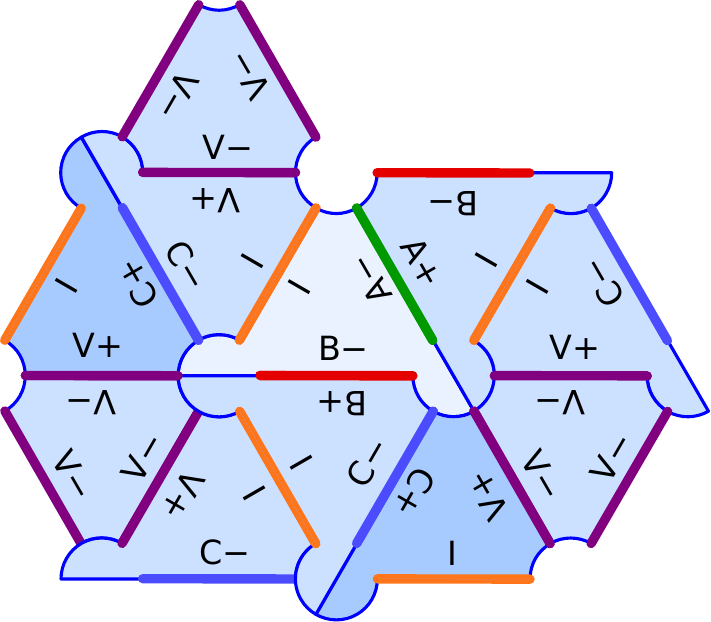}
}

Then on the free $\iB-$, we need a tile with a half-circle on the right of the corresponding $\iB+$, for otherwise the hole would stay empty. There is only one such triangle: $\iB+$, $\iI$, $\iC+$, and then the last one is determined.
\end{proof}

\begin{corollary}\label{cor:T1-tto}
When the first configuration of \Cref{fig:labeled-cc-2} appears in a whole plane tiling, the orientations of the markings of all yellow hexagons are determined:

\image{}{fig:T1-b}{
\includegraphics[scale=0.75]{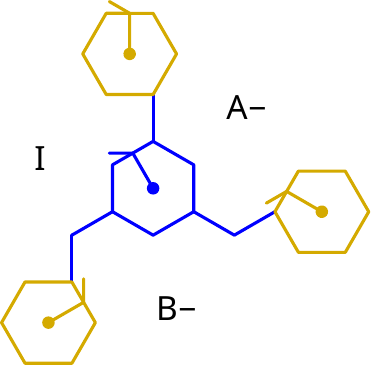}
}
\end{corollary}

Looking only at the type of interfaces, we realize that collapsing the $\iV$ type edges in the previous diagram leads to a hexagon of dots related by $\iI$-type edges:

\nopagebreak

\image{Dotted lines are $\iI$-type interfaces, violet ones are $\iV$-type.}{fig:8-6}{
\includegraphics[scale=0.3]{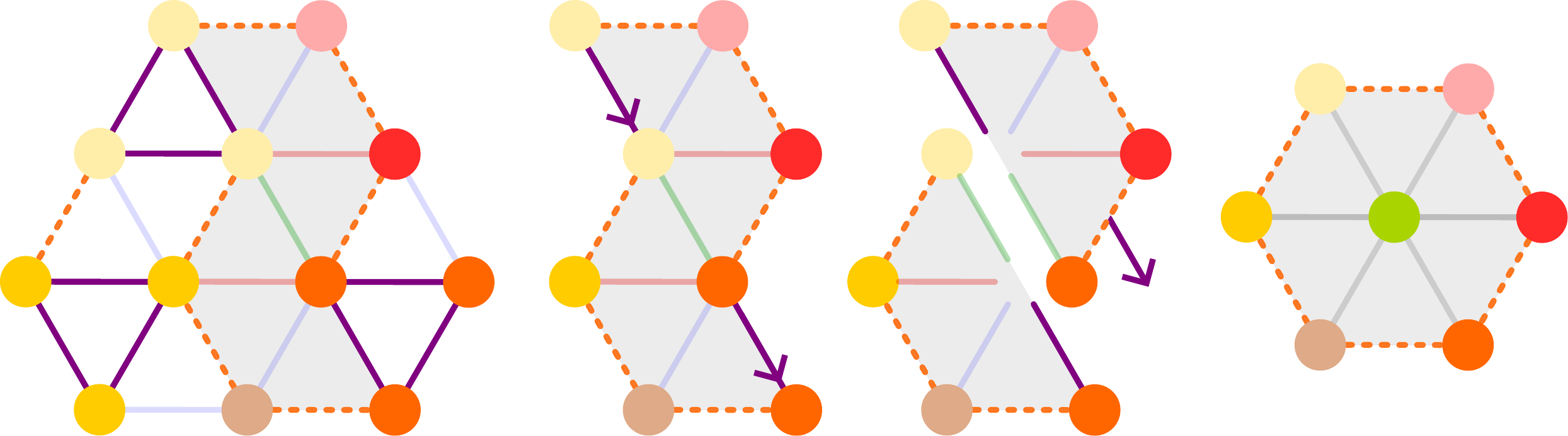}
}

These will play the role of the next generation of hexagons after yellow and blue.
Among the 6 greyed out tiles, the 4 ones having a B interface (red segment) form a parallelogram with 6 triangle vertices, one for each corner of the big hexagon.

\begin{lemma}\label{lem:no-2nd}
The following configuration never appears in a whole plane tiling:

\nopagebreak

\image{}{}{
\includegraphics[scale=0.5]{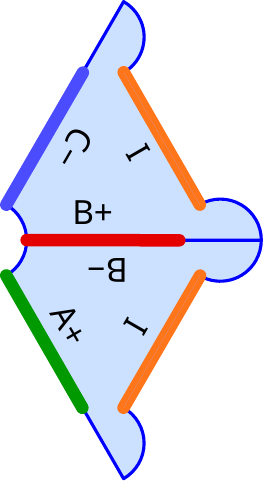}
}
\end{lemma}
\begin{proof}
It appears only once as an internal segment of the list of \Cref{fig:vlist} (the second configuration), rotated upside-down.
As a consequence, its left vertex (right vertex in \Cref{fig:vlist}) circular hole cannot be filled, for the configuration would appear elsewhere in the list.
This can also be directly checked by trying to fill this vertex neighbourhood.
\end{proof}

This excludes the second configuration in \Cref{fig:vlist}.
Also, given that on \Cref{fig:triset}, the third and eighth tile cannot match their $\iB$ interface because of two half circles would overlap, the previous lemma forces the pairing between the $\iB$-type interfaces: on \Cref{fig:triset}, tile 3 and tile 5 on one hand, tiles 4 and 8 on the other hand.

\begin{corollary}\label{cor:AB-into-green}
Any $\iA$ or $\iB$ interface extends into \Cref{fig:N1}.
\end{corollary}

\subsection{Packed tilesets}

\subsubsection{Packs}

Thanks to \Cref{cor:AB-into-green}, we can pack tiles into a smaller tileset and use a simpler notation as follows:

\nopagebreak

\image{Note the difference between the top and bottom of the green pack: it is \emph{not} centrally symmetric.
A dot has been added, as an optional visual aid to locate the ($\iA-,\iI,\iB-$) triangle (the one symbolizing a $T1$ cluster). }{fig:pack}{
\makebox[\textwidth][c]{\includegraphics[scale=0.5]{4-set-5b2.pdf}}
}

\begin{theorem}\label{thm:pack}
Whole plane tilings with those pieces, respecting that $+$ and $-$ must go together, are equivalent to whole plane tilings with any of the triangle tilesets of \Cref{fig:triset,fig:triset-7h,fig:tri-6b}, so to whole plane Spectre tilings.
This corresponds to cutting the plane along the $\iV$ (violet) and $\iI$ (orange) type edges.
\end{theorem}
\begin{proof}
First note that no two distinct pack can intersect since their boundary is only made of $\iV$ and $\iI$ while internally they only have $\iA$, $\iB$ or $\iC$ edges.

Let us prove that the triangle tiling is covered by packs.
We have already seen that triangles with a $\iB$-type edge must be part of a green pack.
Triangles that have no $\iB$-type edge and which are not of type ($\iV-$, $\iV-$, $\iV-$) or ($\iV+$, $\iV+$, $\iV+$), are of two types: ($\iI$, $\iV+$, $\iC+$) and ($\iI$, $\iV+$, $\iC-$).
Consider a triangle $t$ of one of these two types and the triangle that is in contact with it through the $\iC$-type edge: either this triangle is the other type ($\iI$, $\iV+$, $\iC+$) or ($\iI$, $\iV+$, $\iC-$), or it has a $\iB$-type interface.
In the first case $t$ is in a yellow pack, in the second case in a green pack.

The converse statement is immediate.
\end{proof}

\medskip

We can equivalently use the following set with only 3 pieces and no edge markings:

\nopagebreak

\image{}{fig:3-set}{
\includegraphics[scale=0.4]{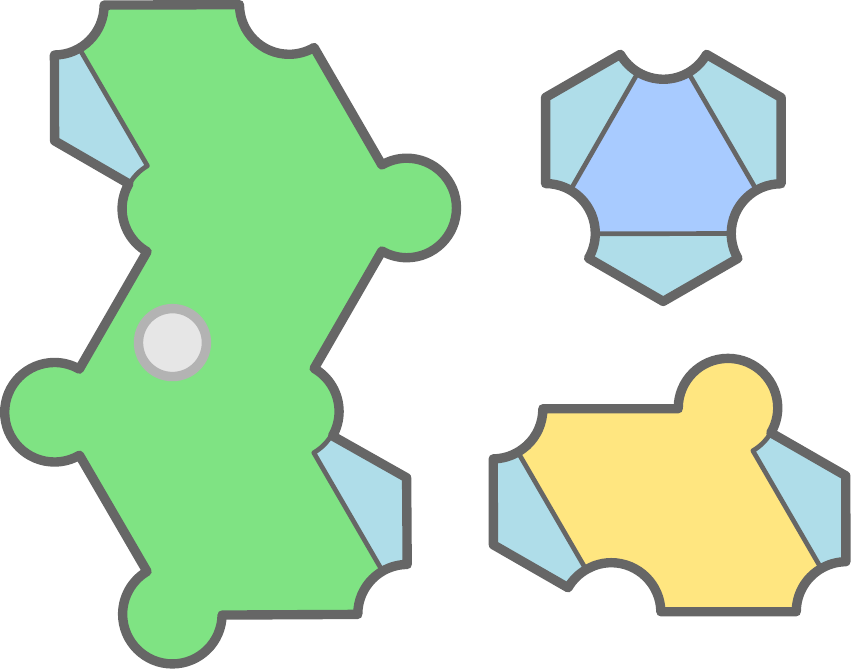}
}

\Cref{prop:around-tri} implies that the big piece is always environed as follows (the environment is actually determined further; see for instance \Cref{lem:gr-pack-env-2}):

\nopagebreak

\image{}{fig:green-pack-env}{
\includegraphics[scale=0.45]{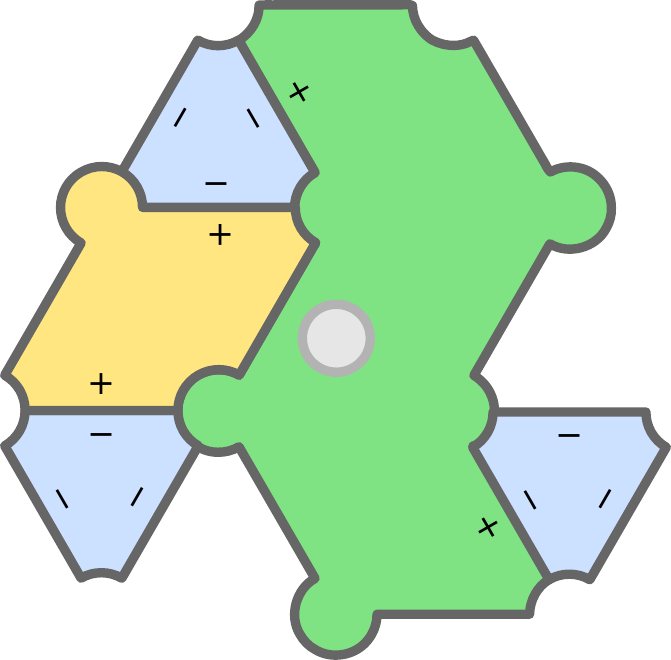}
}

At some point we will use the following decorations on the packed tileset. 

\nopagebreak

\image{}{fig:packed-deco}{
\includegraphics[scale=0.45]{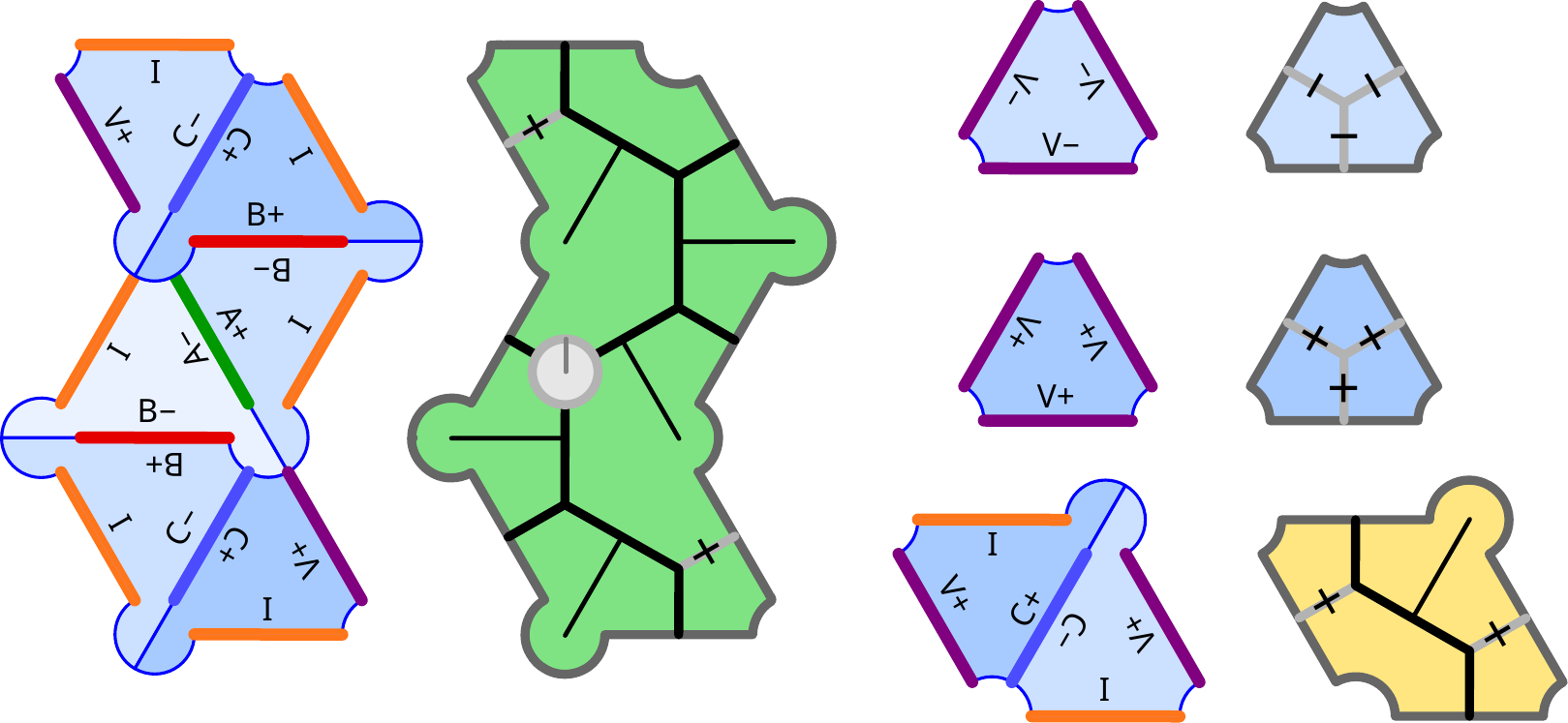}
}

The inner thick lines, in black and in gray, once the packed tiles are assembled in some way, trace hexagons centred at the vertices of the triangle. They are thus associated to yellow hexes, but \textbf{rotated by $1/12$ in the clockwise direction}.
The gray thick lines stem from, and are orthogonal to, the centre of the the edges of type $\iV$. 
The black thick lines will turn out to trace the outline of what will be yellow hexagon clusters: we expect to see triangular clusters of type $1$, $2$ and $3$.

\image{Example of what a yellow pack represents in terms of hexes of the original decoration graph and Spectre tiling. The arrow shows the yellow hex corresponding to the circle. Note the rotation of 1/12 between the marking of that circle and that of the yellow hex.}{fig:wir-1}{
\includegraphics[scale=0.5]{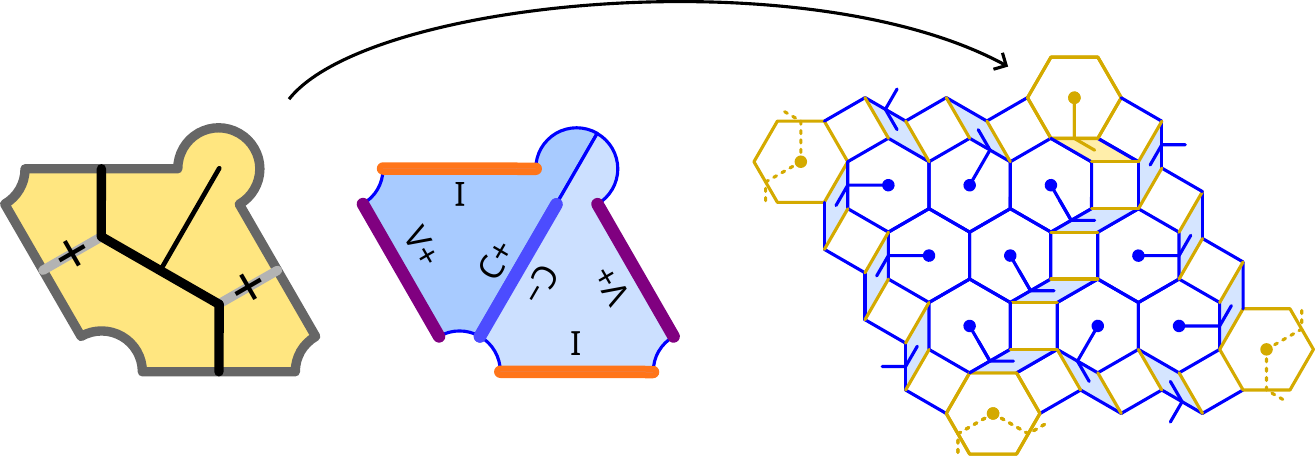}
}

Recall that the triangle pieces represent the blue cc's of the yellow/blue decoration graph, which all have a $\hT n$ blue hex cluster.
Recall also that we put a gray dot in the green pack, at the centre of the triangle of type ($\iA-$, $\iI$, $\iB-$), i.e.\ the light blue one in \Cref{fig:triset,fig:triset-7h}. This triangle type is the only one that represents a cc with a $\hT1$ blue hex cluster 

The thin black lines in the marking of the packs is then the yellow hex marking, rotated by $1/12$ clockwise as earlier.
It is also the extension of the seam between two half circles in the triangular tileset (\Cref{fig:triset-7h}).
The thin gray line on the gray dot represents the orientation of the blue hexagon marking of the corresponding blue hex $\hT1$ cluster, also rotated by $1/12$ in the clockwise direction.

\image{The gray dot represents a blue hex. Note the $1/12$ rotation between the hex markings.}{fig:wir-2}{
\includegraphics[scale=0.5]{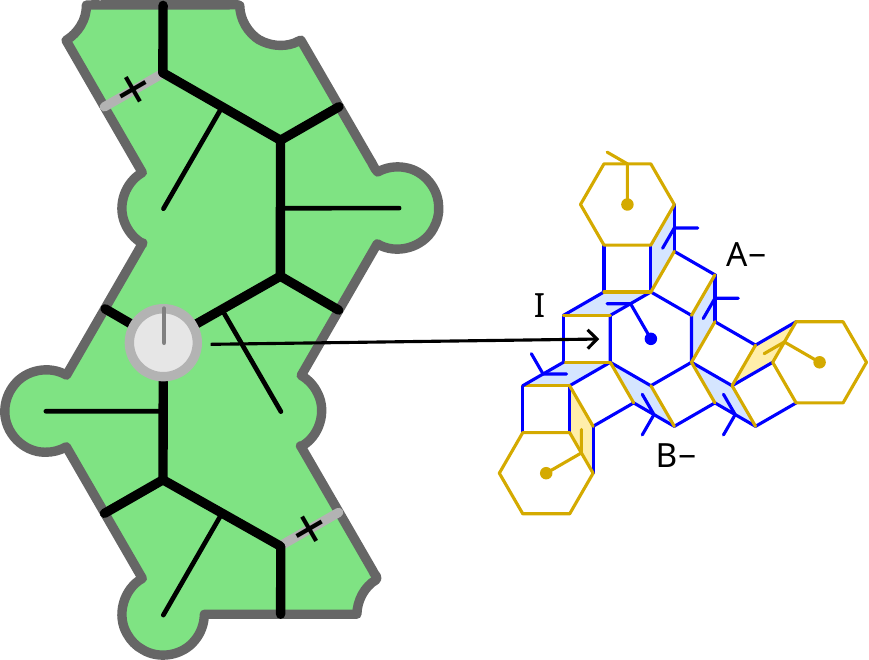}
}

We will sometimes omit the thin black and gray lines on the decorated packed tiles of \Cref{fig:packed-deco}.

\subsubsection{Adjacency and higher level honeycomb}

Let us introduce a third honeycomb into the game, which we call the \term{green honeycomb}. Since we oriented the blue and the yellow honeycombs the same way, it seems appropriate to orient the green one the same way, which is what we do in the sequel.

\image{Green honeycomb}{}{
  \includegraphics[scale=0.666]{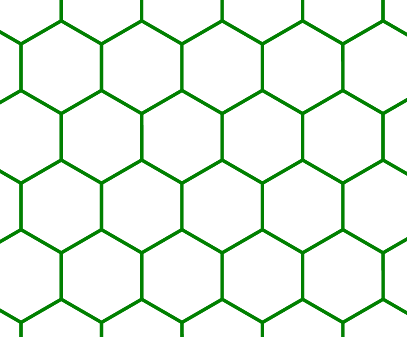}
}

Recall that we called \emph{adjacent} two yellow clusters having two yellow hexagons connected by an $\iI$-type interface in the underlying blue triangle tileset. This interface is supported by a blue triangle edge, and we look at the vector of this edge, directed by a choice of first and second yellow cluster.
Given a whole plane tiling by the blue triangle tileset, we will use these vectors to build a correspondence from each cluster to a \emph{vertex} of this green honeycomb.
We sum this up in the following proposition:

\begin{proposition}\label{prop:yc2gh}
There is a bijection from the set of yellow clusters to the set of vertices of the green honeycomb such that adjacent vertices correspond to adjacent yellow clusters. Moreover:
the yellow cluster points up or down exactly like the triangle of the dual tessellation (to the green honeycomb) that contains the vertex; for two adjacent yellow clusters, the vector between the corresponding vertices of the honeycomb has a direction that is rotated by 1/12 in the anticlockwise direction compared to the direction of the vector supporting the $\iI$-type interface. 
\end{proposition}

The rest of this section is a proof of the proposition. Its mechanism reveals interesting insights and more properties of the above correspondence.

Recall that the yellow honeycomb is \emph{visible} as the thick black and gray lines on the decorated packed triangles, and as such is a tessellation of the whole plane dual to the tessellation by the regular triangles supporting the triangle pieces (in the way of the fourth image of \Cref{fig:dual}). In particular the interfaces link yellow hexagon centres to adjacent yellow hexagon centres.

First note that an $\iI$-type interface relates two yellow hexes whose orientations are necessarily in different classes modulo $1/3$ of a turn.

Recall that all yellow clusters have been proved to be triangular with 1, 3 or 6 hexagons. We define yellow cluster tips exactly as we defined blue cluster tips in \Cref{def:tips}. We define \term{sides} of a cluster as the part of their boundary between two tips.
Since tips are yellow hexagon vertices, they are at triangle centres.

\begin{lemma}
Each green pack intersect exactly 6 different yellow clusters.
Cluster tips are all in green packs.
Green packs contain exactly one tip of each of the six yellow clusters it intersects.
\end{lemma}
\begin{proof}
The first claim comes from yellow clusters being triangular and examination of the thick line in the green pack (see \Cref{fig:shatter}).

Note that the topmost tip of an up pointing yellow triangular cluster is its only boundary vertex such that the two boundary edges going to it are directed as on the figure below.

\nopagebreak

\image{}{}{
\includegraphics[scale=0.5]{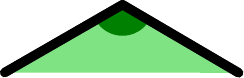}
}

An analogue statement holds for the three tips of up or down pointing triangular clusters.
%
%

Now examination of the clusters intersected by the different pack pieces show that their direction up or down is determined in advance (if one rotates them by an odd multiple of 1/6 then one must permute u and d).
This follows for instance from the dot markings of \Cref{fig:triset-7h} (in the case of the green and the yellow pack, this can also be deduced from the orientation class of the two clusters that have two hexes crossed by the pack; and for the green pack one then uses the rule that $\iI$ interfaces, i.e.\ the orange lines, shall invert $u$ and $d$).
Then using the rule above, we can identify the tips.

\image{u: up, d: down, dark-green: yellow cluster tip}{fig:shatter}{
\includegraphics[scale=0.4]{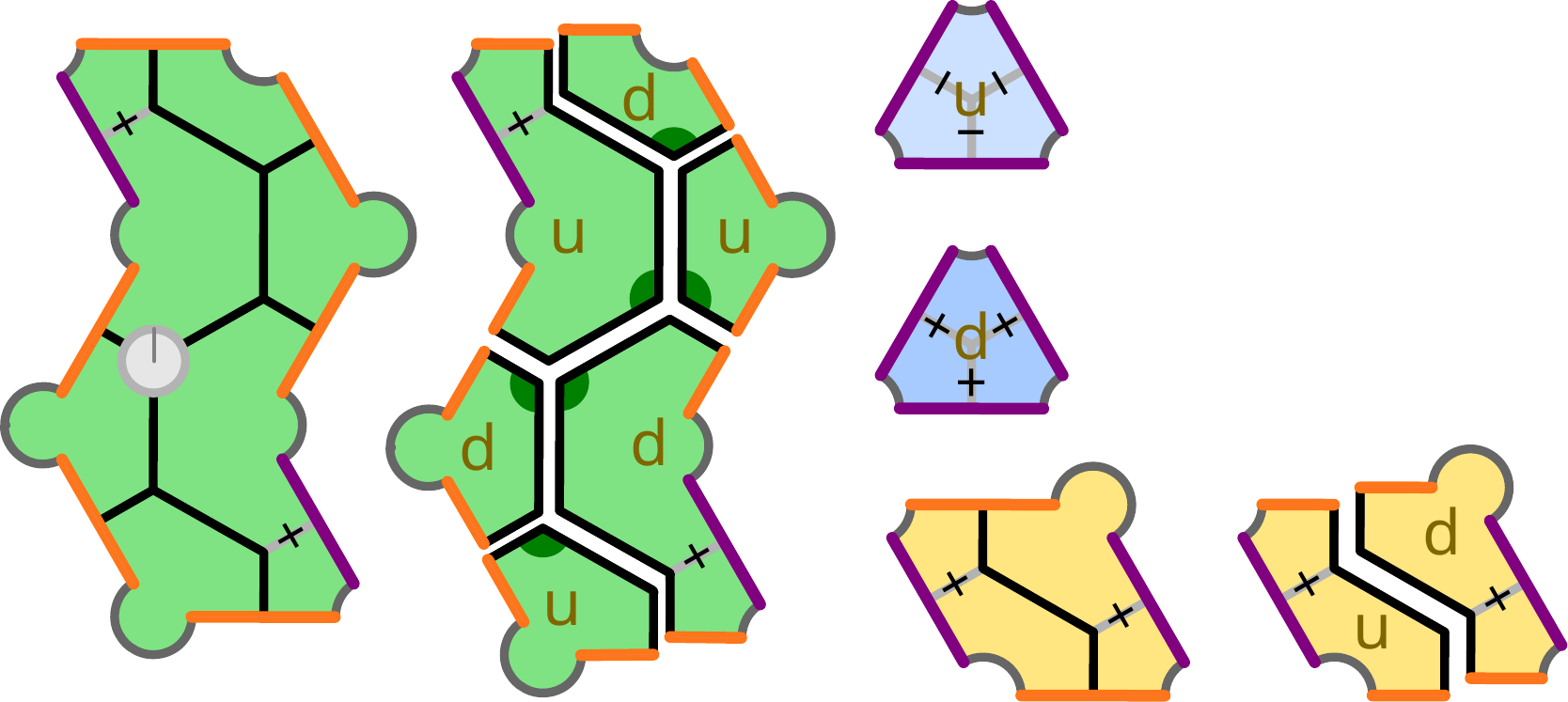}
}
\end{proof}

\begin{corollary}\label{cor:ee}
Each yellow cluster meets exactly three green packs.
There is at least one $\iI$-type interface through each side of a cluster.
\end{corollary}
\begin{proof}
The first claim is an easy logical consequence of the previous proposition and the fact clusters have three tips. 
The second claim above follows from examination of the green pack again.
\end{proof}

But there is more to that: examine the picture below.

\nopagebreak

\image{}{fig:pack-to-gh}{
\includegraphics[scale=0.28]{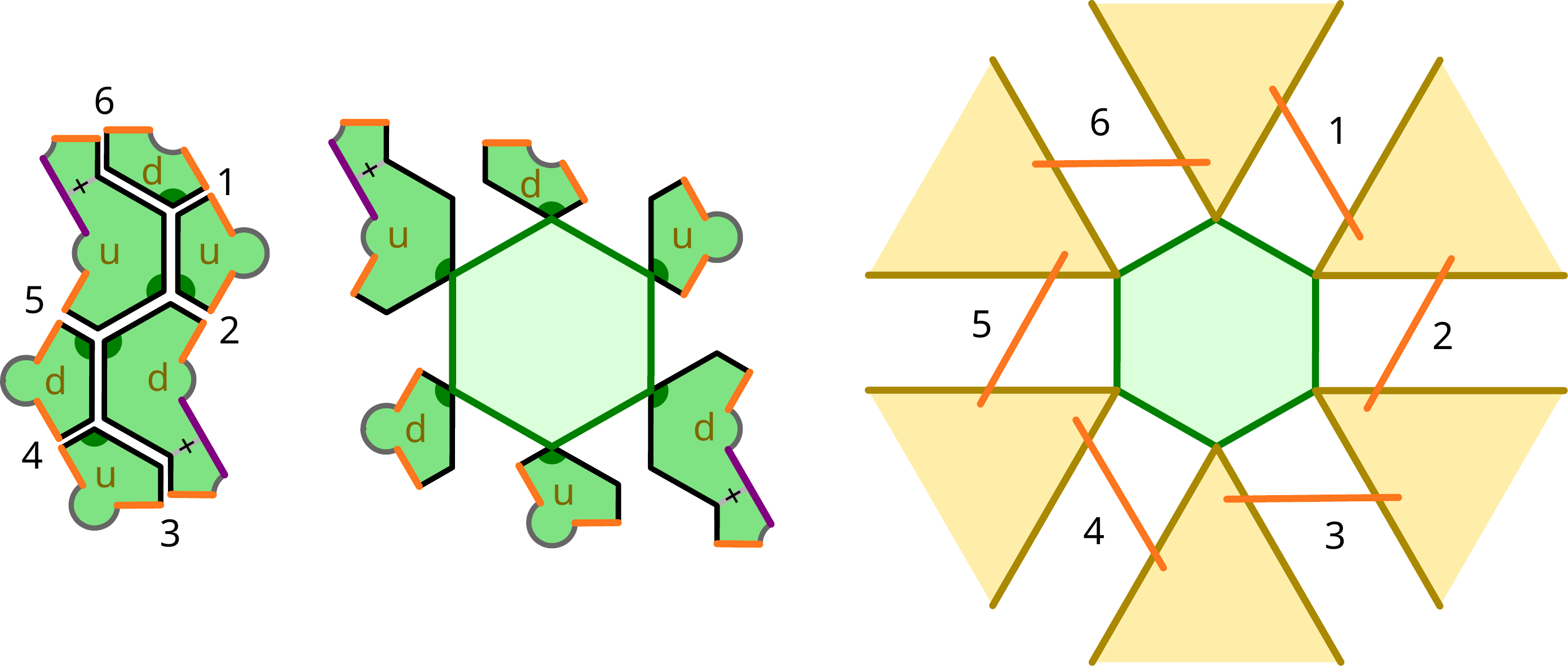}
}

We can associate a green hexagon to the green pack, so that the yellow triangular clusters are related to the hexagons exactly as in the dual regular hexagonal and triangular tessellations are, as on the figure below.

\image{}{fig:dual-2}{
\includegraphics[scale=0.666]{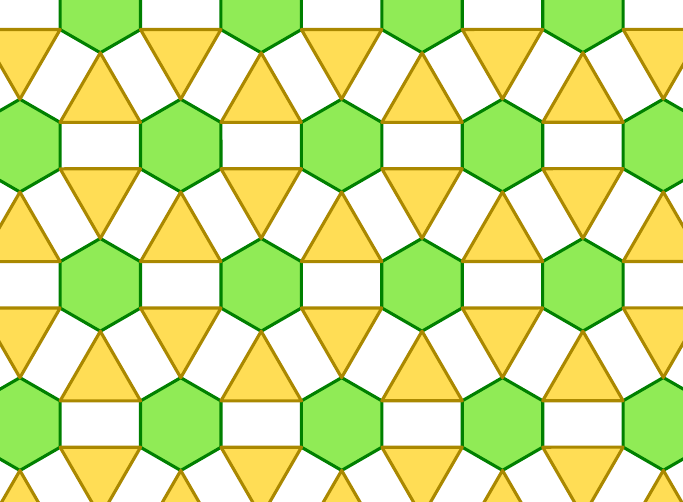}
}

The claim below may now seem obvious, yet we feel it requires an argument.

\begin{proposition}
 The whole set of yellow clusters and green packs are combinatorially arranged, using \Cref{fig:pack-to-gh}, exactly as on \Cref{fig:dual-2}
\end{proposition}
\begin{proof}
There is no ambiguity on the orientation of the hexagons and of the triangles and identification of their respective tips.

Walking around between the hexes and triangles in \Cref{fig:dual-2} corresponds without ambiguity to a walk in the green packs and yellow clusters, and vice-versa.

One main thing is to check that going around a white rectangle in \Cref{fig:dual-2} corresponds to coming back to the same pack or clusters. But the intersection of the boundaries of two adjacent yellow clusters is a zigzag broken line and must enter green packs near each end. The claim follows.

The other main ingredient is simple connectivity of the plane.
Because of this, any walk starting from a green hex (or yellow triangle) and coming back to itself can combinatorially decomposed into elementary loops around hexes, triangles and rectangles.
Similarly any walk starting from a pack (or yellow cluster) and coming back to itself can be combinatorially decomposed into elementary loops. Since these loops do not break the correspondence, we get that it is a bijection.

Another approach is to build a piecewise linear map from the plane with the pack tiling to the plane of \Cref{fig:dual-2} (or the other direction). Such a map will be a covering and since the plane is simply connected, a homeomorphism.
\end{proof}

By collapsing the yellow triangle in \Cref{fig:dual-2} we get the green honeycomb. By the proposition above, this gives a bijective mapping associating a vertex of the green honeycomb to a yellow cluster, such that adjacency in the honeycomb graph corresponds to adjacency as defined before, i.e.\ via the presence of an $\iI$-type edge connecting yellow clusters. Note that for each such adjacency, the $\iI$ edge direction must be rotated by 1/12 in the counterclockwise direction to give the direction of the corresponding hexagon edge in the green cluster.

\subsubsection{Coordinate relation between three levels of honeycombs}\label{ss:3-levels}

We prove here a relation that was guessed in the introduction. It will not be needed for most statements in this article.

We use the same notations as in \Cref{lem:coord}: $\rho=e^{2\pi i /6}\in\C$, we assume that the centre of the hexagons of the blue honeycomb forms the lattice $\Lambda = \Z+\rho\Z$, and hence the sides of the hexagons are of the form $\rho^k c$ with $c=1/(1+\rho)$.
Below we use a yellow and a green honeycomb, and we also assume that their hexagon centres are the points of the lattice $\Lambda$. 
The vertices of the honeycombs are a subset of $c\times\Lambda$.

We denote $HB$, $HY$ and $HG$ the blue, yellow and green honeycombs.

\begin{proposition}\label{prop:B-Y-G-coord}
Consider the data of a partition of $HB$, with dots, associated to a whole plane tiling by the Spectre.
To each dot we associated a hexagon in $HY$. We call $y\in \Lambda$ the coordinate of its centre and $b\in c\times\Lambda$ the coordinate of the dot in $HB$.
The yellow hexagon belongs itself to a yellow cluster $\hT'n$ which we associated to a \emph{vertex} of $HG$.
Denote $g\in c\times\Lambda$ its coordinate.
We choose any dot as a reference and let $b_0$, $y_0$ and $g_0$ denote the associated values.
Then we have the following relation for any other dot:
\[b-b_0 = 3 \times (y-y_0) + \rho^2 \times (g-g_0).\]
(Note that $\rho^2=j=e^{2\pi i/3}$ is the principal primitive third root of unity.)
\end{proposition}
\begin{proof}
The caption of the figure below recalls the scale and orientation convention, which is the same for each of the three honeycombs.

\image{Dark blue vector has affix $1$ and light blue has affix $c=1/(1+\rho)$.}{}{
\includegraphics[scale=0.4]{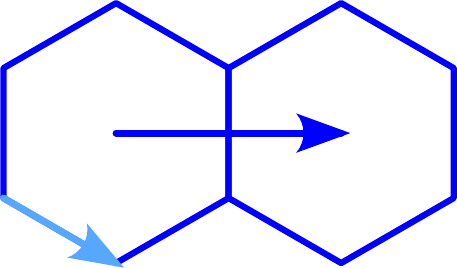}
}

On the figures below the light green vector has affix:
$3$ in $HB$, $1$ in $HY$, $0$ in $HG$;
the dark green vector has affix: $(2+3\rho)\times c$ in $HB$, $1$ in $HY$, $\rho c$ in $HG$.
For the blue coordinates, this was already computed in \Cref{lem:coord}, realizing that the light green is associated to a $\iV$-type interface, and the dark green one to type $\iI$.

\image{}{}{
\includegraphics[scale=0.25]{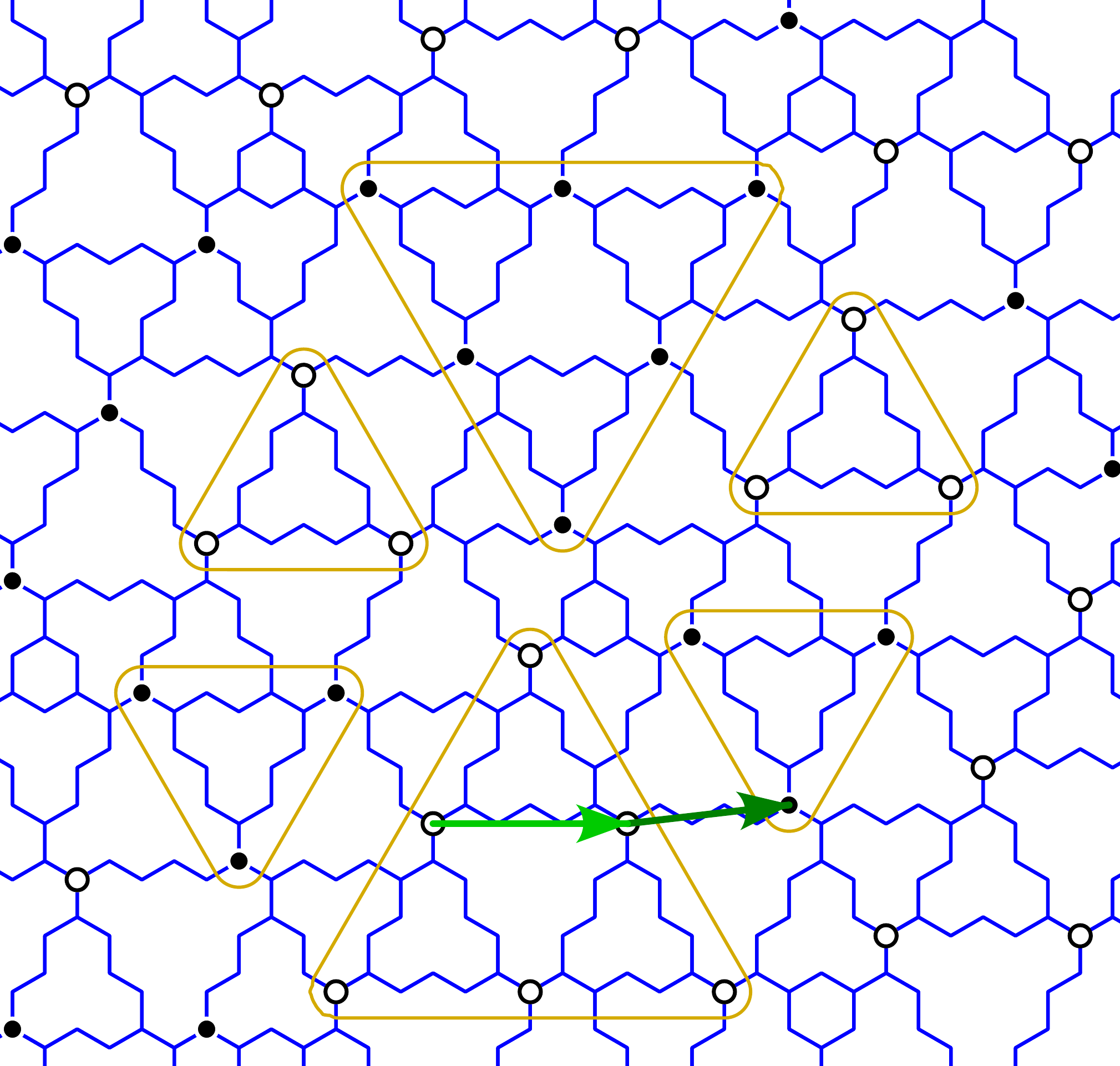}
}

\image{Note that yellow shows reflected possible arrangements compared to possible arrangements of blue or green honeycombs}{}{
\includegraphics[scale=0.35]{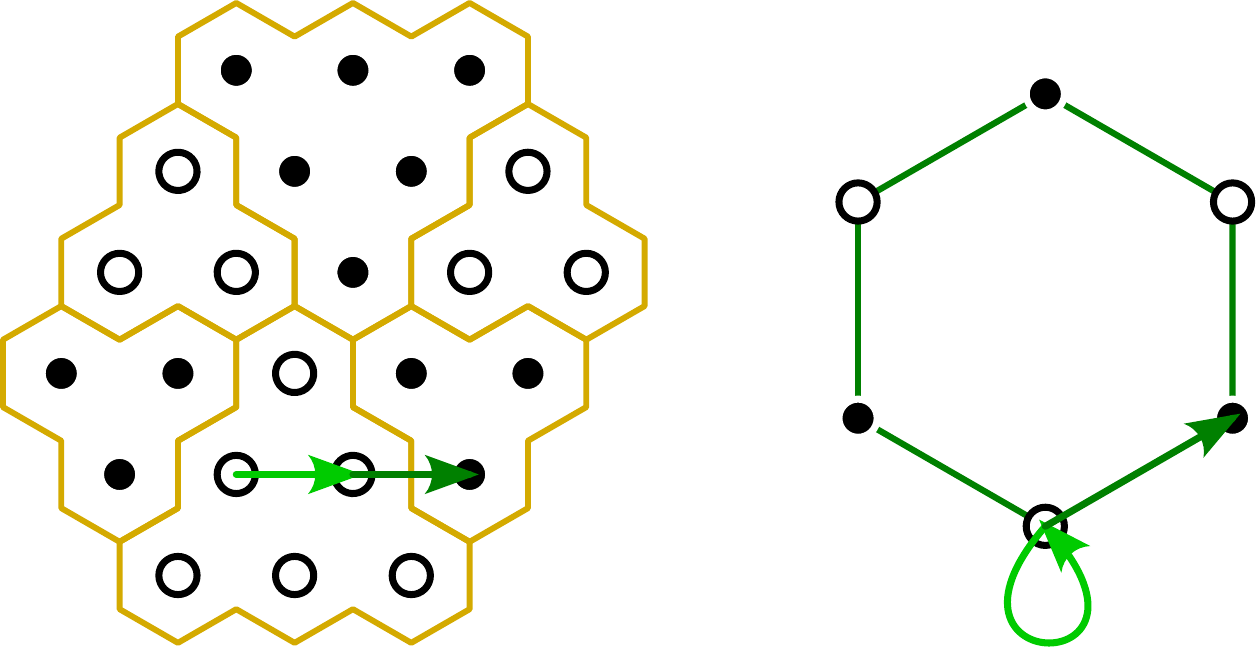}
}

More generally, going from one yellow hex to an adjacent yellow hex in the same cluster moves the centre by some $\rho^k$, $k\in\Z$ (which we may take between $0$ and $5$ included but this does not matter here), while the blue coordinate will move by $\rho^k \times 3$ and the green one will not move.
By definition, two yellow clusters are called adjacent if have a pair of adjacent yellow hexes such that the corresponding dots in the blue honeycomb are extremities of an $\iI$-type interface.\footnote{So having adjacent hexagons is not enough for us to call adjacent two yellow clusters.}
If going from the first to the second is through vector $\rho^{k}$ in the yellow honeycomb, then it is by vector $\rho^{k} (2+3\rho)c$ in the blue and vector $\rho^{k+1} c$ in the green.

In both cases, the quantity $b-3y-\rho^2 g$ is invariant.
Since any two yellow hexes can be related by a finite sequence of moves as above, we are done.
\end{proof}

Note the angle $1/12$ rotation between the $\iI$-type link in the yellow honeycomb and the green one, which was already noticed in the previous section. Be also aware that not all two adjacent yellow hexes of different yellow clusters are related by an $\iI$-type move, yet at least one pair in each adjacent\footnote{For the adjacency notion recalled in the proof above.} clusters are, as the figure below shows on an example.

\image{All the $\iI$-type links for some situation (actually occurring in whole plane tilings).}{}{
\includegraphics[scale=0.35]{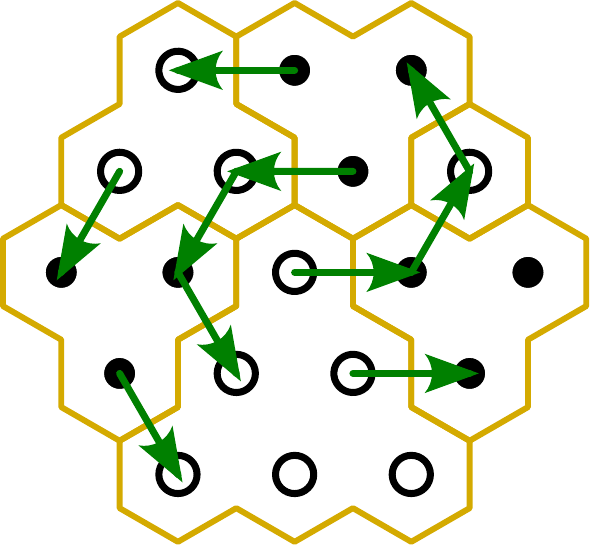}
}

\subsection{Environment of the yellow clusters}\label{sub:env}

Recall that we proved in \Cref{ss:hTp} that every yellow cluster is a $\hT'n$, $n=1$, $2$ or $3$.
We now start the moderately tedious task of classifying all the possible environments of the $\hT'n$ (yellow hex clusters).

\subsubsection{Some visual rules}\label{ss:vis}

We sum up here some configurations that are either impossible or have some implied continuations.

\begin{lemma}\label{lem:fill-1}
Consider a connected partial tiling with the packs. 
Consider a triangle vertex for which among its 6 possible neighbouring triangles, only two are not covered, and they are contiguous.
If the free interfaces have no $+$ or $-$ sign, then the non-covered part can only be filled by a green piece, with the central dot being one of the four dots indicated in the figure below.

\nopagebreak

\image{The thick black lines indicate that the free interface has no $+$ nor $-$ marking. The white circles, including the centre one, may or may not be already filled. (A posteriori the rightmost one cannot.)}{fig:fill-1}{
\includegraphics[scale=0.333]{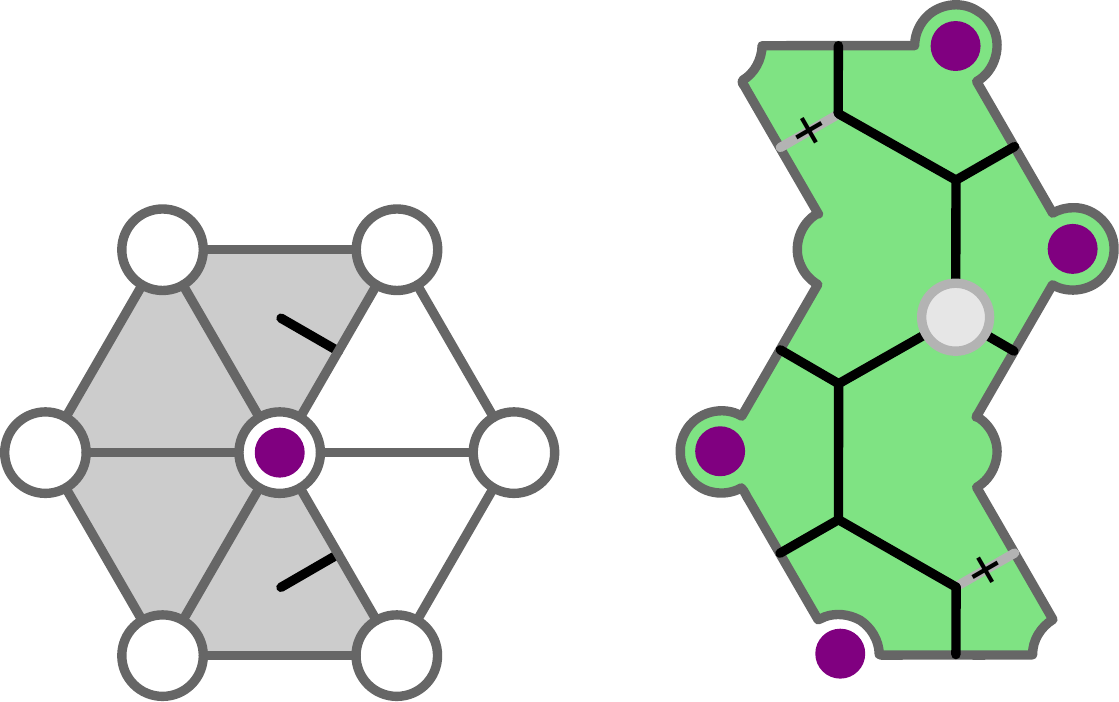}
}
\end{lemma}

This is immediate by examination of all possibilities, and does not take too much time since we have only 4 possible pack pieces.

\begin{lemma}\label{lem:fill-2}
The following configuration can only be filled by a yellow piece:

\nopagebreak

\image{}{fig:fill-2}{
\includegraphics[scale=0.333]{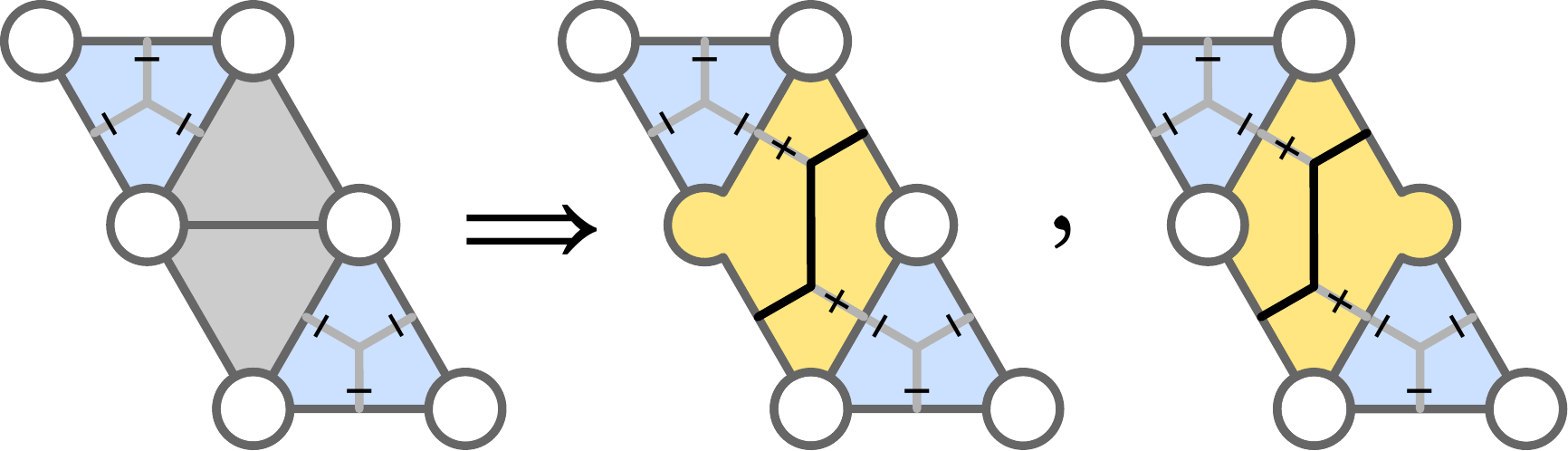}
}
\end{lemma}
\begin{proof}
Again it is a simple examination. Trying to fit a $+,+,+$ piece on some gray triangle implies the other gray triangle is a $-,-,-$ but then some interface is in contact with another $-,-,-$, contradicting the rule that no two $-$ interfaces can fit.
If instead one tries to fit the $+$ of a green piece on one of the initial $-,-,-$ piece, the other $-,-,-$ piece touches it on a side without a sign, which is also forbidden.
\end{proof}

For convenience, we reproduce here \Cref{fig:green-pack-env} and the conclusion of the discussion around this figure:

\begin{lemma}\label{lem:fill-3}
Any green pack is environed as follows in any pack tiling covering a sufficiently big neighbourhood of it (independent of the tiling):

\image{}{fig:green-pack-env-b}{
\includegraphics[scale=0.33]{nec-b.pdf}
}
\end{lemma}

The claim about the neighbourhood independence follows from the fact that in the proof of \Cref{prop:around-tri}, which the lemma is deduced from, we did not need to look at the tiling beyond some radius around the central triangle.

\subsubsection{\texorpdfstring{$\hT'1$}{T'1}}\label{ss:tp1}

We already noted that a $\hT'1$ is the centre of either the \nth{5} or \nth{8} configuration of \Cref{fig:vlist}.

\begin{proposition}\label{prop:no-5th}
The fifth configuration of \Cref{fig:vlist} cannot occur in a whole plane tiling.
\end{proposition}
\begin{proof}
See \Cref{ss:Tpn-env-pfs}.
\end{proof}

So a $\hT'1$ is always the centre of the \nth{8} configuration of \Cref{fig:vlist}, which we recall below left. Below right we indicate the corresponding green packed tiles configuration it induces.

\nopagebreak

\image{The $\hT'1$ is the yellow hex represented by the circular arc at the centre.)}{}{
\begin{tikzpicture}
\node at (0,0) {\includegraphics[scale=0.45]{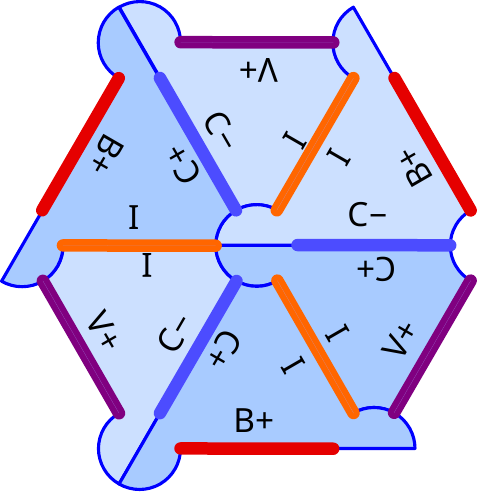}};
\node at (6,0) {\includegraphics[scale=0.3]{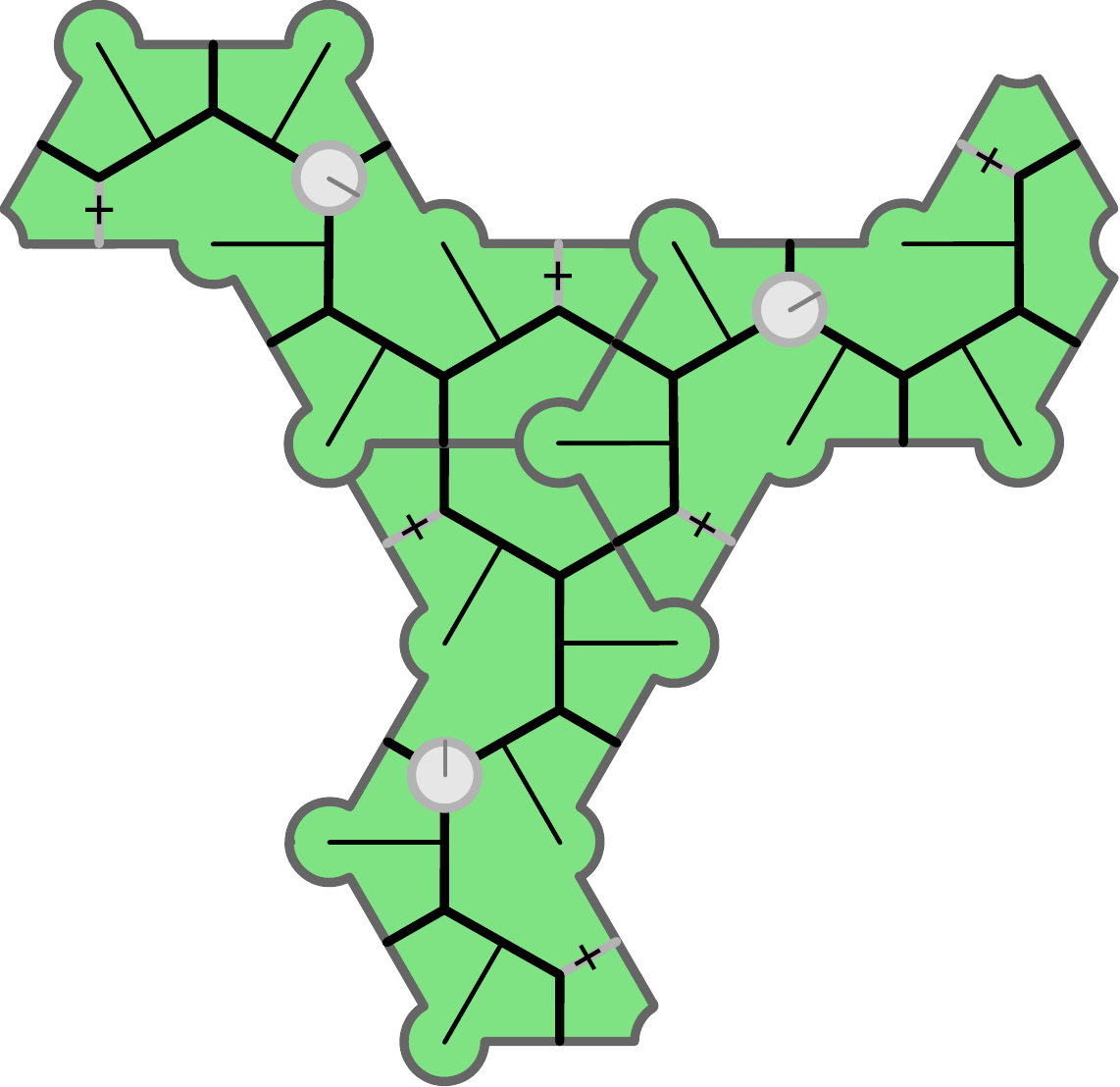}};
\end{tikzpicture}
}

%
%

This immediately resonates with the right part of \Cref{fig:T1-b}, which we recall here, rotated then mirror-reflected on the right. 
What resonates with a yellow hex in the figure below is a gray dot in the figure above.

\nopagebreak

\image{Left: a specific cc type; right: its reflection}{}{
\includegraphics[scale=0.66]{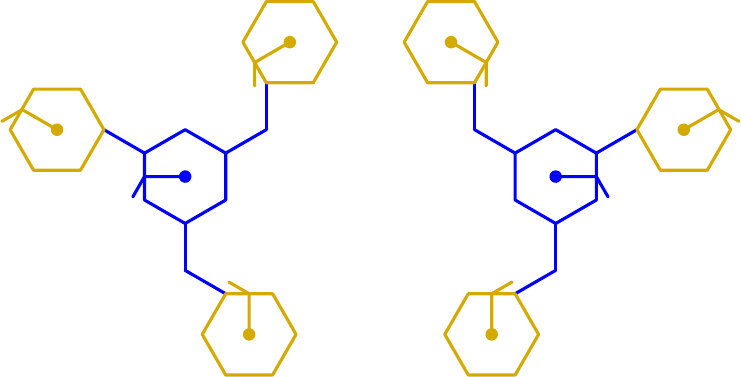}
}

%
%
%

\subsubsection{\texorpdfstring{$\hT'2$}{T'2}}\label{ss:tp2}

Let now us study the environment of a $\hT'2$.
By the analysis of \Cref{ss:hTp} it is associated to a $-,-,-$ triangular piece of the packed tileset of \Cref{fig:packed-deco}.
Recall that by \Cref{prop:Tp3} the $+,+,+$ triangle cannot touch it for a $\hT2'$.

\nopagebreak

\image{What can be in contact with one side of a $\hT'2$}{}{
\includegraphics[scale=0.3]{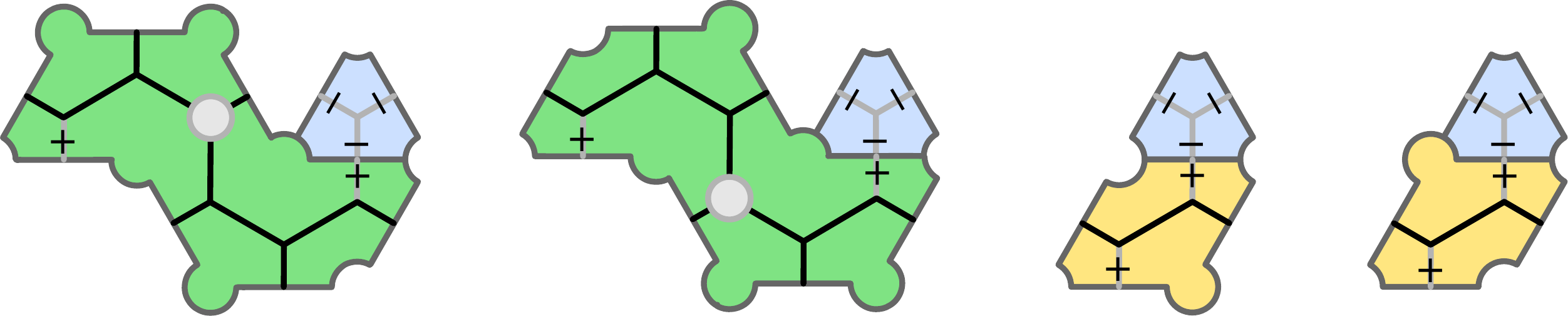}
}

The two diagonal sides of the yellow pack cannot be in contact with another yellow pack, for that would force the presence of a ($-$,$-$,$-$) piece sharing a triangle vertex with the original one, and we would get a $\hT'3$:

\image{}{}{
\includegraphics[scale=0.33]{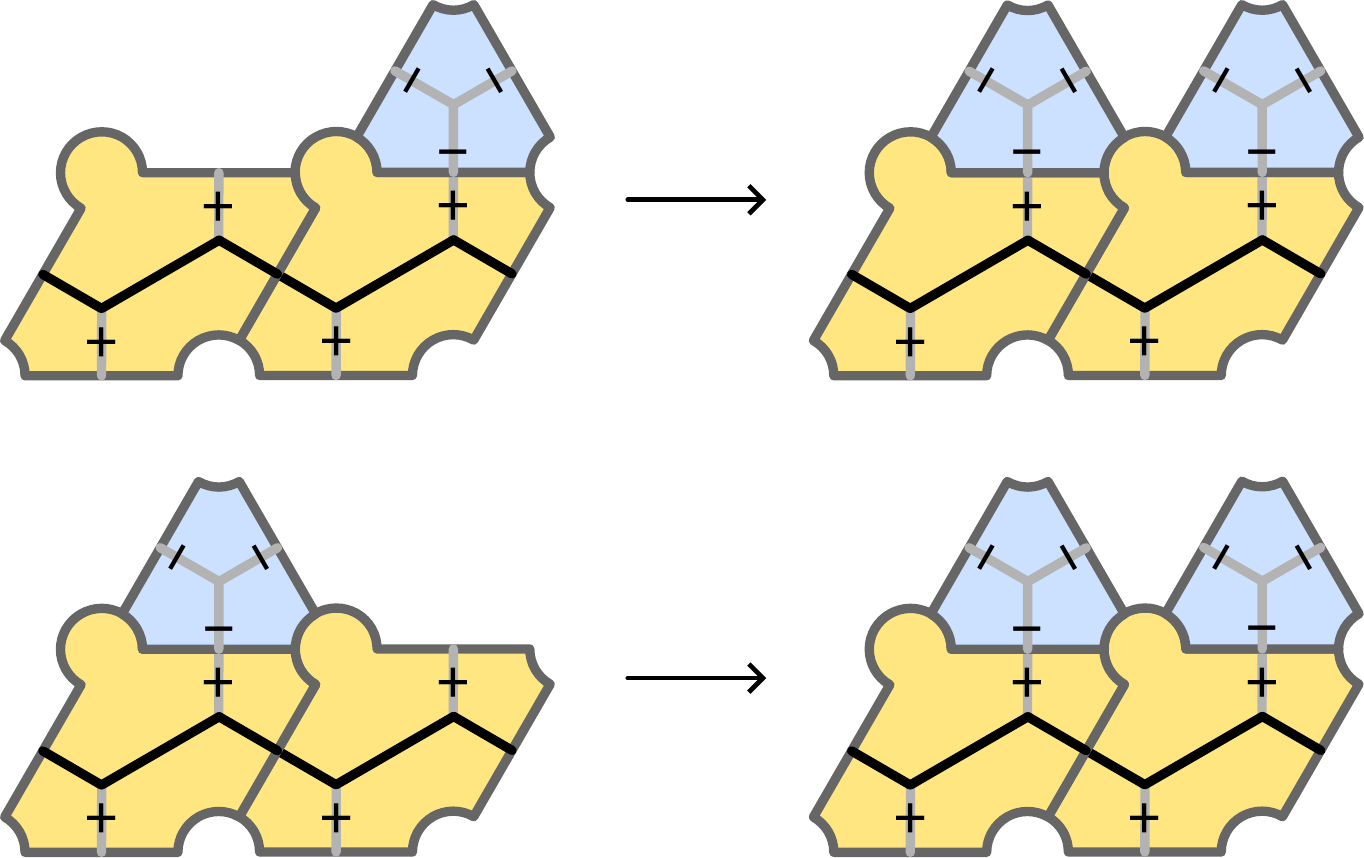}
}

Similarly, the situation below-left is impossible:

\image{}{}{
\includegraphics[scale=0.33]{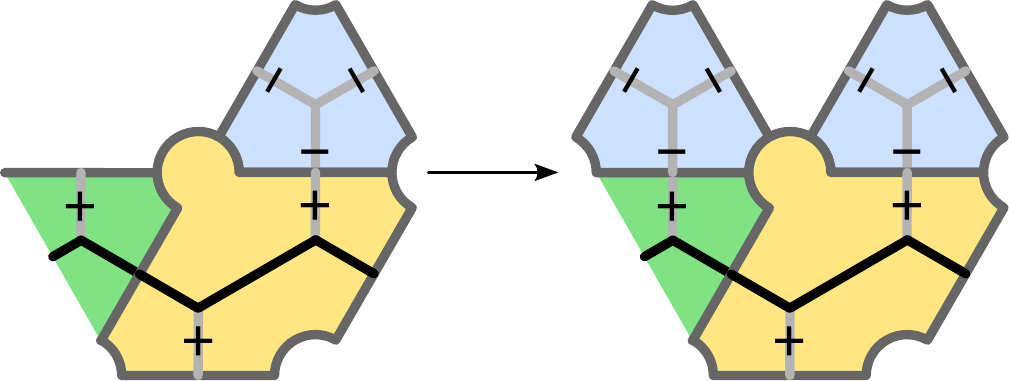}
}

The following is impossible too, because the red dot cannot be filled:

\image{}{fig:Tp2-env-no}{
\includegraphics[scale=0.3]{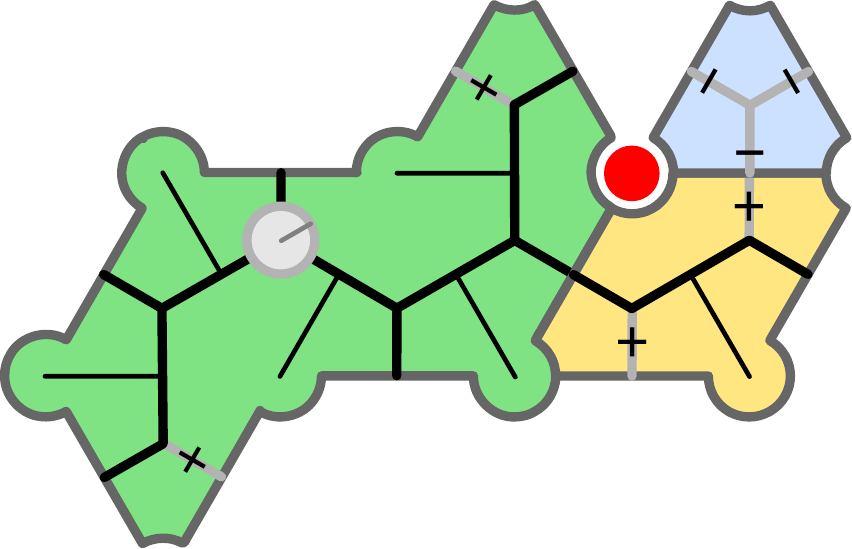}
}
Focusing on the lower left corner, we necessarily are in one of the 6 following cases, where the reader will note that the lower left tip of the $\hT'2$ is occupied by a green triangle, and also that the green piece appears in the 6 possible orientations.

\image{The 6 possibilities have been given labels.}{fig:Tp2-corner}{
\begin{tikzpicture}
\node at (0,0) {\includegraphics[scale=0.26]{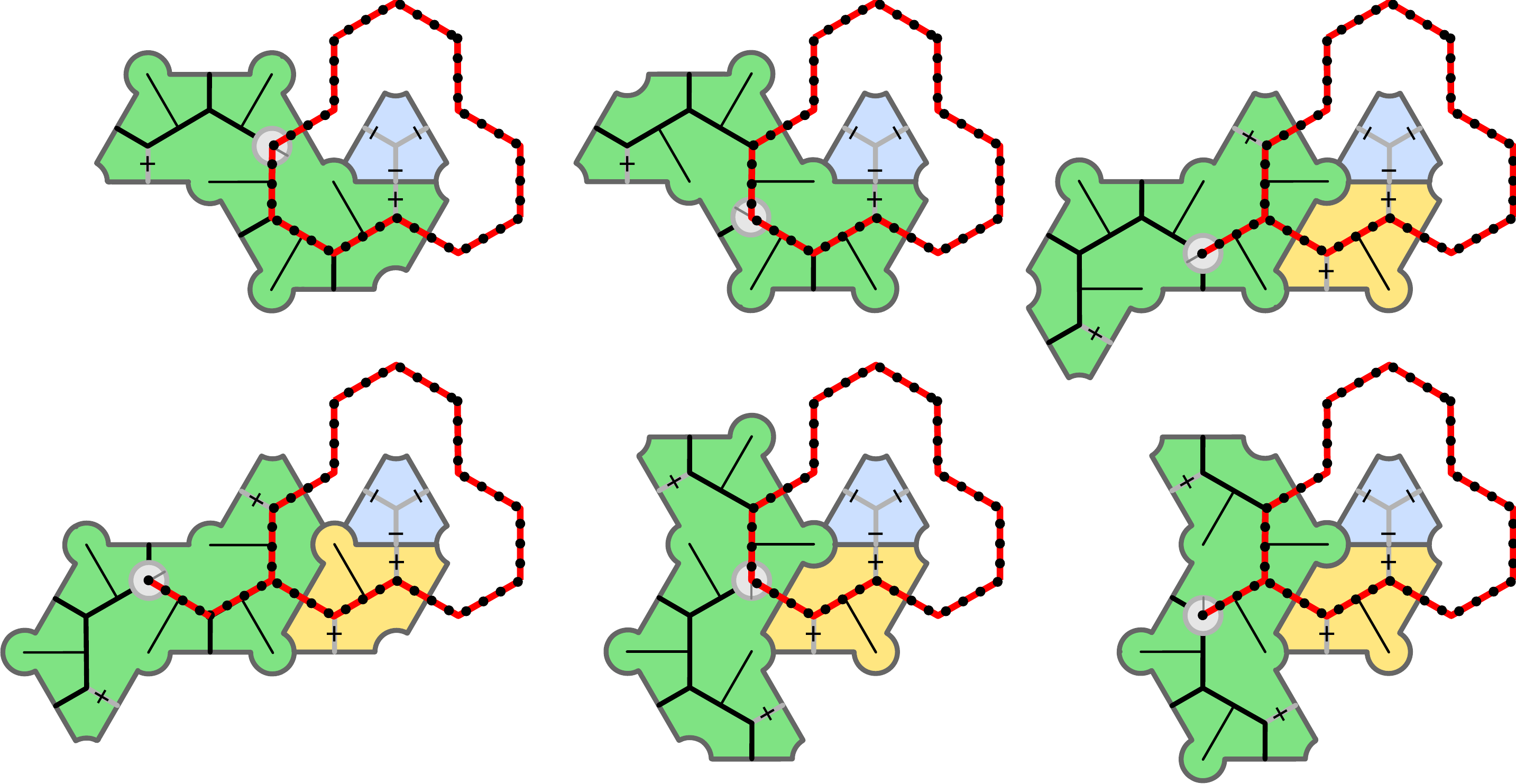}};
\node at (-5,1) {$n$};
\node at (-1,1) {$t$};
\node at (6,0.8) {$a1$};
\node at (-4,-2.8) {$a2$};
\node at (1,-2.8) {$t'$};
\node at (5.5,-2.8) {$a1'$};
\end{tikzpicture}
}

The immediate environment of the central blue triangle is a copy of three of the situations of \Cref{fig:Tp2-corner}, rotated by 0, $1/3$ and $2/3$.

\nopagebreak

We now invite the reader to focus on the circles of position highlighted below:

\image{}{fig:focus}{
\includegraphics[scale=0.3]{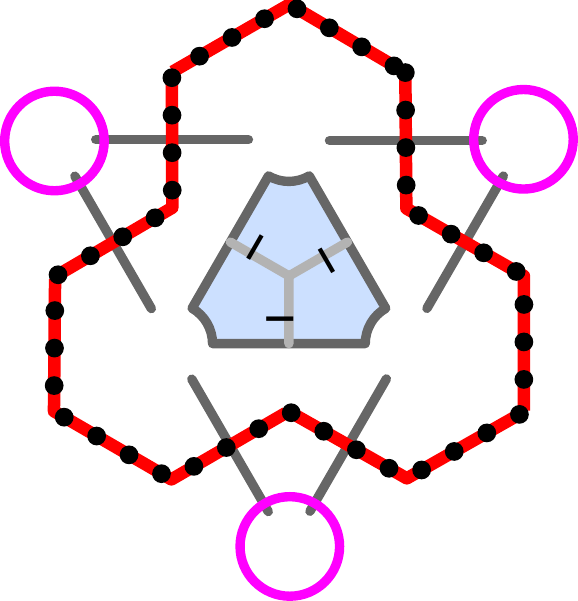}
}

Any configuration in the central column of \Cref{fig:Tp2-corner} occupies two of these sites, so no two central column configurations can occur at the same time, and in fact no two of the following can occur in clockwise order: ($n$ or $t$ or $t'$) then ($t$ or $t'$ or $a1$ or $a1'$).

In the next statement we exclude a few other possibilities. Nearly all concern the presence of a two segments long antenna $a2$ and something else.

\begin{lemma}\label{lem:Tp2-clock}
Consider two tips following each other un the clockwise order around a $\hT'2$ cluster of a whole plane tiling by the packs.
Then their respective types cannot be in the following list (each item is to be read in the clockwise order):
$(n,n)$,
$(a2,a2)$,
$(a2,n)$,
$(a1,a2)$,
$(a1',a2)$.
\end{lemma}
\begin{proof}
See \Cref{ss:Tpn-env-pfs}.
\end{proof}

%

All this restricts the possibilities to exactly the same as in \Cref{prop:T2-cc-1,prop:T2-cc-2}, mirrored, which we recall below. Indeed, using the nomenclature above, we get that no two $n$ nor two $a2$ can be present by the first two case of \Cref{fig:focus}.
Every other excluded case concern an $a2$ (length two antenna) and tell us respectively that, in clockwise order, it cannot be followed by an $n$, nor preceded by $a1$ or $a1'$. 
Recall also that one cannot have in clockwise order ($n$ or $t$ or $t'$) then ($t$ or $t'$ or $a1$ or $a1'$).
So $n$ is followed by $a2$ and $t$ or $t'$ by $n$ or $a2$.
Hence, with no $n$ nor $a2$, there can only be $a1$ or $a1'$.
We thus only get up to circular permutation, in clockwise order: with an $n$: ($n$, $a2$, $t$ or $t'$ or $a1$ or $a1'$) i.e.\ the last two cases of \Cref{fig:Tp2-refl}; with no $n$ but an $a2$: ($t$ or $t'$, $a2$, $a1$ or $a1'$) i.e.\ the second case; with no $n$ nor $a2$: ($a1$ or $a1'$, $a1$ or $a1'$, $a1$ or $a1'$), i.e.\ the first case.
In each case the hex orientation (visible as thin black or gray lines on \Cref{fig:Tp2-corner}) match those of \Cref{fig:Tp2-refl}.

\image{Mirrored versions of the figures of \Cref{prop:T2-cc-1,prop:T2-cc-2}}{fig:Tp2-refl}{
\includegraphics[scale=0.5]{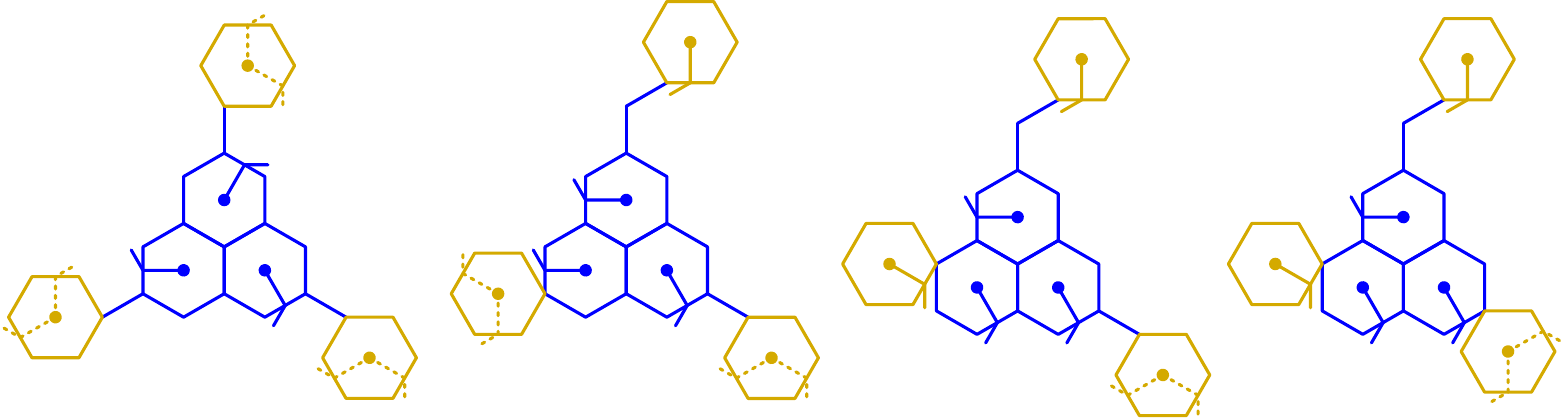}
}

\subsubsection{\texorpdfstring{$\hT'3$}{T'3}}\label{ss:tp3}

\begin{lemma}\label{lem:Tp3-env}
The environment of a $\hT'3$ is as on the figure below:

\nopagebreak

\image{}{}{
\includegraphics[scale=0.3]{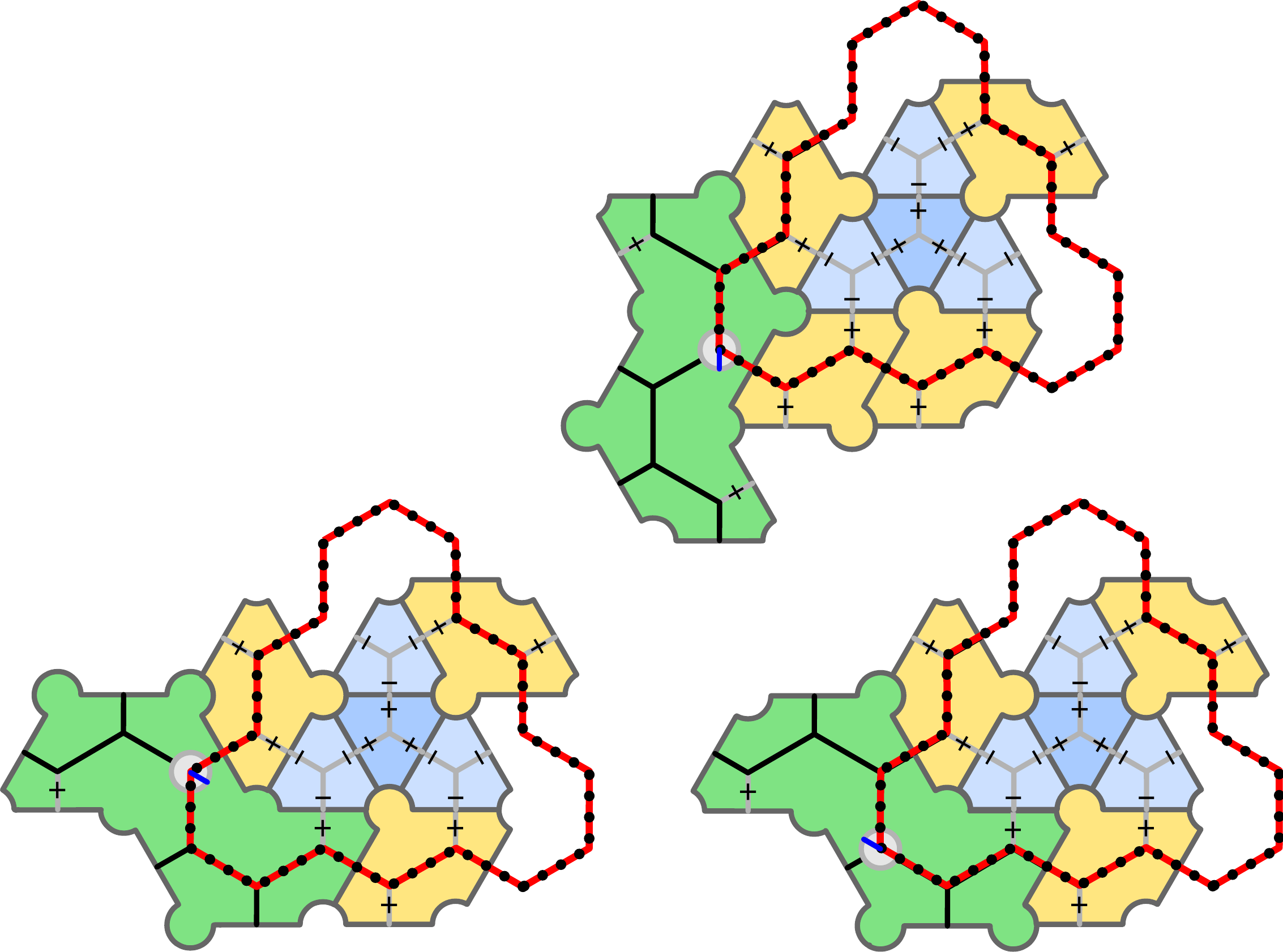}
}
\end{lemma}
\begin{proof}
See \Cref{ss:Tpn-env-pfs}.
\end{proof}

All this restricts the possibilities to exactly the same as in \Cref{prop:T3-cc}, mirrored. For convenience we reproduce this figure here:

\nopagebreak

\image{Mirrored and rotated copies of the two situations of \Cref{fig:T3}}{}{
\includegraphics[scale=0.66]{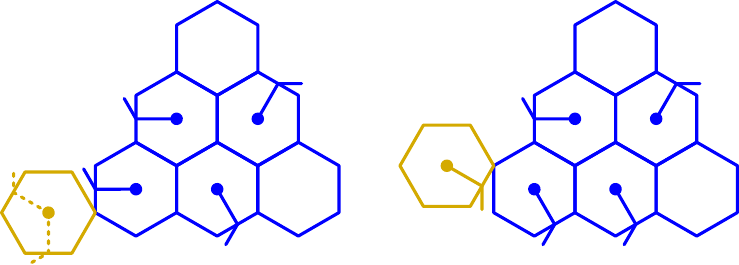}
}

\subsection{Hierarchical structure}\label{sub:he}

We can thus reprove the results in \cite{chiral}:

\begin{corollary}
Whole plane tilings by the Spectre, without reflections, are uniquely hierarchical.
\end{corollary}
\begin{proof}
A whole plane tiling induces a yellow/blue graph which by \Cref{thm:pieces-ne} decomposes into pieces enumerated in \Cref{fig:labeled-cc-2} with superimposing interfaces.

Such a tiling has been seen to be combinatorially equivalent to a tiling with triangular pieces given in \Cref{fig:triset-7h}.
Then the triangles have been proven (\Cref{thm:pack}) to pack into shapes given by \Cref{fig:pack}.
These shapes have been given decorations in \Cref{fig:packed-deco}, which trace in the plane the same type of figure as \Cref{fig:trace}.

Let us mark as pink the central thick black line of the green pack as on the figure below:

\image{}{fig:ps}{
\includegraphics[scale=0.35]{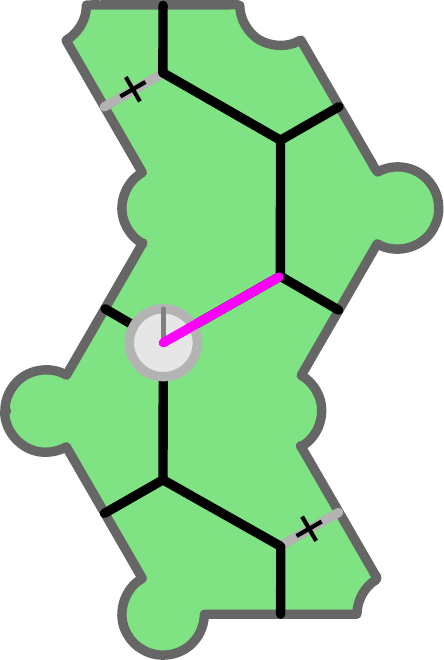}
}

\image{Example of decorated pack tiling, for some whole plane tiling by the Spectre.}{fig:s4}{
\makebox[\textwidth][c]{\includegraphics[scale=0.25]{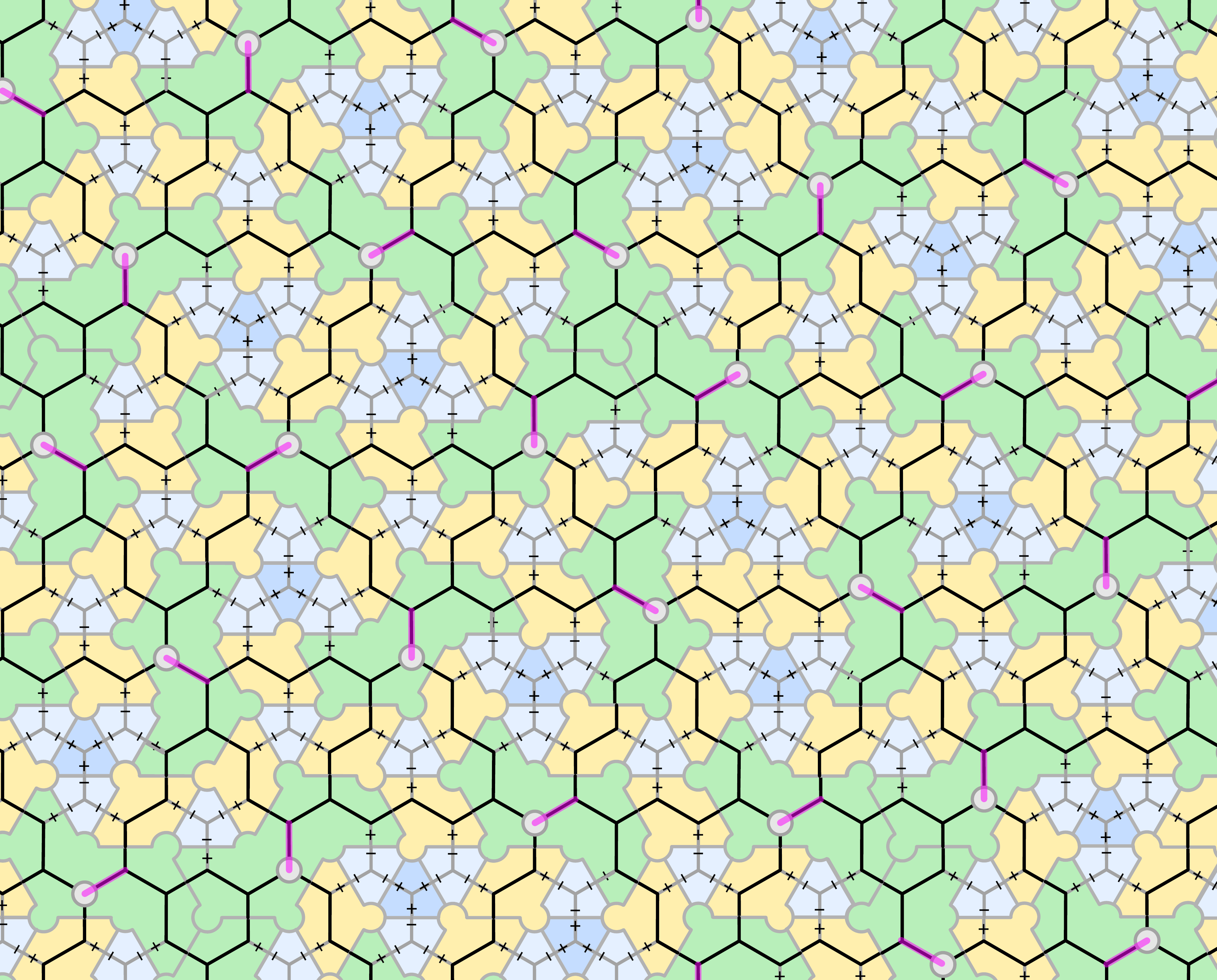}}
}

The thick black and gray lines of those decorations cut the plane into big hexagons and the thick black lines group those hexagons into clusters $\hT'1$,  $\hT'2$,  $\hT'3$, which can be extended within the green packs to reach the big grey dot. Those \term{extended clusters}, which we will also call \term{big cc}, look exactly like the reflection of the blue connected components of the yellow/blue graph that we enumerated in \Cref{sub:cc-list} and called cc.

We can then invoke in \Cref{sub:condensed} the \Cref{prop:ces} (the second rule is respected, using the pink segment an the fact that in the list of possibilities made in \Cref{ss:tp1,ss:tp2,ss:tp3} they always correspond to the pink segment of \Cref{fig:contracted-pieces}, repeated below), which implies that the extended clusters above correspond combinatorially to some reflected Spectre tiling. Then the whole analysis can be carried out again, which proves the hierarchical nature of the tiling.
In every step, the grouping is unique and in the combinatorial equivalence between different kind of tilesets, the correspondence is unique. So the tiling is uniquely hierarchical.

\image{}{table:1}{
\small
\begin{tikzpicture}[every node/.style={align=center}, arrows={[scale=3]}]
\node(a2) at (0,2) {Spectre\\tiling};
\node(a1) at (0,0) {cc $+$\\interfaces};
\node(b) at (2.25,0) {triangles};
\node(c) at (4.25,0) {packs};
\node(d) at (6.5,0) {reflected\\honeycomb\\partition$+$dots};
\node(e2) at (9.25,2) {reflected\\Spectre\\tiling};
\node(e1) at (9.25,0) {reflected\\cc $+$\\interfaces};
\node(1) at (11,0) {etc.};

\draw[<->] [below](a2) -- (a1);

\draw[<->] [right](a1) -- (b);
\draw[<->] [right](b) -- (c);
\draw[->,thick] [right](c) -- (d);
\draw[<->,dotted] [right](d) -- (e1);
\draw[<->] [right](e1) -- (1);

\node at (7.65,1.6) {prop.\\\ref{prop:ces}};
\draw[<->] [below](e2) -- (e1);

\draw[<->] [above right] (d) -- (e2);
\end{tikzpicture}
}

Alternatively, we can bypass the use of Spectres, because the way \Cref{prop:ces} is proved is\ldots\ via the dotted equivalence.
\end{proof}

Moreover we can invert the thick arrow in \Cref{table:1}:

\begin{proposition}\label{prop:inv}
Any whole plane or partial honeycomb partition with dots, which is decomposable as an assembly of the pieces of \Cref{fig:copy_refl} (which are reflections of the pieces of \Cref{fig:contracted-pieces}) such that:
\begin{itemize}
\item pieces are rotated by a multiple of 1/6 and assembled by superimposing their boundary interfaces from dot to dot (they must have the same shape and dots must map to dots),
\item hexes do not overlap and each element of the partition is the set of hexes of a piece.
\item every dot has one and only one of its three incoming segments that is pink in the list on the right of interface pairs of the figure,
\end{itemize}
can be obtained as traced by the thick black line and dots decorations of some whole plane or partial tiling by the packed pieces of \Cref{fig:packed-deco}.

\nopagebreak

\image{This is a copy of the \emph{reflected} pieces of \Cref{fig:contracted-pieces}. It shows a list of allowed cc shapes together with dots. In \Cref{fig:contracted-pieces} these dots marked where yellow hexagons touch them, for decorated graphs associated to Spectre tilings. But here we do not assume initially that there is a Spectre tiling: only that we have an arrangement satisfying the hypotheses if \Cref{prop:inv}.}{fig:copy_refl}{
\includegraphics[scale=0.5]{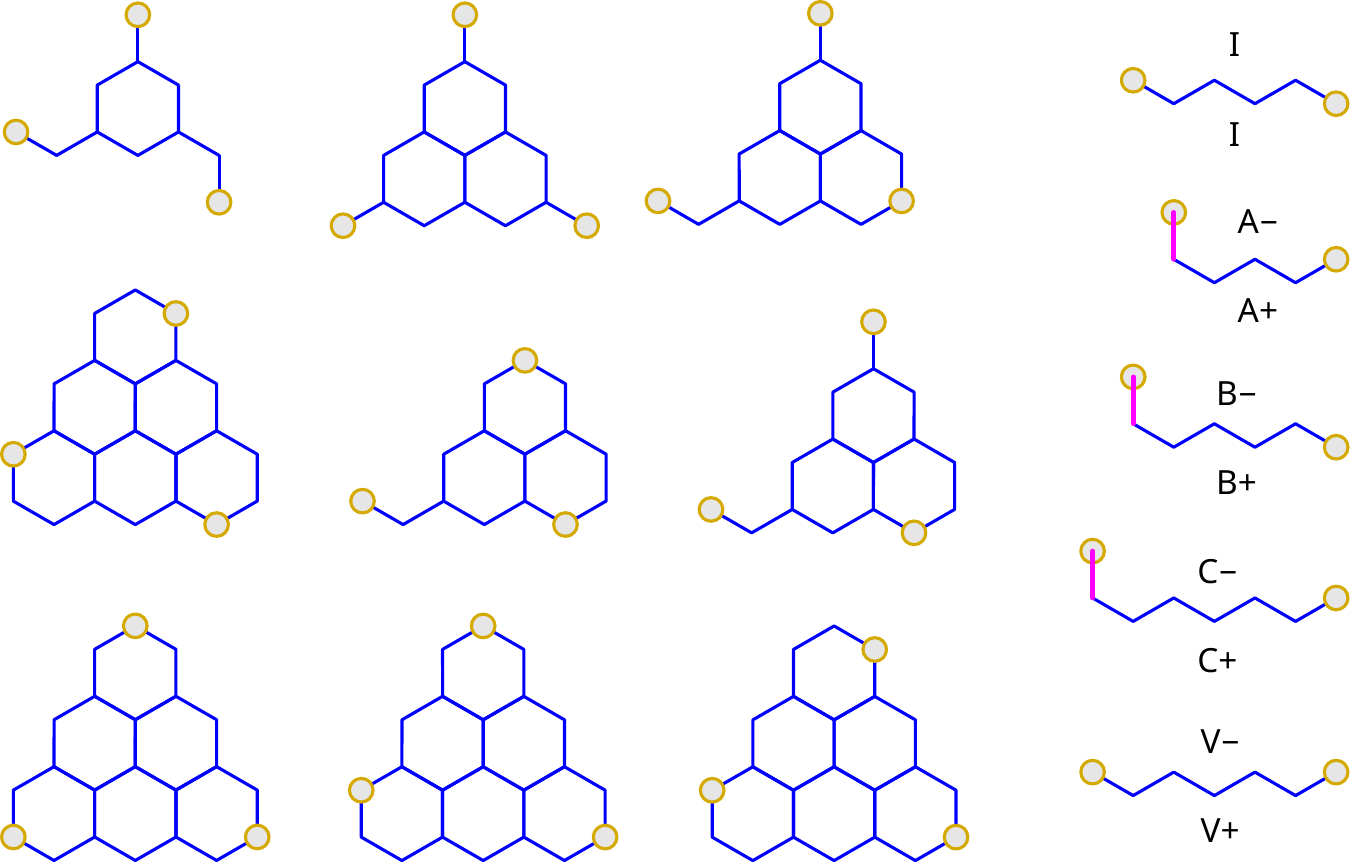}
}
\end{proposition}
\begin{proof}
First we recall the possible tip types for the cc's, according to the position of the associated dot, and the presence or not of an antenna. The tip is associated to the triangle cluster, assumed to point up in the figure below, and we focus here on the lower left tip of the triangle.

\image{Tip types for a lower left tip. Cluster tip circled in green. Dot position indicated by a gray circle, its orientation by a short blue line.}{fig:tip-types}{
\includegraphics[scale=0.25]{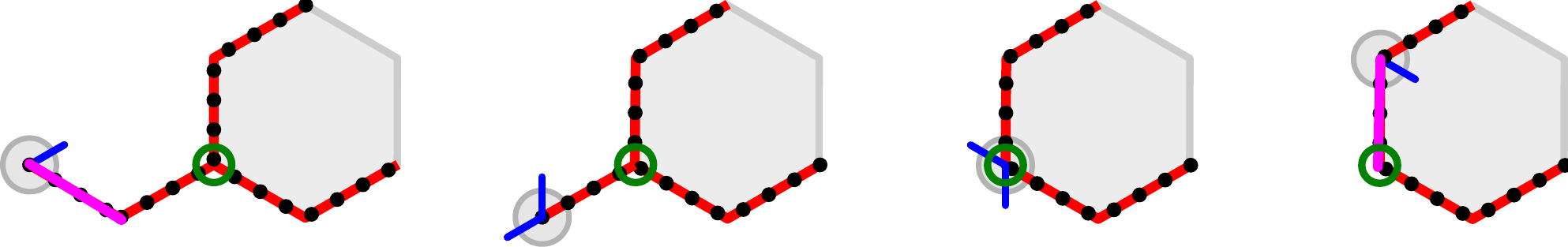}
}

If we are to cover the dot by a green pack respecting the dot orientation, in can only be the as follows.

\image{Ways a cluster tip can be covered by a green pack. Call them type~1 to type~6.}{fig:unsub-1}{
\includegraphics[scale=0.25]{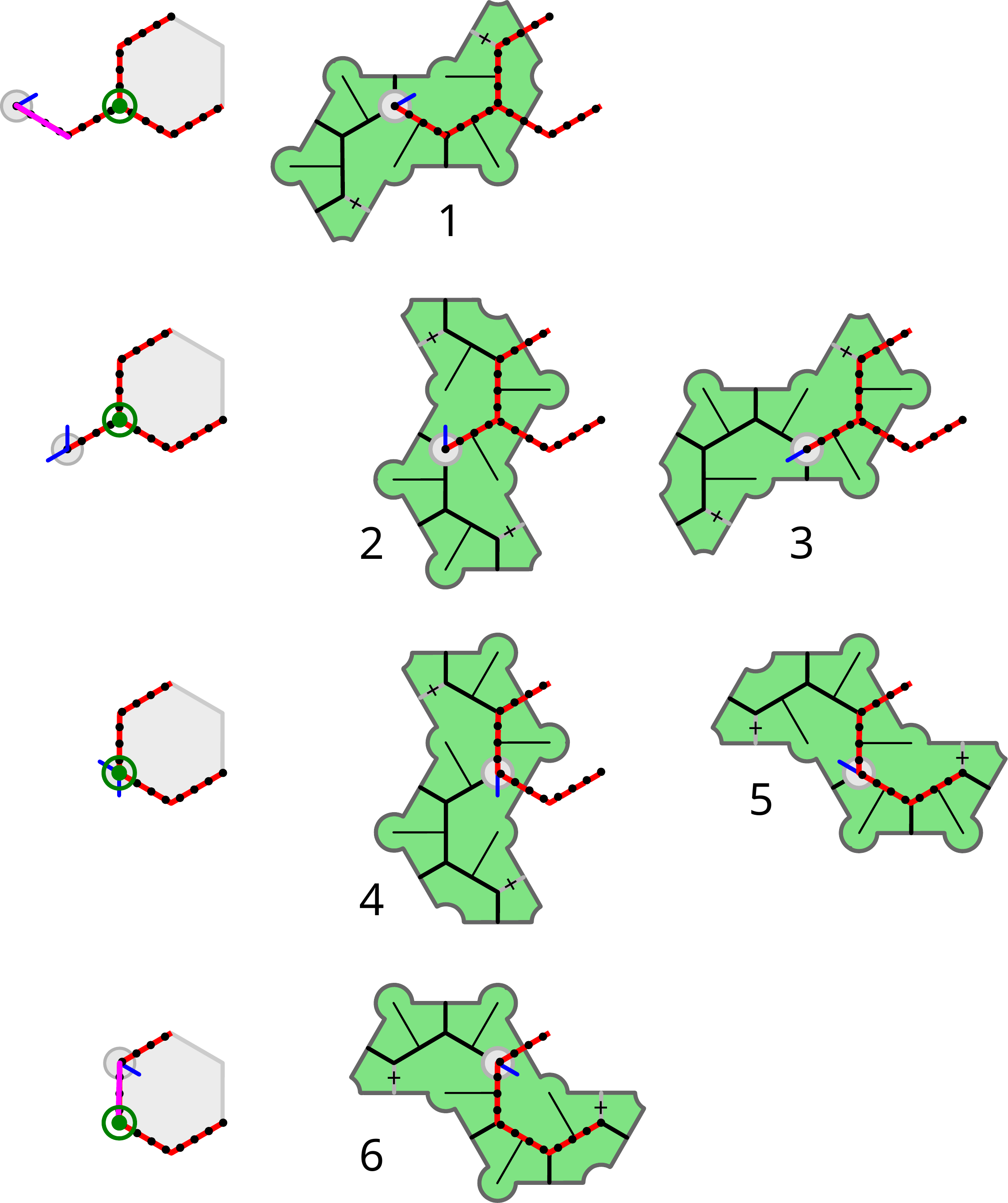}
}

It is remarkable that the dot, hence the pack, comes in the 6 possible orientation once and only once in this figure.

Consider now a dot whose orientation, given by the marking or equivalently the pink edge, we arbitrarily fix as on the figure below with the blue line oriented down. Then there are exactly the 6 possibilities below, where the tip comes in the 6 possible orientation once and only once in this figure.

\image{The 6 types, tips indicated by an orange wedge}{fig:unsub-2}{
\includegraphics[scale=0.25]{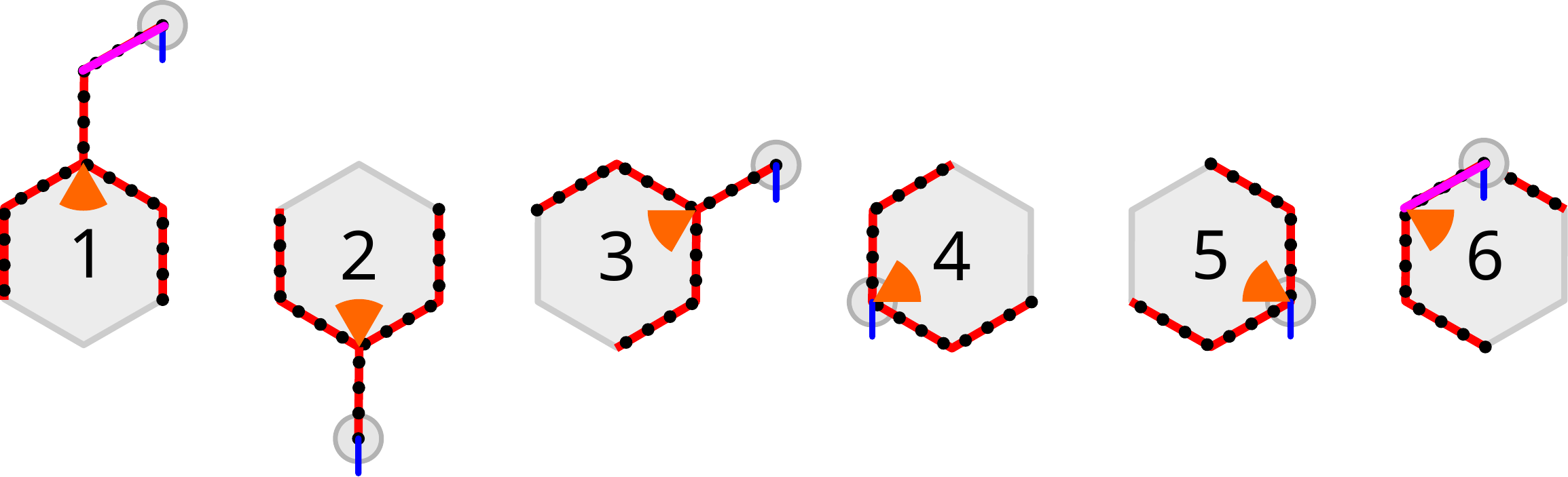}
}

On the figure below we realize that the dot must be environed by a unique assembly. Moreover this assembly uses each possibilities above once.

\image{First, because of the pink segment rule, 1 and 6 must assemble as on the left picture.
The other two hexagons touching the dot can only be 4 and 5 as on the second picture.
Then the hexagons called A and B can only be 3 and 2.
}{fig:unsub-3}{
\includegraphics[scale=0.25]{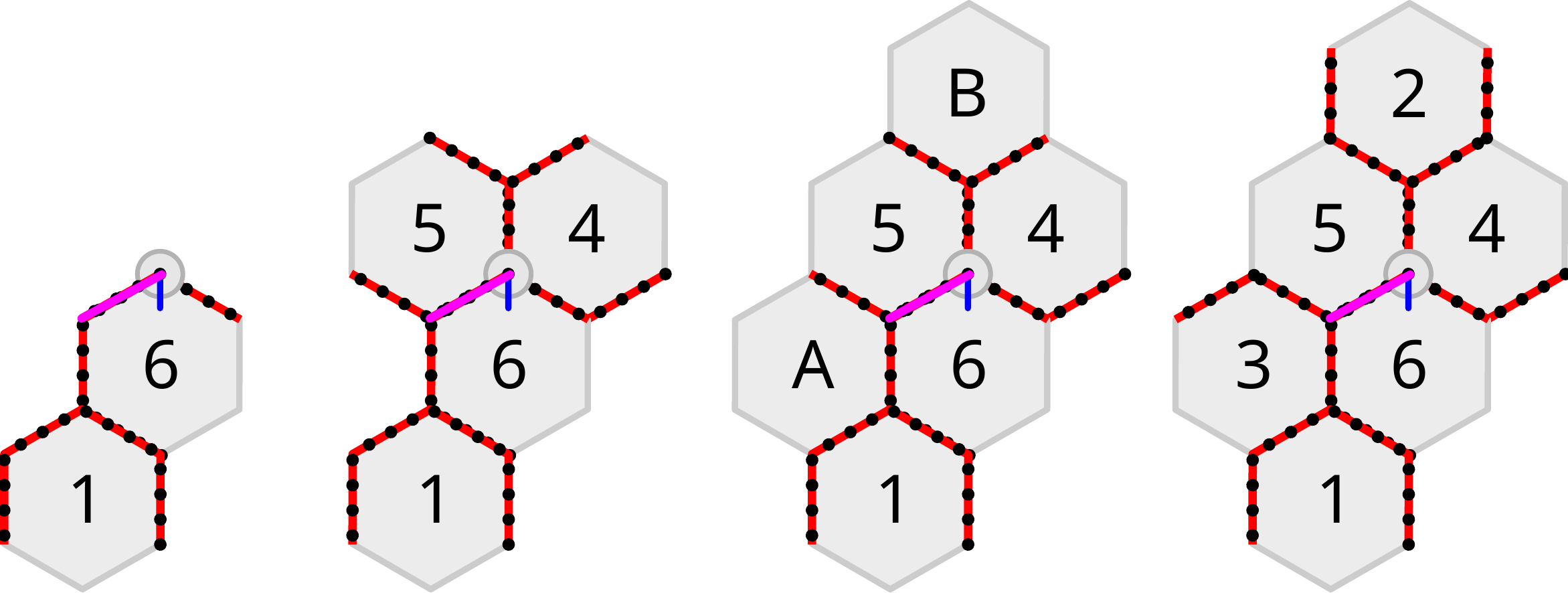}
}

We now examine the effect of placing the green pack on this configuration:

\image{}{fig:unsub-4}{
\includegraphics[scale=0.25]{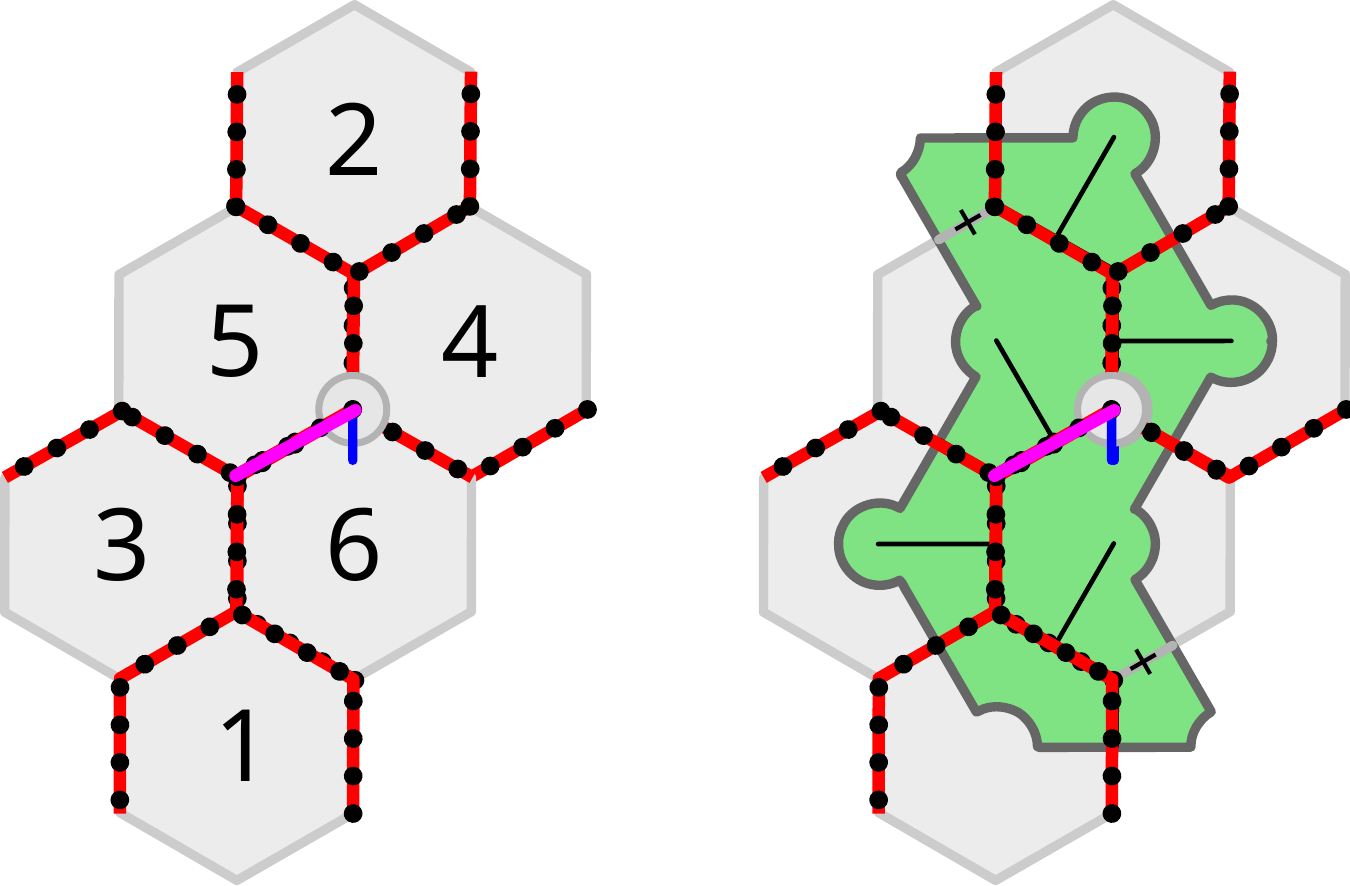}
}

Keeping this in mind we will proceed as follows: first for each dot we put the green pack as above.
We claim that two different dots give disjoint green packs.
Indeed note first that on \Cref{fig:unsub-4}, the only vertices of the honeycomb that a green pack covers belong to the boundary of the reflected cc piece containing its dot.
Then on \Cref{fig:unsub-1} we realize that:

\textit{Fact 1:} the vertices of the cluster that are covered by the pack are in 4 cases the tip plus the next vertex along the cluster boundary in the clockwise direction, and in cases 5 and 6 the two preceding ones.

On a $\hT2$ and a $\hT3$, these sets of vertices, associated to any two different tips, are disjoint.
On a $\hT1$ the three tips have antennae and thus we also have disjointness.
We also must check disjointness of the disks at triangle corners, i.e.\ at hex centres.
Actually disks from the green packs are only placed at hexes that contain a tip of the cluster they belong to.
So for the three dots of a given reflected cc piece, they are different hexagons, unless we are in the case of a $\hT1$.
In that case, since the only reflected cc piece in our collection with a $\hT1$ is as on the figure below, for which in the two allowed cases\footnote{One of which actually never appears if the situation of the proposition covers the whole plane, but that does not matter here.} there is no overlap:

\image{}{fig:unsub-6}{
\includegraphics[scale=0.22]{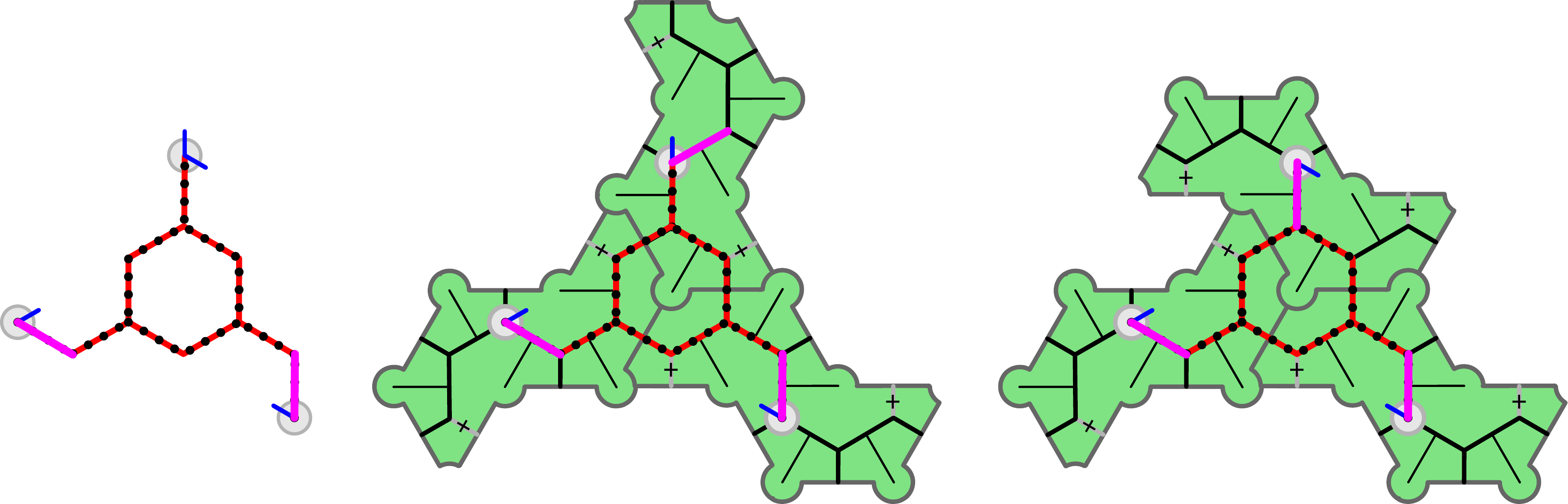}
}

Second, for each $\hT2$ we put a $-,-,-$ tile in the centre and for each $\hT3$ we put a $+,+,+$ one, surrounded by three $-,-,-$ tiles:

\image{}{fig:unsub-5}{
\includegraphics[scale=0.25]{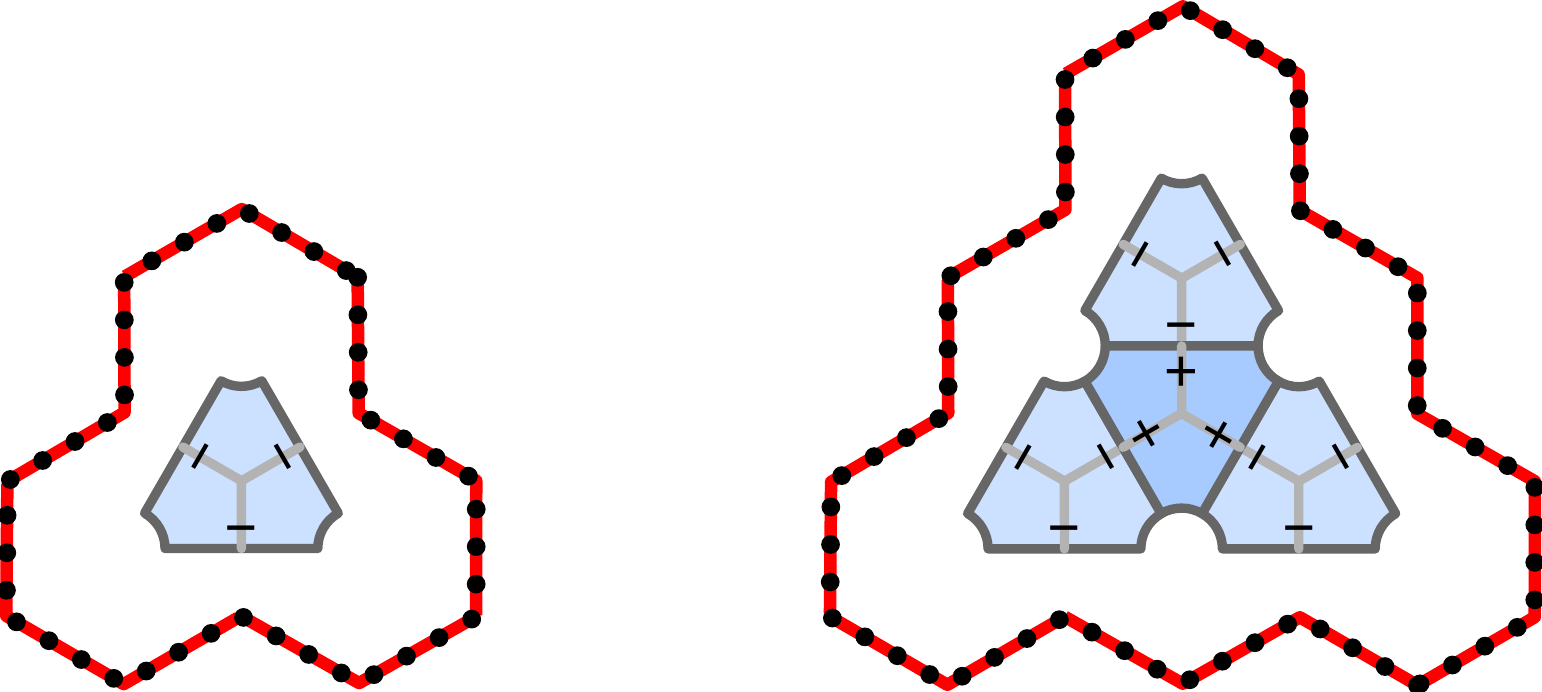}
}

There remains to place yellow tiles. For each $\hT3$ clusters we place three ones as follows:

\image{}{fig:unsub-7}{
\includegraphics[scale=0.25]{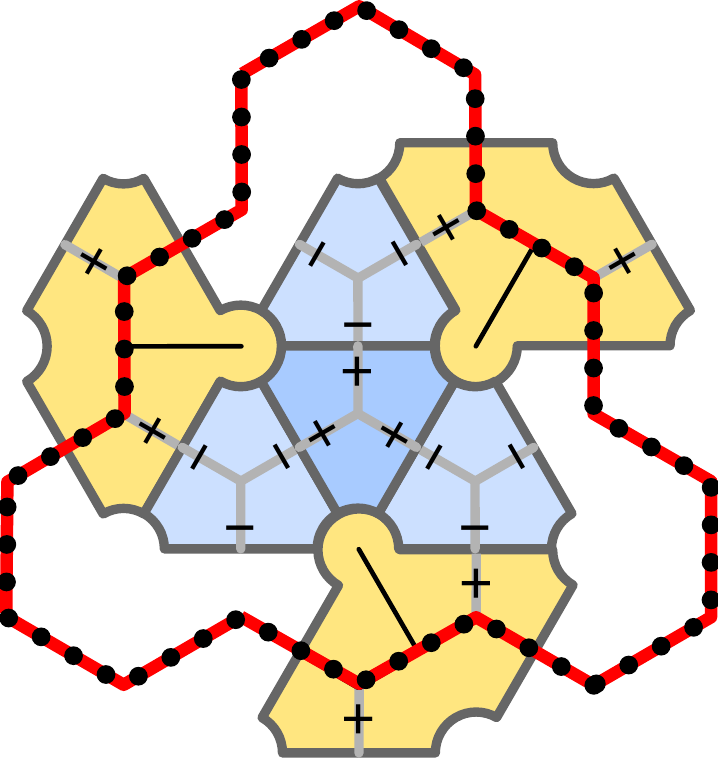}
}

Note that they cannot overlap the hex vertices covered by the green packs, which by Fact~1 above only cover hex vertices that belong to a cluster boundary (this excludes the blue packs) and are either at a tip or one vertex after a tip or one or two vertices before a tip, in the clockwise direction (this excludes the yellow packs).
Note that to two different $\hT3$ clusters we associated disjoint yellow pack pieces: the only way it could fail to be the case is if two $\hT3$ clusters are placed as on (a rotation of) the figure below, but the matching rules of \Cref{fig:copy_refl} do not include such a possibility.

\image{Top: colliding yellow pieces in pink. This situation cannot happen because of the matching rules \Cref{fig:copy_refl}. Bottom: a failed attempt at obtaining the arrangement above.}{fig:unsub-10}{
\begin{tikzpicture}
  \node at(0,0) {\includegraphics[scale=0.25]{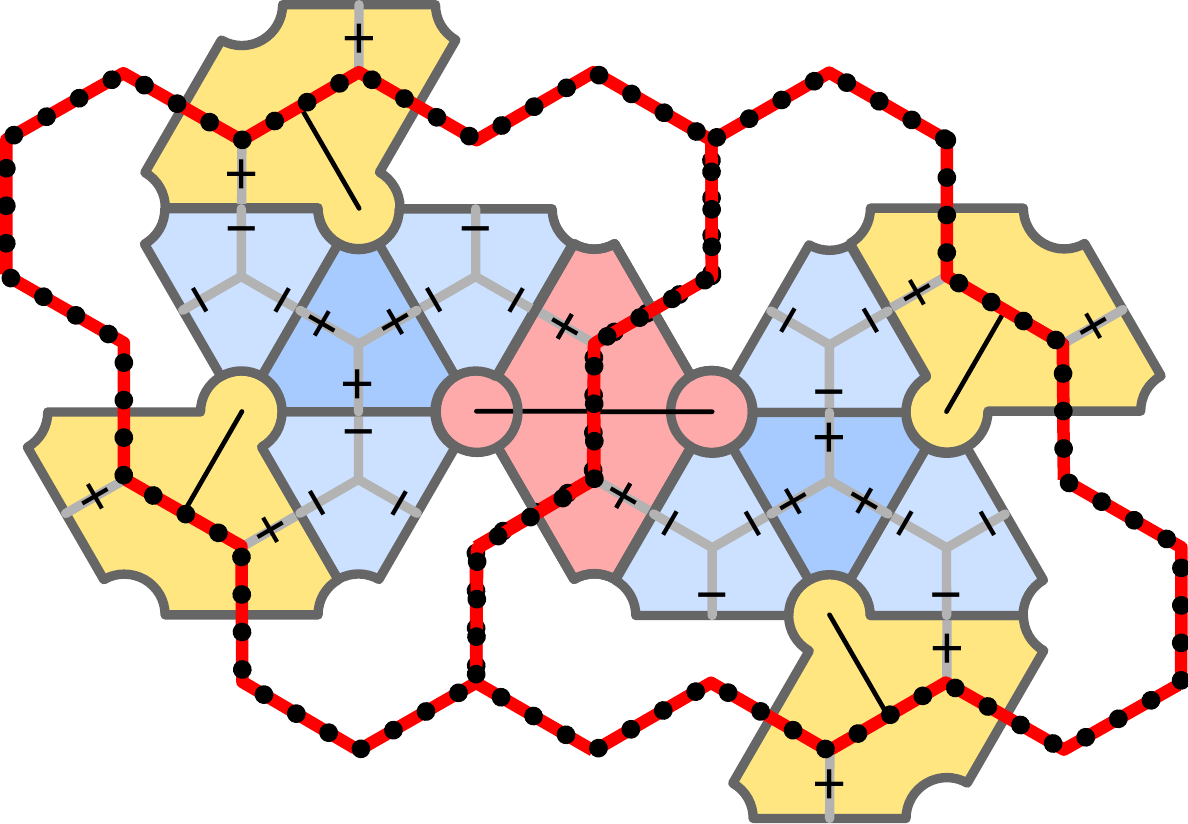}};
  \node at(0,-3.5) {\includegraphics[scale=0.5]{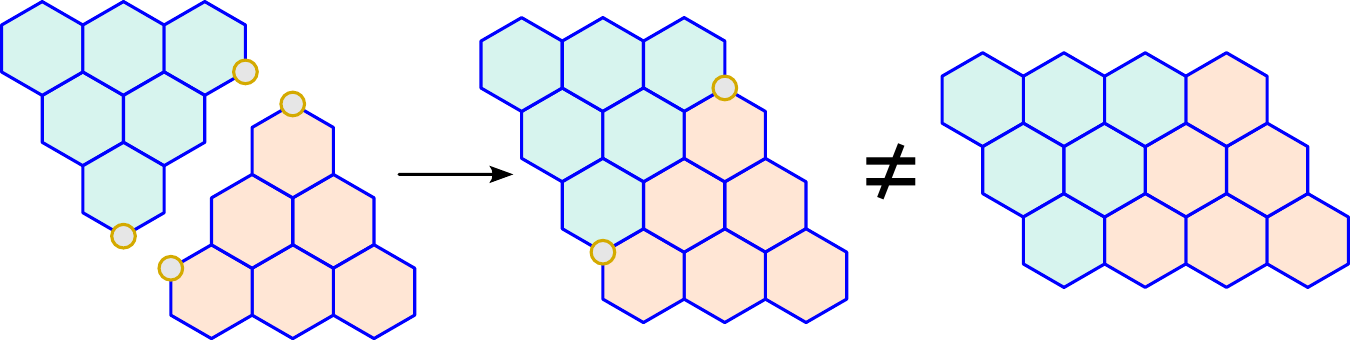}};
\end{tikzpicture}
}

Then we place a yellow tile as follows for each type 1 tip of a reflected cc piece with a $\hT2$, on the side preceding it, in the clockwise order. We show the three allowed cases in the figure: the yellow piece is placed the same way in each.

\image{}{fig:unsub-8}{
\includegraphics[scale=0.25]{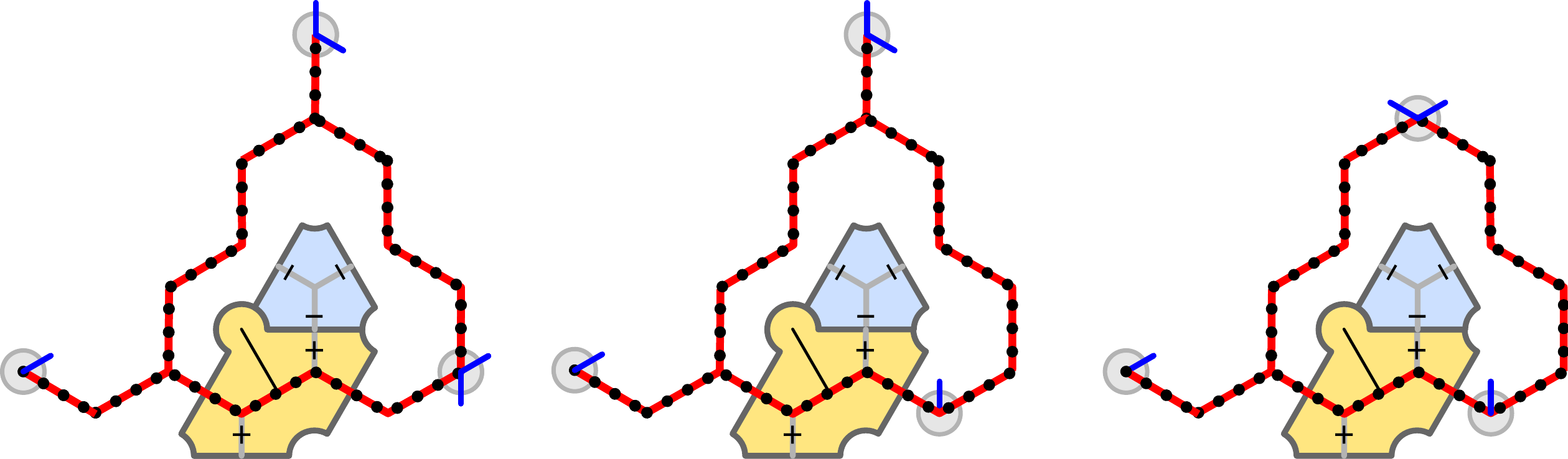}
}

Using Fact 1 again, it covers honeycomb vertices disjoint from the green packs. It is also disjoint from the already placed yellow pack tiles: it is obvious for the disk of the yellow pack tile and the honeycomb vertices they cover belong to an interface that does not match with those of the reflected cc pieces of $\hT3$ type if we take the gray dot orientation into account: the interface would have to be of type $\iI$ but the orientation of the two left dots on the figure below are both imposed and do not match (actually, even if we disregard the orientation of the gray dot and fit the cc pieces, their respective yellow pieces are disjoint).

\image{The interface of a reflected cc piece with a $\hT2$ and preceding a dot at the end of a length 2 antenna, cannot match with a $\hT3$. The only interface shapes that could match would have the shapes above, but then the left gray dots orientation are different.}{fig:unsub-9}{
\includegraphics[scale=0.25]{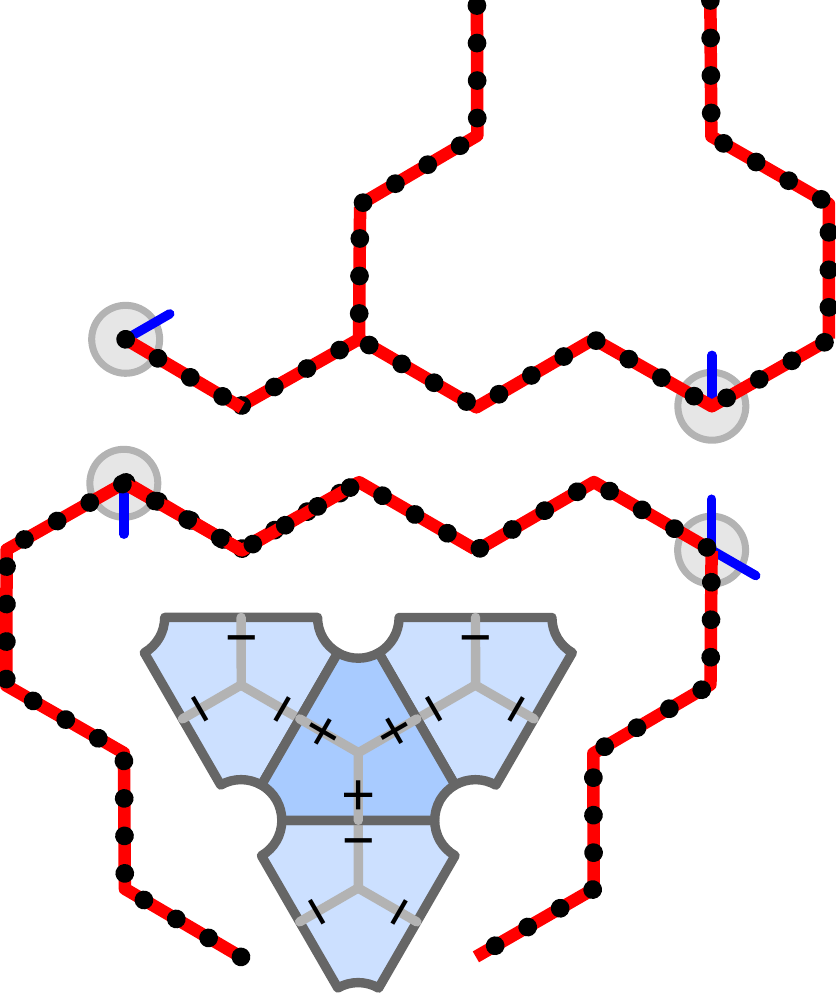}
}

\image{Even if we disregard the gray dot orientation and match the boundaries with dots, the yellow pieces are disjoint.}{fig:unsub-11}{
  \includegraphics[scale=0.25]{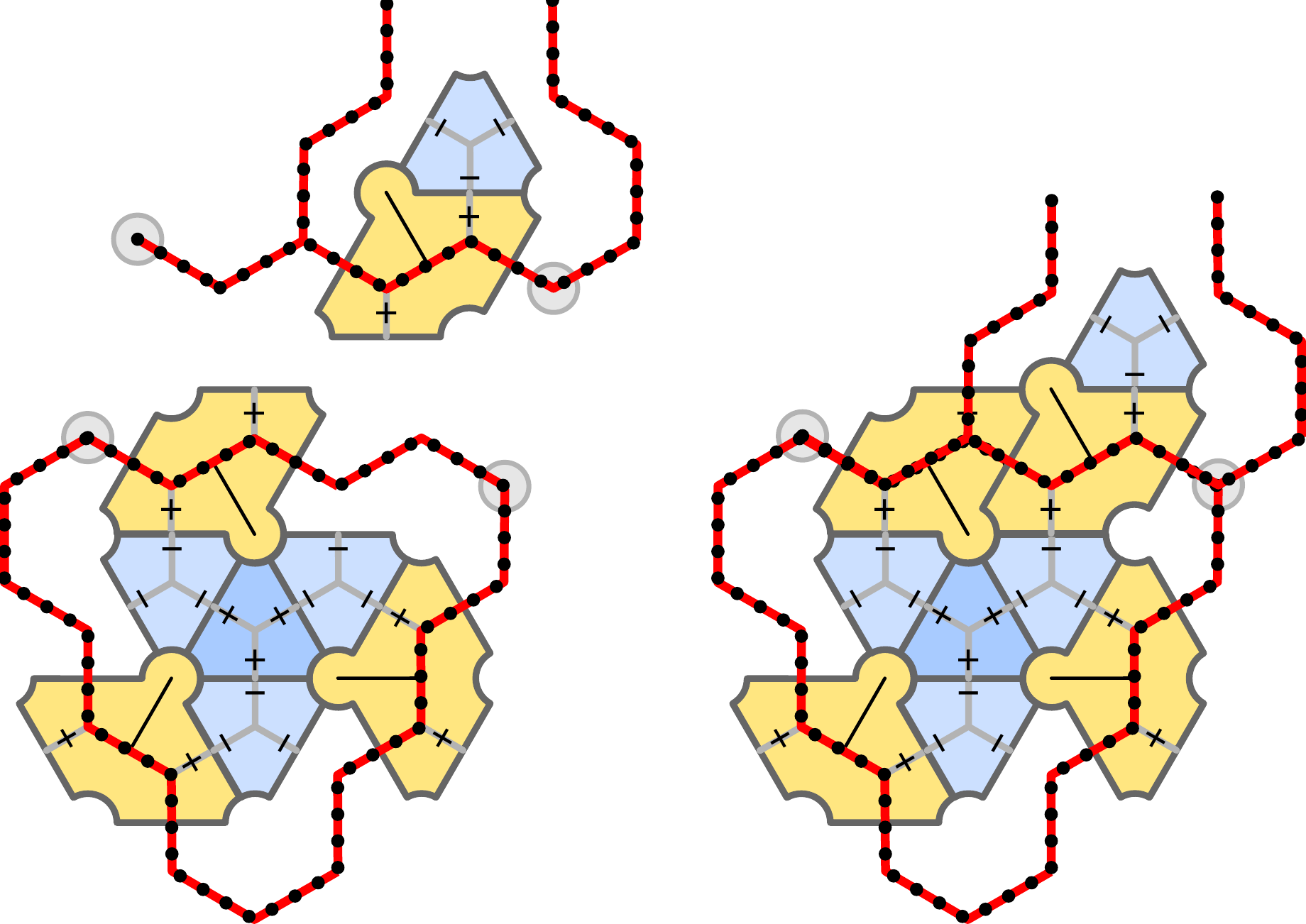}
}

At this stage of the the proof of \Cref{prop:inv}, we have associated to the data described in the hypotheses a set of pack pieces layed on the honeycomb partition of the plane and proved that these pieces cover every hexagon vertex once and only once.

There oly remains to check that every hexagon centre is covered once and only once.
First they are all covered:
\begin{itemize}
  \item Hexes that do not have a tip on their boundaries are necessarily in a $\hT3$ and they have their centre covered by the disk of a yellow pack (see \Cref{fig:unsub-7}).
  \item A hex with a tip of type 2 to 6 has its centre covered by a disk in a green pack, according to \Cref{fig:unsub-4}.
  \item A hex with a tip of type 1 that belongs to a $\hT1$ has its centre is covered exactly once by \Cref{fig:unsub-6}.
  \item A hex with a tip of type 1 that belongs to a $\hT2$ has its centre covered by the yellow piece of \Cref{fig:unsub-8}.
\end{itemize}
Finally, we check that no hexagon centre is covered by more than one pack disk:
for one thing, blue pack pieces have no disk.
We already saw that green packs are disjoint, so no two green disks can overlap.
Yellow pack pieces only have one disk.
If the piece is associated to a $\hT3$, the disk is a the centre of one of its tip-free hexagons, and the three yellow pack pieces cover three distinct hexagon centres.
If the piece is associated to a $\hT2$, the disk covers the hex of this $\hT2$ that has a tip of type~1.
So no two yellow disks overlap.
Last, no yellow disk can overlap a green disk: this is because green disk are at the centre of hexagons that contain a tip of type~2 to~5,\footnote{This is not in contradiction with the fact that the unique hex of a $\hT1$ has three tips, two of which are of type 1, because the third one is of type 2 or 3.} while a yellow disk is either at the centre of a tip-free hexagon of a $\hT3$, or at the centre of a $\hT2$ hexagon whose tip is of type 1 (and a hexagon of a $\hT2$ has only one tip).
\end{proof}

An alternative proof is by looking at one of the substitution systems proposed in \Cref{fig:sr,fig:sr3c} and carefully check that at the interfaces, we get no superimposition and we cover everything.

\begin{remark*}
Another alternative proof could be designed using another interpretation of packs made in \Cref{ss:pts}: considering the data of the proposition, follow the diagonal arrow of \Cref{table:1} to get a reflected Spectre tiling. Then prove that the associated hex centres draw packs as on \Cref{fig:de-2}. We will not go through this approach here. 
\end{remark*}

By \Cref{prop:inv}, we can traverse \Cref{table:1} and thus the hierarchy in the other direction. In particular we reprove\footnote{The fact that the proposition is enough is classical: it goes by extracting subsequences converging on bounded sets and performing a diagonal argument.} the companion to the previous corollary, also a results of \cite{chiral}:
 
\begin{corollary}
There are whole plane tilings by the Spectre, without reflections.
\end{corollary}

In the next section we make the procedure of \Cref{prop:inv} more direct and we propose several equivalent substitution systems.

\subsection{Substitution systems}\label{sub:subsys}

Substitution systems are an explicit way of going backward in the hierarchy, i.e.\ refining a tiling into a tiling with smaller/more pieces.

\subsubsection{From cc to reflected triangle packs}
%

One key step in going backwards is \Cref{prop:inv}.
Either its proof, or a careful inspection of the  the list of possible environments determined in \Cref{ss:tp1,ss:tp2,ss:tp3}, give several substitution systems that can work.
If we allow overlaps\footnote{More precisely total overlaps: pieces that are repeated after the substitution, and for which deletion of the extra pieces gives a valid tiling.} it is easy.
We are most interested in substitution systems without overlaps, and it turns out to be possible: for this we attach the green packs only to the $\hT n$ to type 6 points in the nomenclature of \Cref{fig:unsub-1}:

\nopagebreak

\image{}{fig:sr}{
\makebox[\textwidth][c]{\includegraphics[scale=0.24]{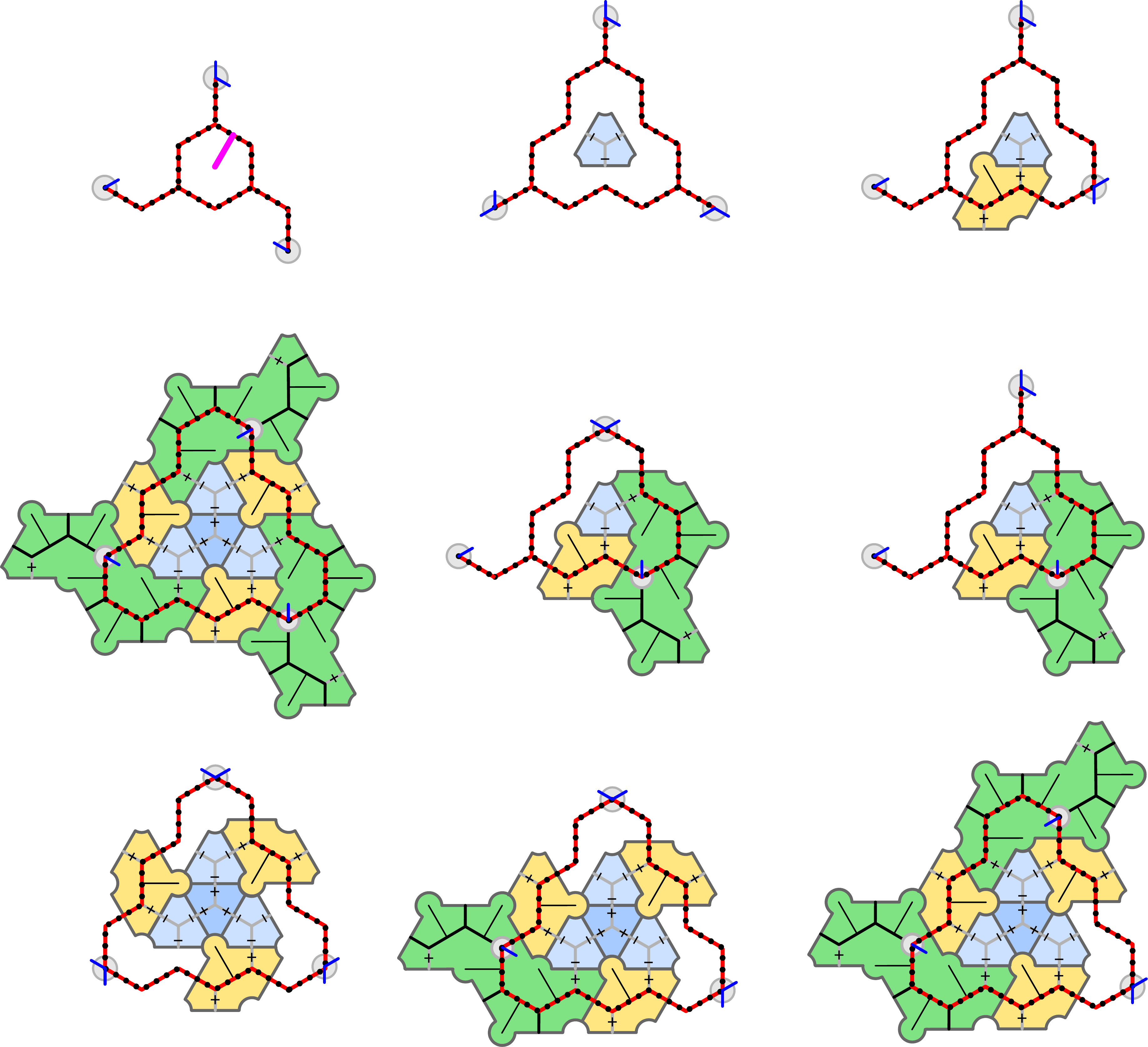}}
}

We get the following variant, somewhat simpler, if we attach the green packs to the length two antenna ends instead:

\nopagebreak

\image{}{fig:sr3c}{
\makebox[\textwidth][c]{\includegraphics[scale=0.24]{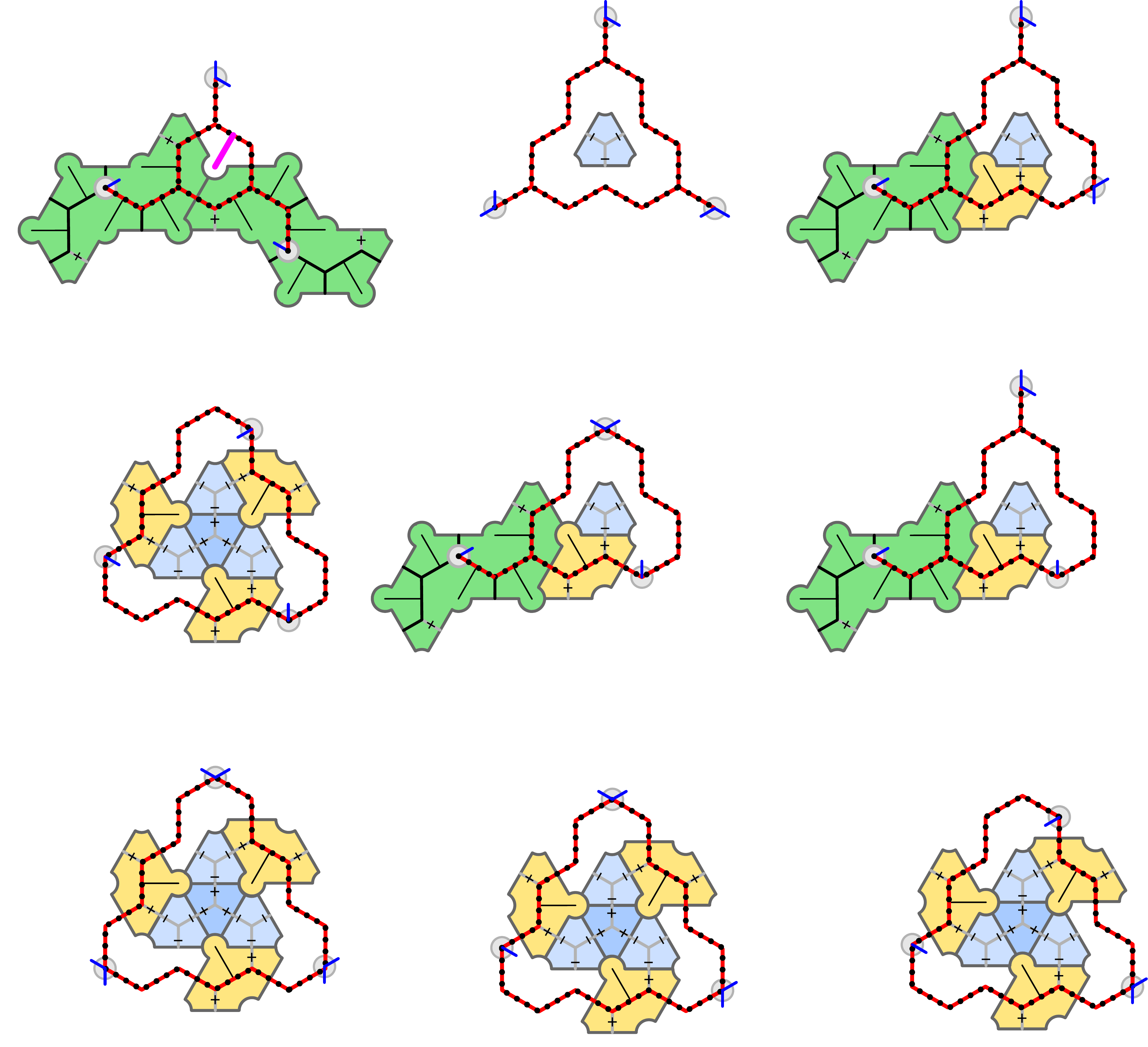}}
}

Conversely, the following figure shows the compacted unreflected blue cc's of the triangle pieces of the packs. Initially this is not a substitution, since this is what the packs represent; however, since we are reversing the procedure, this can be thought of, and used, as a substitution.

\nopagebreak

\image{}{fig:sr-0}{
\includegraphics[scale=0.33]{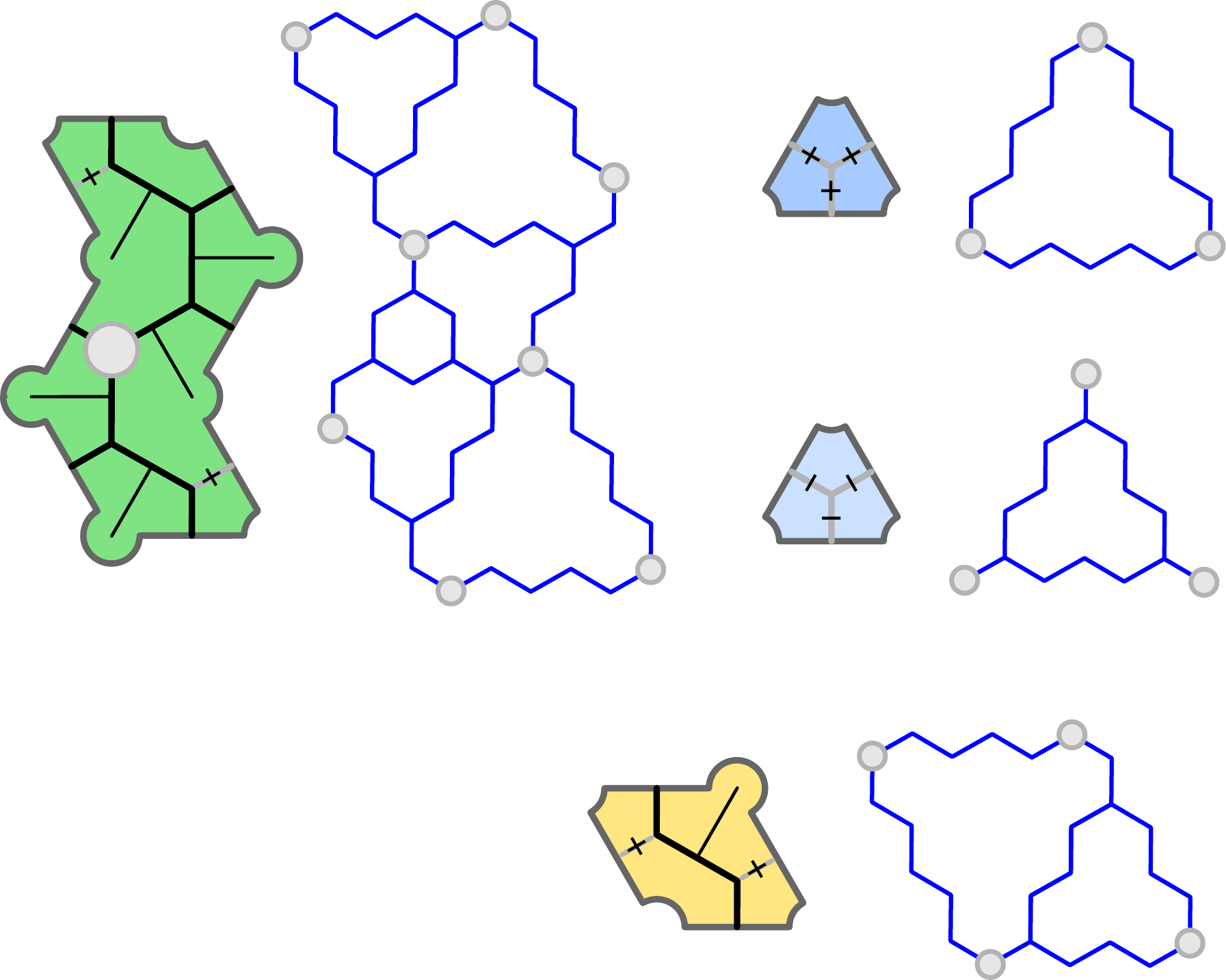}
}

The substitution is then obtained by alternating the two substitutions above (and making reflections).

\image{The two-step substitution procedure}{table:2}{
\small
\begin{tikzpicture}[every node/.style={align=center}, arrows={[scale=3]}]
\node(0) at (-0.75,0) {etc.};
\node(a2) at (1.5,3) {Spectre\\tiling};
\node(a1) at (1.5,1.5) {cc $+$\\interfaces};
\node(a0) at (1.5,0) {honeycomb\\partition$+$dots};
\node(c) at (4,0) {packs};
\node(e2) at (6.5,3) {reflected\\Spectre\\tiling};
\node(e1) at (6.5,1.5) {reflected cc $+$\\interfaces};
\node(e0) at (6.5,0) {reflected\\honeycomb\\partition$+$dots};
\node(f) at (9,0) {packs};
\node(1) at (10.5,0) {etc.};

\draw[<->] (a2) -- (a1);
\draw[<->] (a1) -- (a0);

\draw[<-,thick] (0) -- (a0);
\draw[<-,thick] (a0) -- (c);
\draw[<-,thick] (c) -- (e0);
\draw[<-,thick] (e0) -- (f);
\draw[<-,thick] (f) -- (1);

\draw[<->] (e2) -- (e1);
\draw[<->] (e1) -- (e0);
\end{tikzpicture}
}

On \Cref{fig:sr-0}, it is apparent that to the four packed pieces correspond in the honeycomb polygons traced by the marked points: an octagon with a half-turn symmetry, two regular triangles, and a parallelogram. Actually the octagon is a regular hexagon cut in two halves along a diagonal, where the two halves have been translated away from each other and linked. It will remain of this type under successive substitutions. The two equilateral triangles remain equilateral (and identical) and the parallelogram will remain a parallelogram.

\medskip

To get a Spectre tiling in the end we need to go up in \Cref{table:2}.
One can start from a compacted cc tessellation of all or part of the plane and then space them out using yellow segments oriented as in the interfaces of \Cref{fig:labeled-cc-2}.
Then the rhomb/hex pairing is completely determined, and by replacing in each pair the hex by an $\iD3$ and the rhomb by an $\iD2$ on gets a Spectre tiling.
An example is partially illustrated in \Cref{fig:wir-1}.

This amounts to assembling the following replacement of the cc's:

\image{}{fig:de-4}{
\includegraphics[scale=0.5]{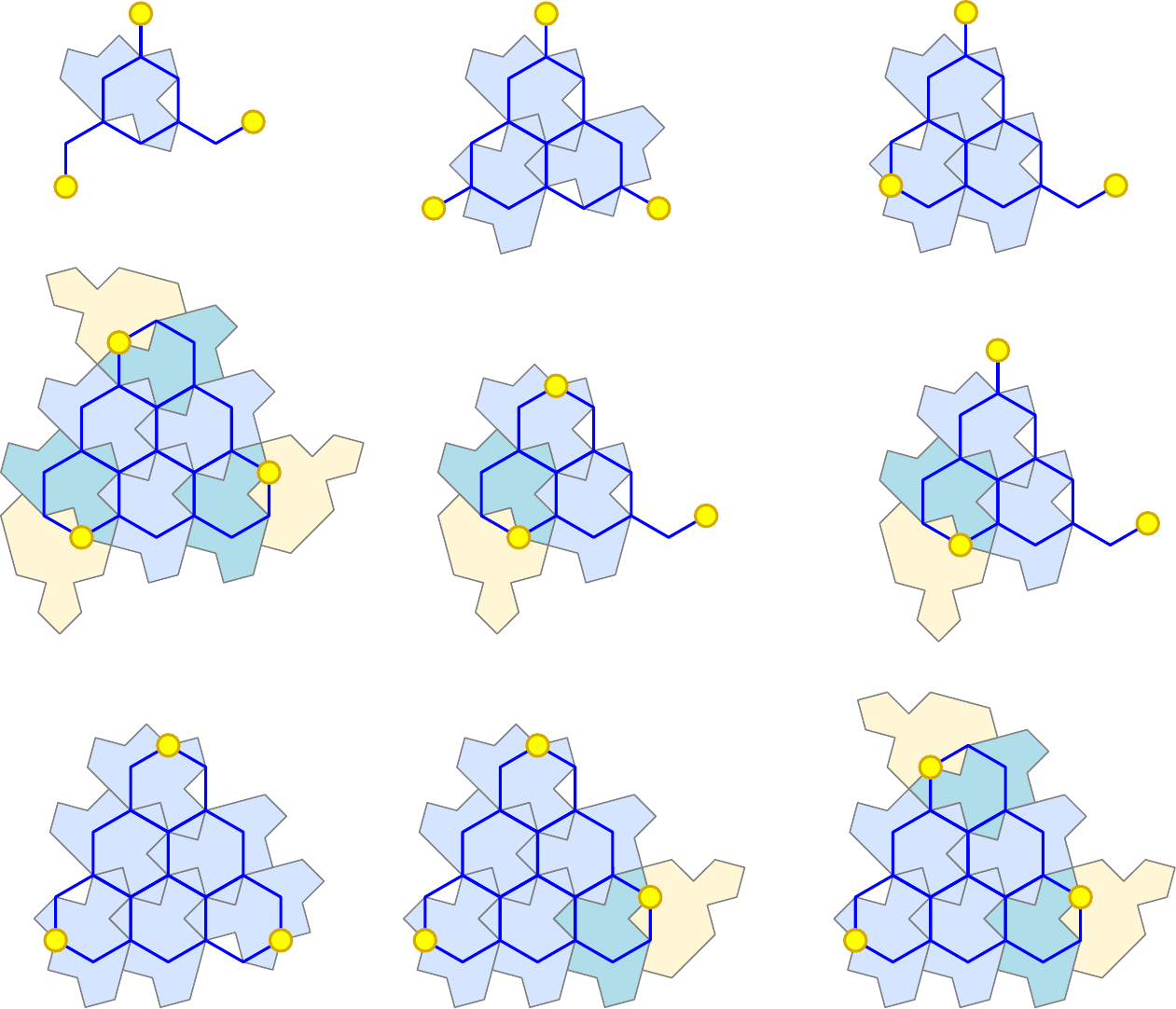}
}

Or equivalently with:

\image{}{fig:de-5}{
\includegraphics[scale=0.5]{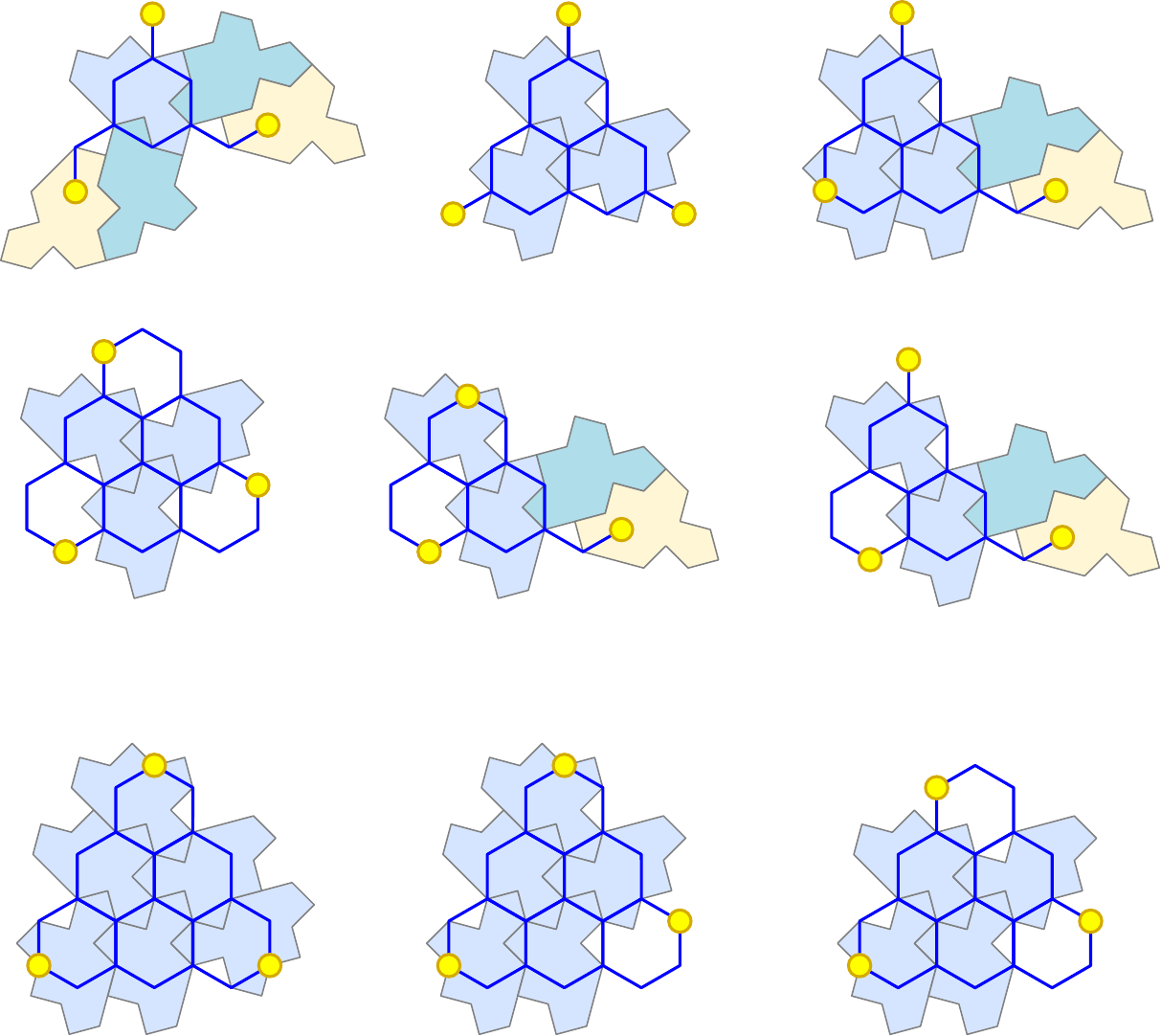}
}

But we can also use another approach: see \Cref{sec:down}.

\medskip

By chaining the two substitutions of \Cref{fig:sr,fig:sr-0} in one order or its opposite, one can express the substitution either  in terms of packed triangle pieces only, or in terms of honeycomb partition~$+$ dots only, as we respectively do in the next two sections.

\subsubsection{In terms of triangle packs only}

Substitution rule in terms of the 4 packed triangle pieces:

\nopagebreak

\image{Keeping track of the dots help to assemble the substituted pieces}{fig:sr-7}{
\makebox[\textwidth][c]{\includegraphics[scale=0.25]{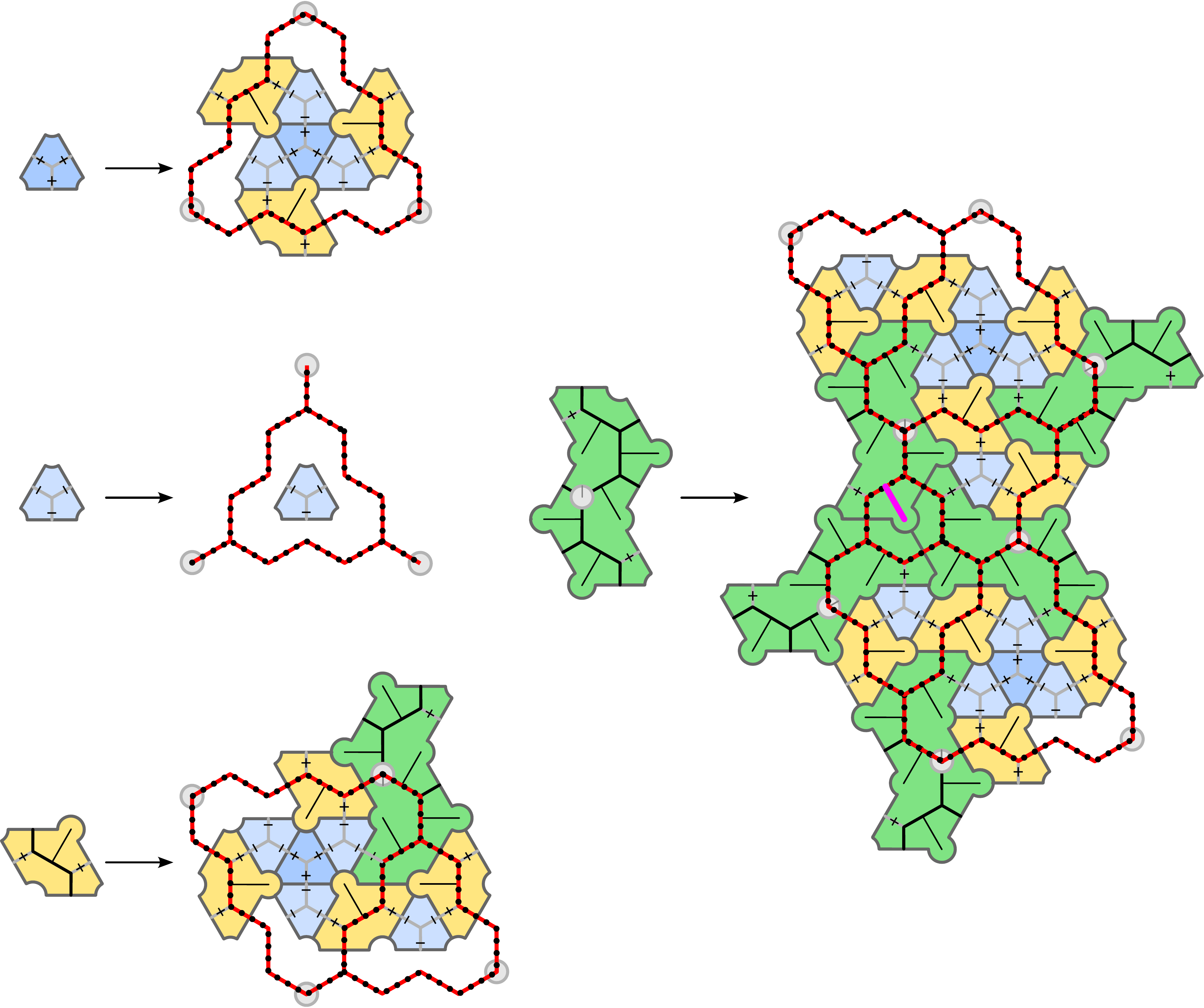}}
}

We see that we can also use the three pieces of \Cref{fig:3-set}.

\nopagebreak

\image{}{fig:sr-8}{
\includegraphics[scale=0.25]{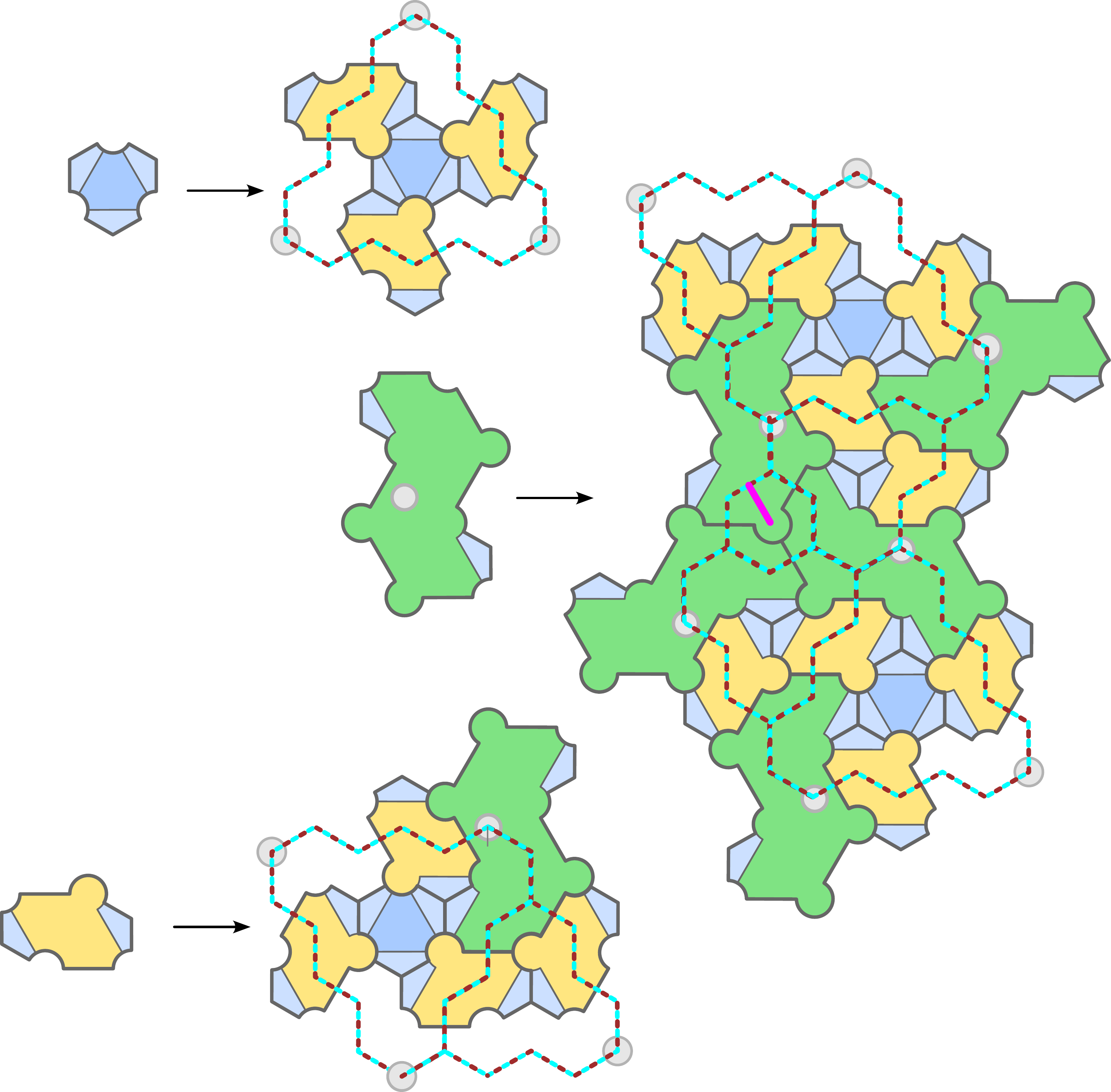}
}

These last two sets are reminiscent of the tile grouping made in the non-chiral tiling of \cite{march}, the H, T, P, F system of what they called meta-tiles.
See more in \Cref{sec:down}.

\subsubsection{Variants}\label{ss:variants}

There are many possible variants.

We can for instance factor each green piece with a yellow one, leading to the turquoise piece and the substitution below: 

\nopagebreak

\image{}{fig:sr-9}{
\makebox[\textwidth][c]{\includegraphics[scale=0.25]{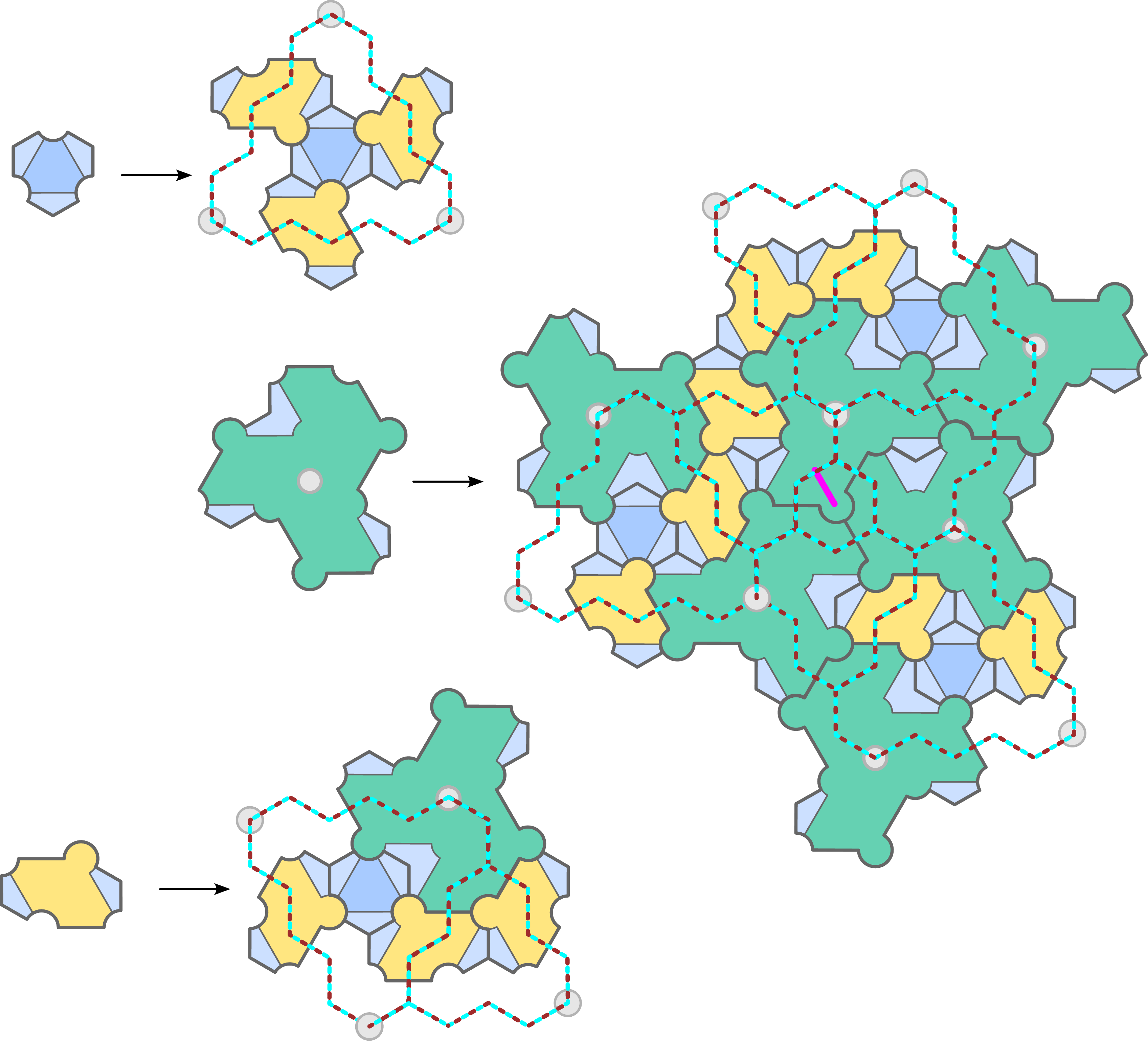}}
}

We could instead start from \Cref{fig:sr-7} and group a yellow and a light blue piece together. 

\nopagebreak

\image{}{}{
\includegraphics[scale=0.33]{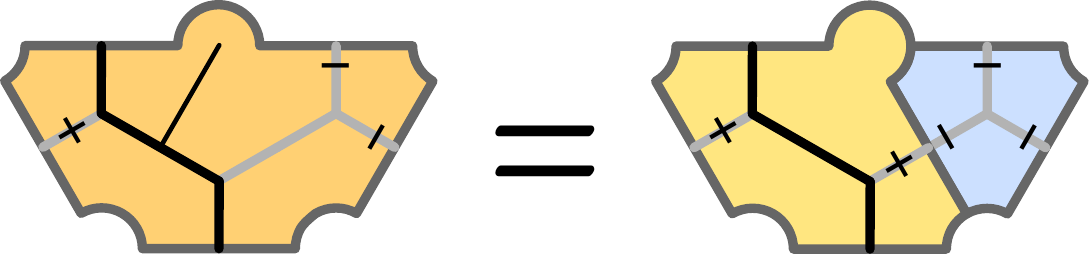}
}

There still remains 4 tiles in this case: 

\nopagebreak

\image{}{fig:sr-7c}{
\makebox[\textwidth][c]{\includegraphics[scale=0.25]{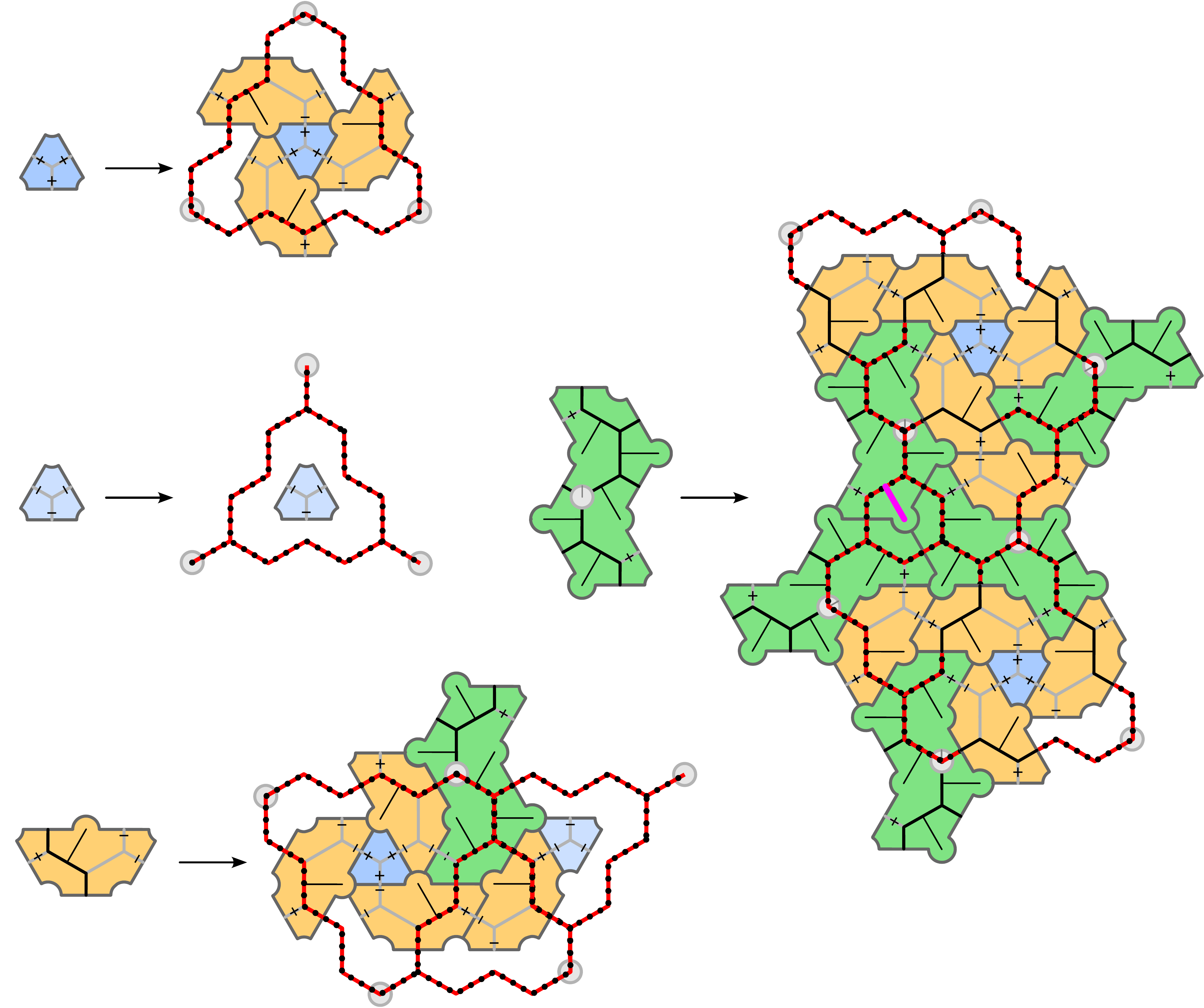}}
}

It is also possible to group tiles as follows (2 variants):

\image{}{fig:violet}{
\makebox[\textwidth][c]{\includegraphics[scale=0.33]{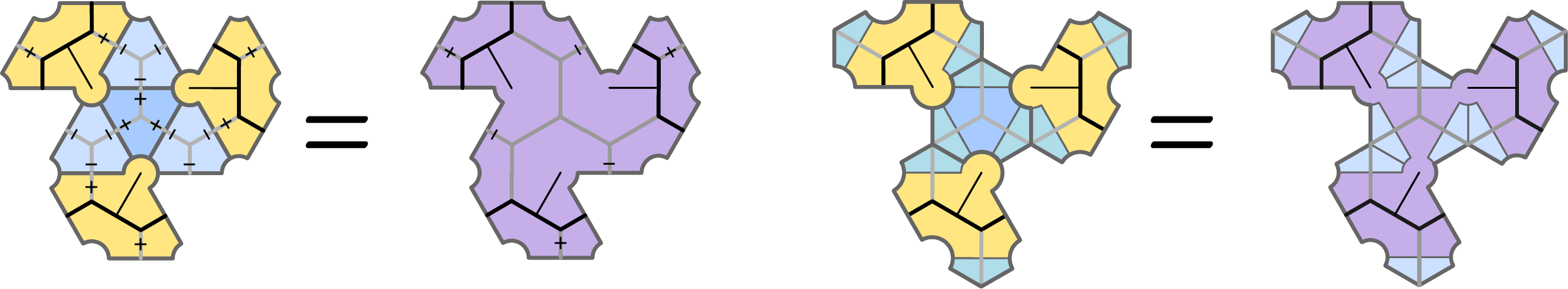}}
}

But we will not show the resulting substitutions here.

\medskip

It is likely that there are a lot of other possible variants.


\subsubsection{In terms of honeycomb partition and dots only}\label{ss:hpdosub}

Finally, we give the substitution rules in terms of $\hT n$+dots only: from \Cref{fig:sr}

\nopagebreak

\image{The \nth{4} configuration never appears in a whole plane tiling.}{fig:sr-10b}{
\makebox[\textwidth][c]{\includegraphics[scale=0.18]{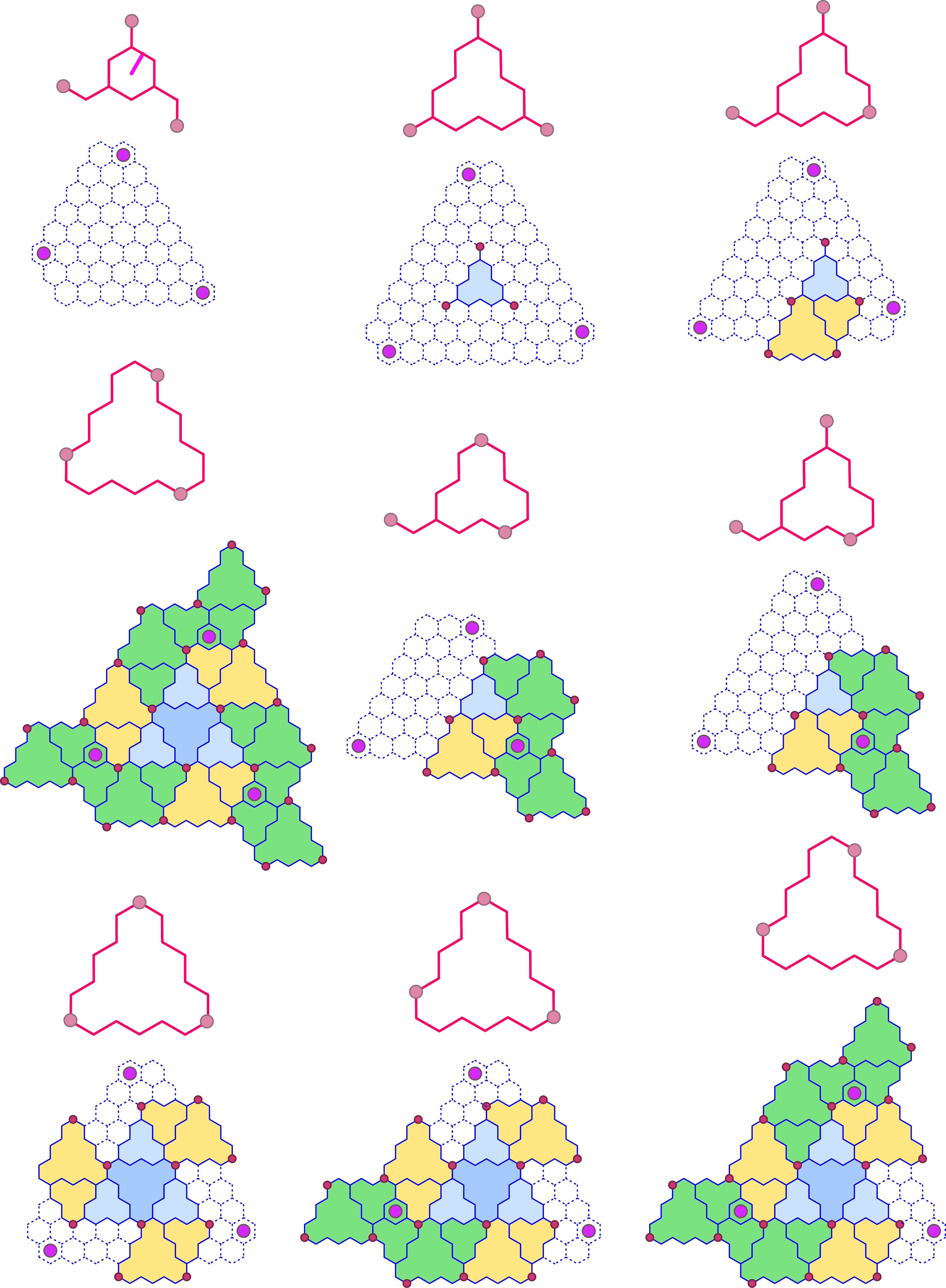}}
}

Or, from the variant \Cref{fig:sr3c}

\nopagebreak

\image{The \nth{4} configuration never appears in a whole plane tiling.}{fig:sr-10c}{
\makebox[\textwidth][c]{\includegraphics[scale=0.18]{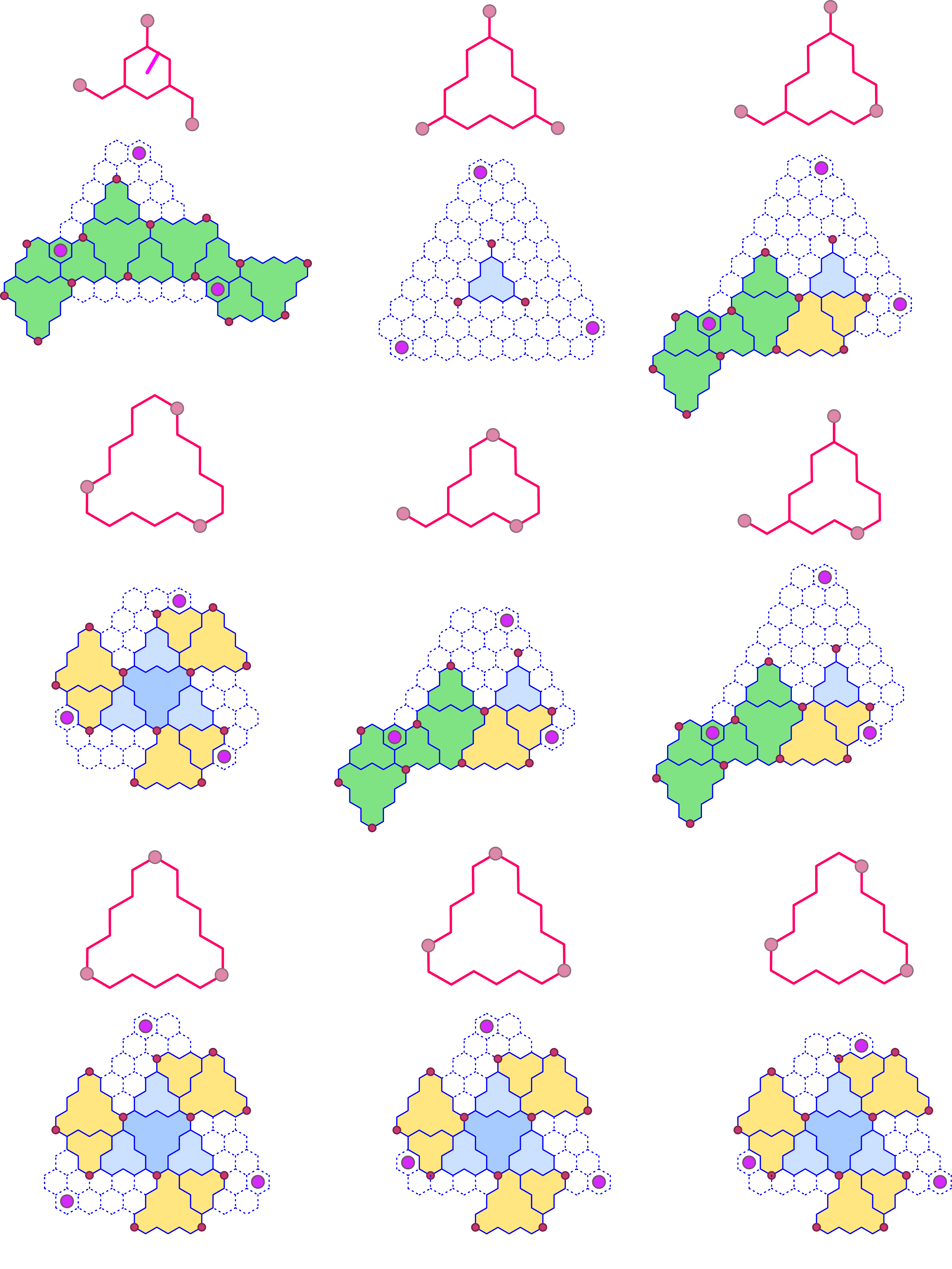}}
}

Consider any of these two substitutions and call it $S$. Consider a whole plane tiling. Its cc's (with dots) can be grouped in packs on which decoration forms a new cc system, which itself can be grouped in packs, etc.
We proved that the initial cc's are necessarily obtained by the substitution $S$ from the bigger ones one level above.
On the figure, it can be seen that the \nth{4} kind of cc with dots does not appear after the substitution. This already proves that no whole plane tiling has the \nth{4} kind of cc.

\begin{proposition}
 In a whole plane tiling by the Spectre, all cc types of \Cref{fig:labeled-cc-2} other than the \nth{4} appear. All triangle packs of \Cref{fig:pack} appear in their triangle representation.
\end{proposition}
\begin{proof}
It makes the argument easier to consider the second substitution system.
Numbering the types of cc from $1$ to $9$, and calling $G$, $Y$, $B$, $D$ the packs (Green with 6 triangles, Yellow with 2, Blue and Dark blue), we get that in the first substitution: from \Cref{fig:sr3c}
$1\to G$,
$2\to B$,
$3 \to G,Y,B$,
$5 \to Y,B,D$,
$6 \to G,Y,B$,
$7 \to Y,B,D$,
$8 \to Y,B,D$,
$9 \to Y,B,D$,
and from \Cref{fig:pack,fig:labeled-cc-2}, the packs
are composed of triangles representing:
for $G$: $1,3,5,6,8,9$; for $Y$: $3$ and $8$; for $B$: $2$; for $D$: $7$.
Starting from anywhere but $2$ and $B$, we see that the set of descendants is everything (apart 4) at every generation.
Starting from $2$ we get only $B$ and from $B$ we get only $2$.
It is impossible to cover the honeycomb only using pieces of type $2$. The claim follows.
\end{proof}

\subsection{Going down one more level}\label{sec:down}

\subsubsection{Packs to Spectres}\label{ss:pts}

Comparison of \Cref{fig:sr,fig:sr-0,fig:odd-env-2,fig:labeled-cc-2} suggests to use the following substitutions from \emph{reflected} packs to Spectre patches.

\nopagebreak

\image{Note: packs are reflected}{fig:p2s}{
\includegraphics[scale=0.5]{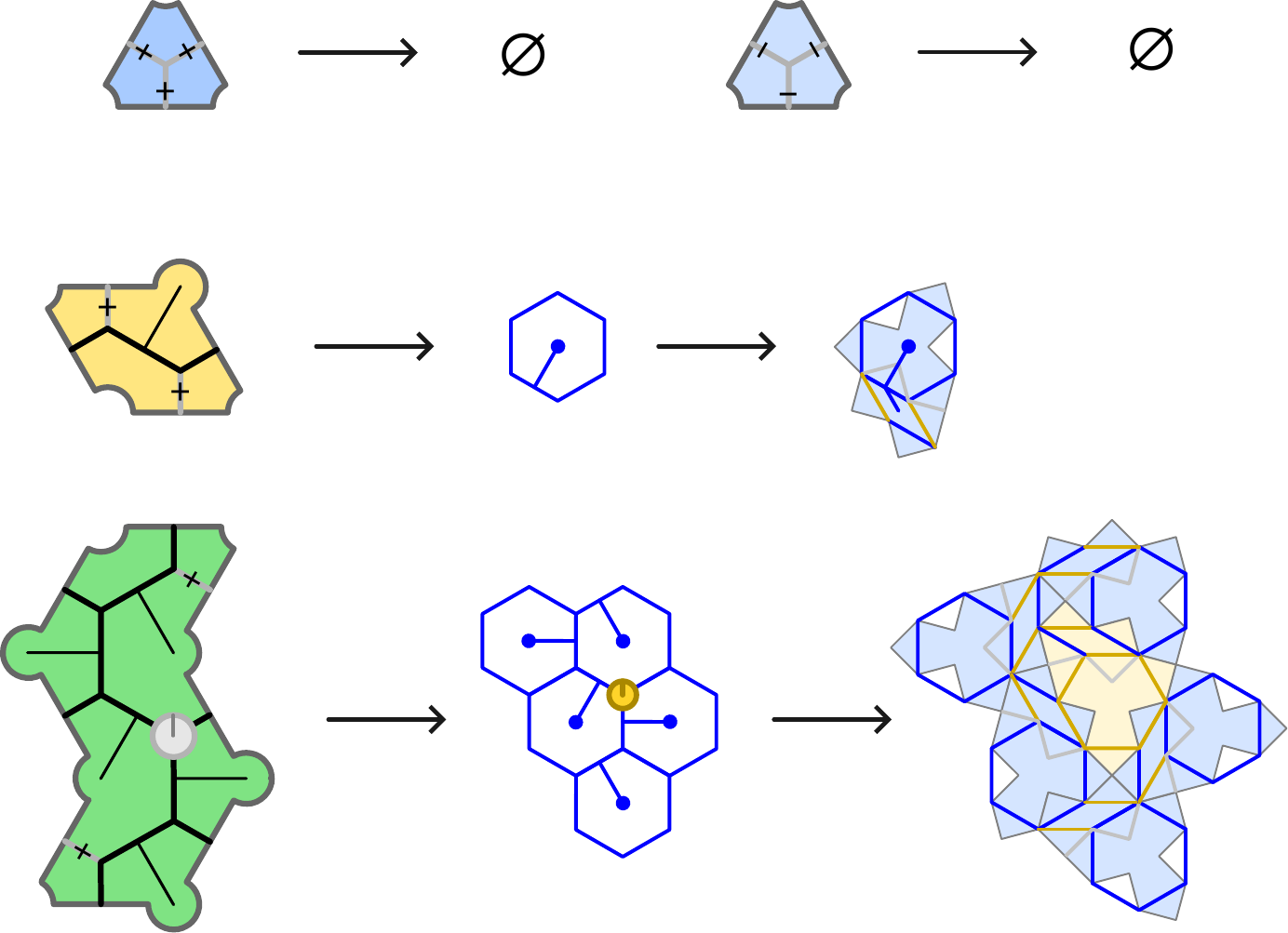}
}

This suggests the following variants of \Cref{fig:de-4,fig:de-5} for the final substitution step after spacing-out cc's:

\nopagebreak

\image{}{fig:de-3}{
\makebox[\textwidth][c]{\includegraphics[scale=0.5]{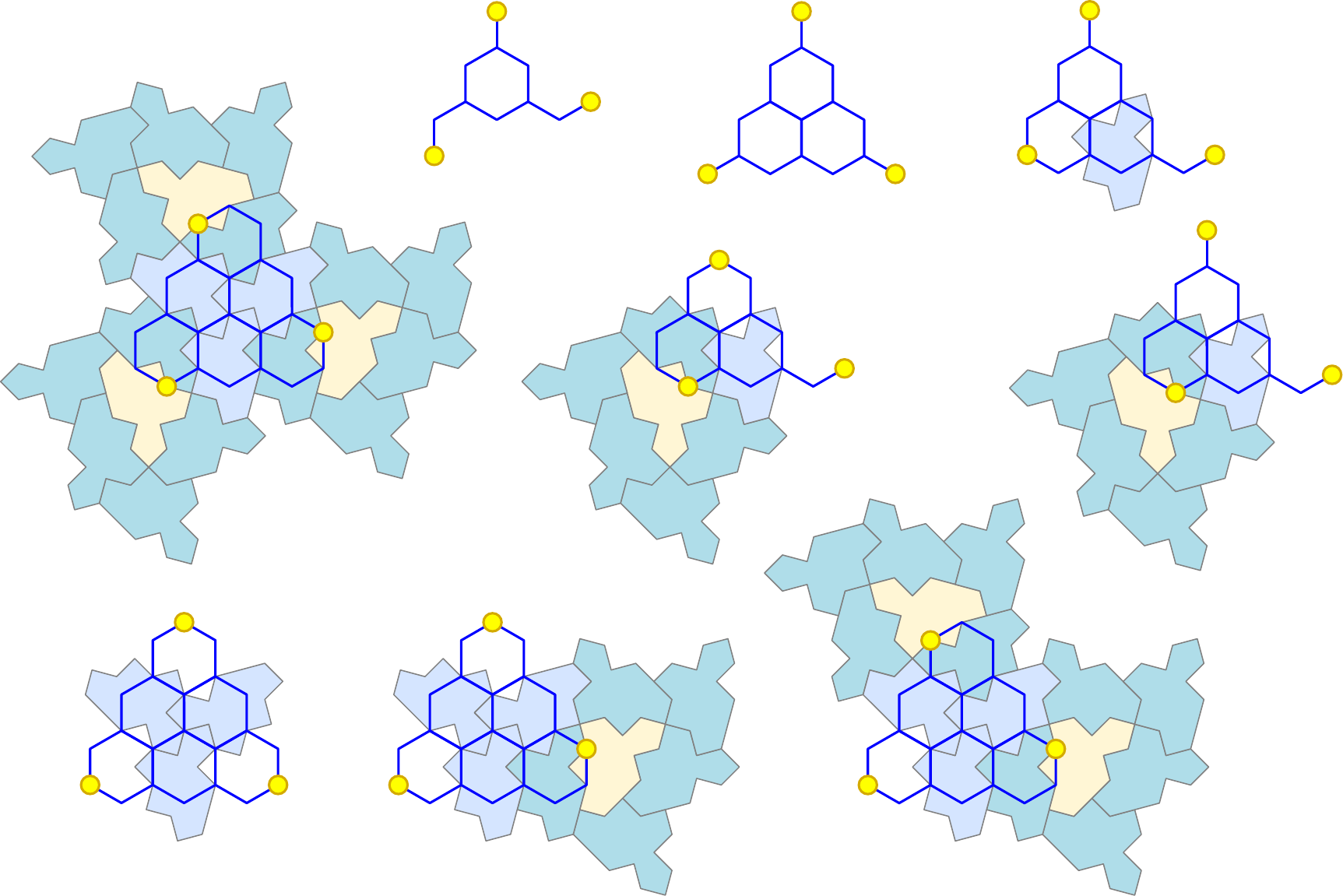}}
}

\image{}{fig:de-6}{
\makebox[\textwidth][c]{\includegraphics[scale=0.5]{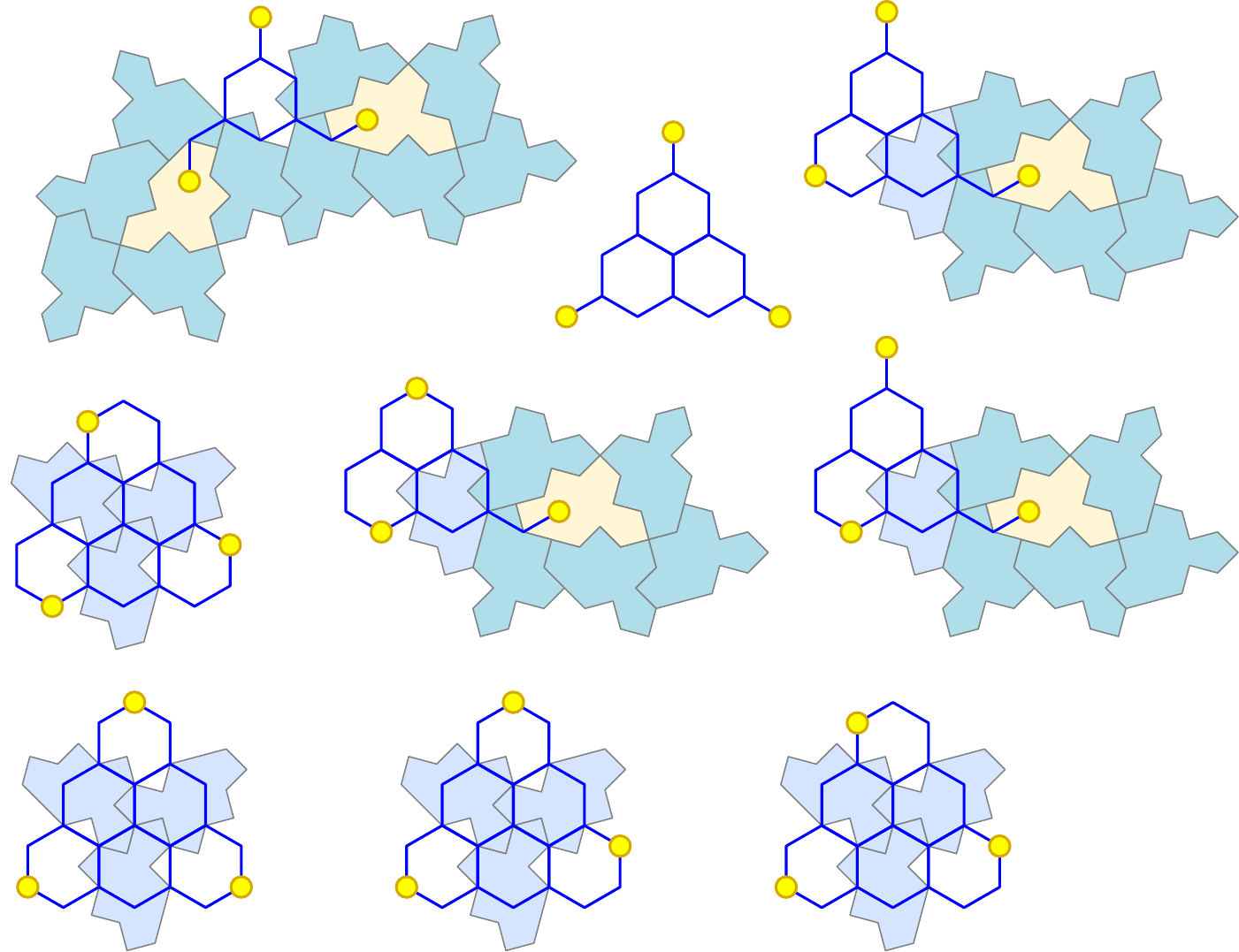}}
}

It also motivates the following grouping into green and yellow ``packs'' on an example of the Spectre tiling (compare with \Cref{fig:s4}):

\nopagebreak

\image{Yellow ``packs'' are single yellow Spectre tiles, while green ``packs'' consist in one light gray odd tile and 5 green Spectre tiles.}{fig:extract}{
\includegraphics[scale=0.54]{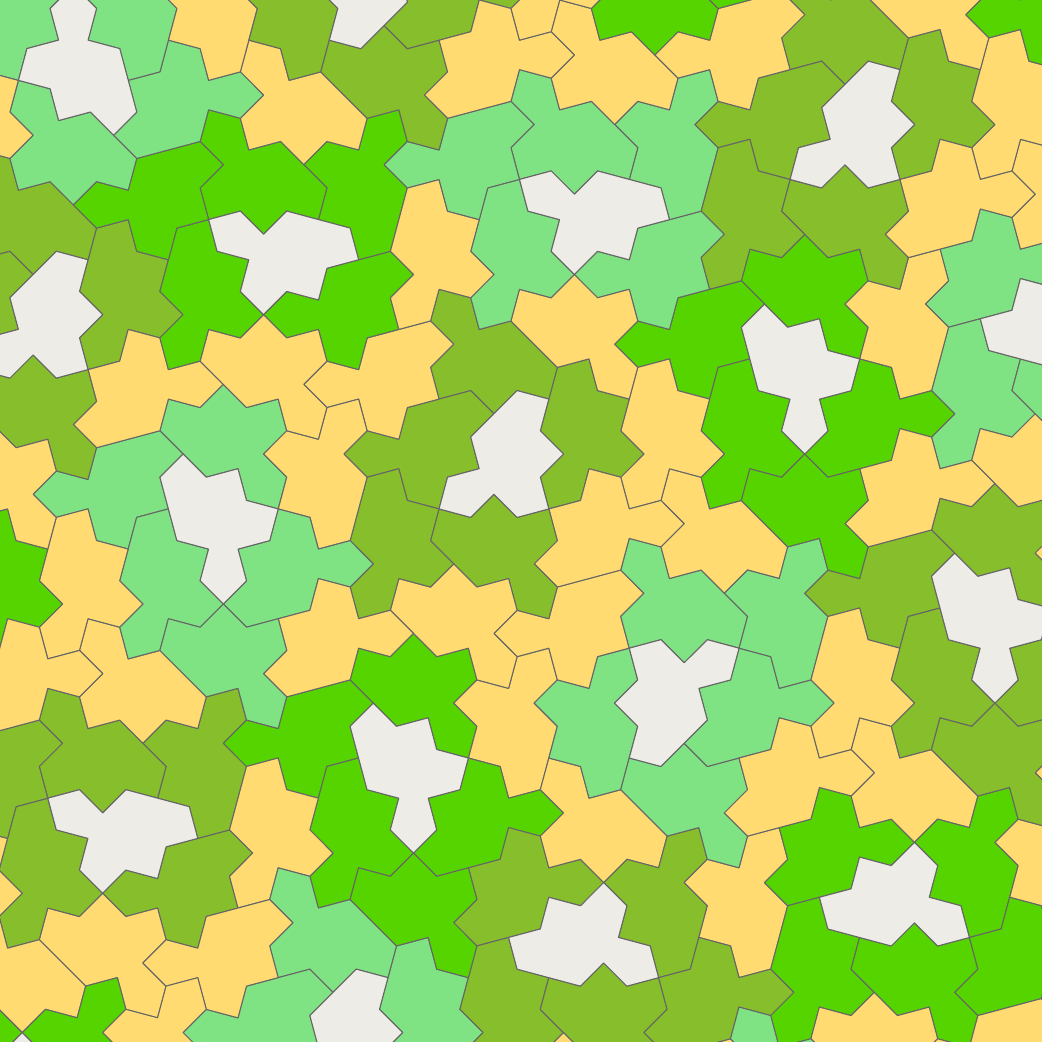}
}

\nopagebreak

\image{Adding the yellow/blue decoration graph and tracing out the triangle patches, using blue hex centres as vertices:}{fig:de-2}{
\includegraphics[scale=0.54]{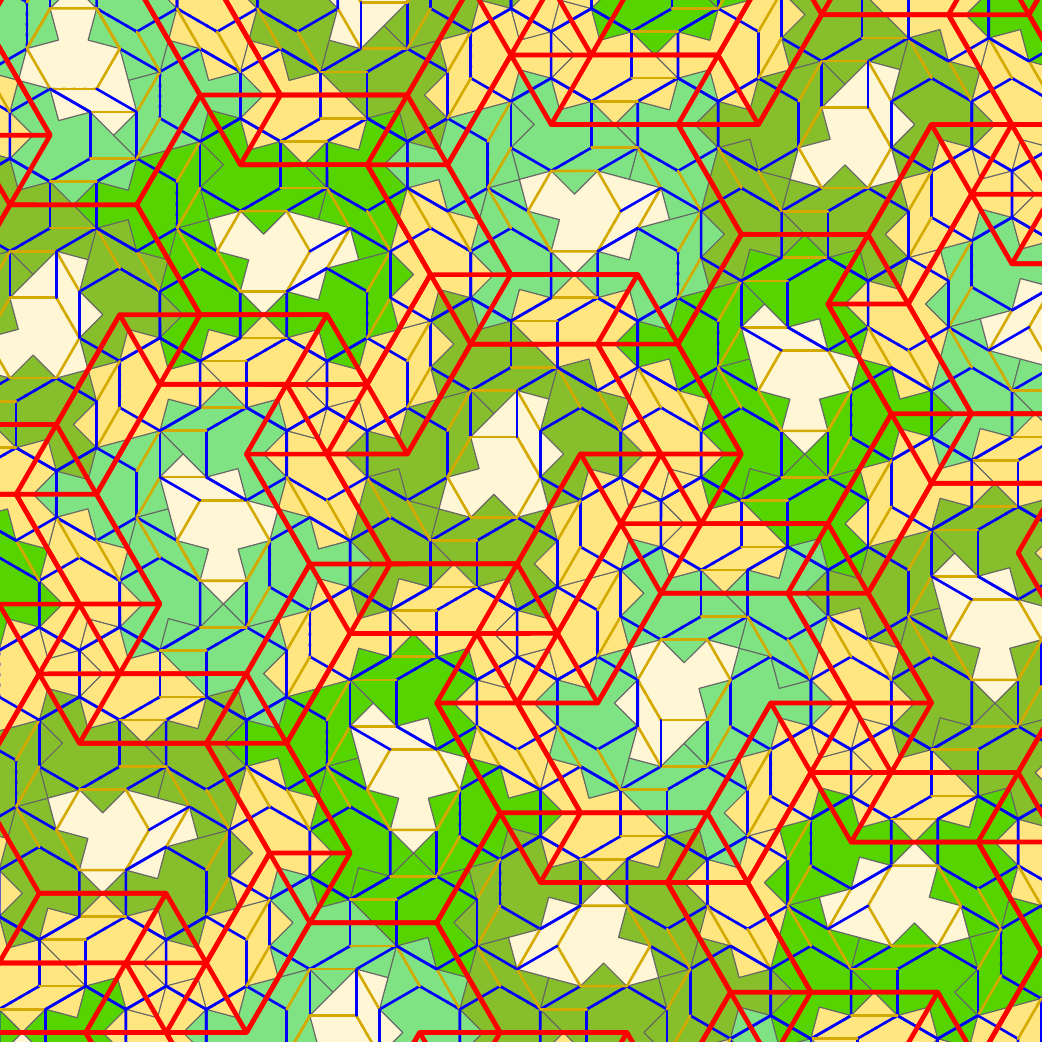}
}

We interpreted packs as a Spectre packing instead of packing of triangles representing cc's. To complete this approach, it would be nice to interpret the triangles in terms of meeting points of 3 even Spectres or $\iD3$, and an optional odd one.
Based on \Cref{fig:de-2}, it seems that the following association can be made for the packs (it extends \Cref{fig:p2s}; if used as a substitution it gives compatible superimpositions).

\nopagebreak

\image{}{}{
\includegraphics[scale=0.5]{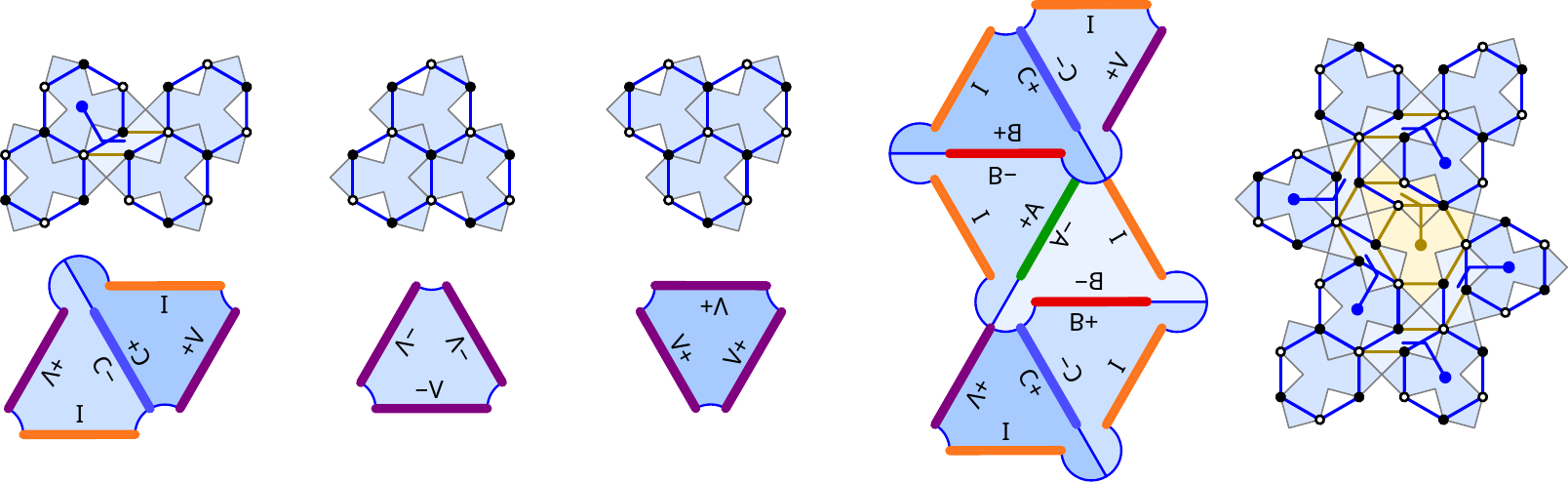}
}

Yielding the following association for the triangles:

\image{}{}{
\includegraphics[scale=0.5]{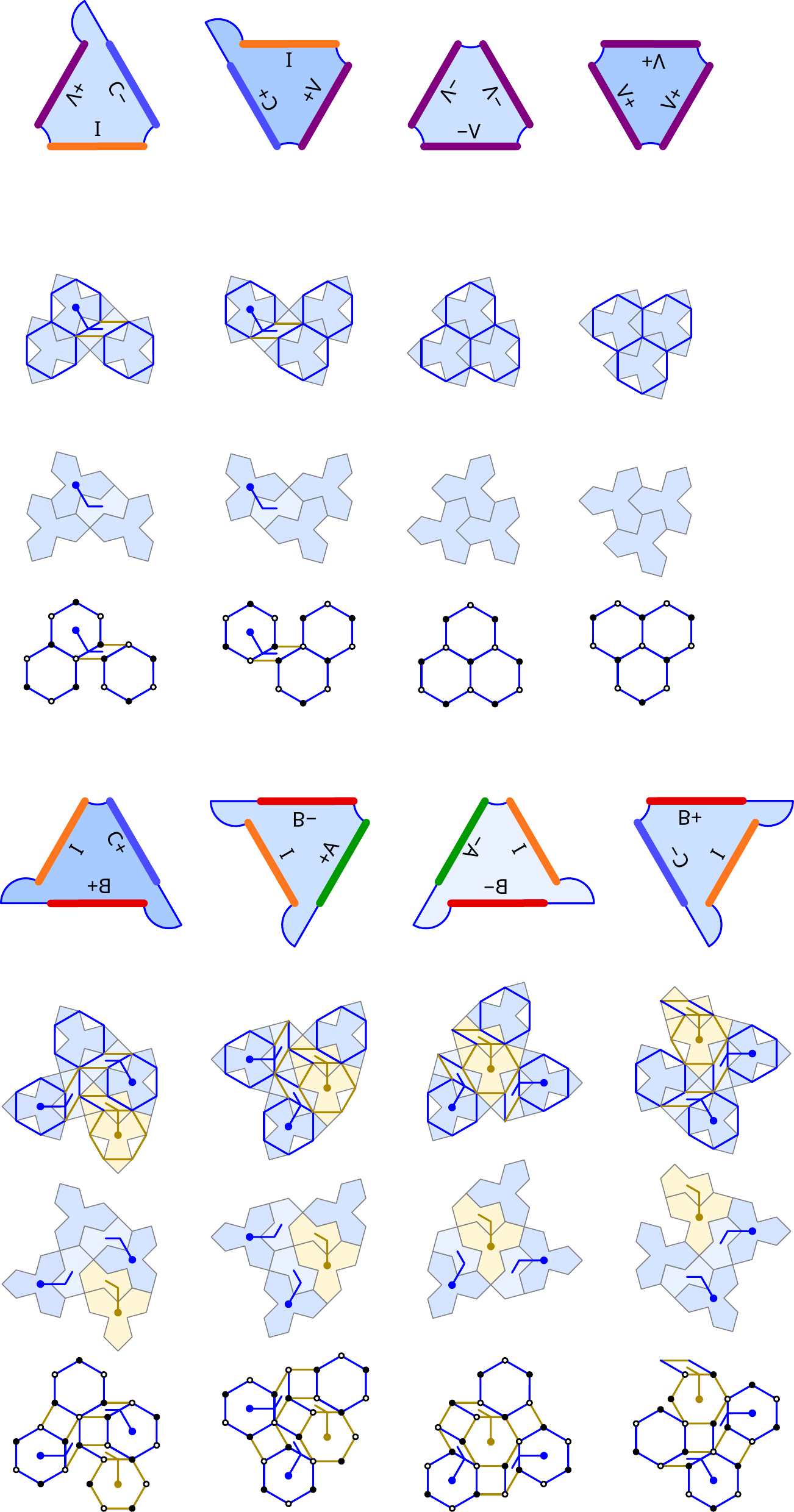}
}

Each triangle is associated to a trio of blue hexes that share a corner under the yellow contraction (see \Cref{sub:condensed}).
We could probably have started from an exhaustive study of that kind of triples instead of the blue cc's, and performed basically the same study as for the triangle tileset of \Cref{sec:triset} (\Cref{fig:triset}), allowing for a shorter deduction path to the substitution structure of the Spectre tiling, as there ought to be repetitions between the sections up to \Cref{sub:cc-list} included, and the study of the properties of the triangle tileset.

\subsubsection{One more spectre around the odd}

As a variant we can use the packs of \Cref{fig:sr-9}:

\nopagebreak

\image{Packs are reflected}{fig:p2s2}{
\includegraphics[scale=0.4]{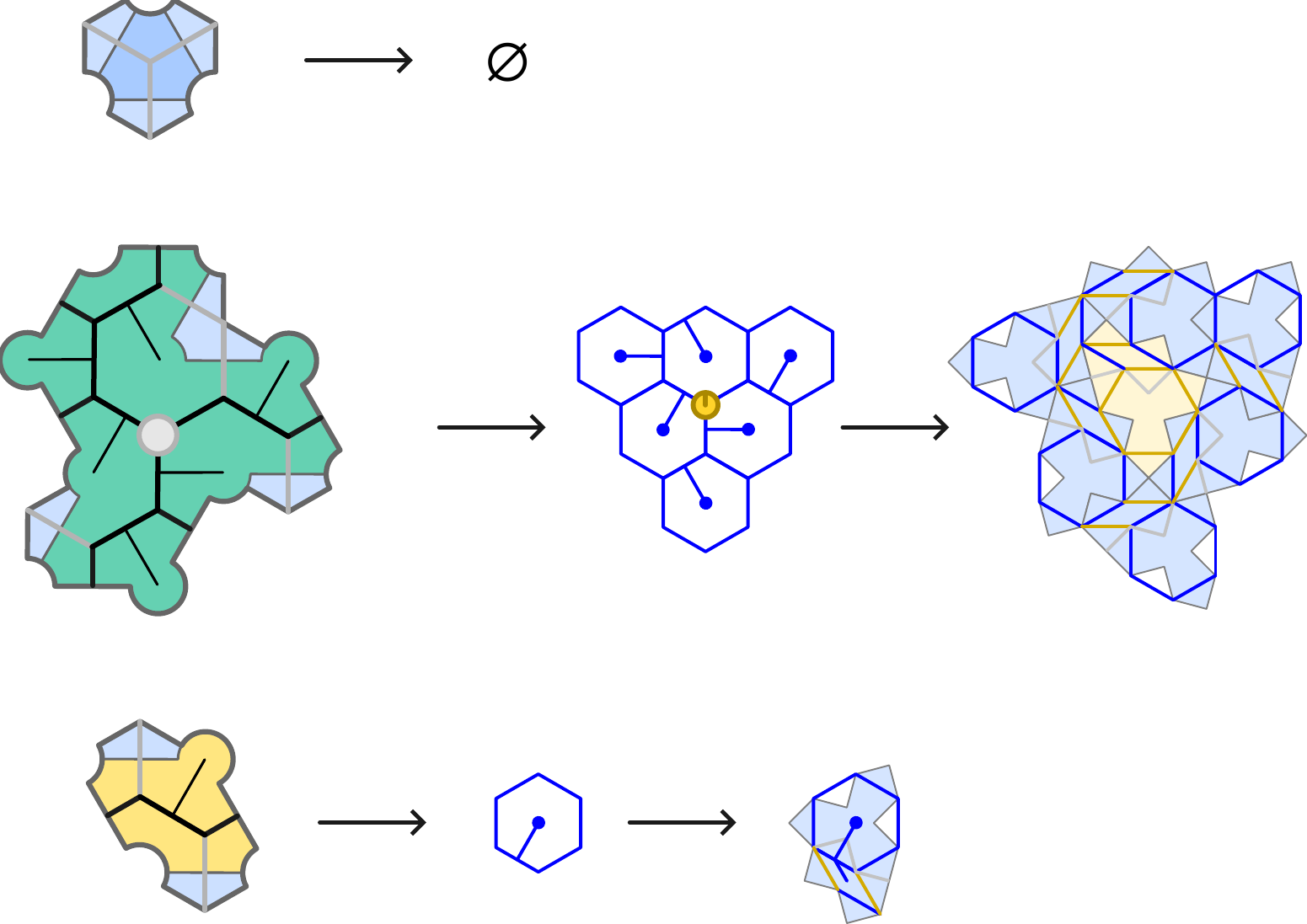}
}

Yielding:

\image{}{fig:de-3b}{
\makebox[\textwidth][c]{\includegraphics[scale=0.5]{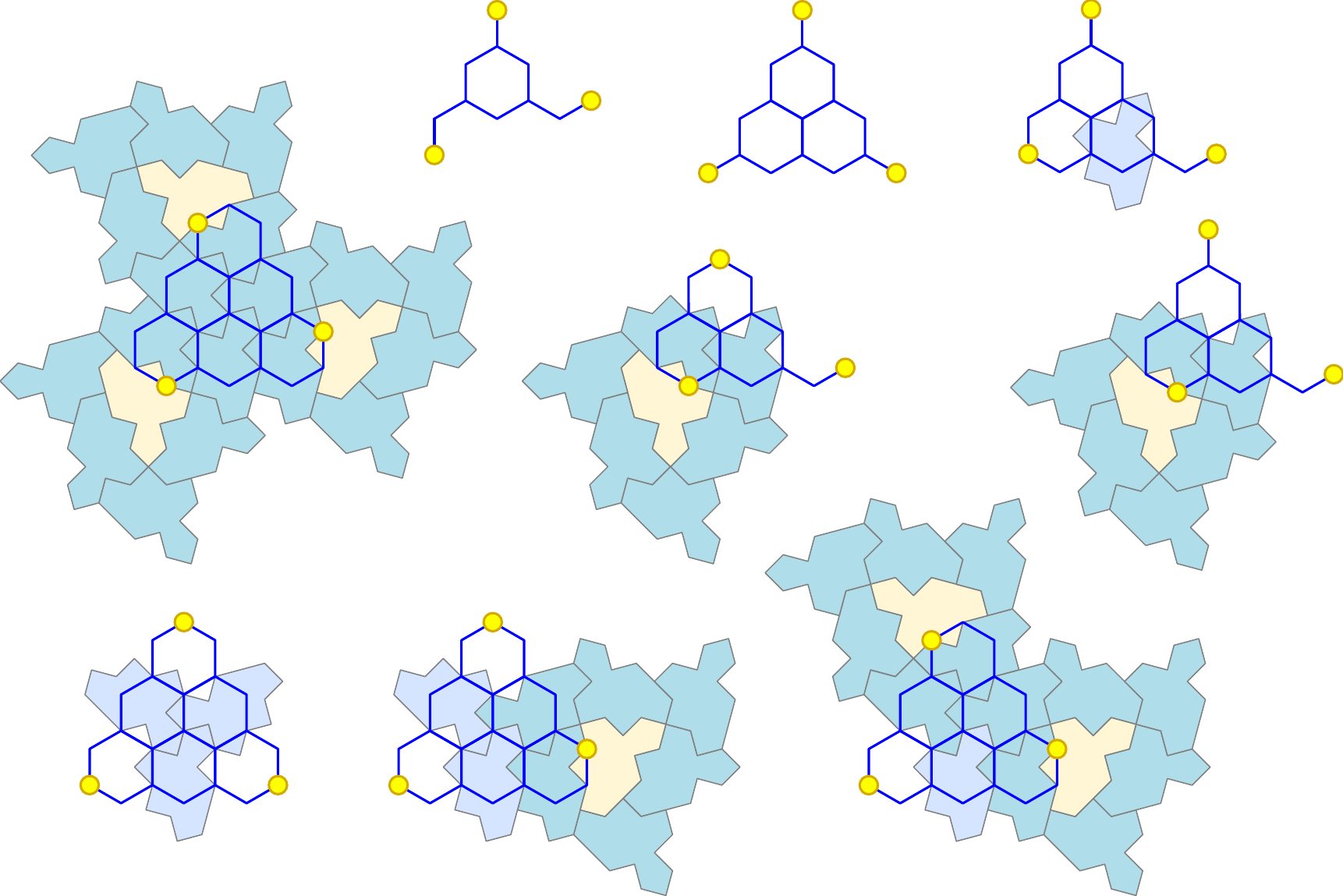}}
}

Or equivalently:

\image{}{fig:de-6b}{
\makebox[\textwidth][c]{\includegraphics[scale=0.5]{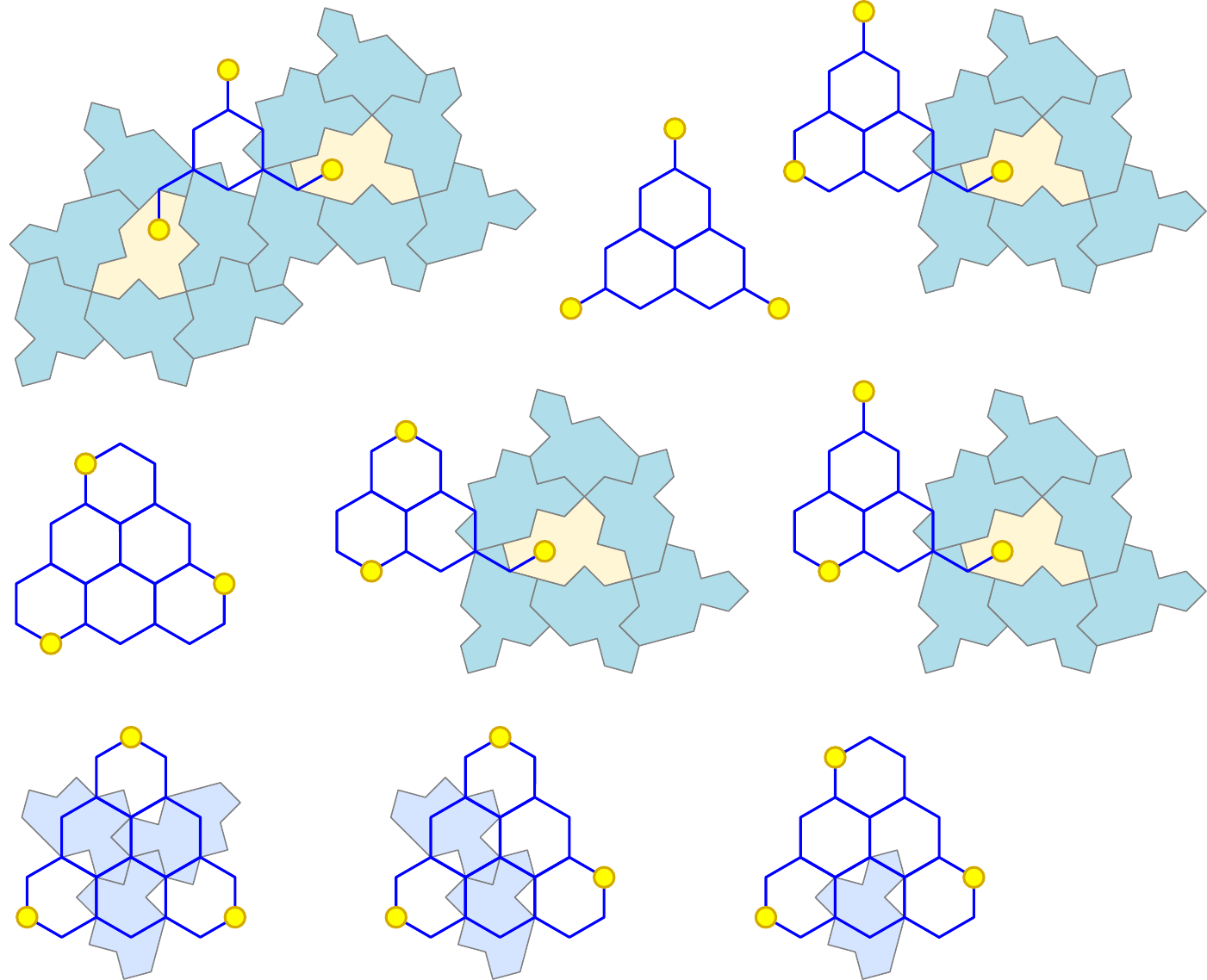}}
}

\subsubsection{Recognizing the hexagons of \texorpdfstring{\cite{chiral}}{[\ref{labelnumber-chiral}]}}\label{ss:oh}

We are now very close to the H8/H9 clusters of \cite{chiral}.
We reproduce them below for convenience:

\image{H9 is H8 with the dark blue tile added}{fig:H8-H9}{
\begin{tikzpicture}
\node at (0,0) {\includegraphics[scale=0.666]{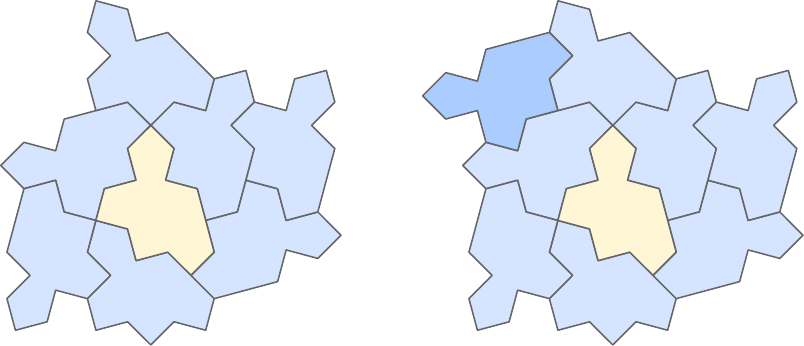}};
\node at (-1.2,-1.8) {H8};
\node at (4.2,-1.8) {H9};
\end{tikzpicture}
}

The H9 is one tile less than the environment of \Cref{fig:odd-env}.

The last version of our clusters above around the yellow Spectre have 7 tiles, including the yellow one.
It is however not possible to integrate the one missing to get an H8, in a way that is compatible with our cc shape$+$dots-based analysis: to do this we need to subdivide shapes of cc into cases, by adding the following information: which orientation (out of two possible) do we have for the yellow hexes at the cc dots for which this orientation is not already known.
We have not tried to determine the implication of this modification on the substitution systems proposed earlier.

\bigskip

To conclude \Cref{sec:analysis}, we propose a way to recognize the nine hex clusters $\Gamma$, $\Delta$, $\Theta$, $\Lambda$, $\Xi$, $\Pi$, $\Sigma$, $\Phi$ and $\Psi$ of \cite{chiral} from the triangle assemblies of \Cref{fig:vlist}.

First we note that in \Cref{fig:vlist}, none of the four cases of the third column can appear in a whole plane tiling.
The one on line~1 has been ruled out by \Cref{lem:no-2nd} and the one one line~2 by \Cref{prop:no-5th}.
The other two are ruled out in the next two lemmas.

\begin{lemma}\label{lem:impo-arr-1}
The following arrangement cannot occur in a whole plane tiling:

\nopagebreak

\image{}{}{
\includegraphics[scale=0.45]{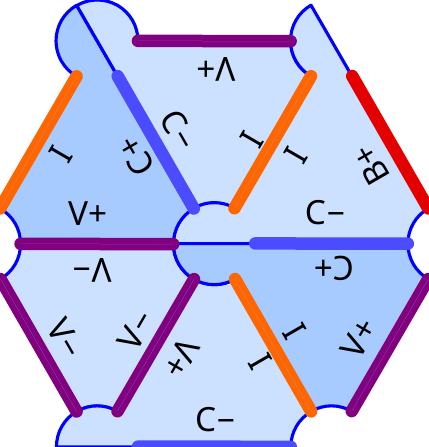}
}
\end{lemma}

\begin{lemma}\label{lem:impo-arr-2}
The following arrangement cannot occur in a whole plane tiling:

\nopagebreak

\image{}{}{
\includegraphics[scale=0.45]{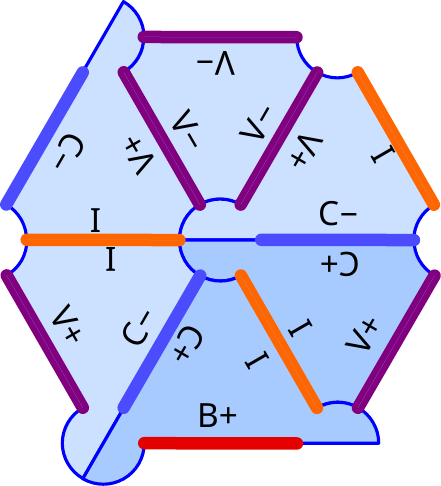}
}
\end{lemma}

These lemmas follow from the following improvement on \Cref{fig:green-pack-env}, which has its own interest.
The proofs of the three lemmas are to be found in \Cref{ss:recog-pfs}.

\begin{lemma}\label{lem:gr-pack-env-2}
  Every green pack is environed as follows:
  
  \nopagebreak
  
  \image{}{fig:gr-pack-env-2}{
    \includegraphics[scale=0.33]{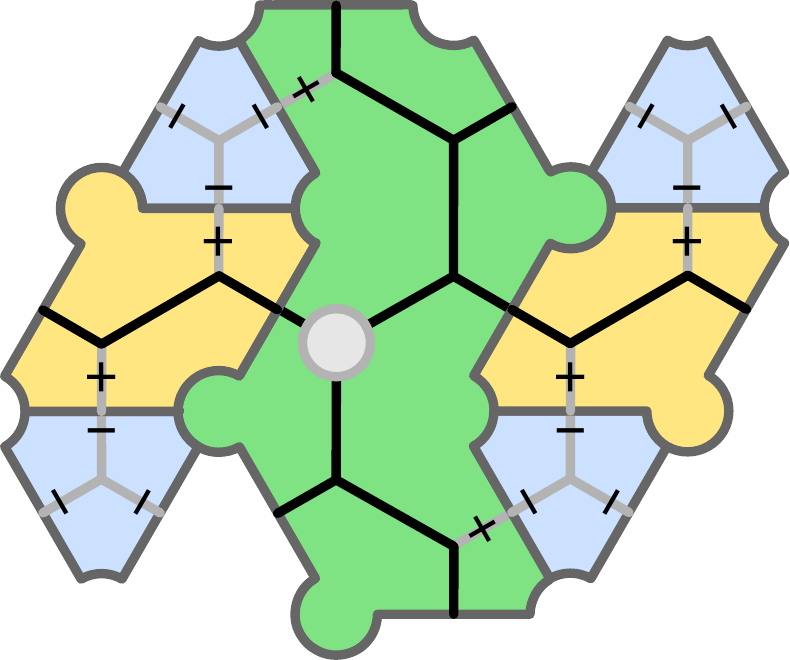}
  }
\end{lemma}

We can then prove the following result:

\begin{lemma}\label{lem:vlist-2}
All the other vertex environments of \Cref{fig:vlist}, i.e.\ those that are not in the third column, occur in whole plane tilings.
\end{lemma}
\begin{proof}
Let us number the remaining environments:

\nopagebreak

\image{}{fig:ab}{
\includegraphics[scale=0.35]{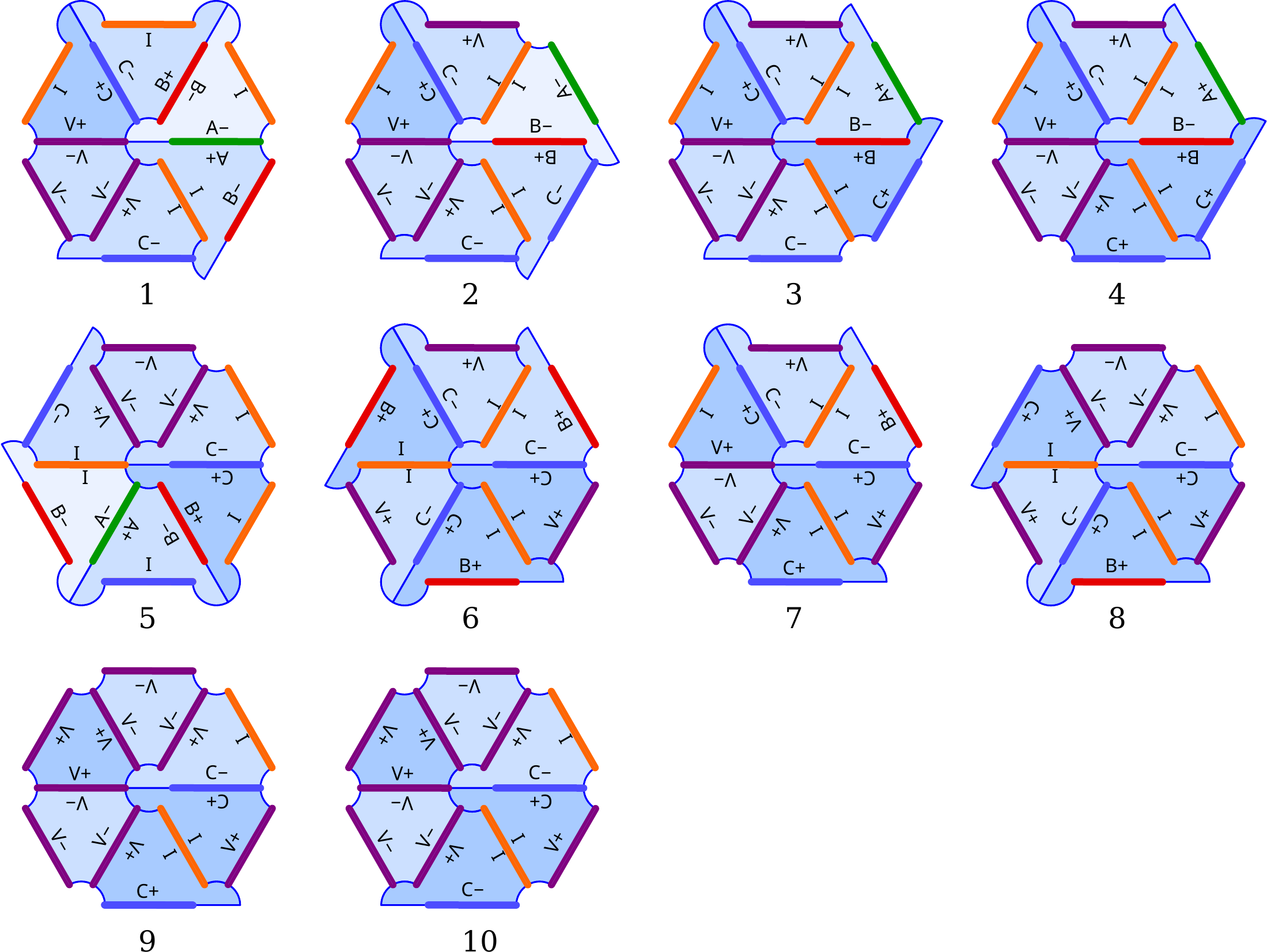}
}

We can interpret them in terms of the packs, which are visually easier to spot:

\image{The centre of the triangle arrangements of \Cref{fig:ab} are indicated by a dark blue spot.}{}{
\makebox[\textwidth][c]{\includegraphics[scale=0.21]{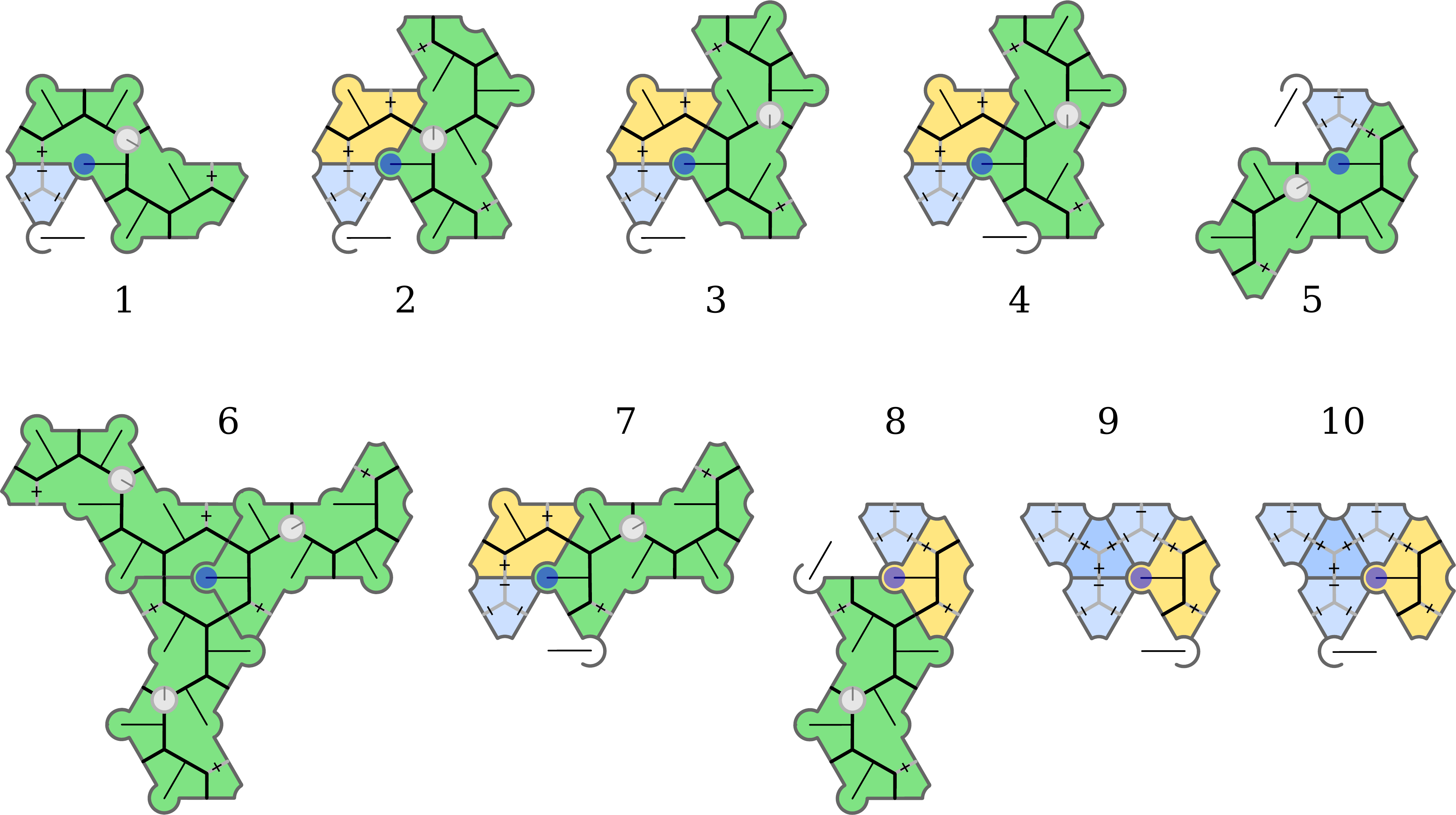}}
}

We look at the the second iteration of a yellow pack and find an instance of each of the situations of the previous figure:

\nopagebreak

\image{}{}{
\makebox[\textwidth][c]{\includegraphics[scale=0.33]{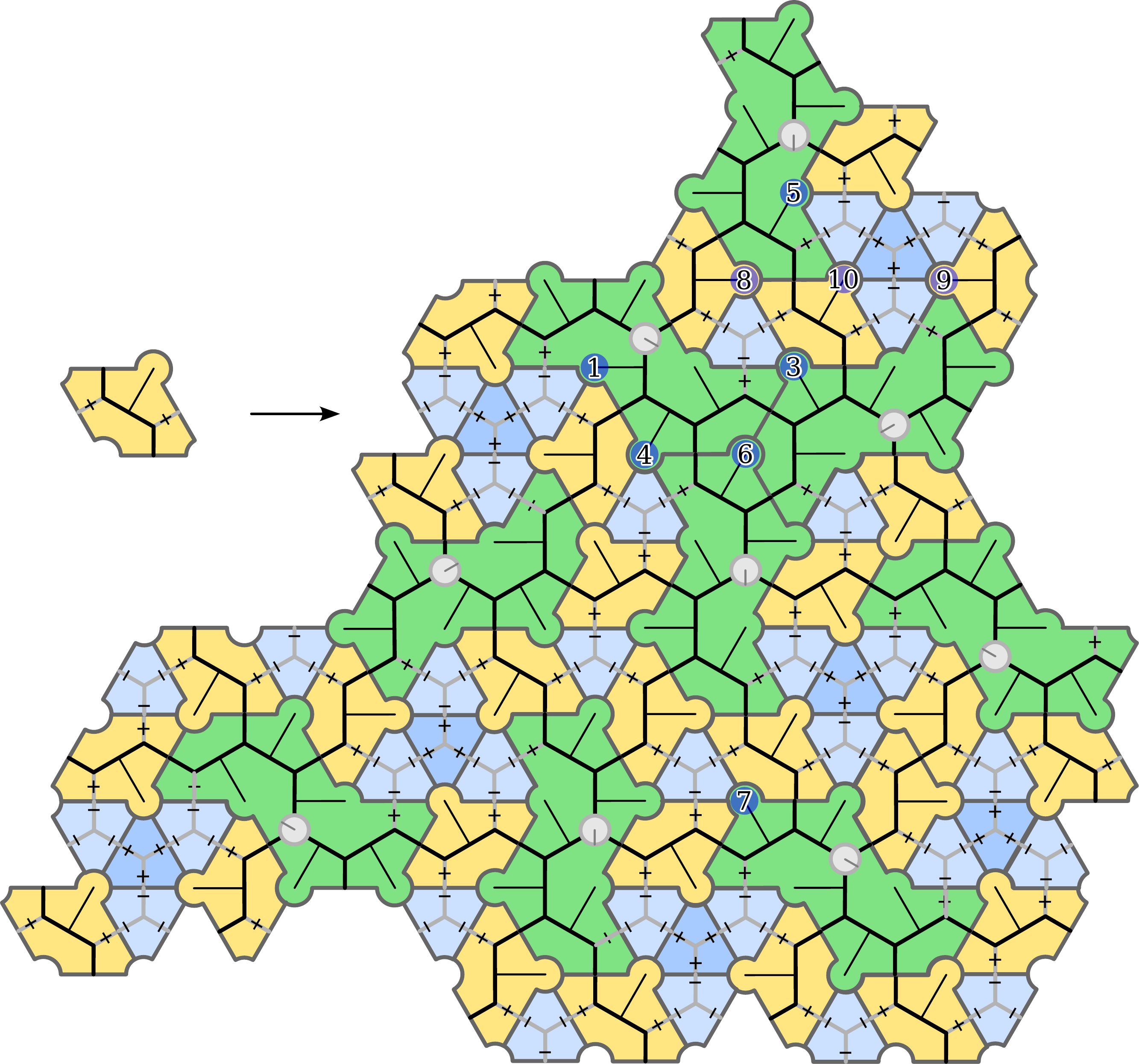}}
}

%
%

\end{proof}

Using \Cref{fig:labeled-cc-2,fig:odd-env-2}, we are able to identify the cluster types for the ten arrangements. On the figure we rotated the arrangements by a half-turn, for easier visual identification with Figure~4.1 of \cite{chiral}.
Curiously, Cluster $\Phi$ appears twice.

\image{}{fig:recog}{
\makebox[\textwidth][c]{\includegraphics[scale=0.36]{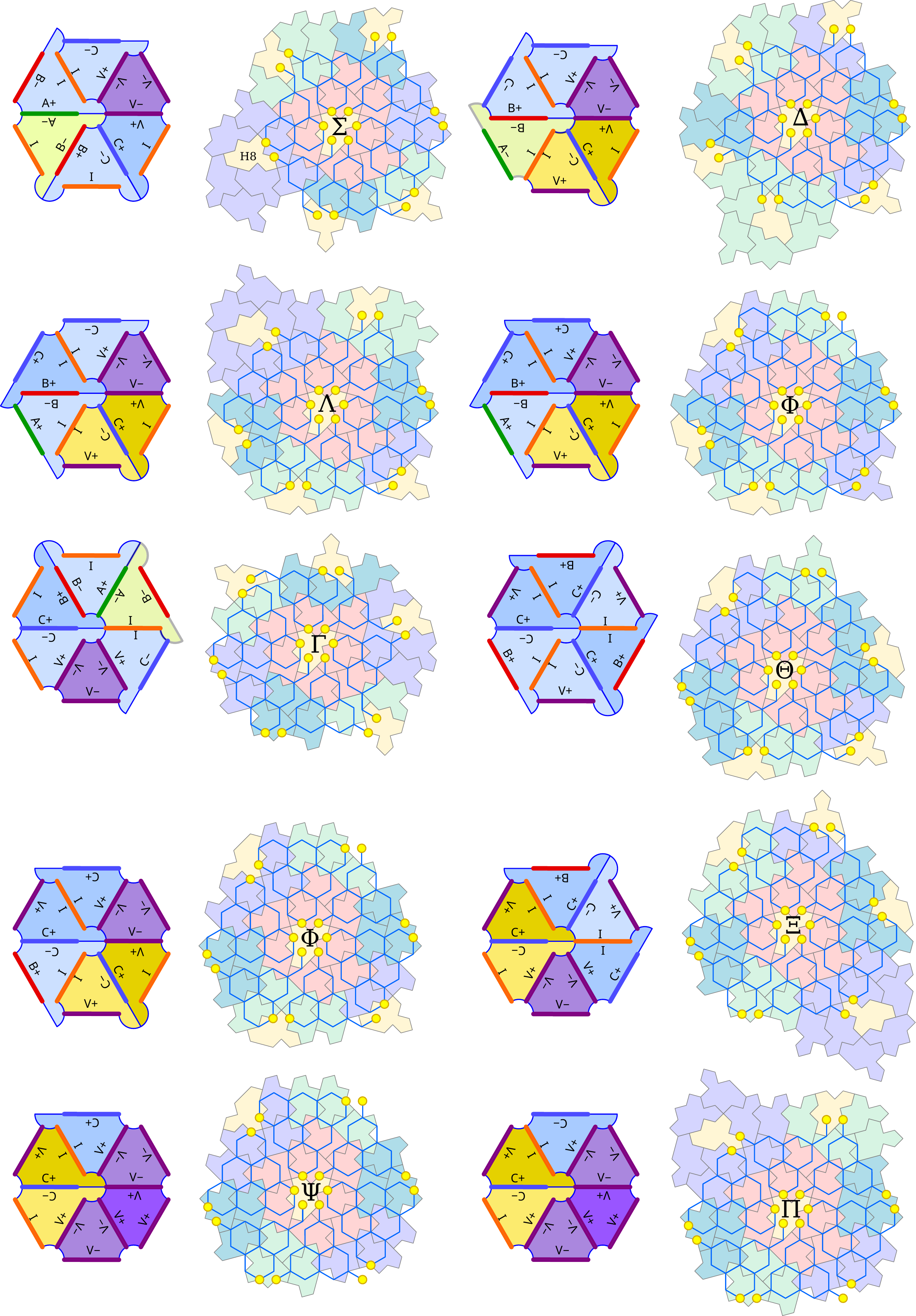}}
}

\section{Deferred proofs}\label{sec:later}

\subsection{Odd tile environment}\label{sub:pf-odd-env-1}

Here we prove \Cref{prop:odd-env-1}.

\begin{proposition}\label{prop:impo-13}
The following arrangement cannot appear in a whole plane tiling by the Spectre.

\nopagebreak

\image{}{}{
\includegraphics[scale=0.75]{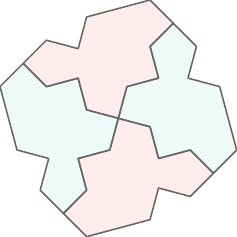}
}
\end{proposition}
\begin{proof}
In the figure below we look at the 1/4 angle of the inward dent on the upper right. It can be filled in 4 ways. Each leads immediately or by repeated use of \Cref{prop:autofill-1} to an unfillable outline as per \Cref{prop:impo-11}.

\image{}{}{
\includegraphics[scale=0.55]{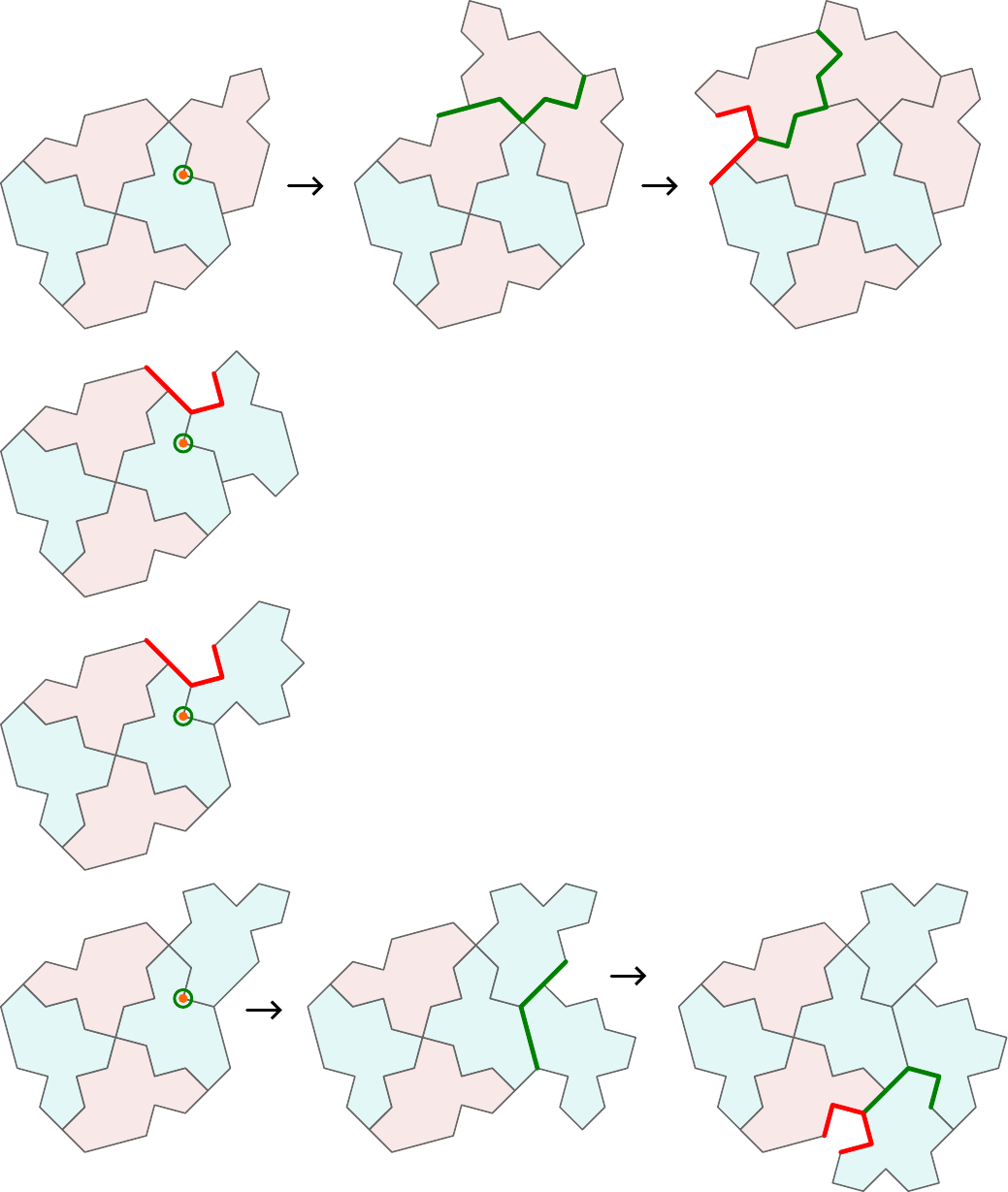}
}
\end{proof}

\begin{proposition}\label{prop:impo-12}
The following arrangements cannot appear in a whole plane tiling by the Spectre.

\nopagebreak

\image{}{}{
\begin{tikzpicture}
\node at (0,0) {\includegraphics[scale=0.666]{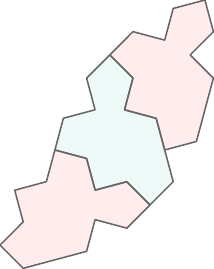}};
\node at (4,0) {\includegraphics[scale=0.666]{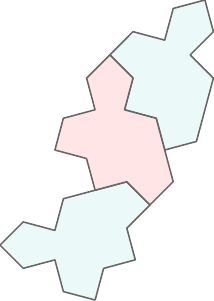}};
\end{tikzpicture}
}
\end{proposition}
\begin{proof}
We first treat the left configuration. In the figure below we look at the 1/4 angle of some inward dent. It can be filled in 4 ways. Each leads immediately or by use of \Cref{prop:autofill-1} to an unfillable outline as per \Cref{prop:impo-11} or the forbidden situation of \Cref{prop:impo-13}.

\image{}{}{
\includegraphics[scale=0.55]{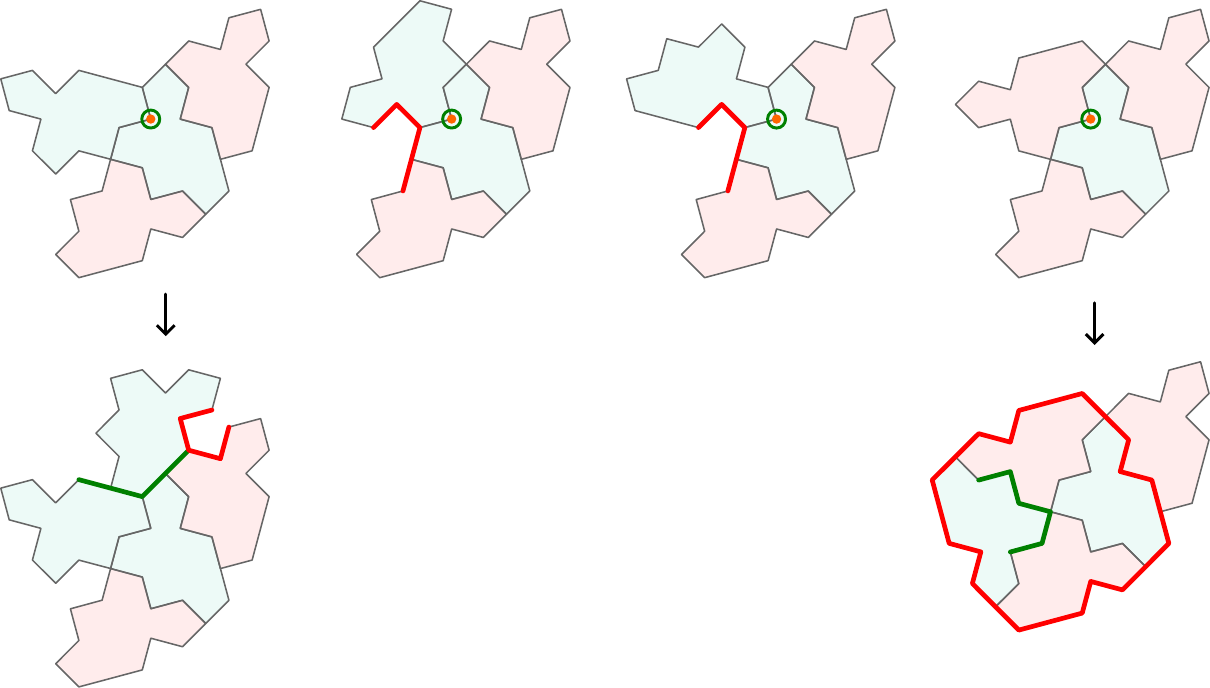}
}

For the right configuration:

\image{}{}{
\includegraphics[scale=0.55]{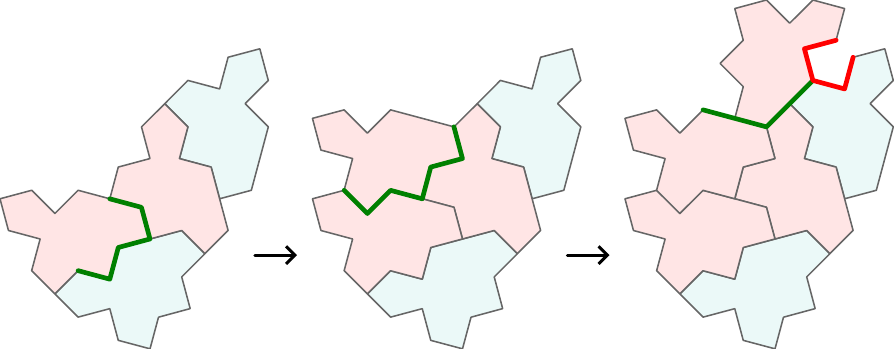}
}
\end{proof}


\begin{proof}[Proof of \Cref{prop:odd-env-1}]
First there are, up to rotation by a multiple of $1/12$, only 3 ways two tiles of different parity can be in contact along at least one edge:

\image{}{fig:3-cases}{
\includegraphics[scale=0.666]{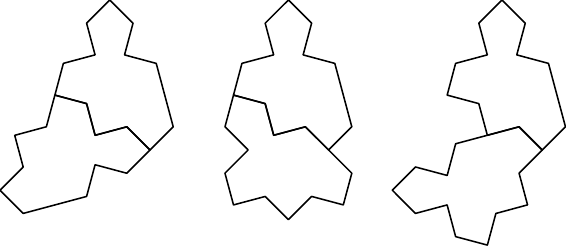}
}

or 6 ways if we only perform turns by multiples of $1/6$:

\image{}{}{
\includegraphics[scale=0.666]{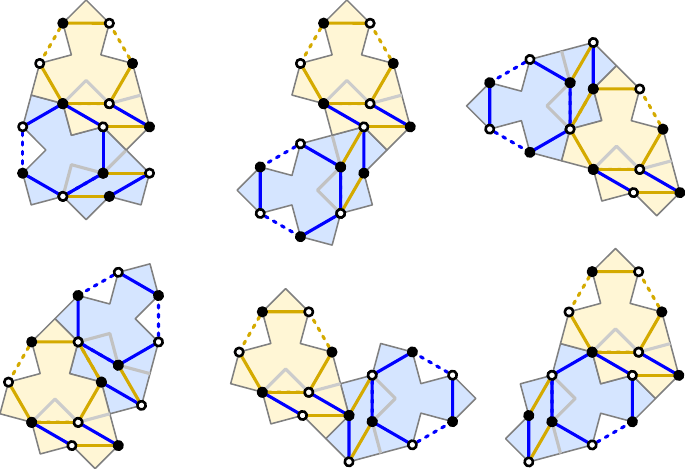}
}

This is proved by fixing the orientation of the yellow tile,  then for each of the 6 possible orientations of the blue tile, trying to place it adjacently.

We now prove that of for each of these three configurations, one of the two tiles is like the red tile in the figure below, up to a rotation of a multiple of $1/12$. (Note that the other of the two tiles may or may not be one of the blue tiles in the figure below.) 

\image{}{fig:oppo-1}{
\includegraphics[scale=0.666]{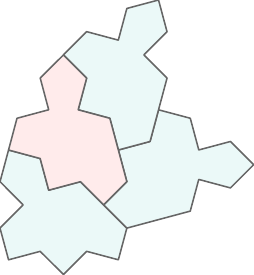}
}

To prove this claim, we first study the central case of \Cref{fig:3-cases} (which has the same shape as the \emph{mystic} of \cite{chiral}):

\image{}{fig:oppo-3}{
\includegraphics[scale=0.55]{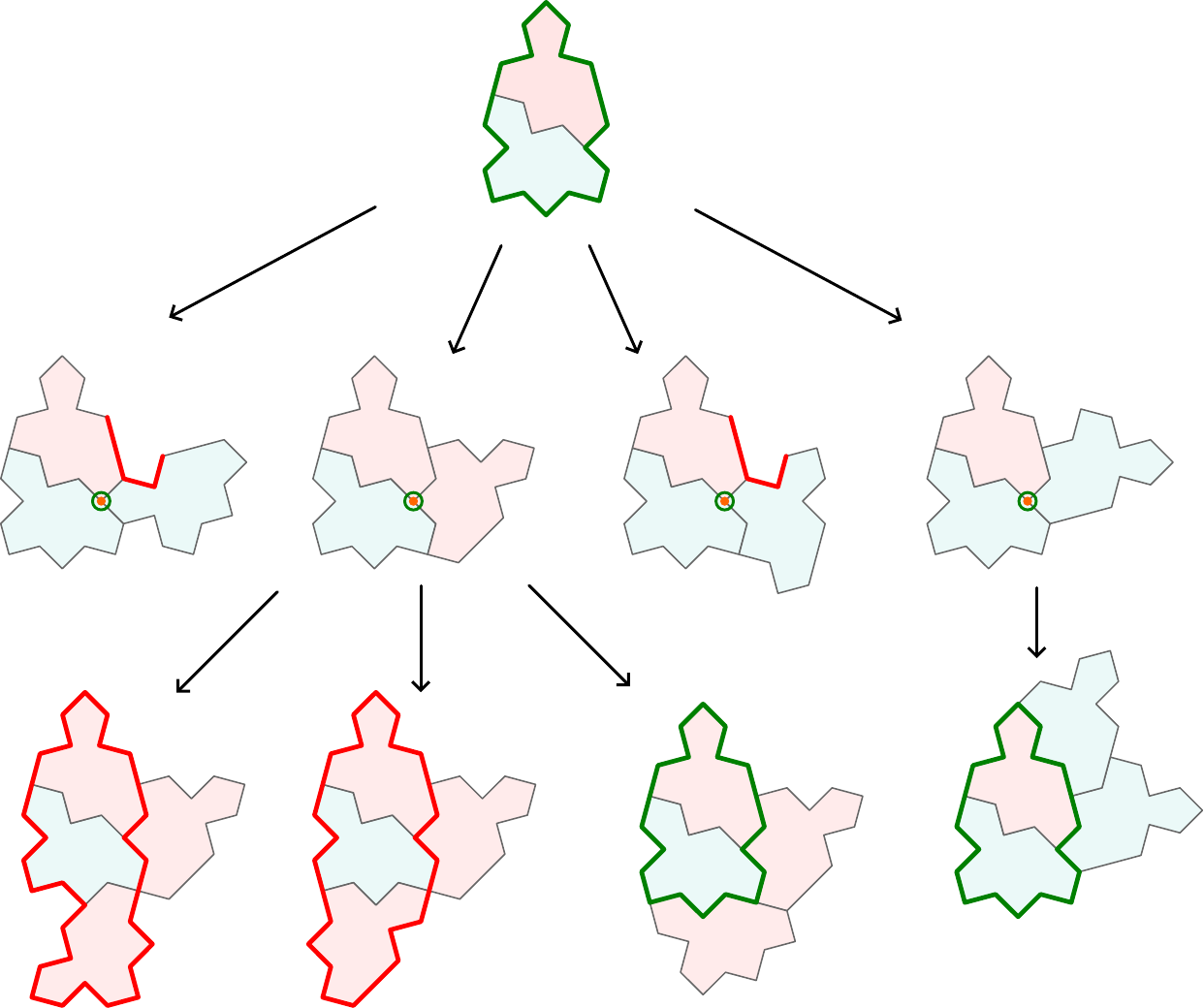}
}

In the figure above, only two cases are possible. The one on the right is exactly as in \Cref{fig:oppo-1}, as is the one left to it after swapping the colours and rotating by $-1/12$.

In the right case of \Cref{fig:3-cases}, we see below that we must have a mystic shape, and can use the previous figure to get the second deduction in the diagram below:

\image{}{}{
\includegraphics[scale=0.55]{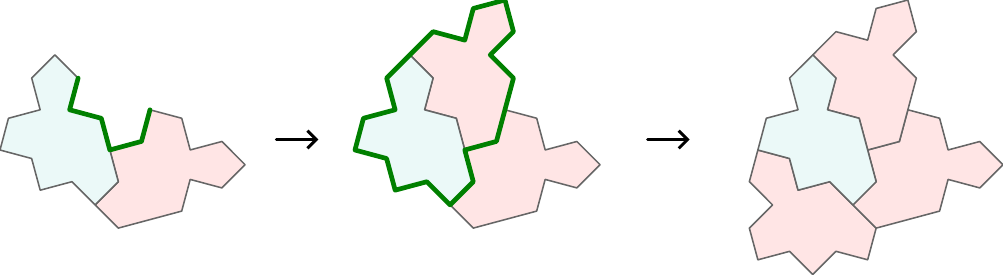}
}

Last, in first case of \Cref{fig:3-cases}, we get, using \Cref{fig:oppo-3} when we see a mystic:

\image{}{}{
\includegraphics[scale=0.55]{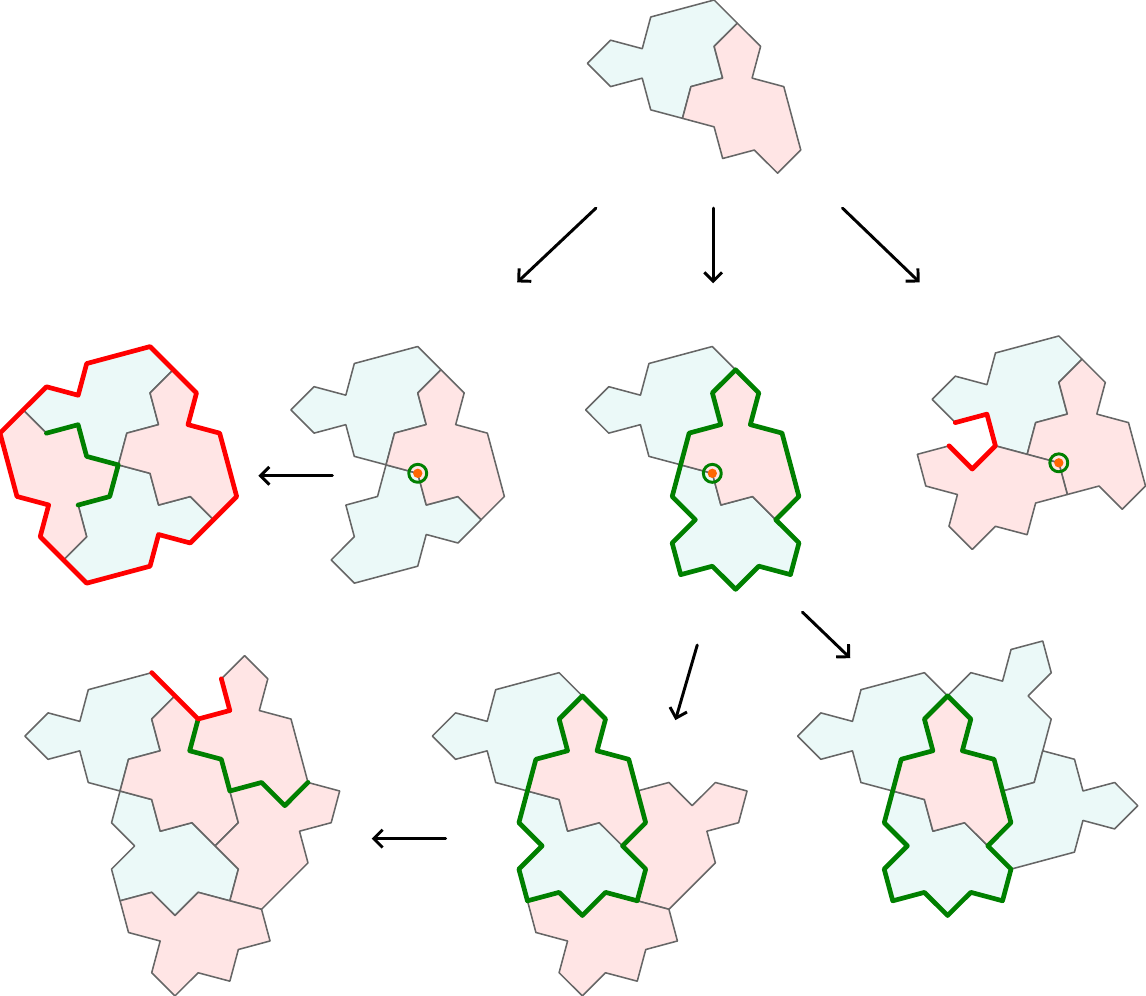}
}

This proves the claim.

From the configuration of \Cref{fig:oppo-1} we deduce:

\image{}{fig:last-1}{
\includegraphics[scale=0.55]{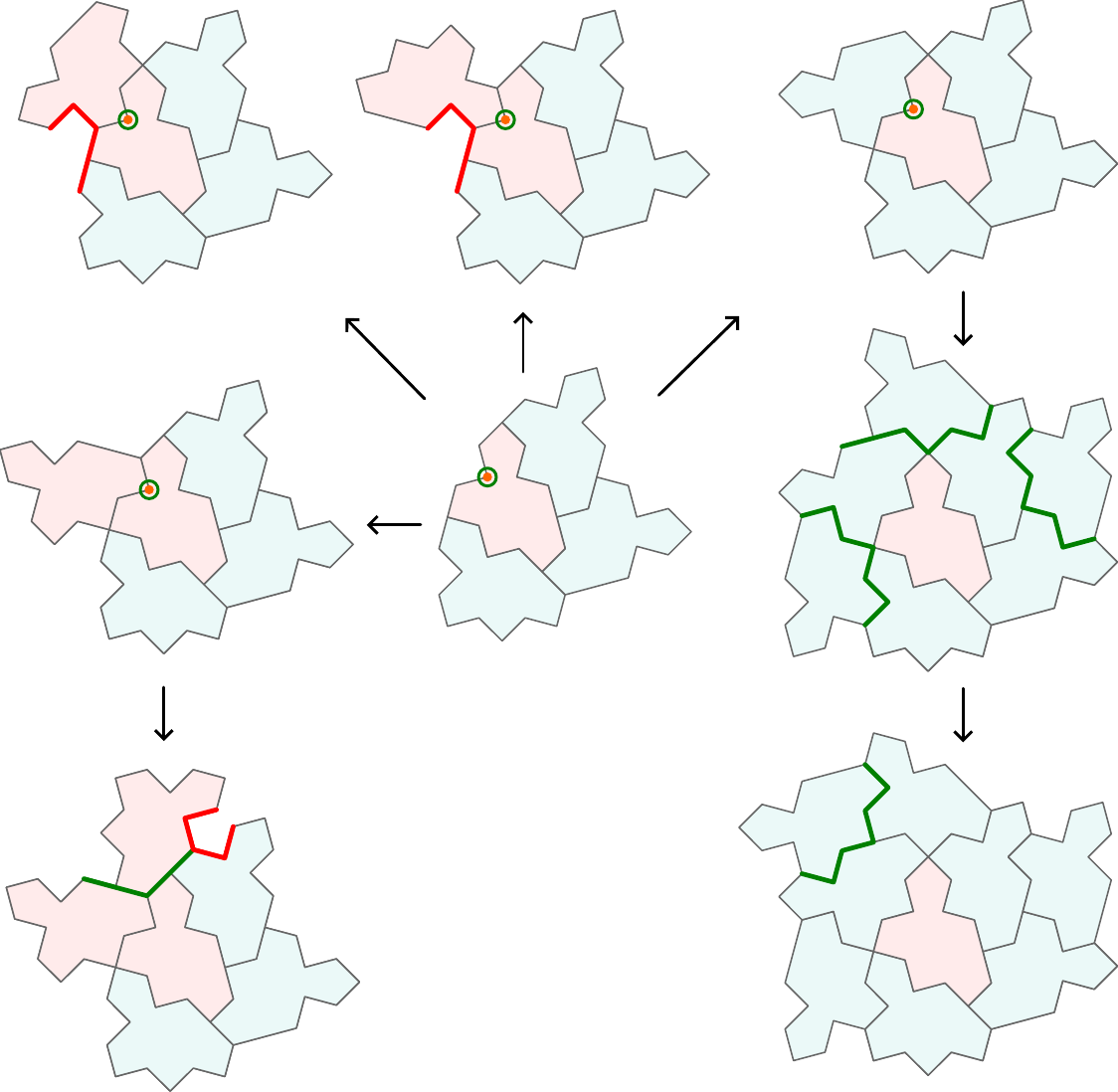}
}

Which is almost the whole wanted picture, except one tile, for which the deduction is a bit long. Note that \emph{this last tile is not needed to deduce that whole plane tilings with the Spectre necessarily have a parity class whose tiles are all isolated}.

The deduction goes as follows: if it were not the case, the only other possibility at the marked indent below would be the tile marked 0.

\image{}{}{
\includegraphics[scale=0.55]{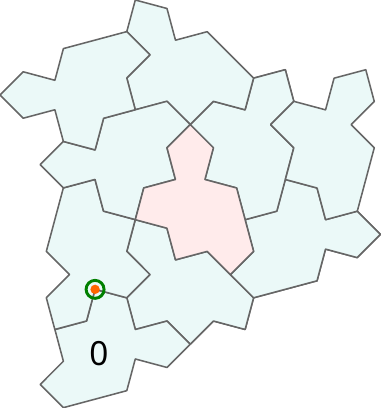}
}

If this were the case, let us look at another indent and at the 4 possible ways to fill it and what this implies. Numbers indicate in which order the tiles can be deduced using \Cref{prop:autofill-1}.

\image{}{}{
\includegraphics[scale=0.48]{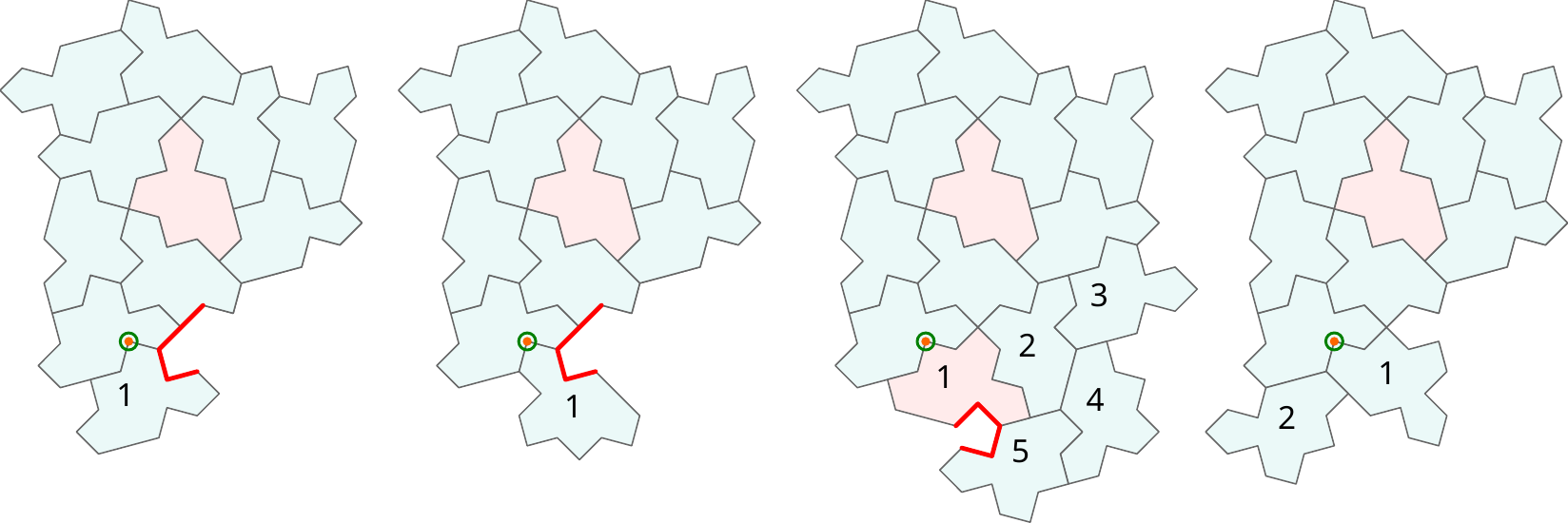}
}

The last case is ruled out as follows:

\image{}{}{
\includegraphics[scale=0.55]{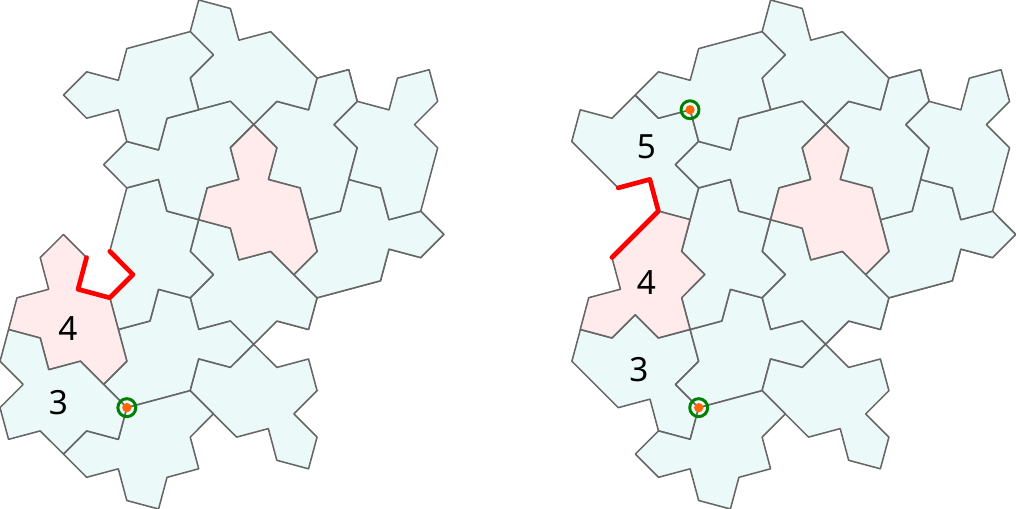}
}
\end{proof}

\subsection{Cc types}\label{ss:pf-cc-types}

The proposition below is \Cref{prop:T1-cc}.

\begin{proposition*}
Every cc of a $\hT 1$ is necessarily as on the picture below:\footnote{Actually we will see in \Cref{cor:T1-tto} that the lower left arrow can only have one orientation.}

\nopagebreak

\image{}{fig:T1-a-copy}{
\includegraphics[scale=0.75]{T1-ant-2b.pdf}
}
\end{proposition*}
\begin{proof}
Let us fix the orientation of the Spectre that the $\hT1$ is associated to:

\image{}{}{
\includegraphics[scale=0.75]{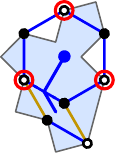}
}

With this orientation, we are in the first two possibilities of \Cref{fig:tips}. Knowing that the second configuration cannot appear, the top tip can only be:

\image{}{}{
\includegraphics[scale=0.75]{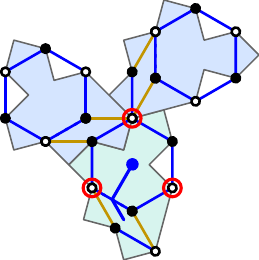}
}

which by repeated use of \Cref{prop:autofill-1} can only complete into

\image{}{}{
\includegraphics[scale=0.75]{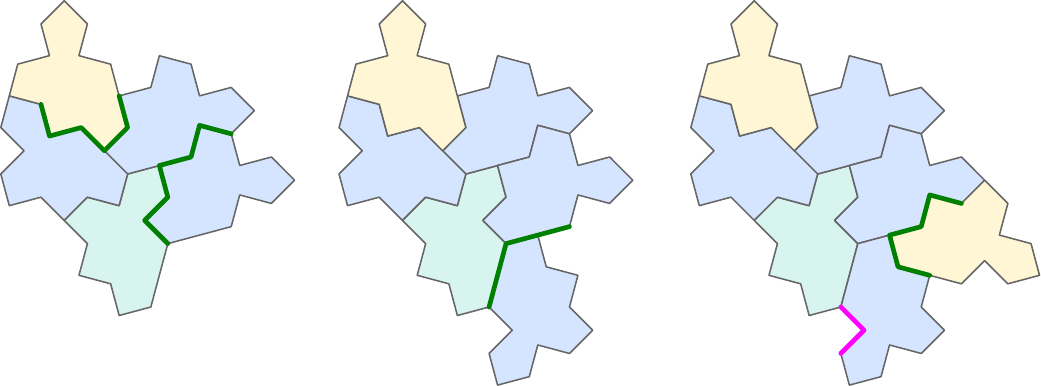}
}

and then looking at what can fit in the pink dent above and completing two more tiles, we get the following two possibilities:

\image{}{}{
\includegraphics[scale=0.75]{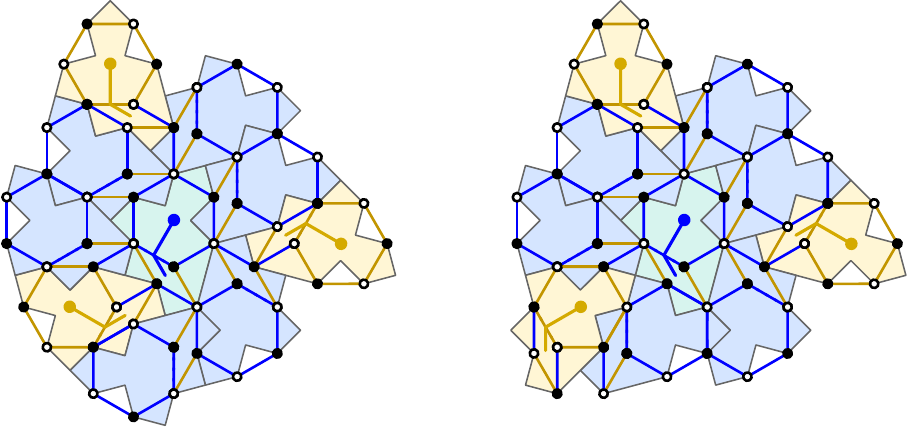}
}
\end{proof}

The proposition below is \Cref{prop:T2-impo}.

\begin{proposition*}
A $\hT 2$ with markings as below cannot occur.

\image{}{}{
\includegraphics[scale=0.75]{T2-impo-b.pdf}
}
\end{proposition*}
\begin{proof}

Initially, there is an order 3 rotational symmetry.
The inward dent of the blue segment highlighted in red below can only have two ways to be filled, because no other blue hex can touch it by definition of a $\hT2$:

\image{}{}{
\includegraphics[scale=0.75]{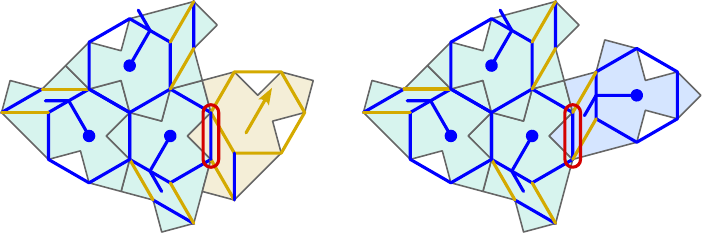}
}

The one of the left contradicts the orientation of the bottom-most green tile in \Cref{prop:odd-env-1} at the tile outlined in red below left.
The one on the right implies the configuration below right, which yields the same contradiction at the tile outlined in red.

\image{}{fig:conf2}{
\includegraphics[scale=0.75]{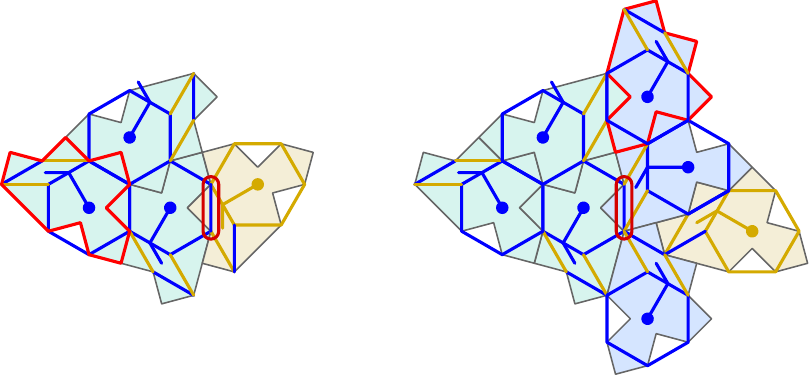}
}

%
%
\end{proof}

The proposition below is \Cref{prop:T2-cc-1}.

\begin{proposition*}
A $\hT 2$ with markings as below left has a cc as below right:

\nopagebreak

\image{}{}{
\begin{tikzpicture}
\node at(-2,0) {\includegraphics[scale=0.75]{T2-1-b.pdf}};
\node at(2,0) {\includegraphics[scale=0.66]{T2-ant-1b.pdf}};
\end{tikzpicture}
}
\end{proposition*}
\begin{proof}
We have the following series of implications, with the red boundaries being places where no tile can fit.

\image{}{}{
\includegraphics[scale=0.63]{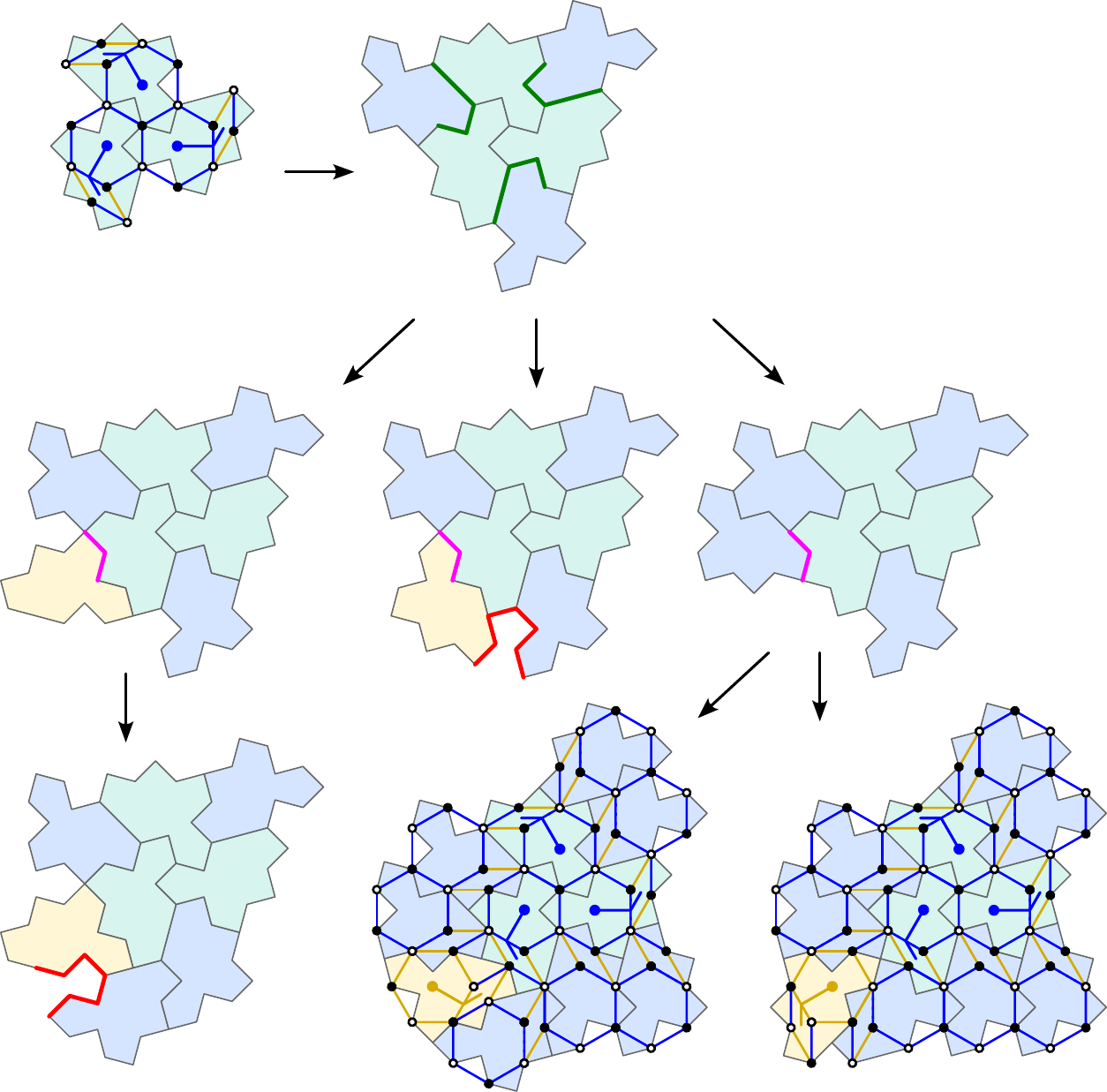}
}

In the end there are two possibilities for the lower left corner of the T2, and the other two corners have the same possibilities by rotation.
\end{proof}

The proof of \Cref{prop:T2-cc-2} is by a similar kind of analysis and is omitted here.

\medskip

The lemma below is \Cref{lem:T3-env}

\begin{lemma*}
A $\hT3$ is environed as follows, in terms of the $\iD3$ and $\iD2$ induced tiling:

\image{}{fig:T3-env-copy}{
\includegraphics[scale=0.75]{T3-env-5.pdf}
}
\end{lemma*}
\begin{proof}
We claim that the pink inward dents below must be filled as follows:

\image{}{}{
\includegraphics[scale=0.75]{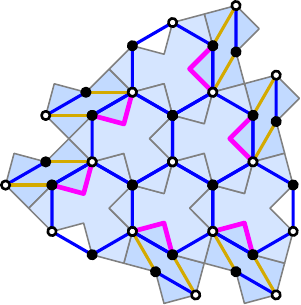}
}

Indeed the dent filling it must be supported by a blue segment.
By definition of a $\hT3$ we cannot use a blue $\iD3$ (it has a blue hex) so we are left with a yellow $\iD2$ or a blue $\iD2$.
But the yellow one does not fit, as the following figure shows:

\image{}{fig:pp-1}{
\includegraphics[scale=0.75]{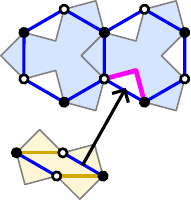}
}

We can immediately see that three $\iD3$ have only one adjacent $\iD2$ so must be paired with them:

\image{}{}{
\includegraphics[scale=0.75]{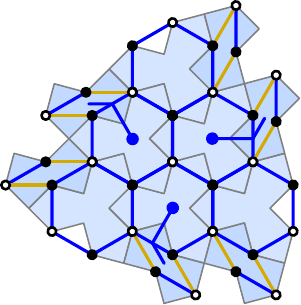}
}

In the following configuration, the pink hollow can only be filled by a $\iD3$, as the only other fitting piece is a yellow $\iD2$ whose paired $\iD3$ cannot fit.

\image{}{}{
\includegraphics[scale=0.75]{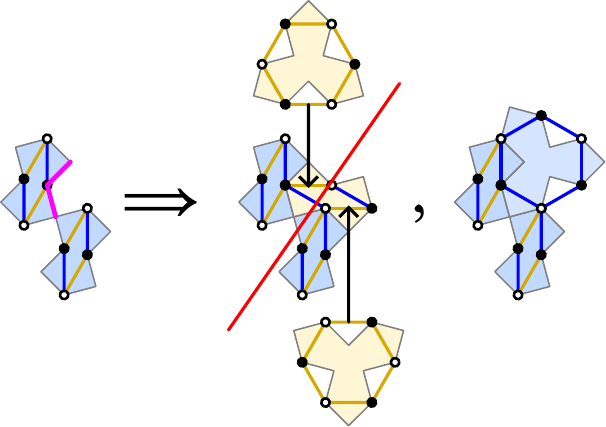}
}

So we get:

\image{}{}{
\includegraphics[scale=0.75]{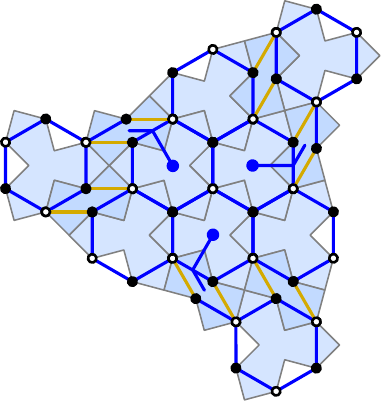}
}

The following configuration on the left leads to the presence of an $\iD2$ (in orange) which cannot be paired.

\image{}{}{
\includegraphics[scale=0.75]{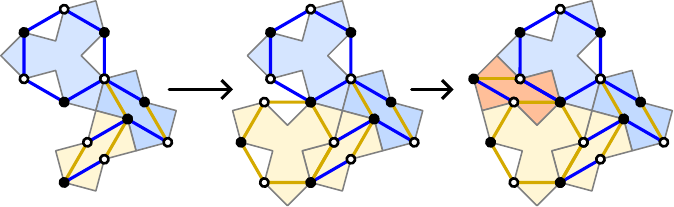}
}

As a consequence we get:

\image{}{}{
\includegraphics[scale=0.75]{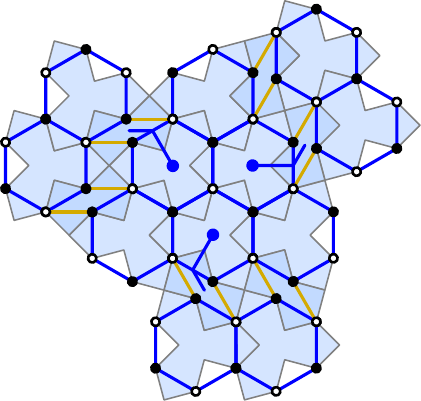}
}

Now by the completion rule of \Cref{lem:completion} we get the left frame below, and then the cyan tile can only have their associated $\iD2$ placed as below right, and this allows to add one more $\iD2$ by the principle of \Cref{fig:pp-1}  

\image{}{}{
\includegraphics[scale=0.75]{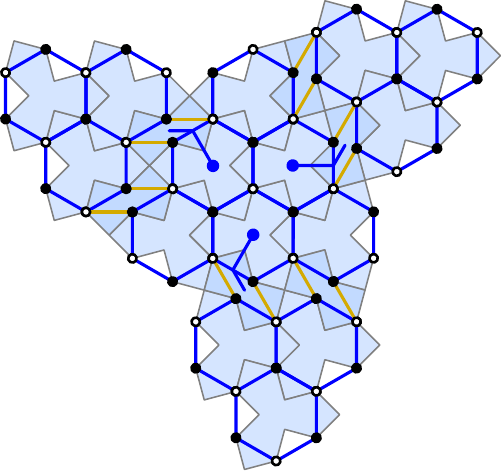}
\quad
\includegraphics[scale=0.75]{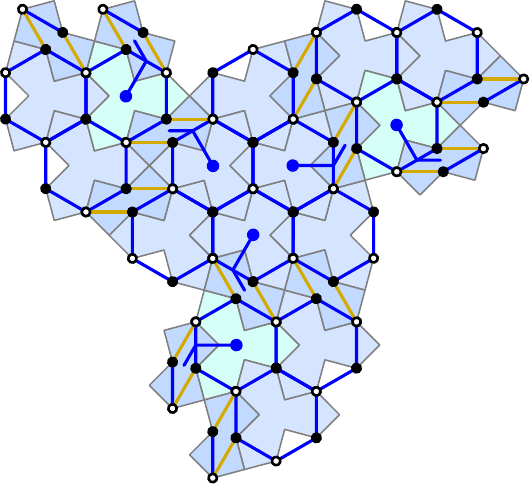}
}
%
%
%
\end{proof}

The proposition below is \Cref{prop:T3-cc}
\begin{proposition*}
A T3 has no antenna: it is its own cc. It can only be as follows on its lower left corner, and similarly for the other two by rotation.

\nopagebreak

\image{}{}{
\includegraphics[scale=0.66]{T3-ant-1b.pdf}
}
\end{proposition*}
\begin{proof}
The Spectre containing the lower left blue hex can be in two positions. For each, let us see what this implies on the environment given by \Cref{fig:T3-env-copy}.

\image{}{}{
\includegraphics[scale=0.75]{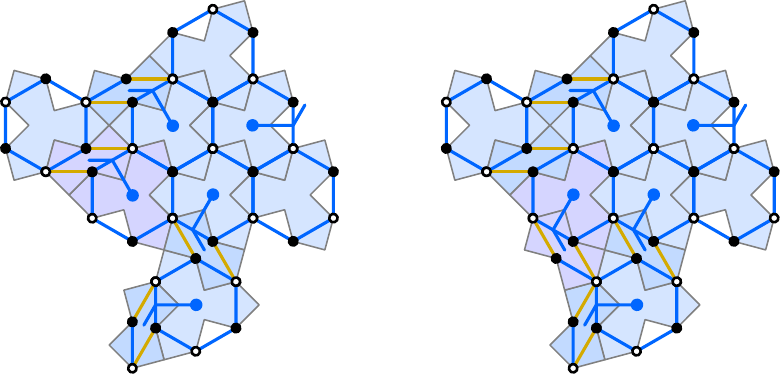}
}

They imply:

\image{}{fig:104}{
\includegraphics[scale=0.68]{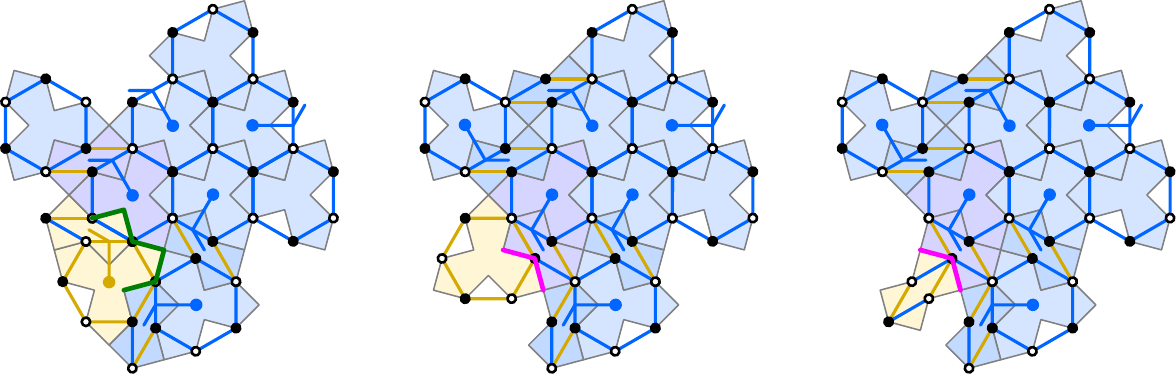}
}

The central case implies the first two below, and the right case implies the right one: first the paired yellow $\iD3$ can only fit on one side of the yellow $\iD2$, then the the mauve tile is a forced blue tile, but its position with respect to the yellow tile contradicts \Cref{prop:odd-env-1}.

\image{}{}{
\includegraphics[scale=0.68]{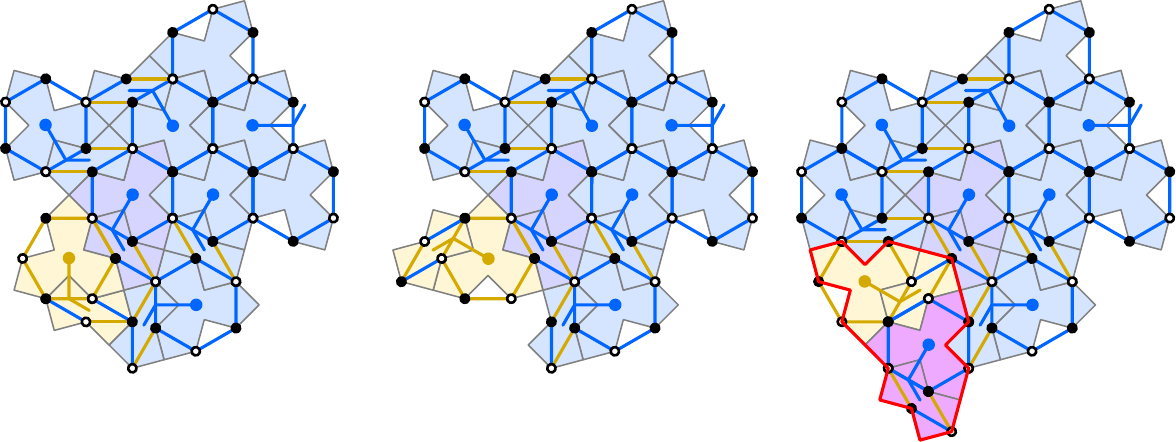}
}
\end{proof}

\subsection{Environment of \texorpdfstring{$\hT'n$}{T'n} clusters}\label{ss:Tpn-env-pfs}

We will make implicit or explicit use of the visual aids of \Cref{ss:vis}, which we encourage the reader to consult again.

\bigskip

The statement below is a copy of \Cref{prop:no-5th}.

\nopagebreak

\begin{proposition*}
The fifth configuration of \Cref{fig:vlist} cannot occur in a whole plane tiling.
\end{proposition*}
\begin{proof}
We have the following sequence of deductions

\nopagebreak

\image{}{fig:exclu}{
\includegraphics[scale=0.21]{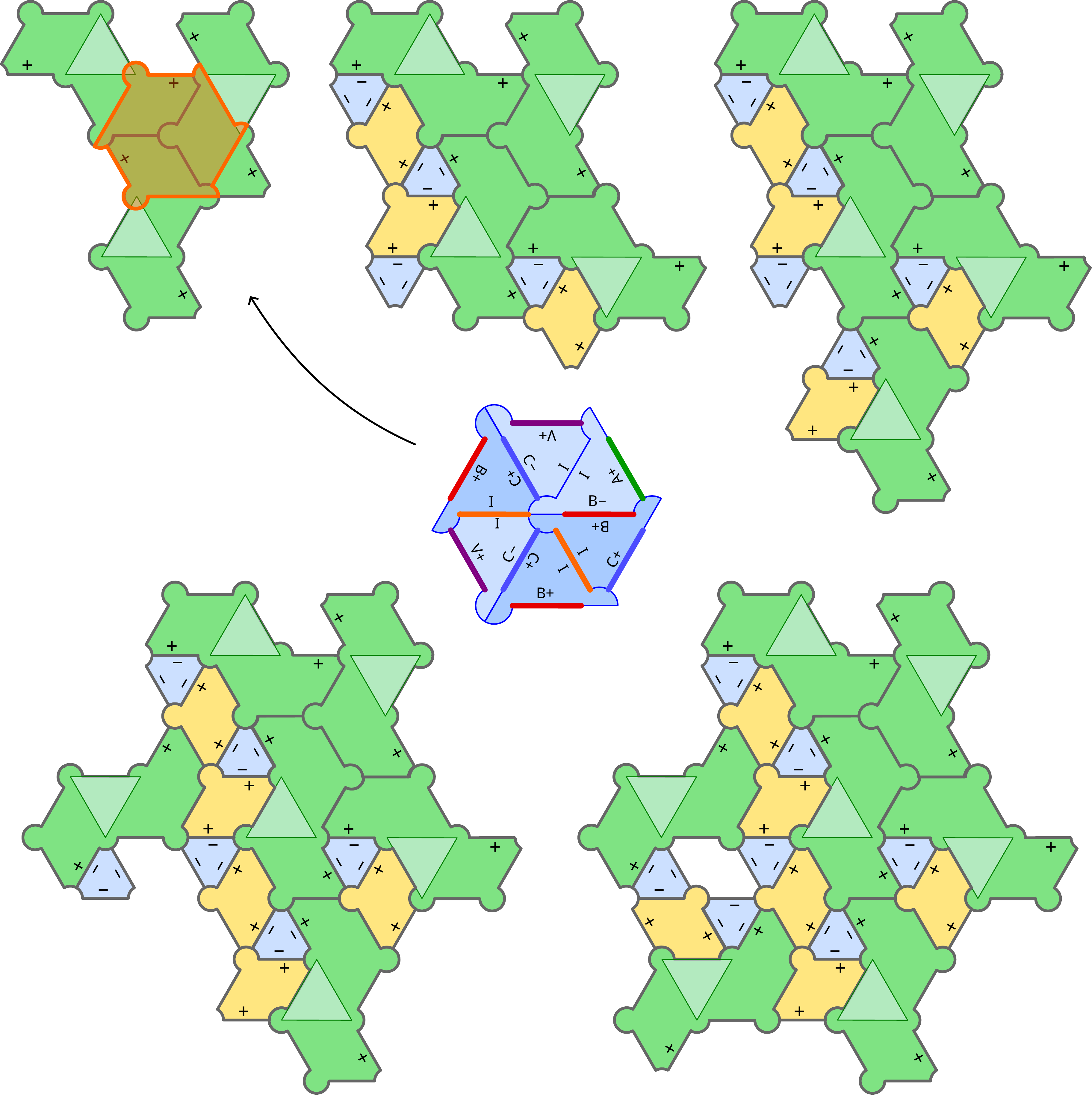}
}

The last one leaves a hole that cannot be filled.
\end{proof}

\medskip

The following statement is a copy of \Cref{lem:Tp2-clock}.

\nopagebreak

\begin{lemma*}
Consider two tips following each other in the clockwise order around a $\hT'2$ cluster of a whole plane tiling by the packs.
Then, using the nomenclature of \Cref{fig:Tp2-corner}, their respective types cannot be in the following list (each item is to be read in the clockwise order):
$(n,n)$,
$(a2',a2')$,
$(a2',n)$,
$(a1,a2')$,
$(a1',a2')$.
\end{lemma*}
\begin{proof}
We have the following deductions for each case, each leading to an impossibility indicated in red:

\image{The implication with an $E$ is deduced from the visual rule of \Cref{lem:fill-3}, those with a $C$ from that of \Cref{lem:fill-1}. For the $A$, the mauve dot can only be filled by a yellow piece, since no green piece can fit that would fill it.}{}{
\makebox[\textwidth][c]{\includegraphics[scale=0.22]{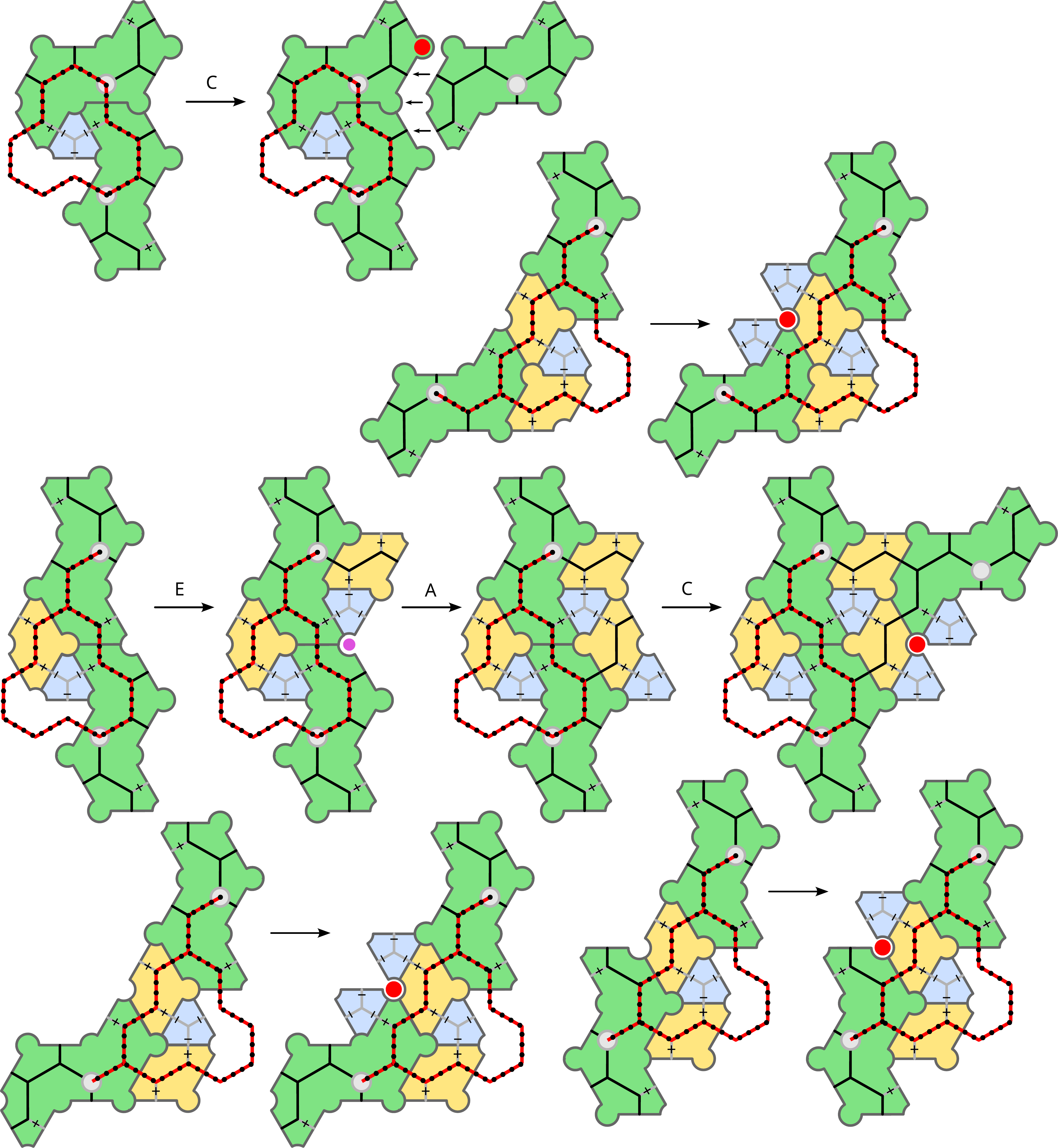}}
}
\end{proof}

\clearpage

The statement below is a copy of \Cref{lem:Tp3-env}.

\nopagebreak

\begin{lemma*}
The environment of a $\hT'3$ is as on the figure below:

\nopagebreak

\image{}{}{
\includegraphics[scale=0.3]{Tp3-env-c.pdf}
}
\end{lemma*}
\begin{proof}
First, the environment of a $\hT'3$, described using the triangle tileset and the packed tileset, must be as follows:

\nopagebreak

\image{}{fig:T3-env-2-b}{
\includegraphics[scale=0.4]{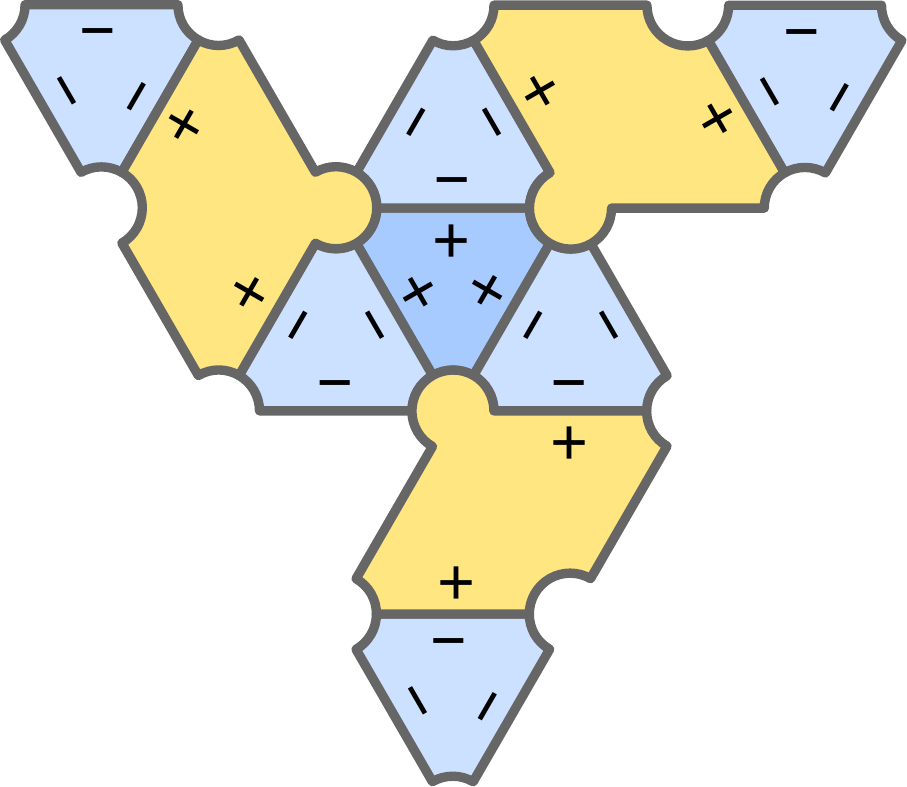}
\quad
\includegraphics[scale=0.4]{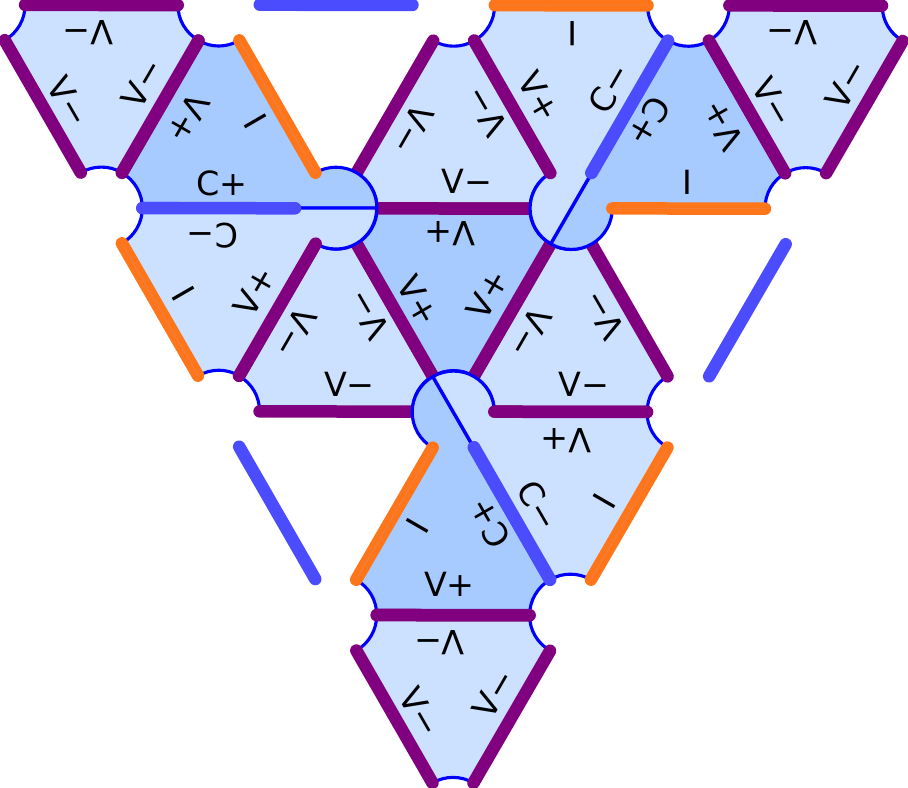}
}

With the packed tileset, it is quite easy to prove: the highlighted segments in the following figure cannot match a green pack, because there is no room (an already placed $-,-,-$ triangle is taking the room). They cannot match a $+,+,+$ triangle because of \Cref{prop:Tp3}. So this must be a yellow pack. The other $+$ side of the yellow pack must match the only pack with a $-$, namely the $-,-,-$ triangle.

\image{Let us focus on the highlighted sides.}{}{
\includegraphics[scale=0.5]{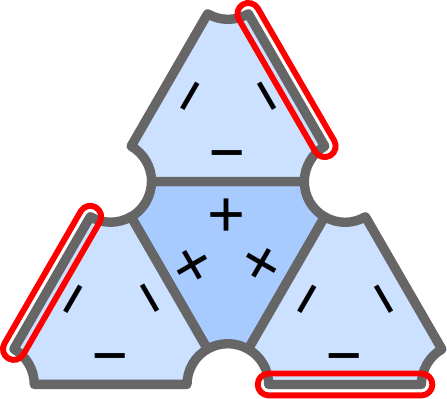}
}

Similarly as for the $\hT'2$, from the $\hT'3$ environment, we get the situations depicted below, where red dots also mean an impossibility, and the explanations are below the figure.

\nopagebreak

\image{}{}{
\makebox[\textwidth][c]{\includegraphics[scale=0.25]{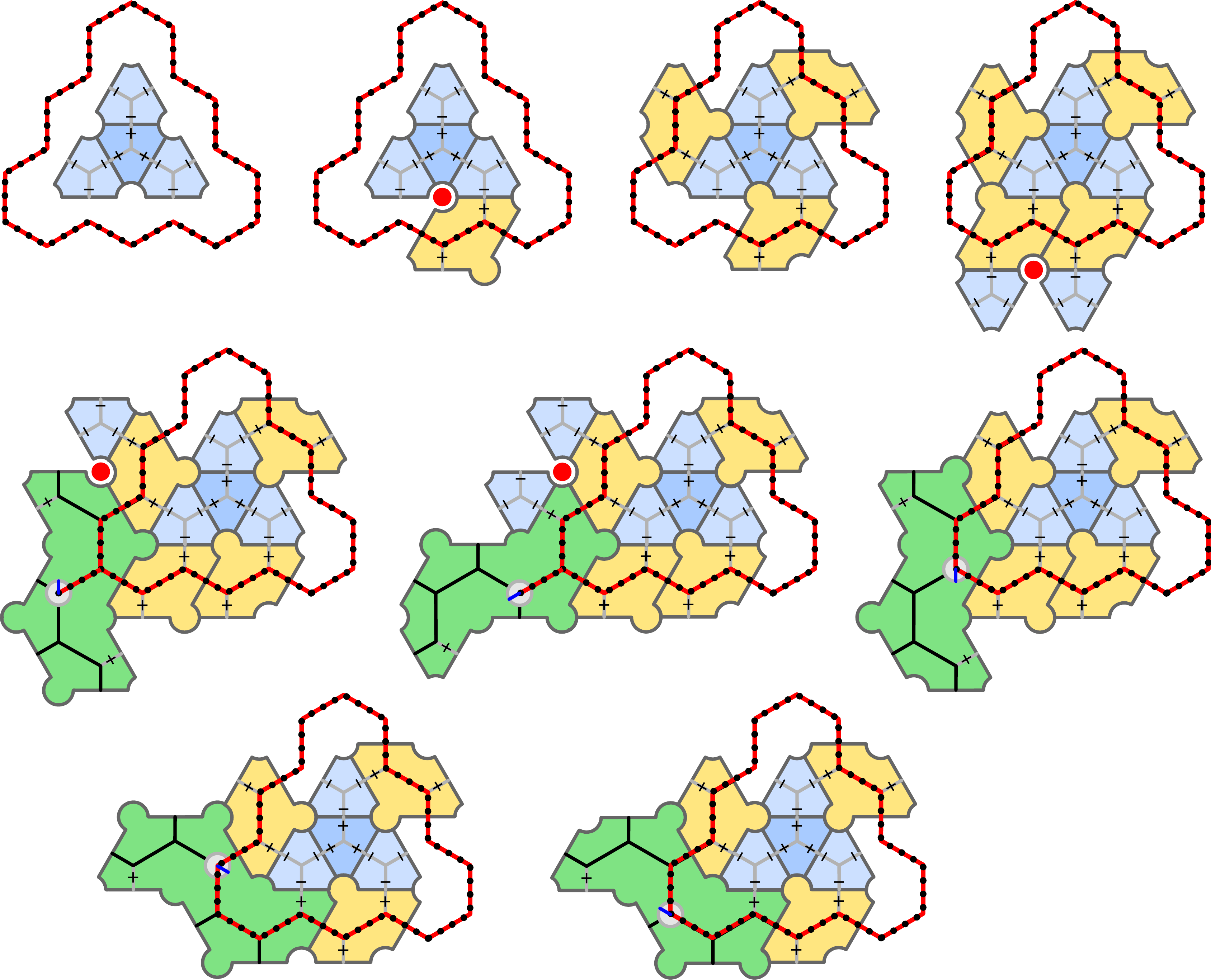}}
}

The image is the initial $\hT'3$. We saw the right `$-$' segment of the bottom part has to be in contact with one of the two $+$ of a yellow one.
The second image shows that one of these yellow $+$ does not leave a possibility to complete the tiling. So we must be in the situation of the third image.
The fourth image and the second row explore (im)possibilities in the case the left `$-$' segment of the bottom part of the initial configuration is in contact with a yellow $+$: only one case is not ruled out.
The last row explores the case where it is in contact with a green $+$. There, no case is ruled out.
\end{proof}

\subsection{Recognizing the hexagons of \texorpdfstring{\cite{chiral}}{[\ref{labelnumber-chiral}]}}\label{ss:recog-pfs}

The statement below is a copy of \Cref{lem:gr-pack-env-2}.

\begin{lemma*}
Every green pack is environed as follows:

\nopagebreak

\image{}{fig:gpe-copy}{
\includegraphics[scale=0.33]{green-pack-env-2.pdf}
}
\end{lemma*}
\begin{proof}
Otherwise we would have the following set of deductions, starting from \Cref{fig:green-pack-env}.

\nopagebreak

\image{}{}{
\includegraphics[scale=0.28]{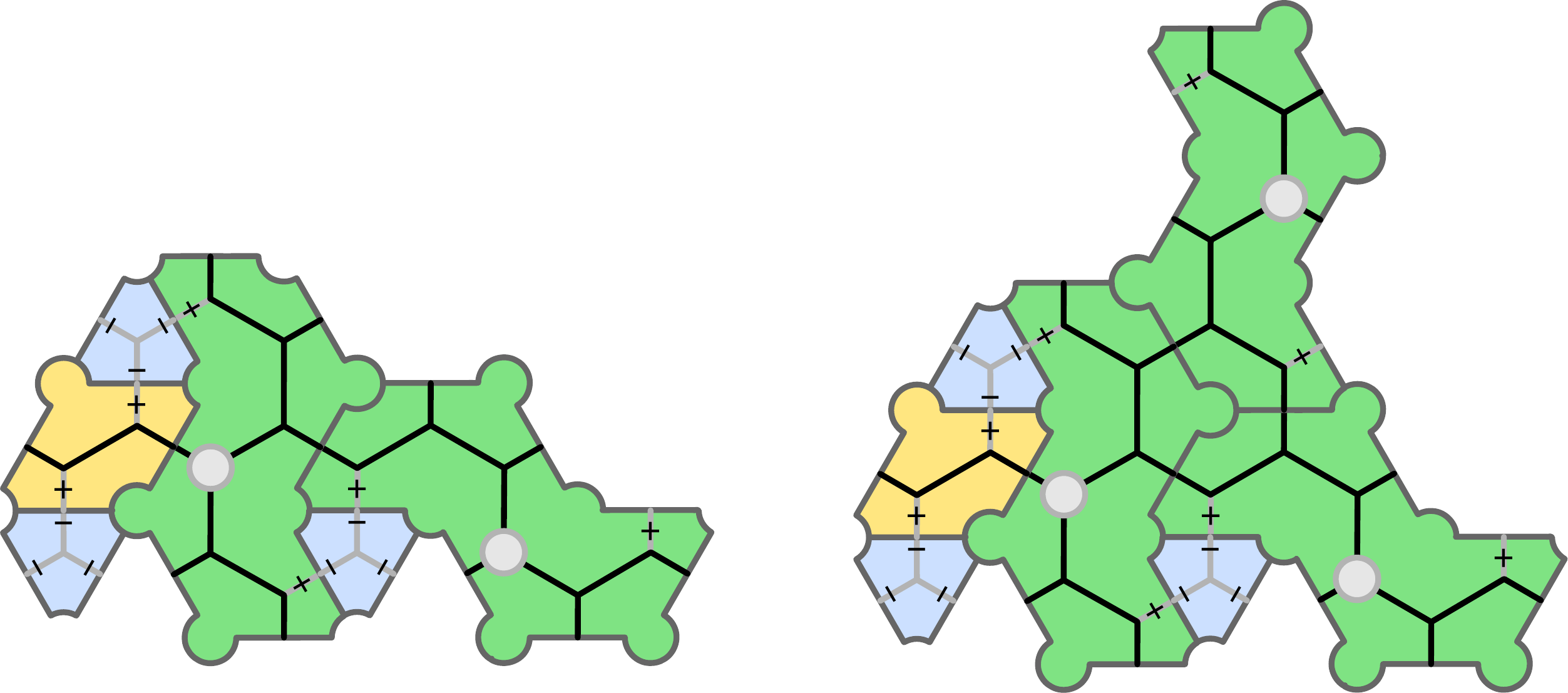}
}

Leading to a configuration of green tiles that was excluded in \Cref{fig:exclu}.
\end{proof}

The following statement is a copy of \Cref{lem:impo-arr-1}.

\nopagebreak

\begin{lemma*}
The following arrangement cannot occur in a whole plane tiling:

\nopagebreak

\image{}{}{
\includegraphics[scale=0.45]{impo-b.pdf}
}
\end{lemma*}
\begin{proof}
If the arrangement appears then there must be packs as follows:

\nopagebreak

\image{}{}{
\includegraphics[scale=0.35]{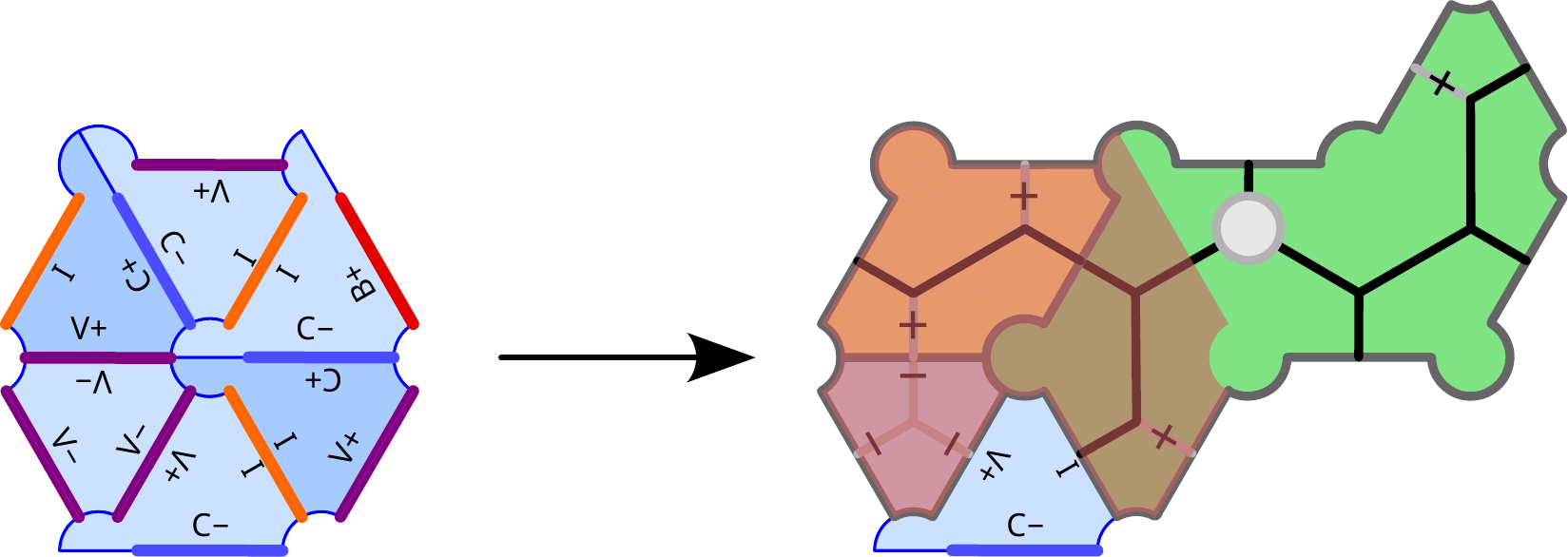}
}

The bottom triangle can only be filled by the tips highlighted in dark green in the figure below (the ones in red have the wrong triangle type).

\image{}{}{
\includegraphics[scale=0.35]{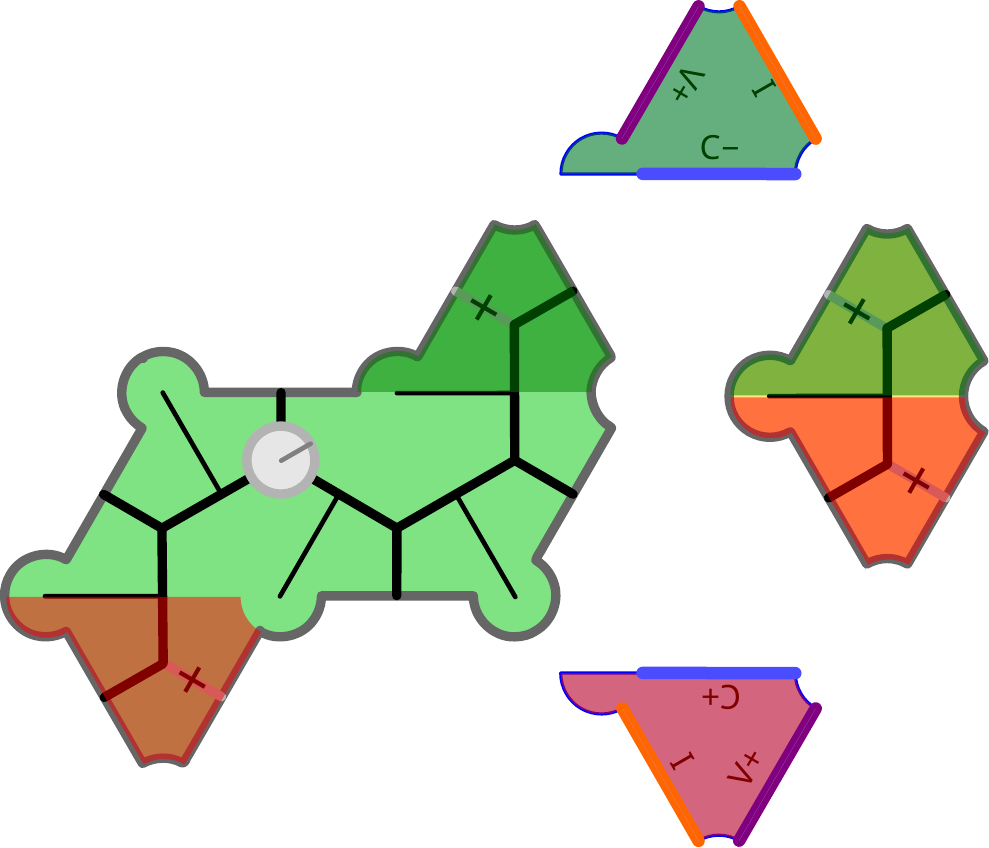}
}

If the bottom is filled by a yellow piece then this leads to the following impossibility:

\image{}{}{
\includegraphics[scale=0.3]{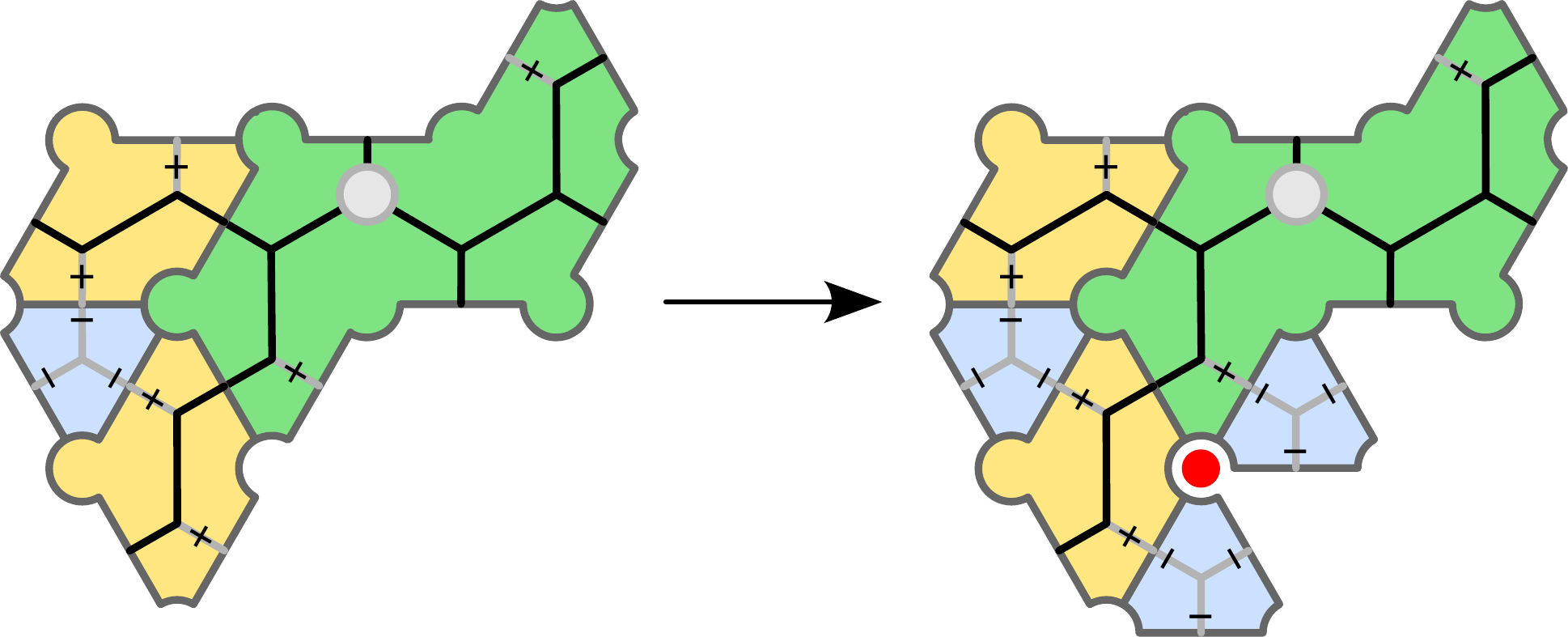}
}

So it is filled by a green pack. We then have the following set of deductions:

\nopagebreak

\image{For the first frame we used \Cref{fig:gpe-copy}. For the second and third one the visual rules of \Cref{ss:vis}.}{}{
\makebox[\textwidth][c]{\includegraphics[scale=0.21]{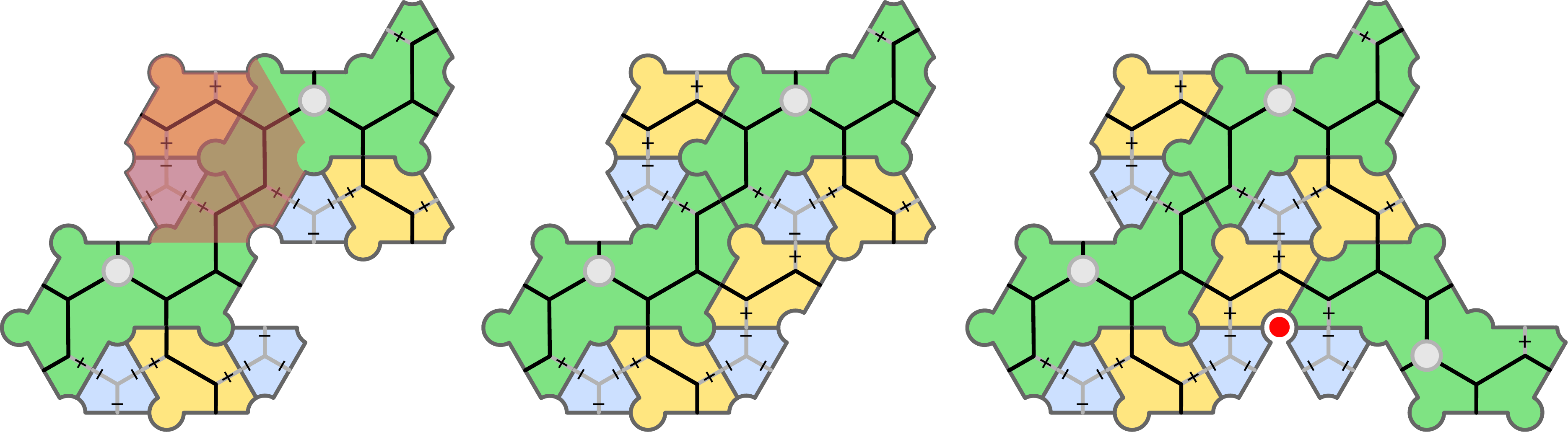}}
}

Leading to a non-fillable disc in red.
\end{proof}

The following statement is a copy of \Cref{lem:impo-arr-2}.

\nopagebreak

\begin{lemma*}
The following arrangement cannot occur in a whole plane tiling:

\nopagebreak

\image{}{}{
\includegraphics[scale=0.45]{impo-c.pdf}
}
\end{lemma*}
\begin{proof}
Otherwise we would have the following set of deductions where we start by an alternative:

\nopagebreak

\image{For the arrow with an $E$ indicates we used \Cref{fig:gpe-copy}}{}{
\includegraphics[scale=0.25]{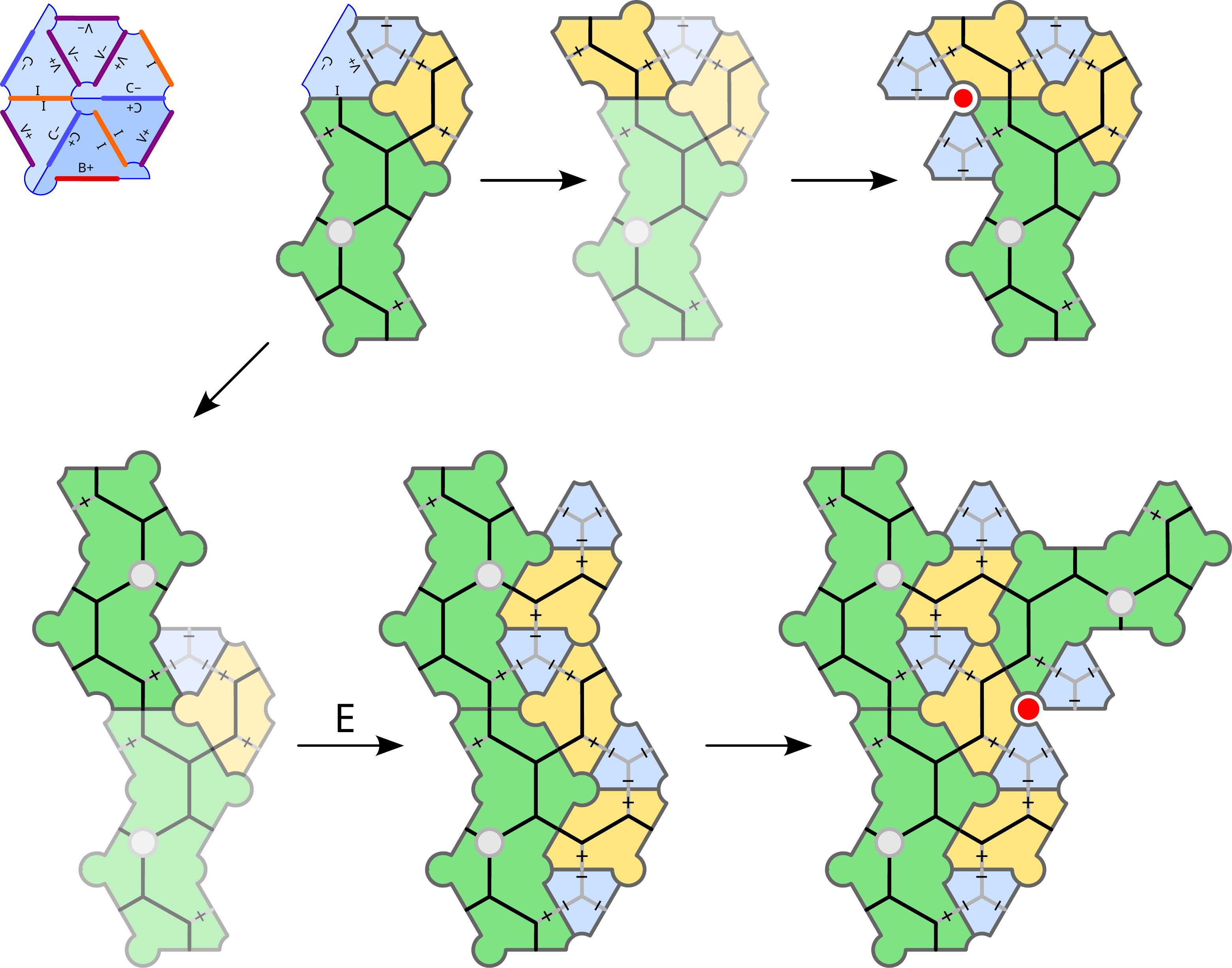}
}

Each alternative lead to to a disc indicated in red and that is impossible to fill.
\end{proof}

\section{Double level hierarchy on the honeycomb partition does not require dots}\label{sec:double}

In this section we come back to the partition of the honeycomb into triangular clusters, as on \Cref{fig:bo}. Recall that hexagons in the honeycomb correspond to blue hexes of the decoration graph associated to a Spectre tiling.
We explained in \Cref{sec:analysis} that if we add the position of the dots corresponding to contracted yellow hexes to the picture of the honeycomb partition, it has a unique hierarchical structure.

\medskip
 
Can we do without the dots?

\medskip

Deducing the position of the dots from the sole partition is not as trivial as one would think. In particular it may require to look at arbitrary large patches.

Let us be more precise. In the partition, the clusters can be thought of as related as are triangles on a regular triangular lattice.
There is one dot corresponding to every vertex of this lattice, i.e.\ to every circular group of 6 adjacent clusters. 
There are no other dots.
We observe only 3 kinds of cluster groups on the honeycomb, up to a rotation of 1/6:

\image{}{fig:ABC}{
\includegraphics[scale=0.15]{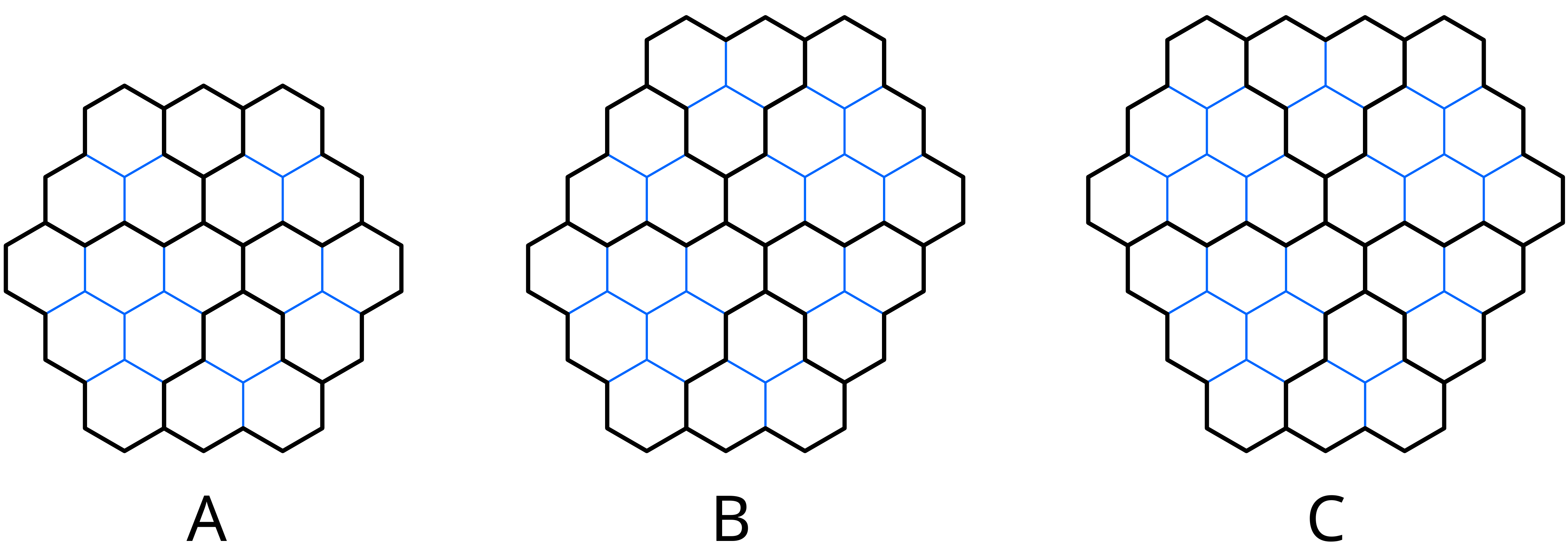}
}

In group $C$ the corresponding dot can only be placed on one vertex, but in the other two, they have two possibilities:

\image{}{fig:ABC-dots}{
\includegraphics[scale=0.15]{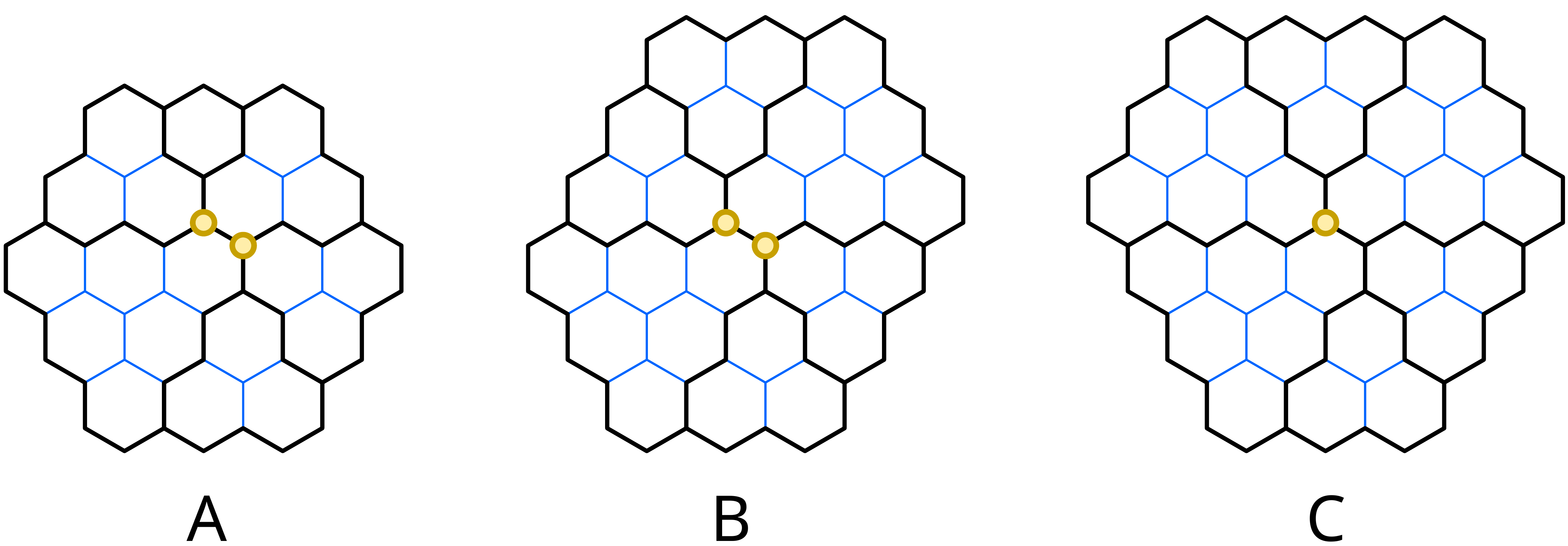}
}

These claims can be deduced from the study of \Cref{sec:analysis}: for instance from the list of possible vertex environments in  \Cref{lem:vlist-2} and the cc's that the triangle pieces represent, see \Cref{fig:labeled-cc-2}.

Looking at larger groups of clusters around a group of 6 adds constraints to the position of the dot.
Unfortunately, determining the placement of a dot on a group A or B may require to look at arbitrary larger groups depending on the situation: this will follow from the end of this section.

\medskip

Remarkably, it is possible to describe the \emph{second iterate} of the hierarchy/substitution uniquely in terms of the partition, i.e.\ without the dots.

To explain this, let us first note that in the partitioned honeycomb with dots, from \Cref{fig:N1} in \Cref{prop:around-tri}, we get that the environment of a $\hT1$ cc is as follows, up to a rotation of a multiple of $1/6$:

\image{On the right, we only retained the clusters: the figure then gets an order 3 rotation symmetry}{fig:T1-env-h}{
\includegraphics[scale=0.25]{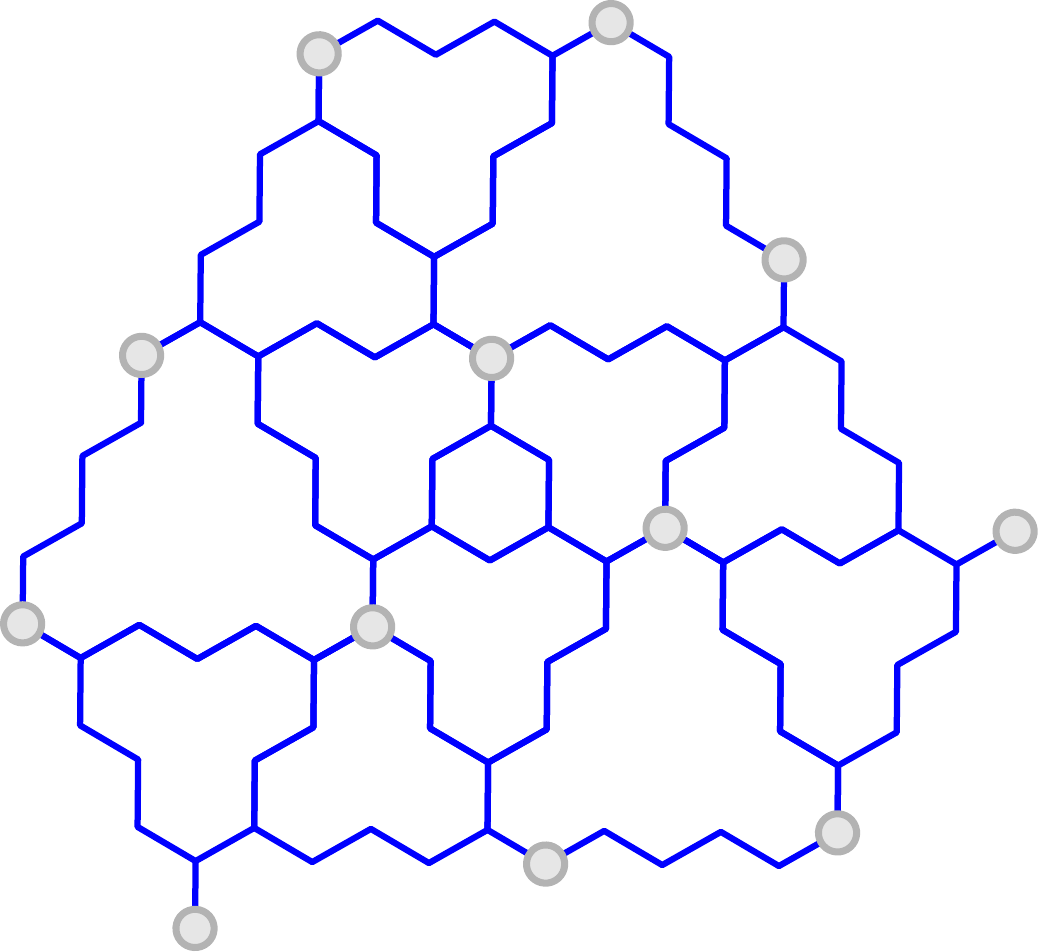}
\quad
\includegraphics[scale=0.25]{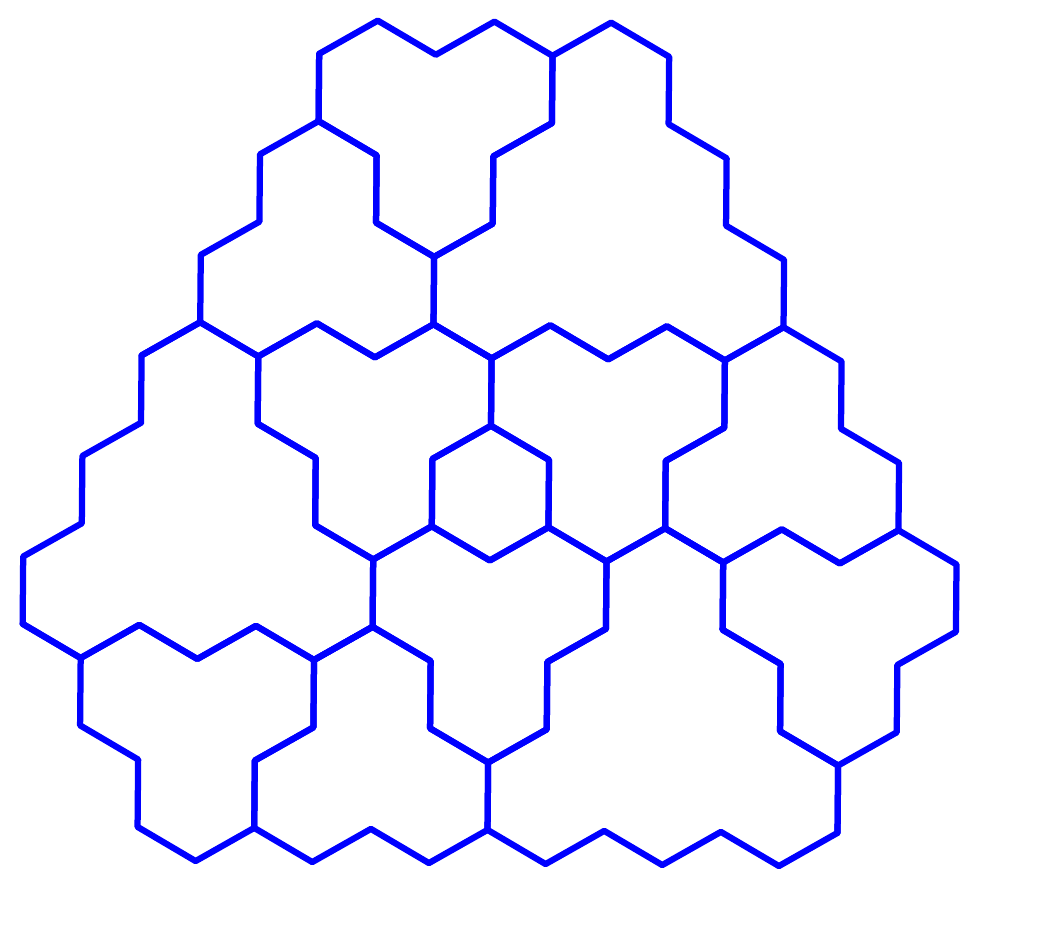}
}

In \Cref{fig:sr-10b,fig:sr-10c} the $\hT1$ are the hexagons containing the purple dots. Let us place the cluster group of the right figure above for each of the 8 occurring (i.e.\ disregarding the \nth{4}) situations of these figures. For each purple dot, it is possible to know the orientation of the group (there are only 2 possibilities because of the order 3 symmetry) from the orientation of the dot before the substitution (indicated in \Cref{fig:sr,fig:sr3c}): this is because it is either known or has two known possibilities that differ from $1/3$.

\image{We placed neighbourhoods at purple dots. Where they overlap, the shade is darker. We added yellow and light blue clusters coming from \Cref{fig:sr-10b,fig:sr-10c}. Then we placed gray clusters in the centre as the only possible ones filling the gaps and respecting \Cref{fig:ABC}.}{fig:su-h}{
\includegraphics[scale=0.11]{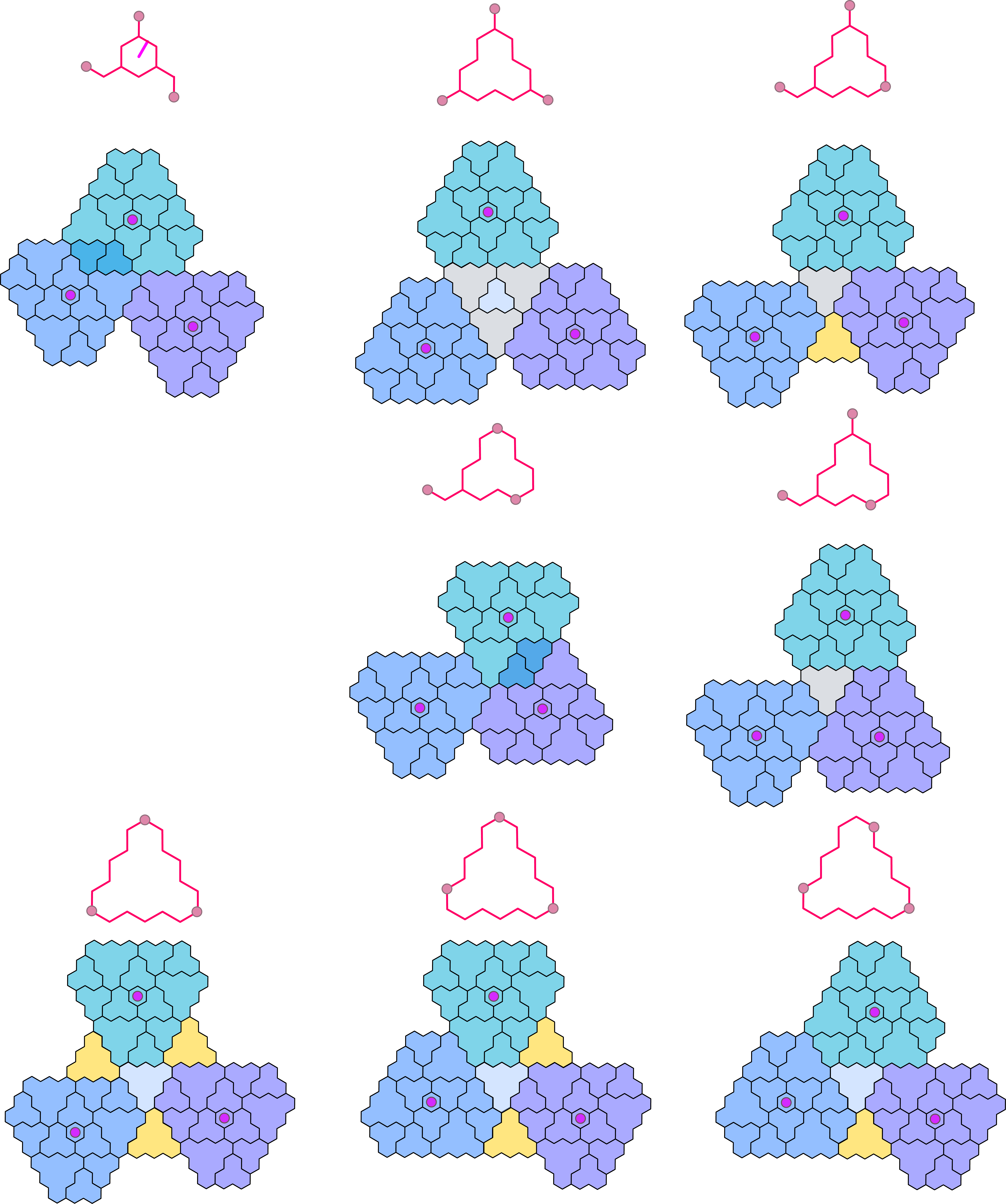}
}

Now the purple dots correspond to hexes two levels up (call them \term{green hexes}). The $\iV$-type interfaces of the cc's one level up (in red on the last figure) link dots whose corresponding ``green'' hex two levels up is in the same green cluster, while $\iI$-type interfaces link dots whose corresponding green hexes are in adjacent green clusters: see \Cref{def:V-clu,def:I-adj,fig:8-6,fig:pack} and the whole analysis we made in of the hierarchy (but apply all this one level up).
\Cref{fig:su-h} implies it is possible to deduce the blue honeycomb partition from the green honeycomb partition two levels above. More precisely the green honeycomb clusters will provide groups of 1, 3 or 6 purple dots in a triangle configuration on the blue honeycomb, with the corresponding blue hexes being separated by a determined vector (associated to $\iV$-type interfaces): using complex numbers, if two adjacent hexagons of the blue honeycomb are separated by vector $1$, and thus the whole blue honeycomb by vectors in $\Lambda = \Z+\rho\Z$ where $\rho = e^{i 2\pi/6}$ is a primitive cube root of $1$, then the vector is in $(8+\rho)\times\Lambda$. The way the triangle points is also known.
Adjacent green honeycomb clusters are separated by a vector that is known too, $7$ times a power of $\rho$ which is determined by which side of the two green clusters touch.
Note also that the partition of the green honeycomb into green clusters follows from the vectors between nearby $\hT1$ in the blue honeycomb: the $\iV$-type and $\iI$-type interfaces are characterized by this vector being respectively in $(8+\rho)\times \rho^\Z$ and $7\times \rho^\Z$. The other interfaces have incompatible vectors : $\iA$ gives $(5+\rho^{-1})\times \rho^\Z$, $\iB$ gives $(6+2\times\rho^{-1})\times \rho^\Z$ and $\iC$ gives $(8+\rho^{-1})\times \rho^\Z$.

\medskip

Finally we explain how to get dot ambiguity from a given partition.
Start from the environment of a $\hT1$ as on \Cref{fig:T1-env-h}.
The partition has order 3 rotation symmetry, while the dots cannot.
Perform the order two substitution for the partition an arbitrary number of times.
Note that a $\hT1$ lies a the centre again, and that the obtained arrangement of clusters must keep the order three symmetry, while as we noted, the dots cannot have this symmetry. There will be differences in the dots near the outer boundary of the big arrangement, and knowing one may determine the orientation at the centre.
This way we can already make an whole plane tiling for which the blue honeycomb partition has order three symmetry around an $\hT1$, whose orientation is thus not determined in advanced.
But better: knowing that there are non-symmetric situations with a $\hT1$ whose orientation is forced in \Cref{fig:su-h}, it is possible to construct from any whole plane tiling, using an even number of substitution, a whole plane partition that has some $\hT1$ whose orientation is only determined if one looks at a big enough situation.

We end with the following picture, showing, for some whole plane tiling, occurrences of the shape in \Cref{fig:T1-env-h}, and the way the same shape may appear two levels up.

\image{}{}{
\includegraphics[width=\linewidth]{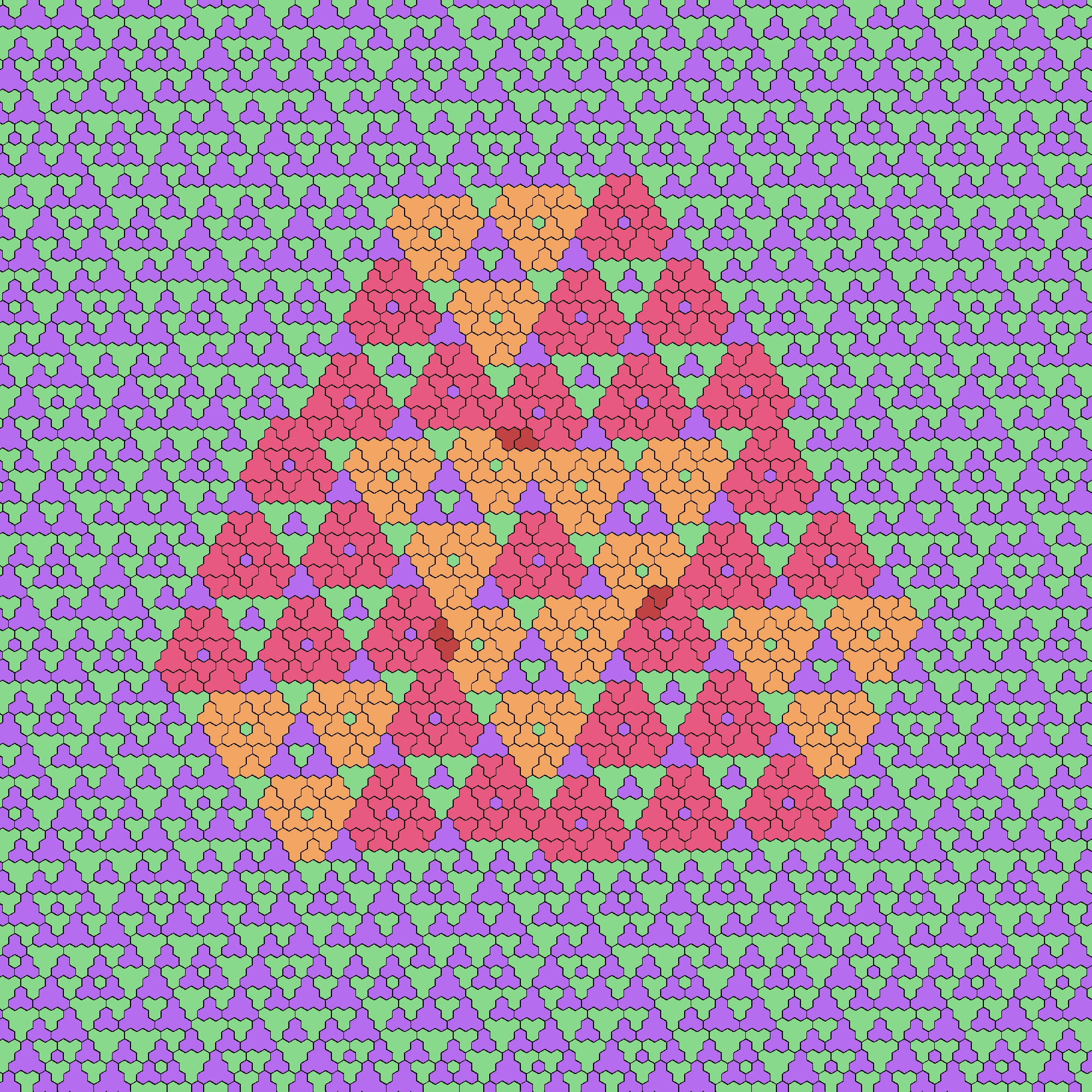}
}

\printbibliography

\end{document}